%% file: pAClassification.tex
\documentclass[12pt,letterpaper]{amsart}

\usepackage{amscd}
\usepackage{amsmath,amsfonts,amsthm,epsfig,latexsym,graphicx,amssymb}
\usepackage[english]{babel}
\usepackage{comment}
\usepackage{epsfig}
\usepackage{tikz, graphics} 
\usepackage{graphics,color}
\usepackage{psfrag}
\usepackage{hyperref}
\hypersetup{colorlinks=true, citecolor=green, linkcolor=blue, hypertexnames=false}

\newtheorem{coro}{Corollary}
\newtheorem{defi}{Definition}
\newtheorem{prop}{Proposition}

\newtheorem{theo}{Theorem}
\newtheorem{lemm}{Lemma}

\newtheorem{prob}{Problem}
\newtheorem{rema}{Remark}

\newtheorem{conv}{Convention}

\newtheorem*{coro*}{Corollary}
\newtheorem*{defi*}{Definition}
\newtheorem*{prop*}{Proposition}
\newtheorem*{conj*}{Conjecture}
\newtheorem*{theo*}{Theorem}
\newtheorem*{lemm*}{Lemma}
\newtheorem*{ques*}{Question}
\newtheorem*{exer*}{Exercise}
\newtheorem*{exem*}{Example}
\newtheorem*{prob*}{Problem}
\newtheorem*{rema*}{Remark}
\newtheorem*{conv*}{Convention}
\newtheorem*{clai*}{Claim}
\newtheorem*{affi*}{Affirmation}
\newtheorem*{claim*}{Claim}
\newtheorem*{exes*}{Exercise}
\newtheorem*{pbm*}{Problem}

\def\AA{{\mathbb A}}

\def\II{{\mathbb I}}  

 \def\NN{{\mathbb N}} 

 \def\RR{{\mathbb R}} 
\def\TT{{\mathbb T}}

 \def\ZZ{{\mathbb Z}}

  \def\cG{{\mathcal G}}  \def\cS{{\mathcal S}} 
  \def\cH{{\mathcal H}}  \def\cT{{\mathcal T}} 
  \def\cI{{\mathcal I}} \def\cO{{\mathcal O}} \def\cU{{\mathcal U}}
  \def\cJ{{\mathcal J}} \def\cP{{\mathcal P}} \def\cV{{\mathcal V}}
    
\def\cF{{\mathcal F}}   \def\cR{{\mathcal R}}

\title{	A Classification of Pseudo-Anosov Homeomorphisms via Geometric Markov Partitions.}

\begin{document}
	
\title[A classification of pseudo-Anosov homeomorphisms]{A classification of pseudo-Anosov homeomorphisms via geometric Markov partitions.}
\author{Inti Cruz Diaz}
\email{incruzd@gmail.com}
%\subjclass{37D20, 37D05, 37C10, 57K30, 57R30}

 \begin{abstract}
We continue with the ideas of Bonatti-Langevin-Jeandenans towards a constructive and algorithmic classification of pseudo-Anosov homeomorphisms (possibly with spines), up to topological conjugacy. We begin by indicating how to assign to every pseudo-Anosov homeomorphism an abstract geometric type through a Markov partition whose rectangles have been endowed with a vertical direction; these are known as geometric Markov partitions. Such assignment is not unique, as it depends on the specific geometric Markov partition, and not every abstract geometric type is realized by a pseudo-Anosov homeomorphism. This poses interesting problems related to the algorithmic computability of the classification. The article contains three main results: $i)$ the geometric type is a total invariant of conjugacy, $(ii)$ a pictorial and combinatorial criterion for determining when a given geometric type is realizable by a pseudo-Anosov homeomorphism is provided, and $(iii)$ we describe an algorithmic procedure for determining when two geometric types correspond to topologically conjugate pseudo-Anosov homeomorphisms.
\end{abstract}

 \maketitle

\tableofcontents

\input{Intro/Intro}

\input{Preliminares/preliminares}

\input{TotalInvariant/Typeconjugacy}

\input{Realization/Realization}

\input{Equivalence/Equivalence}

\bibliographystyle{alpha}
\bibliography{BibPseudoA}

\end{document}

%% file: Intro/Intro.tex
\section{What do we mean by a classification?}

%%%%%%%% La clasificacion de homeomorfismos de superficies
\subsection{First conventions and definitions.} 
A \emph{surface} is a smooth manifold of dimension $2$. Let $S$ be a closed and orientable surface. $\text{Hom}(S)$ denotes the set of homeomorphisms from $S$ to itself, and $\text{Hom}+(S)$ those that preserve the orientation. Two homeomorphisms $f, g \in \text{Hom}(S)$ are \emph{topologically conjugated} if there exists $h \in \text{Hom}(S)$ that conjugates them such that $g = h^{-1} \circ f \circ h$. Topological conjugacy is an equivalence relation on $\text{Hom}(S)$. We denote $\text{Hom}(S)/\sim{\textbf{Top}}$ as the set of conjugacy classes. For every $f \in \text{Hom}(S)$, the conjugacy class $[f]$ is the set of homeomorphisms of $S$ that are topologically conjugated to $f$. From the viewpoint of topological dynamical systems, two topologically conjugated homeomorphisms share the same topological dynamics, such as phenomena like the existence and density of periodic points, positive topological entropy, transitivity, and topological mixing. It is therefore an interesting yet too ambitious problem to assign to each conjugacy class $[f]$ a certain abstract object that allows us to capture all these topological properties while also deciding whether two homeomorphisms are conjugated or not. In this paper, we take a step forward in the classification of a particular class of surface homeomorphisms up to topological conjugacy, namely, \emph{pseudo-Anosov homeomorphisms}. But before introducing them and discussing the problems we must address, we would like to clarify the type of classification we have in mind with a simple example.

\subsection{Two examples of a perfect classification.}

The classification of finite-dimensional linear vector spaces up linear isomorphism trough its dimension is the archetype of classification that are looking for. It  begins with a vector space $V$, and after constructing a base for $V$, we can associate to $V$ the cardinality of that base. It turns out that all bases of $V$ have the same cardinality, and $\text{Dim}(V) \in \mathbb{N}$ is an abstract object that classifies $V$ as a vector space. That is to say, if two spaces $V$ and $V'$ have the same dimension, there exists a linear isomorphism that maps one onto the other; moreover, the isomorphism is well-defined by sending a base of $V$ to a basis of $V'$. It also has the advantage that given a number $n$, there always exists a vector space of that dimension, and we can even provide a canonical model, the Euclidean space $\mathbb{R}^n$. Another interesting example of classification is that of closed and orientable surfaces through their genus. Given a surface in the wild, if one can obtain a triangulation of $S$ we can determine its genus and draw a picture of it, as two closed and orientable surfaces have the same genus if and only if they are homeomorphic. Finally, there are surfaces of every genus $n$, and they are either a sphere or obtained by the connected sum of $n$  $2$-dimensional tori $\TT^n$. Let us synthesize the important and common aspects of these classifications.

\begin{itemize}
	\item  \textbf{Assignation of a combinatorial object}: In both cases, this is done trough another intermediate mathematical object that permit some explicit computations:  either a basis $A$ for $V$ or a triangulation $\Delta$ of $S$. Then, to the pair $(V,A)$ we assign the cardinality of $A$  and for $(S,\Delta)$, can applied the formula of the  Euler characteristic of $\Delta$. In both cases if we know the intermediate object we can compute the combinatorial object, that in this situation it is a natural number $\NN$.
	
	\item \textbf{The combinatorial object is a total invariant}:
	The association shall to distinguish between vector spaces that are not isomorphic(as linear vector spaces) and surfaces that are not homeomorphic. It is a well-known fact that two finite dimensional vector spaces are isomorphic if and only if their dimensions are the same. Similarly, in the classification of closed surfaces, it is stated that two closed and orientable surfaces are homeomorphic if their genera are the same.
	
	\item \textbf{Realization of the invariants}: Whenever I take an $n\in \mathbb{N}$, I can find a vector space or a surface of genus $n$? Yes, all the natural numbers are realized as the dimension and genus of a vector space and a surface, respectively.
	
	\item \textbf{Equivalence of the invariants}: Can a vector space have bases with different cardinalities, or can a surface have two different genus if we compute the Euler characteristic of different triangulations? In this case, the answer is no. The bases of a vector space all have the same cardinality, therefore, its dimension is uniquely defined. Similarly, the Euler characteristic of any triangulation of the same surface coincides, so its genus is also unique.
\end{itemize}

Our intention is to emulate this type of classification for the conjugacy classes of surface homeomorphisms. However, we must clarify that the number of equivalence classes is uncountable and  there is no hope of obtaining a discrete invariant for all of them as they cannot be parameterized by $\mathbb{N}$. On the other hand, there is a class of surface homeomorphisms that has received much attention since the time of the famous Handel-Thurston classification for orientation-preserving surface homeomorphisms: the so-called pseudo-Anosov homeomorphisms. Let us introduce these pseudo-Anosov homeomorphisms and their geometric Markov partitions, which will be the equivalent of a basis or a triangulation when we pose the roots of our classification. Through them, we will be able to associate a pseudo-Anosov homeomorphism with a combinatorial object in a countable set; it will never be simply a natural number, but rather an abstract geometric type. We will gradually introduce these concepts to formulate the problems we will address in this document with precision.

%%%%%%Homeomorfismos pA

\subsection{Surface homeomorphisms up Isotopy.}

 An orientation-preserving homeomorphism in $S$ is \emph{pseudo-Anosov } (or \emph{generalized pseudo-Anosov} as defined in \cite[Exposition 13.1]{fathi2021thurston}) if it preserves a pair of \emph{transverse measured foliations}, $(\mathcal{F}^s, \mu^s)$ and $(\mathcal{F}^u, \mu^u)$, the stable and unstable foliation of $f$ respectively. Each singularity must be a  $k$-pronged and we are allowing \emph{spines} ($1$-prongs). Simultaneously, $f$ shall expands the $\mu^s$-measure of any compact \emph{unstable arc} $J \subset F^u \in \mathcal{F}^u$, by a \emph{stretch factor} $\lambda > 1$, i.e., $\mu^s(f(J)) = \Lambda \mu^s(J)$, and similarly,  if $I \subset F^s$ is a compact stable interval contained in a  leaf $F^s \in \mathcal{F}^s$, $\mu^s(f(I)) = \Lambda^{-1} \mu^s(I)$ and we said $f$ contracts along the stable foliation. In some sense, pseudo-Anosov homeomorphisms exhibit \emph{uniformly hyperbolic} dynamics outside their singularities and we must to exploit this analogy.  Nowadays there are more general constructions of surface maps that generalize one or another property of  pseudo-Anosov homeomorphisms (\cite{Hall2004unimodal}, \cite{Andre2005extensions}, \cite{boyland2023unimodal}). Therefore, we prefer to simply call them pseudo-Anosov reserve the term "generalized" for these more modern constructions. We usually just said that $f$ is \textbf{p-A} homeomorphism to abbreviate our speech.

According to the Nielsen-Thurston classification theorem \cite{thurston1988geometry}, any surface homeomorphism is \emph{isotopic} to one that is either \emph{periodic}, pseudo-Anosov, or \emph{reducible}. In the reducible case, the surface can be divided into finite pieces along a family of (not null-homotopic) closed curves, and in each piece, the homeomorphism is isotopic to one of the first two cases. The mapping class group of the surface $S$ is denoted by $\mathcal{MCG}(S)$ and consists of the isotopy classes of $\text{Hom}_+(S)$. Since Thurston's foundational paper \cite{thurston1988geometry}, the study of $\mathcal{MCG}(S)$ has been the focus of a lot of research (a comprehensive introduction can be found in \cite{farb2011primer}).

 Among these classes of isotopy, the pseudo-Anosov classes are particularly interesting, and it is natural to try to classify them. There is some hope, as two pseudo-Anosov homeomorphisms are topologically conjugated if they are isotopic  \cite[Theorem 12.5]{fathi2021thurston} and the number of elements in the mapping class is countable so it is the number of conjugacy classes that contain a pseudo-Anosov homeomorphism. 

%%%%%%%%%%%%%%%%%% Revision de clasificaciones

In \cite{bestvina1995train}, Bestvina and Handel provide an algorithmic proof of the Nielsen-Thurston classification theorem and subsequently an algorithmic classification of orientation-preserving homeomorphisms up to isotopy. The combinatorial objects used to achieve such classification are \emph{train tracks} and the incidence matrix associated with them. Similarly, using a variation of train tracks and the intimate relation between the mapping class group of the $n$-punctured disk and the braid groups, Los has provided a finite algorithm for determining whether an isotopy class (of the punctured disk) contains a pseudo-Anosov element or not \cite{los1993pseudo}. Roughly speaking, a train track is the isotopy class of a smoothly embedded graph in the surface that satisfies certain switch conditions, and the action of the homeomorphism on the graph can be read from a certain \emph{incidence matrix}. This classification, however, does not necessarily provide an answer to the problem of determining if two pseudo-Anosov homeomorphisms are topologically conjugated. A step forward in that direction with a similar approach using cell decomposition was taken by Mosher in \cite{mosher1983pseudo} and \cite{mosher1986classification} to provide a classification of certain pseudo-Anosov isotopy classes defined over a closed surface $S$ with a set $P$ of marked points $\Sigma:(S,P)$ up to conjugacy by elements in $\mathcal{MCG}(S,P)$. He has considered such isotopy classes that have a \textbf{p-A} that fix each of its singularities and each of the separatrices of its singularities. Since two isotopic pseudo-Anosov diffeomorphisms are conjugated \cite[Theorem 12.5]{fathi2021thurston}, Mosher's classification is, in particular, a partial classification of pseudo-Anosov diffeomorphisms defined over $S$ by orientation-preserving homeomorphisms that fix the set of marked points $P$. To achieve such classification, Mosher associated to each pseudo-Anosov isotopy class a family of integer numbers that he called \emph{pseudo-Anosov invariants}. These invariants permit the identification of whether or not two pseudo-Anosov isotopy classes are conjugated, up certain powers.

We could say that, in general terms, train tracks and their incidence matrices are suitable tools for studying homeomorphisms up isotopy, and it's undeniable that Thurston's school, which emphasizes these objects, has provided many deep results that now greatly aid our understanding of surface homeomorphism up isotopic deformations. However, in Thurston's original works, there's an emphasis on dynamics, and in that sense, topological conjugation between homeomorphisms is the natural equivalence relation, rather than isotopy. In this work, we will use completely different ideas, stemming from the hyperbolic theory of dynamical systems, rooted in Anosov's work on the geodesic flow for surfaces of negative curvature and Smale's Horseshoe. Since those days, much water has flowed under the bridge, and we would like to revisit a classification of $C^1$ structurally stable diffeomorphisms on surfaces that may not be widely known but is the root, inspiration, and foundation of this present article.

%%%%%%%%%% Resultado Bonatti-Lngeving-Jeandenanas y Beguin para Smale diffeos
\subsection{Smale surface diffeomorphisms up topological conjugacy.}

A diffeomorphism $\phi:S \to S$ is $C^1$ \emph{structurally stable} if there exists an open set $U$ in the $C^1$ topology of Diff$(S)$ such that any other $C^1$ diffeomorphism $\phi'\in U$ is topologically conjugate to $f$ trough an homeomorphism isotopic to the identity. We'll call \emph{Smale diffeomorphism} to any $C^1$-structurally stable surface diffeomorphism as they were named in \cite{bonatti1998diffeomorphismes}. In such book  Bonatti and Langevin begin a systematic study of the invariant neighborhoods of their \emph{saddle-type saturated sets}, i.e., non-trivial \emph{uniformly hyperbolic sets} without attractors nor repellers their intention is provide a classification up topological conjugacy not just in the hyperbolic set but in a whole neighborhood. After the works of Béguin (\cite{beguin1999champs}, \cite{beguin2002classification}, \cite{beguin2004smale}), they completed their classification where the main tools were Markov partitions and the classifying object and abstract geometric type.  geometric Markov partitions will play the role of train tracks, and their corresponding geometric types, the role of the incidence matrix in the  Bestvina-Handel classification \cite{bestvina1995train}  classification of the mapping classes, and the geometric type can be thought of as a generalization of the incidence matrix of a Markov partition with enriched information related of the order and orientations of the sub-rectangle in the partition. We'll highlight the ideas and results of the Bonatti-Langevin, Béguin classification for Smale surface diffeomorphisms that are necessary to pose the classification problems as formulated by them in \cite{bonatti1998diffeomorphismes}, but adapted to the case of pseudo-Anosov homeomorphisms.

%%% El tipo geometrico empaquetamiento de la informacion
%%% Tipos geometricos abstractos

Let $\phi:S\to S$ a Smale surface diffeomorphism and  assume it is not Anosov and let $K\subset S$ a non-trivial saddle-type \emph{basic piece} of $f$. A  Markov partition of $(f,K)$ is finite family of disjoint rectangles rectangles $\textbf{R}=(R_i)_{i\in I}$ which cover $K$ and such that their images satisfy the following Markov-type property: If $R_1, R_2 \in \textbf{R}$, then $f(R_i) \cap R_j$ is a sub-vertical rectangle of $R_j$, and $f^{-1}(R_j) \cap R_i$ is a sub-horizontal rectangle of $R_j$.  We call these particular sub-rectangles, respectively, the \emph{vertical} and \emph{horizontal sub-rectangles} of the Markov partition $\textbf{R}$. Is a well know result that $(f,K)$  has a Markov partition $\textbf{R}$ (\cite{BowenAnosovbook}). The stable and unstable lamination of $K$ are denote by $W^s(\phi, K)$ and $W^s(\phi, K)$, the family of intervals obtained as the connected components of $W^s(\phi, K)\cap R$ and $W^s(\phi, K) \cap R$.   Let's call \emph{vertical} to the unstable direction and \emph{horizontal} the stable direction. Then, since the surface is oriented the rectangle $R$ have an orientation and a orientation of the vertical  gives an unique orientation of the horizontals coherent with the orientation of $R$. A Markov partition whose rectangles have been endowed with a vertical direction and a labeling is called  \emph{geometric Markov partition} of $\phi$ and we denote them as  $(\phi,K, \cR={R_i}_{i=1}^n)$ and lets fix this information in the rest of this subsection. 

The geometric type of the geometric Markov partition $(\phi,K, \mathcal{R}=\{R_i\}_{i=1}^n)$ was introduced for the first time in \cite[Definition 5.2.1]{bonatti1998diffeomorphismes}, it is denoted by $T(\phi,K, \mathcal{R}=\{R_i\}_{i=1}^n) := (n,\{h_i,v_i\}_{i=1}^n,\Phi_T, \epsilon_T)$, and essentially comprises: the number $n$ of rectangles in $\mathcal{R}$, the numbers $h_i$ and $v_i$ of horizontal and vertical sub-rectangles of the Markov partition $\mathcal{R}$ contained in $R_i$ (See \ref{Defi: Markov partition} for the  \textbf{p-A} case), respectively. After labeling them with respect to the horizontal and vertical direction of $R_i$ (See Definition \ref{Defi: Label and geometrization of subrectangles}), $\Phi_T$ is a bijection between the set of all ordered horizontal sub-rectangles of $\mathcal{R}$ and its vertical sub-rectangles, capturing the order in which $f$ sends the horizontal into its vertical sub-rectangles. On the other hand, the function $\epsilon_T$ associates to every sub-rectangle a number: $1$ if $f$ preserves the vertical direction or $-1$ if it changes it. An abstract geometric type \ref{Defi: Abstract geometric types} is a quadruple $(n,\{h_i,v_i\}_{i=1}^n,\Phi_T, \epsilon_T)$ where the only restriction is $\sum_{i=1}^n h_1 = \sum_{i=1}^n v_i$, then $\Phi_T$ is a bijection between finite sets, and $\epsilon_T$ is a function from a finite set to $\{-1,1\}$. The set of abstract geometric types is denoted by $\mathcal{G}\mathcal{T}$, and we will say that $T$ is realized as a Smale surface diffeomorphism if there exists a Smale diffeomorphism $\phi:S \to S$ with a non-trivial saddle-type basic piece $K$ that has a geometric Markov partition $\mathcal{R}$ whose associated geometric type is $T$. Finally we say that two geometric types are equivalent if they are realized by basic pieces of Smale diffeomorphism that have invariant neighborhoods where they are topologically conjugated. In \cite[Chapter 0.2]{bonatti1998diffeomorphismes}, Bonatti and Langevin stablish three types of questions related to the classification, we has  translate them  as follow:

\begin{enumerate}
	\item Can we associate a finite combinatorial presentation with a Smale diffeomorphism, such that, if two diffeomorphisms have the same presentations, they are topologically conjugate? At this stage, if the presentations are different, we cannot compare the diffeomorphisms.
	\item Is there a finite algorithm to decide when two different presentations correspond to two conjugate diffeomorphisms?
	\item Which combinatorial presentations correspond to a Smale diffeomorphism (realizability)?
\end{enumerate}

They solved $(1)$ in \cite[Theorem 5.2.2]{bonatti1998diffeomorphismes} by proving that if two basic pieces have geometric Markov partitions of the same geometric type, then they have an invariant neighborhood where they are topologically conjugated (in fact, it is an if and only if condition). The problem $3)$ was approached in the following way: for a geometric type $T$, we can glue a family of $n$ affine rectangles along their horizontal and vertical sub-rectangles to obtain a surface with boundary $\mathcal{R}_T$, called the \emph{realizer} of $T$. There is a natural way to define the $n$-power of the geometric type $T$ so that it is coherent with the $n$-th power of $f$ and construct the surface $\mathcal{R}_{T^n}$. The genus of $T$ is defined as the limit of $gen(\mathcal{R}_{T^n})$, and it results in a non-decreasing sequence; therefore, the genus of $T$, $gen(T)$, is finite or infinite. Additionally, they formulated three conditions that are read on the surface $\mathcal{R}_{T^n}$ and allow us to decide if the genus is finite. They showed that if a geometric type is realizable, it has finite genus. Later, Beguin showed that this is a sufficient condition, completing part 3 of the problem. Finally, Beguin developed an algorithm that allows determining when two geometric types are equivalent, providing a solution to $2)$. 

In \cite[Chapeter 8]{bonatti1998diffeomorphismes} Bonatti and Jeandennas begin to adapt these ideas to the classification of pseudo-Anosov homeomorphisms, this article is in certain manner the continuation of such ideas. In the next subsection  we will outline the classification problem for \textbf{p-A} homeomorphism  state the results contained in this article, which are the roots for a constructive classification of pseudo-Anosov homeomorphisms using Markov partitions and geometric types.

\input{Intro/Clasicationprob} %%%% Aqui explicamos como solucionamos y en que parte del texto se encuentran, cada una de las partes del problema de clasificacion.

%% file: Intro/Clasicationprob.tex
\subsection{A constructive classification of pseudo-Anosov homeomorphisms.}%%% Exposicion de resultados

Let to fix  a closed and orientable surface $S$, in the rest of this article the topological conjugacy it is always by orientation preserving homeomorphisms. The set of pseudo-Anosov homeomorphisms of the surface $S$ will be denoted  $ \text{Hom}_{\textbf{p-A}}(S) \subset \text{Hom}_+(S)$ and the set of such conjugacy classes of $\text{Hom}_+(S)$ that contain a pseudo-Anosov homeomorphism is denoted $[ \text{Hom}_{\textbf{p-A}}(S)]_{\text{Top}_+}$.  In proposition \ref{Prop: pA closed up conjugation}, we must prove that every homeomorphism topologically conjugated to a \textbf{p-A} is  pseudo-Anosov. Therefore, we can  denote the elements of $[ \text{Hom}_{\textbf{p-A}}(S)]_{\text{Top}_+}$ as $[f]$, whenever $f\in  \text{Hom}_{\textbf{p-A}}(S)$.	

	\begin{prob}\label{Prob: Clasification}
		The \emph{classification problem} for  pseudo-Anosov homeomorphisms is divided into three fundamental components.
		\begin{enumerate}
	\item \textbf{Finite presentation:} This aspect aims to provide a concise and combinatorial description of the conjugacy classes that contain a generalized pseudo-Anosov homeomorphism. It can be further divided into the following sub-problems:
	\begin{itemize}
		\item[I)] \textbf{Combinatorial representation:} Associate to every pseudo-Anosov homeomorphism $f$ an abstract geometric type. Even if the association is not unique, it must satisfy that: A geometric type is linked to a pair of \textbf{p-A} homeomorphisms if and only if they are topologically conjugated.
		
		\item[II)] \textbf{Combinatorial models:} Given geometric type $T$ associated with a pseudo-Anosov homeomorphism $f$, we must reconstruct a pseudo-Anosov homeomorphism that is conjugate to $f$ using the information provided by $T$.
	\end{itemize} 
	
	\item \textbf{Realization:} The subset of abstract geometric types that are associated with a pseudo-Anosov homeomorphism is the pseudo-Anosov class and will be denoted by $\mathcal{GT}(\mathrm{pA})$. Generally, $\mathcal{GT}(\mathrm{pA})$ is a proper subset of $\mathcal{GT}$ and the second problem is to provide a constructive algorithm method for determining whether or not an abstract geometric type belongs to the pseudo-Anosov class. Ideally, we would like to set a time bound for the algorithm to provide an answer.
	
	\item \textbf{Equivalence of representation:} If a generalized pseudo-Anosov homeomorphism has associated multiple combinatorial representations, we ask for a method to determine whether two geometric types in the pseudo-Anosov class correspond to conjugate pseudo-Anosov homeomorphisms or not. Similar to the previous problem, we would like to provide an algorithm that can efficiently determine, within finite and bounded time, whether two combinatorial representations represent the same conjugacy class.
	
	A more ambitious task is to identify a finite and distinguished family, denoted $\Theta[f]$, of the set of combinatorial objects associated with the pseudo-Anosov homeomorphism $f$. This family must satisfy three essential criteria: 
	\begin{itemize}
		\item $\Theta(f)$ must be computable from any other combinatorial object associated with $f$.
		\item $\Theta(f)$ must allow the computation of any other invariant $T(f)$ associated with $f$.
		\item $\Theta(f)$ must be, by contention, the minimum set of combinatorial information that satisfies the previous items.
	\end{itemize}
\end{enumerate}
\end{prob}

In \cite{IntiThesis} several parts of this program have been completed. In this article, however, we will focus on the more theoretical  point of view of this classification more than the constructive part. That is, during this manuscript, we will assume the existence of Markov partitions with certain combinatorial properties, and in another pepper preparation (\cite{IntiII}), we will emphasize the constructions and algorithms that are necessary to obtain such partitions. Below, we will show how we approach the problem and the results that we present here.

\subsubsection{A combinatorial object associated to a \textbf{p-A} homeomorphism}
We must to start by associate to every \textbf{p-A} $ f \in \text{Hom}_{\textbf{p-A}}(S)$ a certain combinatorial object in a the   countable set t $\mathcal{G}\mathcal{T}$ of abstract geometric types \ref{Defi: Abstract geometric types}.  Roughly speaking, an \emph{abstract geometric type} $T$ is an ordered quadruple, $T = (n, {h_i,v_i}{i=1}^n, \rho_T, \epsilon_T)$, where $n, h_i, v_i \in \mathbb{N}+$ satisfy $\sum{i=1}^{n} h_i = \sum_{i=1}^n v_i$. These numbers determine a pair of formal sets: $\mathcal{H}(T) = {(i,j) : 1 \leq i \leq n \text{ and } 1 \leq j \leq h_i}$ and $\mathcal{V}(T) = {(k,l) : 1 \leq k \leq n \text{ and } 1 \leq l \leq v_k}$ and  $\rho_T: \mathcal{H}(T) \rightarrow \mathcal{V}(T)$ is a bijection, and $\epsilon_T: \mathcal{H}(T) \rightarrow {-1,1}$ is any function over these sets.  The association will be done via Geometric Markov partitions and its geometric type.

A Markov partition (\ref{Defi: Markov partition}) for a pseudo-Anosov homeomorphism $f$ is a family of immersed rectangles (\ref{Defi: Rectangle}) $\cR = \{R_i\}_{i=1}^n$, with disjoint interiors $\overset{o}{R_i} \cap \overset{o}{R_j} = \emptyset$, whose union is $S$, and such that the closure of every non-empty connected component of $f(\overset{o}{R_i}) \cap \overset{o}{R_i}$ is a vertical sub-rectangle of $R_j$, and similarly, the closure of every connected component of $\overset{o}{R_i} \cap f^{-1}(R_j)$ is a horizontal sub-rectangle of $R_k$ (\ref{Defi: Vertical/horizontal subrec}). These horizontal and vertical sub-rectangles are the horizontal and vertical sub-rectangles of the Markov partition $(f,\cR)$. A geometric Markov partition of $f$ is a Markov partition for which we have chosen a labeling of their rectangles, each of them having a vertical and a horizontal direction whose transverse orientation is coherent with the orientation of the surface (\ref{Defi: Geometric rectangle}).

Any pseudo-Anosov homeomorphism has a Markov partition (see \cite[Proposition]{fathi2021thurston} for the case of non-spines and \cite{IntiII} for a synthetic construction in the generalized case). Therefore, every \textbf{p-A} homeomorphism $f$ has a geometric Markov partition, and we use the notation $(f,\cR)$ to indicate that $\cR$ is a geometric Markov partition of $f$. In Subsection \ref{Subsec: Geo type de particion geometrica}, we describe how to associate a unique geometric type to a geometric Markov partition $(f,\cR)$; we denote such geometric type $T(f,\cR)$ and call it the geometric type of $(f,\cR)$ (see \ref{Defi: geometric type of a Markov partition}). An abstract geometric type $T$ belongs to the \emph{pseudo-Anosov class} ( of geometric types) $\mathcal{G}\mathcal{T}(\textbf{p-A})$ if there exists a surface $S$ and a pseudo-Anosov homeomorphism $f: S \to S$ with a geometric Markov partition $\cR$ such that $T(f,\cR) = T$. We say that a geometric type $T$ in the pseudo-Anosov class is realized by $(f,\cR)$ if $T(f,\cR) = T$. Two geometric types in the pseudo-Anosov class are equivalent if they are realized by topologically conjugated pseudo-Anosov homeomorphisms. With this vocabulary, we are ready to state our results

\subsubsection{Total invariant of conjugacy}

The first result of this article is that this geometric type is a total invariant under conjugation. In the following statement, we do not demand in principle that $S_f$ and $S_g$ be homeomorphism surfaces as it is a trivial consequence.

\begin{theo*}[\ref{Theo: conjugated iff  markov partition of same type}]
	Let $f:S_f\rightarrow S_f$ and $g:S_g \rightarrow S_g$ be two  pseudo-Anosov homeomorphisms. Then, $f$ and $g$ have a geometric Markov partition of the same geometric type if and only if there exists an orientation preserving homeomorphism between the surfaces $h:S_f\rightarrow S_g$ that conjugates them, i.e., $g=h\circ f\circ h^{-1}$.
\end{theo*}

Let $f: S \to S$ be a \textbf{p-A} homeomorphism with a geometric Markov partition $\cR$, and let $T = (f,\cR)$. In order to prove Theorem~\ref{Theo: conjugated iff markov partition of same type}, we are going to use $T$ to construct another dynamical system $(\Sigma_T, \sigma_T)$, and then prove that $f$ is topologically conjugate to $(\Sigma_T, \sigma_T)$ through an explicit orientation-preserving homeomorphism that depends solely on the fact that $f$ has a geometric Markov partition $\cR$ of geometric type $T$. In this manner, the maps $f: S_f \rightarrow S_f$ and $g: S_g \rightarrow S_g$ in Theorem \ref{Theo: conjugated iff markov partition of same type} will be topologically conjugated to the same $(\Sigma_T, \sigma_T)$, and therefore they themselves are topologically conjugated. We call $(\Sigma_T, \sigma_T)$ a \emph{symbolic model} (\ref{Defi: symbolic model}) of $f$. Let us briefly explain the symbolic nature of such models. The geometric type $T$ determines the \emph{incidence matrix} $A = A(T)$ of the Markov partition $(f, \mathcal{R})$ ( \ref{Defi: Incidence matrix markov partition})When the matrix $A$ has coefficients in ${0,1}$ it is called \emph{binary}, and we can take $\sigma: \Sigma_{A(T)} \to \Sigma_{A(T)}$ as the \emph{sub-shift of finite type} induced by $T$. Furthermore, there exists a semi-conjugacy $\pi_{(f,\mathcal{R})}: \Sigma_{A(T)} \rightarrow S_f$ that just depends on the homeomorphism $f$ and the Markov partition $\mathcal{R}$, which associates to each code $\underline{w} = (w_z){z \in \mathbb{Z}} \in \Sigma{A(T)}$ a unique point $x$ on the surface $S$ such that $f^z(x) \in R_{w_z}$, i.e., $x := \pi_f(\underline{w})$ follows the "itinerary" given by $\underline{w}$.  Two codes $\underline{v}, \underline{w} \in \Sigma_{A(T)}$ are $\sim_{(f,\mathcal{R})}$-related if $\pi_{(f,\mathcal{R})}(\underline{v}) = \pi_{(f,\mathcal{R})}(\underline{w})$, this relation determines the quotient space $\Sigma_{(f,\mathcal{R})}: := \Sigma_A/\sim_{(f,\mathcal{R})}$, and it is not difficult to see that $\Sigma_{(f,\mathcal{R})}$ is homeomorphic to $S_f$.  The combinatorial model $(\Sigma_T, \sigma_T)$ is obtained by taking the quotient of $(\Sigma_{A(T)})$ by an equivalence relation $\sim_{T}$ determined by $T$ in such a manner that each equivalence class $[\underline{w}]_{\sim_T}$ is equal to a equivalence class $[\underline{w}]_{\sim_{(f,\mathcal{R})}}$.  In Section \ref{Sec: Total Invariant}, we construct this equivalence relation, and Proposition \ref{Prop: The relation determines projections} below establishes its key properties.

\begin{prop*}[\ref{Prop: The relation determines projections}]
	Let $T$ be a geometric type with an incidence matrix $A:=A(T)$ having coefficients in $\{0,1\}$, and let $(\Sigma_A,\sigma)$ be the associated sub-shift of finite type. Consider a generalized pseudo-Anosov homeomorphism $f:S\rightarrow S$ with a geometric Markov partition $\cR$ of geometric type $T$. Let $\pi_f:\Sigma_A \rightarrow S$ be the projection induced by the pair $(f,\cR)$. Then there exists an equivalence relation $\sim_T$ on $\Sigma_A$, algorithmically defined in terms of $T$, such that for any pair of codes $\underline{w},\underline{v}\in \Sigma_A$, they are $\sim_T$-related if and only if their projections coincide, $\pi_f(\underline{w}) = \pi_f(\underline{v})$.
\end{prop*}

\begin{rema}
	If two generalized pseudo-Anosov homeomorphisms $f: S_f \to S_f$ and $g: S_g \to S_g$ have Markov partitions $\cR_f$ and $\cR_g$ of the same geometric type $T$, they will share the same incidence matrix and thus have associated identical sub-shifts of finite type. This implies the existence of projections $\pi_f: \Sigma_A \to S_f$ and $\pi_g: \Sigma_A \to S_g$ that semi-conjugate the shift with the respective homeomorphisms. However, it is not clear how we can establish a direct conjugation between $f$ and $g$ based solely on these projections.
\end{rema}

We must use \ref{Prop: The relation determines projections} to obtain Proposition \ref{Prop:  cociente T}, and then Theorem \ref{Theo: conjugated iff  markov partition of same type}  as a direct consequence.

\begin{prop*}[\ref{Prop:  cociente T}]
	Consider a geometric type $T$ in the pseudo-Anosov class, whose incidence matrix $A:=A(T)$ is binary. Let $f:S \rightarrow S$ be a generalized pseudo-Anosov homeomorphism with a geometric Markov partition $\cR$ of geometric type $T$.  Also, let $(\Sigma_A,\sigma)$ be the sub-shift of finite type associated with $A$ and denote by $\pi_f:\Sigma_A \rightarrow S$ the projection induced by the pair $(f,\cR)$. Then, the following holds:
	
	\begin{itemize}
		\item The quotient space $\Sigma_T = \Sigma_A/\sim_T$ is equal to $\Sigma_f := \Sigma_A/\sim_f$. Thus, $\Sigma_T$ is a closed and orientable surface.
		\item  The sub-shift of finite type $\sigma$ induces a homeomorphism $\sigma_T: \Sigma_T \rightarrow \Sigma_T$ through the equivalence relation $\sim_T$. This homeomorphism is a generalized pseudo-Anosov and topologically conjugate to $f: S \rightarrow S$ via the quotient homeomorphism $[\pi_f]: \Sigma_T = \Sigma_f \rightarrow S$
	\end{itemize}	
\end{prop*}

\subsubsection{The realization of the pseudo-Anosov class}

Let $T = (S, {h_i, v_i}_{i=1}^n, \phi, \epsilon) \in \mathcal{G}\mathcal{T}$ be an abstract geometric type. $T$ is in the \emph{pseudo-Anosov class} if there exists a closed surface $S$ and a pseudo-Anosov homeomorphism $f: S \to S$ with a geometric Markov partition $\mathcal{R}$ whose geometric type is $T$, i.e., $T = T(f, \mathcal{R})$. How can you determine if such a pseudo-Anosov homeomorphism exists just by looking at $T$? This is the realization problem for \emph{p-A}, but it was previously considered for surface diffeomorphisms by Bonatti and Jenadanans \cite{bonatti1998diffeomorphismes} and solved by Béguin in \cite{beguin2002classification} for Smale diffeomorphisms. If $T$ is the geometric type of a geometric Markov partition of saddle-type basic pieces of a Smale diffeomorphism $\phi: S \to S$, it is necessary that the realizer of $T$, $\mathcal{R}T$ (\ref{Defi: m realizer}, be embedded in the surface $S$. Such a surface is constructed by gluing a family of rectangles through vertical and horizontal sub-rectangles following the rule dictated by the geometric type $T$. In this manner, the genus of $T$ is the maximum of the genus of the realizers ${gen(\mathcal{R}{T^n})}$. It was proved by Bonatti-Langevin and Béguin that a geometric type is realized by a basic piece of a Smale diffeomorphism of surfaces if and only if $T$ has finite genus. Our initial task is to establish a connection between Markov partitions for pseudo-Anosov homeomorphisms and those defined by saddle-type basic pieces on surfaces. In Proposition \ref{Prop: pseudo-Anosov iff basic piece non-impace} the non-impasse condition is quite technical and can be look at \ref{Defi: Impasse geo}.

\begin{prop*}[ \ref{Prop: pseudo-Anosov iff basic piece non-impace}]
	Let $T$ be an abstract geometric type. The following conditions are equivalent.
	\begin{itemize}
		\item[i)] The geometric type $T$ is realized as a mixing basic piece of a surface Smale diffeomorphism without impasse.
		\item[ii)] The geometric type $T$ is in the pseudo-Anosov class.
		\item[iii)] The geometric type $T$ satisfies the following properties:
		\begin{enumerate}
			\item  The incidence matrix $A(T)$ is mixing
			\item The genus of $T$ is finite
			\item $T$ does not have an impasse.
		\end{enumerate}
	\end{itemize}
\end{prop*}

In Proposition \ref{Prop: mixing+genus+impase is algorithm}, we discuss a  procedure to determine if $T$ has finite genus and impasse and we can find a certain bound on the number of iterations of the geometric type $T$ in order to corroborate these properties. This is obtained by reformulating our pictorial criterion to finite genus given as is given in \ref{Defi:  Type 3 obtruction geo}, \ref{Defi: Type 1 obstruction top} and \ref{Defi: Type 2 obtruction top} and can be read in the realizer of $T$ into  pure combinatorial terms using the permutations and indices of the geometric type, this is done in \ref{Sub-sec: Comb Impasse genus}. In this manner, we obtain the following result that provide a answer to the realization problem in \ref{Prob: Clasification}.

\begin{theo*}[\ref{Theo: caracterization is algoritmic}]
	There exists a finite algorithm that can determine whether a given geometric type $T$ belongs to the pseudo-Anosov class. Such algorithm requires calculating at most $6n$ iterations of the geometric type $T$, where $n$ is the first parameter of $T$.	
\end{theo*}

\subsubsection{Equivalent geometric types}v

In order to determine if two geometric types in the pseudo-Anosov class represent pseudo-Anosov homeomorphisms that are topologically conjugated, we must rely on the \emph{Béguin's algorithm}. This algorithm, originally developed in \cite{beguin2004smale}, provides a solution for determining whether two geometric types represent conjugate basic pieces. Its essence can be summarized as follows:

\begin{theo*}[The Béguin's Algorithm]
	Let $f$ be a Smale surface diffeomorphism and $K$ be a saddle-type basic piece with a geometric Markov partition $\mathcal{R}$ of geometric type $T=(n,\{(h_i,v_i)\}_{i=1}^n,\Phi_T)$. The steps of the algorithm are as follows:
	\begin{enumerate}
		\item Begin by defining the \emph{primitive geometric types of order} $n$ of $f$, denoted as $\cT(f,n)$, for all $n$ greater than a certain constant $n(f)\geq 0$.
		\item Prove the existence of an upper bound $O(T)$ for $n(f)$ in terms of the number $n$ in $T$, i.e. $n(f)\leq O(T)$.
		\item For every $n>n(f)$, there exist an algorithm to compute all the elements of $\cT(f,n)$ in terms of $T$.
		\item Let $g$ be another Smale surface diffeomorphism and $K'$ be a saddle-type basic piece of $g$ with a geometric Markov partition of geometric type $T'$. Choose $n$ to be greater than or equal to the maximum of $O(T)$ and $O(T')$. After applying Step 3 of the algorithm, we obtain two finite lists of geometric types, $\cT(f,n)$ and $\cT(g,n)$. These lists will be equal if and only if $f$ is topologically conjugate to $g$ in some invariant neighborhoods of $K'$ and $K'$.
	\end{enumerate}
\end{theo*}

Let us explain how we address our problem by shifting the analysis to the formal \textbf{DA}$(T)$ (Definition \ref{Defi: Formal DA}). Consider a geometric type $T$ in the pseudo-Anosov class. According to Proposition \ref{Prop: pseudo-Anosov iff basic piece non-impace}, $T$ has finite genus. Interestingly, this condition is necessary and sufficient for $T$ to be realized as a basic piece of a Smale surface diffeomorphism. The \emph{formal derived from Anosov} of $T$ refers to this realization and is represented by the triplet:

\begin{equation}\label{Equa: DA of T}
	\textbf{DA}(T):=(\Delta(T),K(T),\Phi_T),
\end{equation}

The components are as follows: $\Delta(T)$ is a compact surface with boundary of finite genus, $\phi_T$ is a Smale diffeomorphism defined on $\Delta(T)$; and $K(T)$, the unique nontrivial saddle-type basic piece of $(\Delta(T),\Phi_T)$ with a Markov partition of geometric type $T$.  We call $\Delta(T)$ the \emph{domain} of the basic piece $K(T)$. It is a profound result of Bonatti-Langevin \cite[Propositions 3.2.2 and 3.2.5, Theorem 5.2.2]{bonatti1998diffeomorphismes} that the \emph{DA}$(T)$ is unique up to conjugacy and serves as a bridge between the realm of pseudo-Anosov homeomorphisms and the basic pieces of Smale surface diffeomorphisms. Let $f$ a generalized pseudo-Anosov homeomorphism and $\cR$ a geometric Markov partition, a periodic point of $f$ that is in the stable (unstable) boundary of $\cR$ is a $s$-boundary point ($u$-boundary point). Let $f$ and $g$ be two generalized pseudo-Anosov homeomorphisms, and let $\cR_f$ and $\cR_g$ be geometric Markov partitions with geometric types $T_f$ and $T_g$, respectively. Let  $p(f)$ be the maximum period of  periodic points of $f$ on the boundary of $\cR_f$, and $p(g)$ be the maximum period of periodic points of $g$ on the boundary of $\cR_g$. Take $p = \max\{p(f), p(g)\}$.  We are going to construct  geometric Markov partitions $\cR_{\cR_f,\cR_g}$ of $f$ and $\cR_{\cR_g,\cR_f}$ of $g$ such that: 

\begin{itemize}
	\item The set of periodic boundary points of $\cR_{\cR_f,\cR_g}$ ( $\cR_{\cR_g,\cR_g}$ ) coincides with the set of periodic points of $f$ ( $g$ ) whose period is less or equal than $p$.
	\item Every periodic boundary point in $\cR_{\cR_f,\cR_g}$ is a \emph{corner point} i.e. is $u$ and $s$-boundary.
\end{itemize}

These refined partitions are referred to as the \emph{compatible refinements} of $\cR_f$ and $\cR_g$, and their respective geometric types are denoted $T_{(\cR_f,\cR_g)}$ and $T_{(\cR_g,\cR_f)}$. Two geometric types, $T_1$ and $T_2$, which can be realized as basic pieces, are considered to be \emph{strongly equivalent} if there exists a Smale diffeomorphism $f$ of a compact surface and a nontrivial saddle-type basic piece $K$ of $f$, such that $K$ has a geometric Markov partition of geometric type $T_1$ as well as a geometric Markov partition of geometric type $T_2$.  We have proven the following corollary , which establishes a connection between the formal \textbf{DA} of the geometric types of the compatible refinements and the underlying pseudo-Anosov homeomorphisms.

\begin{coro*}\ref{Coro: equivalence pA and DA}
	Let $f$ and $g$ be generalized pseudo-Anosov homeomorphisms with geometric Markov partitions $\cR_f$ and $\cR_g$ of geometric types $T_f$ and $T_g$, respectively. Let $\cR_{f,g}$ be the joint refinement of $\cR_f$ with respect to $\cR_g$, and let $\cR_{g,f}$ be the joint refinement of $\cR_g$ with respect to $\cR_f$, whose geometric types are $T_{f,g}$ and $T_{g,f}$, respectively.  Under these hypotheses: $f$ and $g$ are topologically conjugated through an orientation preserving homeomorphism if and only if $T_{f,g}$ and $T_{g,f}$ are strongly equivalent.
\end{coro*}

In the proof we use the formal \textbf{DA} in the next manner: If $T_{(\cR_f,\cR_g)}$ and $T_{(\cR_g,\cR_f)}$ are strongly equivalent, then \textbf{DA}$(T_{(\cR_f,\cR_g)})$ and \textbf{DA}$(T_{(\cR_g,\cR_f)})$ are conjugate through a homeomorphism $h$. This homeomorphism induces a conjugation between $f$ and $g$.  Proposition \ref{Prop: preimage is Markov Tfg type} establish  that if $f$ and $g$ are topologically conjugate, then \textbf{DA}$(T_{(\cR_f,\cR_g)})$ has a Markov partition of geometric type $T_{(\cR_g,\cR_f)}$. Hence they  are strongly equivalent.
Finally we use the Béguin's algorithm to determine if $T_{(\cR_f,\cR_g)}$ and $T_{(\cR_g,\cR_f)}$ are strongly equivalent. This provides a solution for the fist part of Item $III$ in Problem \ref{Prob: Clasification} through the following theorem.

\begin{theo*}[\ref{Theo: algorithm conjugacy class}]
	Let $T_f$ and $T_g$ be two geometric types within the pseudo-Anosov class. Assume that $f: S \rightarrow S_f$ and $g: S_g \rightarrow S_g$ are two generalized pseudo-Anosov homeomorphisms with geometric Markov partitions $\cR_f$ and $\cR_g$, having geometric types $T_f$ and $T_g$ respectively.  We can compute the geometric types $T_{f,g}$ and $T_{g,f}$ of their joint refinements through the algorithmic process described in Chapter \ref{Chap: Computations}, and the homeomorphisms $f$ and $g$ are topologically conjugated by an orientation-preserving homeomorphism if and only if the algorithm developed by Béguin determines that $T_{f,g}$ and $T_{g,f}$ are strongly equivalent.
\end{theo*}

\subsection*{Acknowledgements:}
The results of this article are part of my PhD thesis, under the supervision of Christian Bonatti. I want to express my gratitude to him for proposing this problem and guiding me towards its solution. During my studies abroad, I was funded by a CONACyT scholarship, thank you for trusting me. Finally, I want to thank Professor Lara Bossinger, who through her program PAPIIT "Álgebras de conglomerado, geometría tropical y degeneraciones tóricas" has supported the writing this manuscript during my research stay at IMATE-UNAM Oaxaca.

%% file: Preliminares/preliminares.tex
\section{Preliminaries.}\label{Sec: Preliminares}

In this paper, $S$ is a closed and oriented \emph{topological surface} and $\text{Hom}_+(S)$ is the set of orientation-preserving homeomorphisms of $S$. We are going to introduce our \emph{MVP}'s: pseudo-Anosov homeomorphisms and abstract geometric types.

\subsection{The conjugacy class of a pseudo-Anosov homeomorphism.}

We would like to define a singular foliation on any topological surface; therefore, we must avoid any reference to a certain differentiable structure on $S$, as, for example, it is done in \cite{farb2011primer} by demanding the foliated charts to be smooth. In \cite{hiraide1987expansive}, Hiraide has introduced a definition of singular foliation that makes sense on any topological surface and discusses the notion of transversality for curves and foliations in this setting; we must adopt his topological point of view. However, since it is not our intention to reproduce such an article, we refer to it for any doubts about the foliated chart around the $p(x)$-pronged-saddle that is mentioned in \ref{Defi: Singular foliation}, and adopt the following definition of singular foliation as given by Farb and Margalit in \cite{farb2011primer} but taking $C^0$ charts.

\begin{defi}\label{Defi: Singular foliation}
A decomposition $\mathcal{F}$ of $S$ by $1$-dim curves called \emph{leaves} and a finite number of \emph{singular points}, $\textbf{Sing}(\mathcal{F})$, is a \emph{singular foliation} if every \emph{leaf} $L \in \mathcal{F}$ is path connected and for every $x \in S$ there exists $p(x) \in \mathbb{N}^+$ and a $C^0$-\emph{foliated chart} around $x$, $\phi_x: U_x \rightarrow \mathbb{R}^2$, such that $\phi_x(x)=0$ and
\begin{enumerate}
	\item For every leaf $L \in \mathcal{F}$, $\phi_x$ sends every non-empty connected component of $U_x \cap L$ into a level set of a $p(x)$-pronged-saddle singularity with $p(x) \in \mathbb{N}^+$ (See~\ref{Fig: k-saddle}).
	\item There is subset of the singular points where $p(x) \neq 2$.
\end{enumerate}
\end{defi}

If $p(x) \geq 2$, we call $x$ a $p(x)$-prong. If $p(x) = 1$, it is called a \emph{spine}, and if $p(x) = 2$, it is a regular point.

\begin{figure}[ht]
	\centering
	\includegraphics[width=0.4\textwidth]{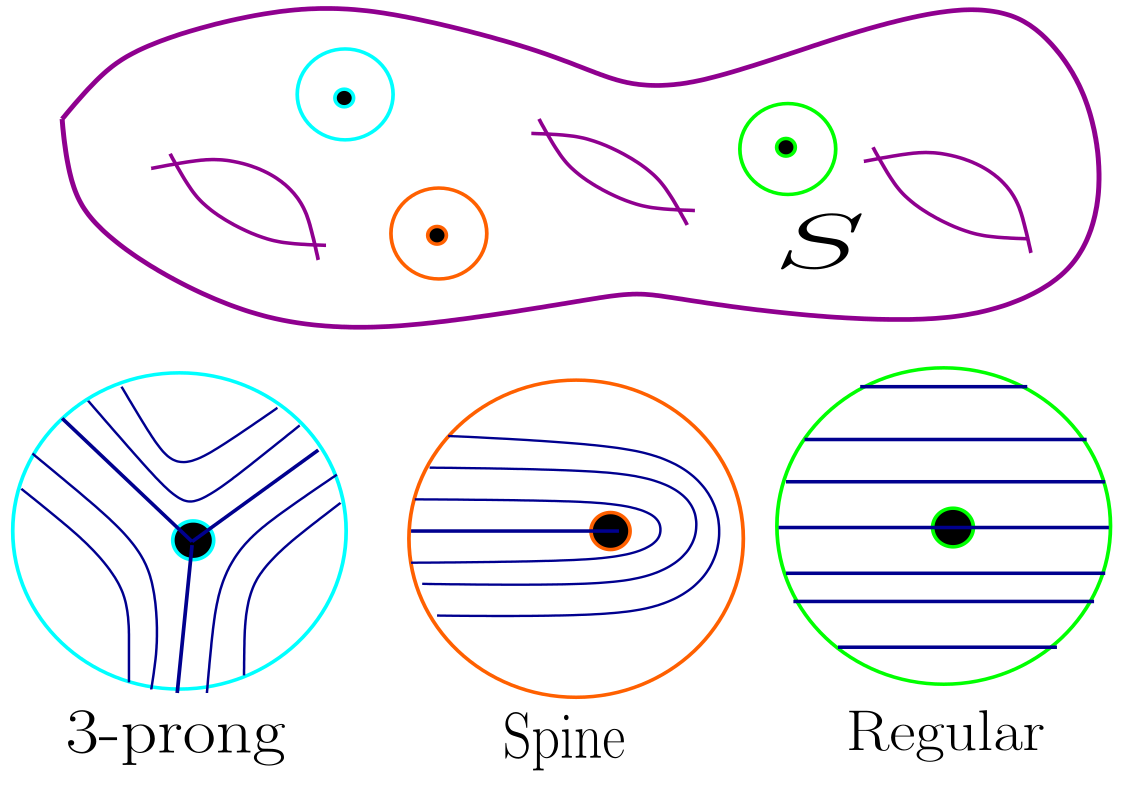}
	\caption{ Saddles  and regular points.}
	\label{Fig: k-saddle}
\end{figure}

An \emph{arc} $\gamma: I \to S$ is a topological embedding from the unit interval $[0,1]$ into $S$; usually, we just write $\gamma$ to indicate the image of such an arc. Two arcs $\alpha$ and $\beta$ are transverse at $x \in \alpha \cap \beta$ if there is a chart around $x$, $\phi_x: U_x \to \mathbb{R}^2$, where the images of $\alpha$ and $\beta$ are $\phi_x(\alpha)$ and $\phi_x(\beta)$, respectively, and if $\pi_1(x,y) = (x)$ and $\pi_2(x,y) = (y)$ are the projections in their respective components, then $\pi_1 \circ \pi_1 \circ \alpha$ and $\pi_1 \circ \pi_2 \circ \beta$ are injective.  The notions of \emph{transverse arcs}, \emph{arcs transverse to the foliation}, and \emph{leaf isotopic arcs} in the $C^0$ framework were properly discussed  by Hiraide in \cite{hiraide1987expansive}. We refer the reader to such articles for more details about our conventions, but they don't differ so much from the classical point of view presented in \cite{fathi2021thurston} and \cite{farb2011primer}.

\begin{defi}\label{Defi: Transverse measure}
A \emph{transverse measure} $\mu$ of the singular foliation $\mathcal{F}$ is a non-atomic \emph{Borel measure}, finite and positive defined over any (non-trivial) compact arc that is transversal to every leaf of $\mathcal{F}$ and for which two \emph{leaf-isotopic} arcs have the same measure.
\end{defi}

\begin{defi}\label{Defi: Measured foliation}
A \emph{measured foliation} of $S$ is the structure $(\mathcal{F}, \mu)$, given by a singular foliation $\mathcal{F}$ of $S$ and a transverse measure $\mu$ of $\mathcal{F}$.
\end{defi}

Let $h:S\rightarrow S'$ be a homeomorphism, and let $(\mathcal{F},\mu)$ be a measured foliation of $S$. Take $h(\mathcal{F})$ as the decomposition of $S'$ given by the image of leaves and singularities of $\mathcal{F}$ under $h$. We claim $\mathcal{F}$ is a singular foliation of $S'$. In effect, take a point $y \in S'$ and compose a foliated chart $\phi_{h^{-1}(y)}:S\to \mathbb{R}^2$ of $\mathcal{F}$ around $h^{-1}(y)$ (as in Definition \ref{Defi: Singular foliation}) with $h^{-1}$ to obtain a $C^0$ foliated chart $(\phi_{h^{-1}(y)} \circ h^{-1})$ of $h(\mathcal{F})$ around $y$.  If $J$ is a non-trivial arc (topologically) transverse to $h(\mathcal{F})$, then $h^{-1}(J)$ is a non-trivial arc (topologically) transverse to $\mathcal{F}^{s,u}$ since $h$ is a homeomorphism. The preimage of every Borel set contained in $J$ is a Borel set in $h^{-1}(J)$, so $h_*(\mu)(J):=\mu(h^{-1}(J))$ is meaningful and we pose the following definition.

\begin{defi}\label{Defi: Induced measured foliation}
Let $h:S\rightarrow S'$ be a homeomorphism, and let $(\mathcal{F},\mu)$ be a measured foliation of $S$. Then $h_*(\mathcal{F},\mu):=(h(\mathcal{F}), h_*(\mu))$ is the measured foliation of $S'$ induced by $h$.
\end{defi}

\begin{defi}\label{Defi: pseudo-Anosov}
	A homeomorphism $f \in \text{Hom}_+(S)$ is \emph{pseudo-Anosov} (abbr.  \textbf{p-A}  ) if it preserves two transverse measured foliations, $(\mathcal{F}^s, \mu^s)$ and $(\mathcal{F}^u, \mu^u)$, the \emph{stable} and \emph{unstable foliations} of $f$ respectively, with the same set and types of singularities, and there exists a \emph{stretch factor} $\lambda > 1$ such that:
	$$
	f_*(\mu^u)=\lambda\mu^u \text{ and } 	f_*(\mu^s)=\lambda^{-1}\mu^s.
	$$
	The singularities of $f$ are given by  $\textbf{Sing}(f):=\textbf{Sing}(\cF^s)=\textbf{Sing}(\cF^u)$.
\end{defi}

A \emph{stable interval} is any arc contained in a single leaf of $\mathcal{F}^s$, and similarly, an \emph{unstable interval} is contained in $\mathcal{F}^u$. Definition \ref{Defi: pseudo-Anosov} states that $f$ uniformly expands the $\mu^u$-measure of any stable arc $I$ by a factor $\lambda$; $f_*\mu^u(I) = \mu^u(f^{-1}(I))$, or, equivalently, it is uniformly contracted by $\lambda^{-1}$.  The main facts about pseudo-Anosov homeomorphisms that we need to reference are contained in Proposition \ref{Prop: pseudo-Anosov properties.}. In the case of surfaces with boundaries, the result is proved by Farb and Margalit in \cite{farb2011primer}. However, as mentioned in \cite[Exposition 13]{fathi2021thurston}, given a pseudo-Anosov homeomorphism with spines (they call them generalized pseudo-Anosov), we can perform a \emph{blow-up} on the spine to obtain a pseudo-Anosov homeomorphism (without spines) on a surface with boundaries. Therefore, such results can be applied to the invariant foliations of pseudo-Anosov homeomorphisms that may contain spines.

\begin{prop}\label{Prop: pseudo-Anosov properties.}
	Let $f: S \rightarrow S$ be a  \textbf{p-A}  homeomorphism with invariant foliations $(\cF^s, \mu^s)$ and $(\cF^u, \mu^u)$. They have the following properties:
	\begin{enumerate}
		\item   None of the foliations contains a closed leaf, and none of the leaves contains two singularities (\cite[Lemma  14.11]{farb2011primer})	
		\item They are \emph{minimal}, i.e., any leaf of the foliations is dense in $S$  \cite[Corollary 14.15]{farb2011primer}).
	\end{enumerate}
\end{prop}

Since we focus on orientable surfaces and orientation-preserving maps, we only consider the topological conjugacy of orientation-preserving homeomorphisms.

\begin{defi}\label{Defi: Top conjugacy}
	Two orientation-preserving surface homeomorphisms $f:S\rightarrow S$ and $g:S'\rightarrow S'$ are \emph{topologically conjugated} if there exists an orientation-preserving homeomorphism $h:S\rightarrow S'$ such that $g = h\circ f\circ h^{-1}$.
\end{defi}

The reason for making our definitions with minimal regularity is the following proposition, that implies  the conjugacy class of a \textbf{p-A} consists only of \textbf{p-A}   homeomorphisms. This contrasts with  pseudo-Anosov isotopy classes, which contain elements that are not pseudo-Anosov in the sense of Definition \ref{Defi: pseudo-Anosov}.

\begin{prop}\label{Prop: pA closed up conjugation}
	Let $f: S \rightarrow S$ be a pseudo-Anosov homeomorphism whose invariant foliations are $(\mathcal{F}^{s,u},\mu^{u,s})$ and let $h: S \rightarrow S'$ be any orientation preserving homeomorphism.  Then $g := h \circ f \circ h^{-1}$ is also a pseudo-Anosov homeomorphism over $S'$ with the same dilation factor as $f$ and whose invariant measured foliations are given by $h_*(\mathcal{F}^{s,u},\mu^{u,s})$.
\end{prop}

\begin{proof}
The conjugacy between $f$ and $g$ implies that $h(\mathcal{F}^{s,u})$ is invariant under the action of $g$. By definition, $\mathcal{F}^{s}$ is topologically transverse to $\mathcal{F}^{u}$, meaning that around any point $x \in S$, there exists a $C^0$-chart $\phi_{x}:U_x \subset S \to \mathbb{R}^2$ where the foliations $\mathcal{F}^{s}$ and $\mathcal{F}^{s}$ are send to the vertical and horizontal foliations of $\mathbb{R}^2$. If $y \in S'$, the map $\phi_{h^{-1}(y)} \circ h^{-1}$, restricted to a certain open set $V_y$, is a $C^0$ chart where $h(\mathcal{F}^{s})$ and $h(\mathcal{F}^{s})$ are trivial foliations of $\mathbb{R}^2$. Therefore, $h(\mathcal{F}^{s})$ is transverse to $h(\mathcal{F}^{u})$ away from the singularities.
  
 Let $\lambda$ be the stretch factor of $f$. The following computation in the stable case is the same as in the unstable case, and they imply that $g$ has the same dilation factor as $f$.
 
	\begin{eqnarray*}
		g_*h_*(\mu^s)(J)=h_*(\mu^s)(g^{-1}(J))\\	
		\mu^s(h^{-1} \circ ¨h \circ f^{-1} \circ h^{-1}(J))=\\
		\mu^s(f^{-1}(h^{-1}(J)))=\lambda^{-1}\mu^s(h^{-1}(J))=\lambda^{-1}h_*\mu^s(J).
	\end{eqnarray*}
Since $g$ is the composition of orientation-preserving maps, it must be an orientation-preserving homeomorphism itself. This ends  our proof.

\end{proof}

\subsection{Abstract geometric types.}

The objects that will be associated with a pseudo-Anosov homeomorphism and a geometric Markov partition are abstract  types, which we will introduce below.

\begin{defi}\label{Defi: Abstract geometric types}
	An \emph{abstract geometric type} $T$ is an ordered quadruple:
	$$
	T = (n, \{h_i,v_i\}_{i=1}^n, \rho_T, \epsilon_T),
	$$
	where the parameters $n, h_i, v_i \in \mathbb{N}_+$ are positive integers, satisfying $\alpha(T):=\sum_{i=1}^{n} h_i = \sum_{i=1}^n v_i$. These parameters define the sets of \emph{horizontal} and \emph{vertical labels} of $T$ as follows:
	$$
	\mathcal{H}(T) = \{(i,j) : 1 \leq i \leq n \text{ and } 1 \leq j \leq h_i\}
	$$
	and
	$$
	\mathcal{V}(T) = \{(k,l) : 1 \leq k \leq n \text{ and } 1 \leq l \leq v_k\},
	$$
	respectively. Additionally, $\rho_T: \mathcal{H}(T) \rightarrow \mathcal{V}(T)$ is a bijection and $\epsilon_T: \mathcal{H}(T) \rightarrow \{-1,1\}$ is any function. The \emph{set of abstract geometric types} is denoted by $\mathcal{GT}$.
\end{defi}
 Since $\rho_T$ is bijective, we can define the \emph{inverse of an abstract geometric type} $T = (n, \{h_i,v_i\}_{i=1}^n, \rho_T, \epsilon_T)$ as the abstract geometric type:
$$
T^{-1} := (n, \{v_i,h_i\}_{i=1}^n,\rho_{T^{-1}}, \epsilon_{T^{-1}}),
$$
where $n$, $v_i$, and $h_i$ are the same parameters as those of $T$. The horizontal and vertical labels are interchanged: $\mathcal{H}(T^{-1})=\mathcal{V}(T)$ and $\mathcal{V}(T^{-1})=\mathcal{H}(T)$. The maps are given by $\rho_{T^{-1}} := \rho_T^{-1}: \mathcal{V}(T) \to \mathcal{H}(T)$ and $\epsilon_{T^{-1}}:= \epsilon_T\circ \rho_T^{-1}: \mathcal{V}(T) \to \{-1,+1\}$. Sometimes, we may use the notation $\Phi_T$ to refer to $(\rho_T, \epsilon_T): \mathcal{H}(T) \to \mathcal{V}(T) \times \{-1,+1\}$ as a whole structure.

\subsection{Geometric Markov partitions.}

A classic \emph{Markov partition} of $f$ is a decomposition of the surface $S$ into pieces called \emph{rectangles} that behave well under iterations of $f$ and have a notion of upper and lower boundaries. Let's describe such pieces.

\begin{defi}\label{Defi: Rectangle}
	A compact subset $R$ of $S$ is a rectangle adapted to $f$ if there exists a continuous function $\rho:\II^2 \rightarrow S$ whose image is $R$  and satisfying the following conditions:
	\begin{itemize}
		\item $\rho:\overset{o}{\II^2} \to S$ is a homeomorphism onto its image that we call \emph{interior of the rectangle } and is denoted by $\overset{o}{R}:=\rho(\overset{o}{\II^2})$.
		\item For every $t\in [0,1]$, $I_t:=\rho([0,1]\times \{t\})$ is contained in a unique leaf of $\mathcal{F}^s$, and $\rho\vert_{[0,1]\times \{t\}}$ is a homeomorphism onto its image. The \emph{horizontal foliation} is the decomposition of $R$ given by these \emph{stable intervals}: $\cI(R):=\{I_t\}_{t\in[0,1]}$.
		\item For every $t\in [0,1]$, $J_t:=\rho(\{t\} \times [0,1])$ is contained in a unique leaf of $\mathcal{F}^u$, and $\rho\vert_{\{t\}\times [0,1]}$ is a homeomorphism onto its image. The \emph{vertical foliation}  is the decomposition of $R$ given by these \emph{vertical intervals}: $\cJ(R):=\{J_t\}_{t\in[0,1]}$
	\end{itemize}
	A function like $\rho$ is called a \emph{parametrization} of $R$.
\end{defi}

If the parametrization of $R$ is not necessary for our discussion, we refer to $R$ as the rectangle, and we usually omit the dependence on $f$ if it's clear. The parametrization of $R$ is not unique, and it could preserve or reverse the orientation of $\overset{o}{\II^2}$ over its image contained in $S$ with the fixed orientation. However, any parametrization must still send horizontal and vertical intervals of $\II^2$ to stable and unstable intervals  contained in a single leaf of $\cF^{s,u}$. This leads to the notion of \emph{equivalent parametrizations}.

\begin{defi}\label{Defi: equivalent parametrizations}
	Let $\rho_1$ and $\rho_2$ be two parametrizations of the rectangle $R$, and let
	\begin{equation}\label{Equa: Cambio cordenadas}
	\rho_2^{-1} \circ \rho_1 := (\varphi_s, \varphi_u) : (0,1)\times (0,1) \rightarrow (0,1)\times (0,1).
	\end{equation}
	
	The parametrizations are equivalent if $\varphi_{s,u}: (0,1) \to (0,1)$ is an increasing homeomorphism.
\end{defi}

In other words, the change of parametrization preserves the orientation in both the stable and unstable foliations simultaneously. Take two equivalent parametrizations $\rho_1$ and $\rho_2$ of $R$, and suppose that $I_t := \rho_1([0,1]\times \{t\})$ and $I_s := \rho_2([0,1]\times \{s\})$ intersect within their interiors, then they must coincide and both parametrizations generate the same horizontal and vertical foliations of $R$, providing their definition only depends on the invariant foliations of $f$. With this in mind, we can define the \emph{interior of a horizontal leaf} $I_t$ as $\overset{o}{I_t} = \rho(\{t\} \times (0,1))$, and the \emph{interior of a vertical leaf} by $\overset{o}{J_t} = \rho((0,1)\times \{t\})$ for any parametrization $\rho$ of $R$.

The following lemma is immediate,  if the parametrization doesn't preserve the orientation, we can simply compose such parametrization with a reflection, for example $(x,y)\to (1-x,y)$.

\begin{lemm}\label{Lemm: Orientation preserving parametrization}
	Every rectangle $R$ has a parametrization $\rho: \II^2 \to S$ such that $\rho\vert_{\overset{o}{\II^2}}$ is an orientation-preserving homeomorphism.
\end{lemm}

Let $R$ be a rectangle, and let $\rho:\II^2 \to R$ be a parametrization that preserves the orientation when restricted to $\overset{o}{\II^2}$. Choose a vertical direction for $\mathcal{J}$ in $\II^2$ and orient $\mathcal{I}$ in the unique direction that is coherent with the orientation of $\II^2$ and that chosen for $\mathcal{J}$. The parametrization $\rho$ takes the  direction of $\mathcal{J}$ and induces a vertical direction in $\mathcal{J}(R) \cap \overset{o}{R}$. This procedure of \emph{choosing a vertical direction} for $R$ is completed by taking the orientation of $\mathcal{I}(R) \cap \overset{o}{R}$ induced by $\rho$.

\begin{defi}\label{Defi: Geometric rectangle}
	A \emph{geometric rectangle} $R$ is a rectangle for which we have chosen a vertical direction.
\end{defi}

It is not difficult to see that different orientation-preserving transformations give rise to the same geometric rectangle, as the homeomorphisms $\varphi_{s,u}$ given in Definition \ref{Defi: equivalent parametrizations} preserve the direction in the vertical and horizontal directions. This justifies the following definition.

\begin{defi}\label{Defi: L-R sides}
	Let $R$ be a geometric rectangle adapted to $f$, and let $\rho:\II^2\to R$ be any orientation-preserving parametrization of $R$. Consider the following subsets of $R$:
	\begin{itemize}
		\item Its \textbf{left and right} sides: $\partial^u_{-1}R:=\rho(\{0\}\times [0,1])$ and $\partial^u_{1}R=\rho(\{1\}\times [0,1])$, respectively. Each of them is called $s$-boundary component of $R$. 
		
		\item Its \textbf{lower and upper} sides: $\partial^s_{-1}R=\rho([0,1] \times \{0\})$ and $\partial^s_{+1}R=\rho([0,1] \times \{1\})$, respectively. Each of them is called $u$-boundary component of $R$. 
		
		\item The \textbf{horizontal (or stable) boundary} $\partial^s R=\partial^s_{-1}R  \cup \partial^s_{+1}R$ and the \emph{vertical (or unstable) boundary of the rectangle} $\partial^u R=\partial^u_{-1}R\cup \partial^u_{+1}R$.
		
		\item The \textbf{boundary} $\partial R=\partial^s R \cup \partial^uR$.
		
		\item The \textbf{corners} of $R$ that are defined by pairs: if $t,s\in \{0,1\}$, the respective corner labeled $C_{s,t}=\rho(s,t)$.
		
		\item For all $x\in \overset{o}{R}$, $I_x$ denoted the horizontal leaf of $\mathcal{I}(R)$ that passes through $x$, and $J_x$ to the vertical leaf of $\mathcal{J}(R)$ that passes through $x$.
		
	\end{itemize}
	
\end{defi}

\begin{defi}\label{Defi: Vertical/horizontal subrec}
	Let $R$ be a rectangle. A rectangle $H \subset R$ is a \emph{horizontal sub-rectangle} of $R$ if for all $x \in \overset{o}{H}$, the horizontal leaf of $\mathcal{I}(H)$ passing through $x$ and the horizontal leaf of $\mathcal{I}(R)$ passing through $x$ coincide. A rectangle $V$ is a \emph{vertical sub-rectangle} of $R$ if for all $x \in \overset{o}{V}$, the vertical leaf of $\mathcal{J}(V)$ passing through $x$ and the vertical leaf of $\mathcal{J}(R)$ passing through $x$ coincide.
\end{defi}

\begin{defi}\label{Defi: Markov partition}
	Let $f: S \rightarrow S$ be a pseudo-Anosov homeomorphism. A \emph{Markov partition} of $f$ is a family of rectangles $\mathcal{R} = \{ R_i \}_{i=1}^n$ with the following properties:
	\begin{itemize}
		\item[i)] The surface is the union of the rectangles: $S = \cup_{i=1}^n R_i$.
		
		\item[ii)] They have disjoint interior, i.e. for all $i \neq j$,  $\overset{o}{R_i} \cap \overset{o}{R_j} = \emptyset$.
		
		\item[iii)] For every $i, j \in \{1, \cdots, n\}$, the closure of each non-empty connected component of $\overset{o}{R_i} \cap f^{-1}(\overset{o}{R_j})$ is a horizontal sub-rectangle or $R_i$. 
		
		\item[iv)] For every $i, j \in \{1, \cdots, n\}$, the closure of each non-empty connected component of  $f(\overset{o}{R_i}) \cap \overset{o}{R_j}$ is a vertical sub-rectangle of $R_j$.
	\end{itemize}
	The family of horizontal and vertical sub-rectangles that appears in Items $iii)$ and $iv)$ are the \emph{horizontal} and \emph{vertical sub-rectangles} of the Markov partition $(f,\cR)$ respectively.
\end{defi}

We must introduce some other interesting subsets inside a rectangle.

\begin{defi}\label{Defi: boundary points}
	Let $\mathcal{R} = {R_i}_{i=1}^n$ be a geometric Markov partition of $f$, and let $p$ be a periodic point of $f$.
	\begin{enumerate}
		\item $p$ is an $s$-\emph{boundary periodic point} of $\mathcal{R}$ if $p \in \partial^s \mathcal{R}$, a $u$-\emph{boundary periodic point} if $p \in \partial^u \mathcal{R}$, and a \emph{boundary periodic point} if $p \in \partial \mathcal{R}$. These sets are denoted as $\text{Per}^{s,u,b}(f,\mathcal{R})$ accordingly.
		\item $p$ is an \emph{interior periodic point} if $p \in \overset{o}{\mathcal{R}}$, and this set is denoted as $\text{Per}^I(f,\mathcal{R})$.
		\item $p$ is a \emph{corner periodic point} if there exists $i \in {1, \ldots, n}$ such that $p$ is a corner point of $R_i$. This set is denoted as $\text{Per}^C(f,\mathcal{R})$.
	\end{enumerate}
\end{defi}

The following lemma will be used many times in our future explanations.

\begin{lemm}\label{Lemm: Boundary of Markov partition is periodic}
	Both the upper and lower boundaries of each rectangle in the Markov partition $\cR$ lay on the stable leaf of some periodic point of $f$. Similarly, the left and right boundaries are contained in the unstable leaf of some periodic point.
\end{lemm}

\begin{proof}
	Let $x$ be a point on the stable boundary of $R_i$. For all $n\geq 0$, $f^{n}(x)$ remains on the stable boundary of some rectangle because the image of the stable boundary of $R_i$ coincides with the stable boundary of some vertical rectangle in the Markov partition, and such vertical rectangle have stable boundary in $\partial^s \cR$. However, there are only a finite number of stable boundaries in $\cR$. Hence, there exist $n_1, n_2\in \mathbb{N}$ such that $f^{n_1}(x)$ and $f^{n_2}(x)$ are on the same stable boundary component. This implies that this leaf is periodic and corresponds to the stable leaf of some periodic point. Therefore, $x$ is one of these leaves and is located on the stable leaf of a periodic point.
	A similar reasoning applies to the case of vertical boundaries.
\end{proof}

\subsection{The sectors of a point.}

There is a natural local stratification of the surface $S$ given by the transverse foliations of a pseudo-Anosov homeomorphism $f$. We are going to exploit it. The following result corresponds to \cite[Lemme 8.1.4]{bonatti1998diffeomorphismes}

\begin{theo}\label{Theo: Regular neighborhood}
	Let $f: S \rightarrow S$ be a generalized pseudo-Anosov homeomorphism, and let $p \in S$ be a singularity with $k \geq 1$ separatrices. Then there exists $\epsilon_0 > 0$ such that for all $0 < \epsilon < \epsilon_0$, there is a neighborhood $D$ of $p$ with the following properties:
	\begin{itemize}
		\item The boundary of $D$ consists of $k$ segments of unstable leaves alternated with $k$ segments of stable leaves.
		\item For every (stable or unstable) separatrice $\delta$ of $p$, the connected component of $\delta \cap D$ which contains $p$ has a measure (stable or unstable) equal to $\epsilon$.
	\end{itemize}
	
	We say that $D$ is a \emph{regular neighborhood} of $p$ with side length $\epsilon$. If necessary, to make it clear, we write $D(p,\epsilon)$.
\end{theo}

\begin{figure}[h]
	\centering
	\includegraphics[width=0.2\textwidth]{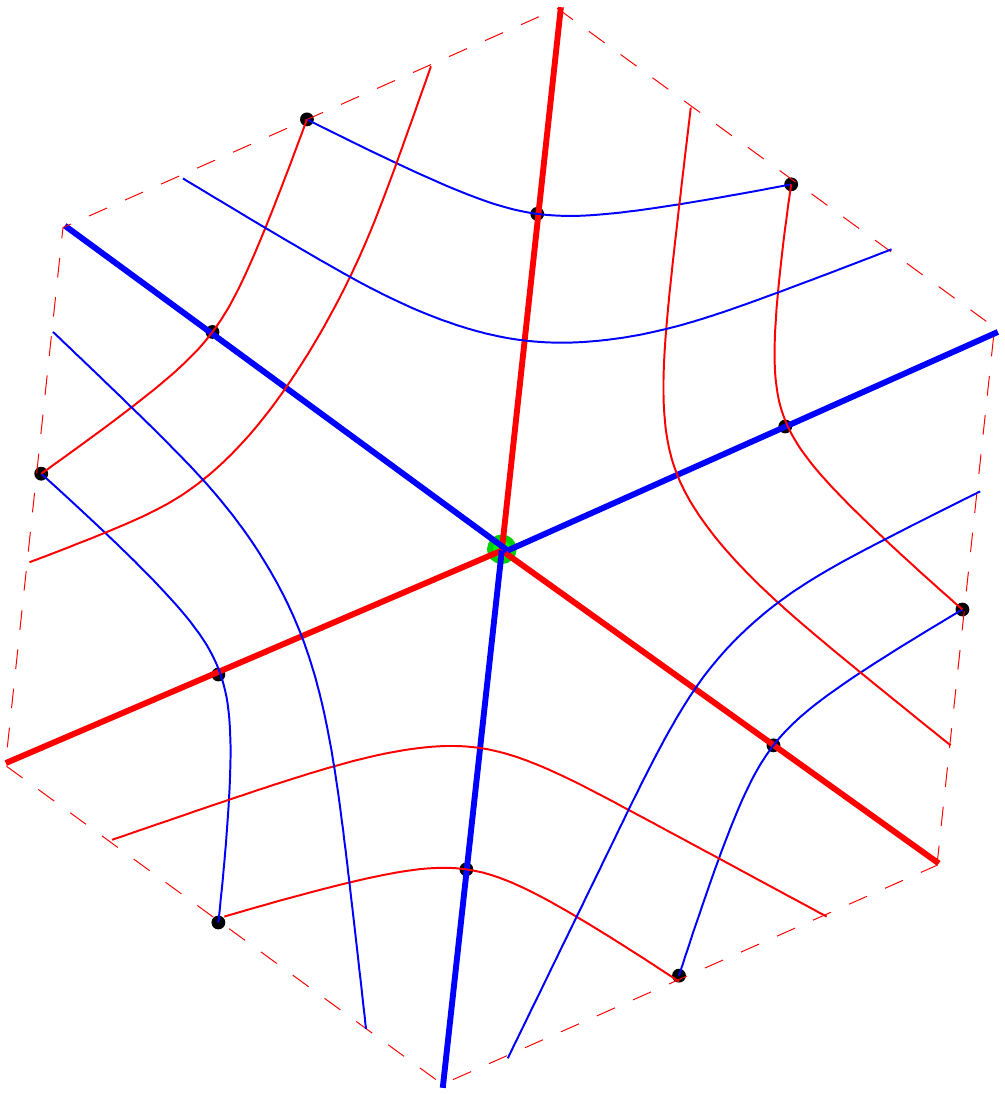}
	\caption{Regular neighborhood of a $3$-prong}
	\label{Fig: Regular neighborhood}
\end{figure}

Let $D(p,\epsilon)$ be a regular neighborhood of $p$ with side length $0<\epsilon\leq\epsilon_0$. We assume that $p$ has $k$ separatrices. Using the orientation of $S$, we label the stable and unstable separatrices of $p$ cyclically counterclockwise as $\{\delta_i^s\}_{i=1}^k(\epsilon)$ and $\{\delta_i^u\}_{i=1}^k(\epsilon)$.

We define $\delta_i^s(\epsilon)$ as the connected component of $\delta_i^s\cap D(p,\epsilon)$ containing $p$, and $\delta_i^u(\epsilon)$ as the connected component of $\delta_i^u\cap D(p,\epsilon)$ containing $p$. We assume that $\delta_i^u(\epsilon)$ is located between $\delta_i^s(\epsilon)$ and $\delta_{i+1}^s(\epsilon)$, where $i$ is taken modulo $k$.

The connected components of $\text{Int}(D(p,\epsilon_0)) \setminus (\cup_{i=i}^k \delta_{i=1}^s(\epsilon_0) \cup \delta_{i=1}^u(\epsilon_0))$ are labeled with a cyclic order and denoted as $\{E(\epsilon_0)_j(p)\}_{j=1}^{2k}$, where the boundary of $E(\epsilon_0)_1(p)$ consists of $\delta^s_1(\epsilon_0)$ and $\delta^u_1(\epsilon_0)$. These conventions lead to the following definitions.

\begin{defi}\label{Defi:converge in a sector }
	Let $\{x_n\}$ be a sequence convergent to $p$. We say $\{x_n\}$ \emph{converges to $p$ in the sector} $j$ if and only if there exists $N \in \mathbb{N}$ such that for every $n > N$, $x_n \in E(\epsilon_0)_j(p)$. The set of sequences convergent to $p$ in the sector $j$ is denoted by $E(p)_j$.
\end{defi}

The set of sequences that converge to $p$ in a sector is $\cup_{j=1}^{2k} E(p)_j$. We are going to define an equivalence relation on this set.

\begin{defi}\label{Defi: Sector equiv}
	Let $\{x_n\}$ and $\{y_n\}$ be sequences that converge to $p$ in a sector. They are in the same sector of $p$, and we denote ${x_n} \sim_q {y_n}$ if and only if $\{x_n\}$ and $\{y_n\}$ belong to the same set $E(p)_j$.
\end{defi}

\begin{rema}\label{Rema: caracterisation sim-p}
	The previous definition is equivalent to the existence of $j \in \{1, \ldots, 2k\}$ and $N \in \mathbb{N}$ such that for all $n \geq N$, $x_n, y_n \in E(\epsilon_0)_j(p)$. We are using this characterization to prove the following lemma.
\end{rema}

\begin{lemm}\label{Lemm: Equiv relation}
	In the set of sequences that converge to $p$ in a sector, $\sim_q$ is an equivalence relation. Moreover, every equivalence class consists of the set $E(p)_j$ for certain $j \in \{1, \ldots, 2k\}$.
\end{lemm}

\begin{proof}
	The less evident property is transitivity. If $\{x_n\}\sim_q \{x_n\}$ and $\{y_n\}\sim_q \{z_n\}$, suppose $x_n,y_n \in  E(\epsilon_0)_j(p)$ for $n>N_1$ and $y_n,z_n \in  E(\epsilon_0)_{j'}(p)$ for $n>N_1$. For every $n>N:=\max{N_1,N_2}$, $y_n\in  E(\epsilon_0)_j(p)\cap  E(\epsilon_0)_{j'}(p)$This is only possible if and only if $j = j'$.
\end{proof}

\begin{defi}\label{Defi: Sector}
	The equivalent class of  successions convergent to $p$ in the sector $j$  is the \emph{sector}  $e(p)_j$ of $p$.
\end{defi}

This notion is important in view of the next proposition, which establishes that generalized pseudo-Anosov homeomorphisms have a well-defined action on the sectors of a point.

\begin{prop}\label{Prop: image secto is a sector}
	Let $f$ be a generalized pseudo-Anosov homeomorphism and $p$ any point in the underlying surface with $k$ different separatrices. If  $\{x_n\}\in e(p)_j$, there is a unique $i\in \{1,\cdots, 2k\}$ such that $\{f(x_n)\}\in e(f(p))_i$. This can be summarized by saying that the image of the sector  $e(p)_j$ is the sector  $e(f(p))_{i}$.
\end{prop}

\begin{proof}
	Note that $p$ is a $k$-prong a regular points or a spine if and only if $f(p)$ is a $k$-prong, a regular point or a spine, therefore $p$ and $f(p)$ have same number of sectors. Let $0<\epsilon<\epsilon_0$ be such that $f(D(p,\epsilon)) \subset D(f(p),\epsilon_0)$, such $\epsilon$ exist because $f$ is continuous.

	Let $\delta^s(p)_j$ and $\delta^u(p)_j$ be the separatrices of $p$ which bound the set $E(\epsilon)_{j}(p) \subset S$, then $f(\delta^s(p)_j)$ and $f(\delta^u(p)_j)$ are in two contiguous separatrices of $f(p)$, which determine a unique set $E(\epsilon_0)_{i}(f(p))$ for some $i\in \{1,\cdots,2k\}$.

	Let $N\in \NN$ such that, for all $n>N$, $x_n\in  E(\epsilon)_j(p)$ This implies that, for every $n>N$, $f(x_n)\in E(\epsilon_0)_{i}(f(p)$ and the succession $\{f(x_n)\}$ is in the sector $e(f(p))_{i}$.
\end{proof}

\begin{lemm}\label{Lemm: sector contined unique rectangle}
	Let $f: S \rightarrow S$ be a generalized pseudo-Anosov homeomorphism with a Markov partition $\cR$. Let $x \in S$ and let $e$ be a sector of $x$. Then, there exists a unique rectangle in the Markov partition that contains the sector $e$.
\end{lemm}

\begin{proof}
	Let $\{x_n\}$ be a sequence that converges to $x$ within the sector $e$. Consider a canonical neighborhood $U$ of size $\epsilon > 0$ around $x$, and let $E$ be the unique connected component of $U$ minus the local stable and unstable manifolds of $x$ that contains the sequence $\{x_n\}$.
	
	By choosing $\epsilon$ small enough, we can assume that the local stable separatrix $I$ of $x$ that bound $E$ is contained in at most two rectangles of the Markov partition, and similarly, the local unstable separatrix $J$ of $x$ is contained in at most two rectangles. By take the right side of the local separatrices, we can take:  rectangles $R$ and $R'$ in the Markov partition, a horizontal sub-rectangle $H$ of $R$ that contains $I$ in its upper or lower boundary, and a vertical sub-rectangle $V$ of $R'$ whose upper or left boundary contains $J$. These rectangles can be chosen small enough such that the intersection of their interiors is a rectangle contained within $E$, denoted as $\overset{o}{Q} := \overset{o}{H} \cap \overset{o}{V} \subset E$. This implies that $R = R'$ since the intersection of the interiors of rectangles in the Markov partition is empty.
	
	Furthermore, by considering a sub-sequence of $\{x_n\}$, we don't change its equivalent sector class. Therefore, $\{x_n\} \subset \overset{o}{Q} \subset R$. This completes our proof.
\end{proof}

\input{Preliminares/SubBL}

%% file: Preliminares/SubBL.tex
 \section{Survey on Bonatti-Langevin theory for Smale diffeomorphism.}\label{Sec: Bonatti-Langevin theory}

Geometric types were introduced by Christian Bonatti and Rémi Langevin to study the dynamics of surface Smale diffeomorphisms in the neighborhood of a saddle-type saturated set without double boundaries. In a series of papers \cite{beguin2002classification}, \cite{beguin2004smale}, and his Thesis \cite{beguin1999champs}, François Béguin gives a \emph{Characterization of realizable geometric types} and develops an algorithm to determine when \emph{two geometric types represent the same saturated set}.

At some point, we will transfer our problems of \emph{realization} and \emph{decidability} from homeomorphisms to the context of basic saddle-type pieces for Smale's surface diffeomorphisms. Therefore, we will provide a review of the concepts and results that we will use, emphasizing the simplified version of the Bonatti-Langevin theory that applies to  basic pieces that don't have double boundaries. This obey that in the future we are going to deal with iteration of the diffeomorphism restricted to mixing basic piece.

The \emph{Smale Spectral Decomposition Theorem} asserts that the non-wandering set of an Axiom A diffeomorphism $f$ can be decomposed into a finite number of basic pieces. Each of these basic sets $K$ can further be decomposed into a finite disjoint union of invariant sets $\{K_i\}$, that under the iteration of $f^n$ for some power $n$, are mixing. As a consequence, these $K_i$ become basic sets of $f^n$.

This phenomenon introduces complexity, as now a basic set $K$ might only be saturated for $f^n$, and not for the original $f$. However, this complication doesn't arise if $K$ is mixing, allowing us to work within the same theoretical framework as before. We can still consider powers of $f$, their Markov partitions, and the geometric types of these partitions in a coherent manner.

\input{Preliminares/DomainBasicpiece.tex}
\input{Preliminares/Thegenus.tex}

\input{Preliminares/Geoimpassegenus.tex}

\input{Preliminares/Comimpassegenus.tex}

%% file: Preliminares/DomainBasicpiece.tex
\subsection{ The domain of a basic piece.}\label{Sub-sec: Domain basic piece}

Once again, let $f: S \rightarrow S$ be a Smale surface diffeomorphism, and let $K$ be a saddle-type basic piece of $f$. The objective of this subsection is to define an invariant neighborhood of $K$ that has $K$ as its maximal invariant set. Such a surface has finite genus, is minimal in a certain sense, and is uniquely determined up to topological conjugacy. This distinguished neighborhood is called the \emph{domain of $K$} and is denoted by $\Delta(K)$. Later on, we will define the \emph{formal derived from pseudo-Anosov} of a geometric type in terms of this domain. More details about this construction can be found in \cite[Chapter 3]{bonatti1998diffeomorphismes}.

If our domain is intended to be an invariant neighborhood of $K$, it needs to contain the stable and unstable manifolds of $K$. Moreover, for it to have finite topology, it needs to contain all polygons bounded by arcs of stable and unstable manifolds. This is the motivation for the following definition.

\begin{defi}\label{Defi: restricted domain}
The \emph{restricted domain} of $K$ is the set $\delta(K)$, consisting of the union of the stable and unstable manifolds of $K$ together with all disks bounded by a polygon formed by arcs of the stable and unstable manifolds of $K$.
\end{defi}
 
 A basic piece $\Lambda$ is always an isolated set, which means that there exists an open neighborhood $U$ of $\Lambda$ such that $\Lambda = \cap_{n=-\infty}^{\infty} f^n(U)$. This $U$ is an invariant neighborhood of $K$. This property can be deduced from the fact that every Smale diffeomorphism admits a \emph{filtration adapted} to $f$ (see \cite[Chapter 2]{shub2013global}). Such a filtration produces an invariant neighborhood $U$ of $f$ with the following properties:
  
\begin{itemize}
\item $U$ is $f$ invariant.
\item $K \subset U$, and for all compact subsets of $U$, their maximal invariant sets are contained in $K$. In other words, $K$ is the maximal invariant set of $U$.
\item $U$ admits a compactification that is a surface of finite type $S'$. Such a surface carries a Smale diffeomorphism, and the complement of $U$ in $S'$ consists of a finite number of periodic attractor or repeller points.
\item The invariant manifolds $W^s(K)$ as $W^u(K)$ are closed sub-sets of $U$.
\item We could suppose every connected component of $U$ intersect $K$.
\item Finally, $\delta(K)\subset U$.
\end{itemize}

Such $U$ is not canonical in the sense that it depends on the filtration. To obtain the domain $\Delta(K)$, Bonatti-Langevin describe the complement of $\delta(K)$ in $U$. The domain of $K$ results from extracting a neighborhood of $\delta(K)$ in $U$. This procedure is discussed in \cite[Proposition 3.1.2]{bonatti1998diffeomorphismes}.

It turns out that $U \setminus \delta(K)$ has a finite number of connected components, all of which are homeomorphic to $\mathbb{R}^2$ and periodic under the action of $f$. The specific construction of $\Delta(K)$ is described in \cite[Subsection 3.2]{bonatti1998diffeomorphismes}, but it essentially involves removing some "strips" from the complement of $U \setminus \delta(K)$. This results in a compact surface, $\Delta(K)$, from which we have removed a finite number of periodic points. We will consider the class of surfaces homeomorphic to $\Delta(K)$ with some additional properties restricted to $K$ as the domain of $K$.

\begin{defi}\label{Defi: Domain of K}
We call a neighborhood $O$ of $K$ a \emph{domain} of $K$ if it is homeomorphic to $\Delta(K)$ under a homeomorphism that coincides with the identity over $K$ and conjugates the restrictions of $f$ to $O$ and $\Delta(K)$.
\end{defi}

The uniqueness, or the canonical nature, of $\Delta(K)$ is summarized in the following proposition, corresponding to \cite[Proposition 3.2.2]{bonatti1998diffeomorphismes}.

\begin{prop}[uniqueness of domain]\label{Prop: unicity of domain}
For every neighborhood $D$ of $K$ that is $f$-invariant and contains $\delta(K)$, there exists an embedding from $\Delta(K)$ to $D$ that is the identity over the restricted domain and commutes with $f$ restricted to $\delta(K)$.
\end{prop}

\begin{defi}\label{Defi: versatile neighboorhood}
A versatile neighborhood of $K$ refers to any $D$-neighborhood of $K$ that is $f$-invariant, has finite topology (i.e., a compact surface minus a finite number of points), and finite genus. Furthermore, it satisfies the minimality property described in Proposition \ref{Prop: unicity of domain}, which states that for every $f$-invariant $D'$-neighborhood of $K$ with finite topology, there exists an embedding from $D$ to $D'$ that commutes with $f$ and is the identity over $\delta(K)$.
\end{defi}

The importance of the domain, apart from its minimality, lies in the fact that $\Delta(K)$ is unique in a specific sense. The main challenge is that the domain is neither closed nor open. In fact, it is through the use of the boundary components of the domains that we are able to connect different domains of basic pieces together and reconstruct the dynamics of the entire diffeomorphism.  Bonatti-Langevin introduce the concept of a \emph{closed module} for $f$ (refer to \cite[Definition 3.2.4]{bonatti1998diffeomorphismes}). While we won't delve into the details of this definition, the significance of it is captured in the following corollary (\cite[Corollary 3.27]{bonatti1998diffeomorphismes}).

\begin{coro}[Uniqueness of the Domain]\label{Coro: Uniqueness of domain}
Any $f$-invariant neighborhood of $K$ with finite topology that is versatile, closed modulo $f$, and such that every connected component of it intersects $K$, is conjugate to $\Delta(K)$ by a homeomorphism that coincides with the identity on $\delta(K)$.
\end{coro}

This corollary allows us to provide a representation of the Smale diffeomorphism on a basic piece in terms of its geometric type. This result is presented in  \cite[Theorem 5.2.2]{bonatti1998diffeomorphismes}, and we reproduce it below:

\begin{theo}\label{Theo: Presentation in a domain}
Let $f$ and $g$ be two Smale diffeomorphisms on compact surfaces, and let $K$ and $L$ be two basic pieces for $f$ and $g$ respectively. We denote their domains as $(\Delta(K), f)$ and $(\Delta(L), g)$. If $K$ and $L$ have Markov partitions with the same geometric types, then there exists a homeomorphism from $\Delta(K)$ to $\Delta(L)$ that conjugates $f$ and $g$.
\end{theo}

It is clear that if $f$ and $g$ are conjugate restricted to $\Delta(K)$ and $\Delta(L)$, and if $K$ admits a Markov partition of geometric type $T$, then its image under conjugation is a Markov partition of $L$ with geometric type $T$. Therefore, Theorem  \ref{Theo: Presentation in a domain} provides a necessary and sufficient condition for the conjugacy between $f$ and $g$ in terms of their geometric types and domains.

Two basic pieces $K$ and $L$ have Markov partitions of the same geometric type if and only if they have the same domain (up to conjugation), as stated in Theorem \ref{Theo: Presentation in a domain}. This allows us to provide the definition of the formal derived-from Anosov of a geometric type in the pseudo-Anosov case (see Definition \ref{Defi: Formal DA}).

%% file: Preliminares/Thegenus.tex
\subsection{ The genus of a geometric type.}\label{Sub-sec: The genus}

The pseudo-Anosov class of geometric type refers to the geometric type of Markov partitions of a generalized pseudo-Anosov homeomorphism. There exists a completely analogous notion for the geometric type of saddle-type basic pieces, which we introduce below.

\begin{defi}\label{Defi: realizable as Basic piece}
Let $T$ be a geometric type. $T$ is \emph{realizable as a (surface) basic piece} if: There exists a Smale surface diffeomorphism $f$ with a \emph{non-trivial} saddle-type basic piece $K$ and a Markov partition $\mathcal{R}$ of $K$ with geometric type $T$. In this case, we say that $K$ realizes the geometric type $T$.
\end{defi}

If we start with a geometric type $T$, we would like to know whether or not it is realizable as a basic piece of surfaces. Christian Bonatti, Rémi Langevin, and Emmanuelle Jeandenans have formulated in \cite[Chapter 7]{bonatti1998diffeomorphismes} a necessary condition for $T$ to be realizable as a basic piece of surface: the genus of $T$ must be finite. Subsequently, François Béguin has shown in \cite{beguin2004smale} that such a condition is sufficient. In this subsection, we define the genus of a geometric type and formulate a criterion that allows us to decide whether the genus is finite.

Let $T=(n,\{h_i,v_i\}_{i=1}^,\Phi)$ be a geometric type. The strategy for constructing a diffeomorphism on a compact surface (of finite type) with a basis piece realizing $T$ begins by considering a family of disjoint, non-degenerate rectangles $\cR=\{R_i\}_{i=1}^n$ together with a function $\phi$ defined piecewise on a family of horizontal sub-rectangles of $\mathcal{R}$ whose image is a family of vertical sub-rectangles of $\mathcal{R}$ satisfying the combinatorics of $\Phi$.

Any saddle-type basis piece has no double $s,u$-boundary and does not admit degenerate rectangles, i.e., rectangles reduced to a point or an interval. If we want to have hope that such a family of rectangles has, as its maximal invariant set, a non-trivial saddle-type basic piece, there is a combinatorial phenomenon dictated by $\Phi$ that we have to avoid.

Suppose there exists a subfamily of rectangles $\{R_{i(k)}\}_{k=1}^m$ of $\cR$, such that they all have a single horizontal sub-rectangle, and the function $\phi$ defined on the horizontal sub-rectangles of the family satisfies the following conditions: $\phi(R_{i(k)}\subset R_{i(k+1)}$ for $k \in \{1,\cdots, m-1\}$ and $\phi(R_{i(m)})\subset R_{i(1)}$. This condition implies that $\phi(R_{i(1)})\subset R_{i(1)}$ is a horizontal sub-rectangle of $R_{i(1)}$. Therefore, in order to have a uniform expansion (as we desire) in the unstable direction, it is necessary that $R_{i(k)}$ reduces to a stable interval. However, this is incompatible with the requirement that the desired basic piece should not have double $s$-boundaries or double $u$-boundaries. This property can be read from the geometric type $T$, which we describe in the following definition

	\begin{defi}\label{Defi: double boundary}
	A geometric type $T$ has a \emph{double $s$-boundary} if there exists a cycle of indices between $1$ and $n$:
	$$
	i_i\rightarrow i_2 \rightarrow\cdots \rightarrow i_k \rightarrow i_{k+1}=i_i.
	$$
	Such that, for all $t\in \{2,\cdots,k\}$, $h_{i_t}=1$ and $\phi_T(i_{t},1)=(i_{t+1},l)$.\\
	Using the inverse of the geometric type is possible to define \emph{double-$u$-boundary} as a double $s$-boundary of $T^{-1}$.
\end{defi}

  It was shown in \cite[Proposition 7.2.2]{bonatti1998diffeomorphismes} that a double boundary of geometric type $T$ corresponds to double boundaries on a hyperbolic set that has a Markov partition of geometric type $T$, so the analogy we present is compatible with the combinatorial formulation. Don't have double boundaries is the fist obstruction of $T$ to be realized as a saddle-type basic piece.

\begin{comment}
There is a first  obstruction of $T$ which is formulate in therms of the incidence matrix of $T$, $A(T)$.

\begin{defi}\label{Defi: Transitive matriz}
Let $A$ a square non-negative matrix. $A$ is transitive if there exist $n\in \NN$ such that all the coefficients of $A^n$ are positive, i.e. $a_{ij}^{(n)}>0$. 
\end{defi}

\begin{lemm}\label{Lemm: T realizable implies A non-boundaries}
If $A(T)$ is transitive, $T$ have no-double $s,u$-boundaries.
\end{lemm}
\begin{proof}
Suppose $T$ have a double $s$-boundary and $A$ is a matrix of $n\times n$. Let $\{e_i\}_{i=1}^n$ the canonical base of $\RR^n$. Let $\{i_j\}_{j=1}^k$, is clear the $A(e_{i_j})=e_{i_{j+1}}$ this implies that for all $m\in \NN$ $A^m(e_{i_j})=e_{i_{j+m}}$ where we take $j+m$ modulus $k$. This implies $A^m$ is not positive, a contradiction with the fact $A$ been transitive.\\
A similar analysis prove the non-existence of $u$-boundaries.
\end{proof}
\end{comment}

Once we have introduced our first obstruction for $T$, in order to construct a surface with boundaries and corners on which there exists an (affine) diffeomorphism realizing $T$, the next step is to formalize what we mean by a diffeomorphism with a combinatorics induced by $T$.
	
	\begin{defi}[Concretization]\label{Defi: Concretization}
			Let $T=(n,\{h_i,v_i\}{i=1}^n, \Phi_T)$ be a geometric type without double boundary. A \emph{concretization} of $T$ is a pair $(\cR=\{R_i\}_{i=1}^n,\phi)$ formed by a family of $n$ oriented, non-degenerate, and disjoint rectangles $\cR$  and a function $\phi$ defined over a subset of $\cup \cR:=\cup_{i=1}^n R_i$ with the following characteristics:
		
		\begin{enumerate}
			\item Every rectangle is endowed with the trivial vertical and horizontal foliations. We give an orientation to the vertical leaves of $R_i$ and let the horizontal leaves be oriented in such a way that a pair of horizontal and vertical leaves gives the orientation chosen for $R_i$.
			
			\item For every $1\leq i \leq n$ and $1\leq j \leq h_i$, there is a horizontal sub-rectangle $H^i_j$ of $R_i$. The sub-rectangles are disjoint and their order is compatible with the vertical orientation in $R_i$. We demand that the lower (upper) boundary of $H^i_1$ ($H^i_{h_i}$) coincides with the lower (upper) boundary of $R_i$.
					
			\item For every $1\leq k \leq n$ and $1\leq l \leq v_k$, we have a vertical sub-rectangle $V^k_l$ of $R_k$. The sub-rectangles are disjoint and the order is compatible with the horizontal orientation in $R_k$. We demand that the left (right) boundary of $V^k_1$ ($V^k_{v_k}$) coincides with the left (right) boundary of $R_k$.
			
			\item The function $\phi:\cup_{i,j}H^i_j \rightarrow \cup_{k,l}V^k_l$ is  orientation-preserving diffeomorphism restricted to every $H^i_j$ that preserves the horizontal and vertical foliations.
			
			\item If $\Phi_T(i,j)=(k,l,\epsilon)$, then $\phi(H^i_j)=V^k_l$.

			\item If $\epsilon(i,j)=1$, $f$ preserves the orientation of the vertical direction restricted to $H^i_j$, and $f$ reverses it in the case $\epsilon(i,j)=-1$.
		\end{enumerate}
	\end{defi}
	
	We have obtained a graphical representation of $T$, but another problem arises: the function $\phi$ is not necessarily hyperbolic. To ensure hyperbolicity, we require the existence of an integer $m\in\mathbb{N}$ such that, for all points in its maximal invariant set, the derivative of $\phi^m$ expands vectors in the vertical direction and contracts horizontal vectors uniformly. This leads us to the following definition.
	
	\begin{defi}[realization]\label{Defi: realization}
Let $T$ be a geometric type without double boundaries. A concretization of $T$ is considered a realization if the maximal invariant set of $\phi$ is hyperbolic.
	\end{defi}

An easy way to define $\phi$ is as an affine transformation over every horizontal sub-rectangle. Such concretizations are called \emph{affine concretizations}. It turns out that every affine concretization is a realization, and every geometric type without a double boundary admits an affine concretization. This implies the result of \cite[Proposition 7.2.8]{bonatti1998diffeomorphismes}, which we reproduce below.
	
	\begin{prop}\label{Prop: Types transitive have a realization}
	Every geometric type without a double boundary admits a realization.
	\end{prop}

Starting from a realization $(\cR=\{R_i\}_{i=1}^n,\phi)$ of $T$, and for each positive natural number $m \in \mathbb{N}+$, we will construct the "minimal" surface with boundaries and corners that supports $m-1$ iterations of $\phi$. This construction involves gluing horizontal and vertical sub-rectangles of the realization according to the rules determined by the diffeomorphism $\phi$. Let us introduce a general definition that will be helpful in our discussions.

\begin{defi}\label{Defi: Vetical-horizontal stripes}
For every $i\in \{1,\cdots,h_i-1\}$, the \emph{horizontal stripe} $\tilde{H^i_j}$ is the vertical sub-rectangle of $R_i$ that is bounded by $H^i_j$ and $H^i_{j+1}$. Similarly, we define the \emph{vertical stripe} $\tilde{V^k_l}$ with $l\in \{1,\cdots,v_k-1\}$.
\end{defi}

\begin{figure}[h]
	\centering
	\includegraphics[width=0.5\textwidth]{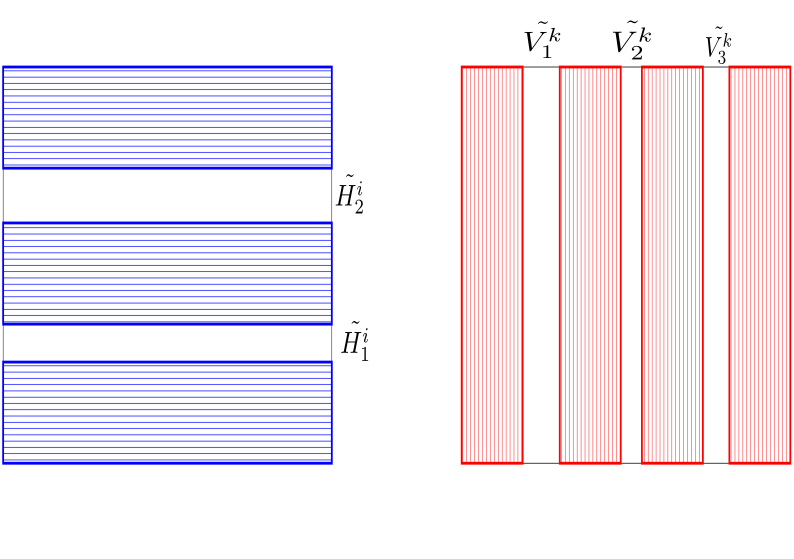}
	\caption{Horizontal and vertical stripes}
	\label{Fig: Stripes}
\end{figure}

We refer to the compact surface $S$ with boundary and vertices as an \emph{HV-surface} if it has an atlas whose charts are open sets contained in the upper half-plane, the first quadrant of the origin in $\RR^2$, or the union of the three quadrants of the origin in the real plane. Furthermore, it is required that the changes of charts are elements of Hom$(\RR)\times$Hom$(\RR)$. This is the type of surface where we will define and iterate $\phi$. It is not immediate that we can extend the dynamics to a finite-type compact surface, but this is the next problem we are going to address.

	\begin{defi}\label{Defi: m realizer }
	Let $T$ be a geometric type without double boundaries, and let $(\{R_i\}_{i=1}^n,\phi)$ be a realization of $T$. We denote the union of the rectangles as $\cR=\cup_{i=1}^n R_i$, and for each $m \in \NN$, we consider the disjoint union of $m$ copies of $\cR$ given by:
		$$
	\cup_{i=0}^m \cR \times \{i\}.
		$$
We define $\cR_m$ as the quotient space of this set by the equivalence relation: for all  $ 1\leq i \leq m-1$
		$$
		 (x,i)\sim(\phi^{-1}(x),i+1),
		$$
for each point where $\phi^{-1}$ is defined, and we leave $(x,i)$ unidentified otherwise.		
 The space $\cR_m$ is the $m$-th \emph{realizer} of $T$ relative to the realization $(\{R_i\}_{i=1}^n,\phi)$. 
	\end{defi}

The space $\cR_m$ is an oriented HV-surface and is endowed with two transverse foliations: the vertical and horizontal foliations (refer to \cite[Section 7.1]{bonatti1998diffeomorphismes} and \cite[Lemma 7.3.3]{bonatti1998diffeomorphismes} for more details). An HV-homeomorphism is a homeomorphism between HV-surfaces that conjugates the vertical and horizontal foliations.
	
For all $m\in \NN$  and every $i\leq m$, there exist an embedding:  
$$
\Psi :\cR\times \{i\} \rightarrow \cR_{m'},
$$

defined as $\Psi(x,i)=[x,i]_{m'}$, which turns out to be an HV-homeomorphism. We denote $R_i$ as $R_i:=\Psi(R_i\times \{0\})$. Let's consider two natural numbers $m' > m$ and $\Psi: \cR \times {i} \rightarrow \cR_{m'}$ as described above. Let's take
 $$
 \phi_m: \cup_{i=0}^{m-1} \Psi(\cR\times \{i\}) \subset \cR_{m'} \rightarrow \cup_{i=1}^m \Psi(\cR\times \{i\}) \subset \cR_{m'},
 $$
given by $\phi_m(\pi(x,i))=\Psi(x,i+1)$.  Provided $\phi(x)$ is defined, $\phi_m(\pi(x_i))=[(x,i+1)]_{m'}=[\phi(x),i]_m'=\Psi(\phi(x,i))$. In this way, the space $\cR_m$ is considered immersed in $\cR_{m'}$ and we have endowed $\cR_m$ with some 'dynamics' that is possible to iterate at most $m$ times.

	\begin{rema*}
 When it makes sense, $\phi_m([x,i])=[\phi(x),i]$, and for all $0 \leq k \leq m$, the image by $\Psi$ of $R_i \times {k}$ is equal to $\phi^k_m(R_i\times {0})$.
	\end{rema*}
	
The following proposition essentially states that there is a unique $m$-realizer up to HV-conjugation. It is proved in \cite[Proposition 7.3.5]{bonatti1998diffeomorphismes}.
 
	\begin{prop}\label{Prop: unique m-realizer}
		Let $T$ be a geometric type,$(\{R_i\},\phi)$ and $(\{R_i'\},\phi')$ be two realizations of $T$, and $\cR_m$ and $\cR_m'$ be their respective $m$-realizers. Then there exists an HV-homeomorphism $\theta:\cR_m \rightarrow \cR_m'$ which conjugates $\phi_m$ to $\phi_m'$. Moreover, the image of $R_i$ under $\theta$ is $R_i'$, and $\theta$ preserves the orientation of the two foliations on the rectangles.	
	\end{prop}
	
Suppose there exists a basic piece $K$ of a Smale surface diffeomorphism $f$ with a geometric Markov partition $\cR$ of geometric type $T$. We would like to compare the surface generated by the iterations $f^n(\cR)$ ($n\geq 0$) of the Markov partition and the $m$-realizer of the geometric type $T$. The \cite[Proposition 7.3.9]{bonatti1998diffeomorphismes} that we cite below provides insight into this problem.

	\begin{prop}\label{Prop: realizer and Markov partition}
	Let $T$ be the geometric type of a geometric Markov partition $\cR:=\{R_i\}$ of a basic piece $K$ of a Smale surface diffeomorphism $f$. If $\cR_m$ is the $m$-realizer of $T$ relative to any realization of $T$, then $\cup_{j=0}^m(\cup_{i}f^j(R_i))$  is an HV-surface, HV-homeomorphic to $\cR_m$.
	\end{prop}
	
The dynamics of $f$ restricted to the Markov partition $\cR$ induces a realization of $T$. Proposition \ref{Prop: unique m-realizer} implies that $f^m$ restricted to $\cup_{j=0}^m(\cup_{i}f^j(R_i))$ is conjugate to $\phi_m$. Each realizer $\cR_m$ has finite genus, and we define $g_m := \text{gen}(\cR_m)$ as the genus of $\cR_m$. Since $\cR_m$ is embedded in $\cR_{m+1}$, the sequence $\{g_m\}_{m\in \NN}$ is non-decreasing. Therefore, it is either stationary or tends to infinity. This observation allows us to define the genus of $T$.
	
	\begin{defi}\label{Defi: Genus of T}

Let $T$ be a geometric type without double boundaries. The genus of $T$ is the upper bound of the sequence $\{g_m\}_{m\in \NN}$ and is denoted by $\text{gen}(T)$.
	\end{defi}

	If the geometric type $T$ is realized as a basic piece, then there exists a surface of finite genus containing the realizer $m$ of $T$ for all $m in \NN$. It follows that the genus of $T$ is finite, as shown in \cite[Corollary 7.3.10]{bonatti1998diffeomorphismes}. Later, François Béguin proved in \cite[Corollary 4.4]{beguin1999champs} that this is a sufficient condition. Thus, we obtain the following theorem:
	
	\begin{theo}\label{Theo: finite genus iff realizable}

Let $T$ be a geometric type. $T$ is realizable as a saddle type basic piece of a surface Smale diffeomorphism if and only if $T$ has no double boundaries and gen$(T)$ is finite.
	\end{theo}
	
This is our main tool for determining whether or not a geometric type is realizable. But how can we check if $T$ does not have double boundaries and has finite genus? The next few sections are devoted to formulating a mechanism to verify this.

%% file: Preliminares/Geoimpassegenus.tex
\subsection{ Topological formulation of finite genus and impasse.}\label{Sub-sec: Top impasse genus}

In this subsection, we assume that $T$ is a geometric type without double boundaries, and $\cR_m$ is the $m$-realizer of $T$ with respect to a realization $(\{R_i \}_{i=1}^n,\phi)$. Our goal is to describe three topological obstructions for the geometric type $T$ to have finite genus. These obstructions are formulated in terms of the $m$-realizer and certain subsets called \emph{ribbons}.

\begin{defi}\label{Defi: k Ribbon}
 For each $k$ in $\{1,\ldots,m\}$, we define a \emph{ribbon} of $k$-generation of $\cR_m$ as the closure of any connected component of $\cR_k \setminus \cR_{k-1}$. Thus, a ribbon of $\cR_m$ can be any ribbon from the $k$-generation for $0 < k \leq m$.
\end{defi}

\begin{comment}

\begin{rema}\label{Rema: Caracterization of ribbon}
Is proved in \cite[Lemma 7.1.6]{bonatti1998diffeomorphismes} that the $k$-generation ribbons are disjoint rectangles whose horizontal boundaries are include in the horizontal boundary of $\cR_m$, denoted $\partial^h\cR_m$ and in fact (\cite[Lemma 7.1.5]{bonatti1998diffeomorphismes} ), they are contain in $\partial^h \cR_0$. \\
Note that: a ribbon of $k$-generation is characterize by be contain in $\pi(\cR_0 \times \{k\})$ but it is not in the image by $\phi_m$ of  $\pi(\cR_0\times  \{k-1\})$, then  the ribbon correspond to a c.c. of $\pi^k_m(\cR_0\setminus \pi^{-1}_m(\cR_0))$, i.e. the ribbon is the image by $\pi^k_m$ of some horizontal stripe $\tilde{H^i_j}$ of $\cR_0$.
\end{rema}

\cite[Lemma 7.4.4]{bonatti1998diffeomorphismes} ensure something more: for all $m>0$ the closure of every c.c. of $\cR_m\setminus \cup_{i=0}^n (R_i\times \{0\})$ correspond to a ribbon of $k$-generation for $k\leq m$. This permit to give the following equivalent definition

\begin{defi}\label{Defi: Ribbon}
We cal ribbon of $\cR_m$ to any c.c. of  $\cR_m\setminus \cup_{i=0}^n (R_i\times \{0\})$ .
\end{defi}
\end{comment}

The following proposition summarizes some of the properties of the ribbons that are relevant to our discussion. A detailed exposition of these properties can be found in \cite[Chapter 7.4]{bonatti1998diffeomorphismes}.

\begin{prop}\label{Prop: properties ribbons}
Let $T$ be a geometric type without double boundaries, and let $(\{R_i\}_{i=1}^n, \phi)$ be a realization of $T$. Consider $0 < k \leq m$. The following propositions hold: 
\begin{itemize}
\item Every ribbon of $k$-generation is a rectangle whose boundary is contained in the horizontal boundary of $\cR$: $\partial^h\cR_0 = \cup_{i=1}^n \partial^h(\cup \pi(R_i \times {0}))$. Moreover, if $B$ is the complement of the interior of the vertical stripes and their positive iterations by $\pi$ contained in the horizontal boundary of  $\partial^h\cR_0$, then for all $m > 0$, the stripes of $\cR_m$ have horizontal boundaries in $B$.

\item A $k$-generation ribbon is the image of the horizontal stripes in $\cR_0$ under $\phi^k_m$, or equivalently, the image of the horizontal stripes $\tilde{H^i_j}\times \{k\}$ under $\pi$.

\item For all $m > 0$, the closure of each connected component of $\cR_m \setminus \cR_0$ is a $k$-generation ribbon for $0 < k \leq m$.
\end{itemize}
\end{prop}

From this result, we can deduce that there is a correspondence between $k$-ribbons and consecutive horizontal sub-rectangles of $\phi^k$. This correspondence is established as follows. Let $0 < k \leq m$ and $r$ be a $k$-generation ribbon of $\cR_m$. Suppose $r = \phi^k_m(\tilde{H^i_j} \times {0})$. Let $H^{i,(k)}{j}$ and $H^{i,(k)}{j+1}$ be two consecutive horizontal sub-rectangles of the rectangle $R_i = \pi(R_i \times {0})$ determined by the map $\phi^k$, such that $\tilde{H^i_j}$ is between them. Since $\phi^k$ is well-defined for such sub-rectangles, we have:

$$
\phi^k_m(H^{i,(k)}_{j}\times \{0\} )=\phi^k(H^{i,(k)}_{j})\times \{0\}=V^{z,(k)}_{l} \times  \{0\}
$$
and 
$$
\phi^k_m(H^{i,(k)}_{j=1}\times \{0\} )=\phi^k(H^{i,(k)}_{j})\times \{0\}=V^{z',(k)}_{l'}\times \{0\}.
$$
Where $V^{z,(k)}_{l}$ and $V^{z',(k)}_{l'}$ are horizontal sub-rectangles of the Markov partition $\cR_0=\{R_i=\pi(R_i\times \{0\})\}$  for the map $\phi^k$. i.e $(\cR_0,\phi^k)$.

Now we are ready to provide the topological criteria, in terms of the realization $\cR_m$, for $T$ to have finite genus. In the following definitions, as usual, we consider a geometric type  without double boundaries:
$$
T=(n,\{(h_i,v_i)\}_{i=1}^n, \Phi:=(\rho,\epsilon)),
$$
 and $(\{R_i\}_{i=1}^n, \phi)$  is a realization of $T$.
 
We reiterate that when we refer to the Markov partition on $\cR_m$, we mean the family of rectangles $\cR_0=\{\pi(R_i\times \{0\})\}$ under the action of $\phi_m$. However, formally speaking, it is not a Markov partition in the strict sense of the definition, as there is currently no complete diffeomorphism for which it satisfies all the properties of a Markov partition according to our definition.

\begin{defi}\label{Defi: Type 1 obstruction top}
Let $R$ be a rectangle in the Markov partition $\cR_0$, $A$ be a horizontal boundary component of $R$, and let $r$ be a ribbon with both horizontal boundaries in $A$.
The Markov partition in $\cR_m$ has a topological obstruction of type-$(1)$ if there exists another ribbon $r'$, distinct from $r$, with a single horizontal boundary contained in $A$ and lying between the horizontal boundaries of $r$. (See Figure \ref{Fig: obstruction 1}) for a visual representation.)

\begin{figure}[hh]
	\centering
	\includegraphics[width=0.3\textwidth]{Preliminares/Img/Type1obs.png}
	\caption{ $(1)$-type obstruction.}
	\label{Fig: obstruction 1}
\end{figure}
\end{defi}

\begin{defi}\label{Defi: Type 2 obtruction top}
Let $A$ and $A'$ be two distinct horizontal boundary components of the Markov partition $\cR_0$. Consider a ribbon $r$ with one horizontal boundary contained in $A$ and the other boundary in $A'$.

The ribbon $r$ has a horizontal orientation when viewed as a horizontal sub-rectangle of a certain rectangle in the Markov partition, $\pi(R_i\times \{k\})$ (with $k\leq m$). This orientation may not correspond to the orientations of $A$ or $A'$ directly. However, we can assign orientations to $A$ and $A'$ that are compatible with the orientation of the horizontal boundary of $r$.
 
The Markov partition has a topological obstruction of type-$(2)$ in $\cR_m$ if there exists another ribbon $r'$ with one horizontal boundary $\alpha$ contained in $A$ and the other horizontal boundary $\alpha'$ contained in $A'$. With the orientations previously fixed on $A$ and $A'$, the order of $\alpha$ and $r\cap A$ in $A$ is the reverse of the order of $\alpha'$ and $r\cap A'$ in $A'$. (See Figure \ref{Fig: obstruction 2} for a graphical representation.)

\begin{figure}[hh]
	\centering
	\includegraphics[width=0.4\textwidth]{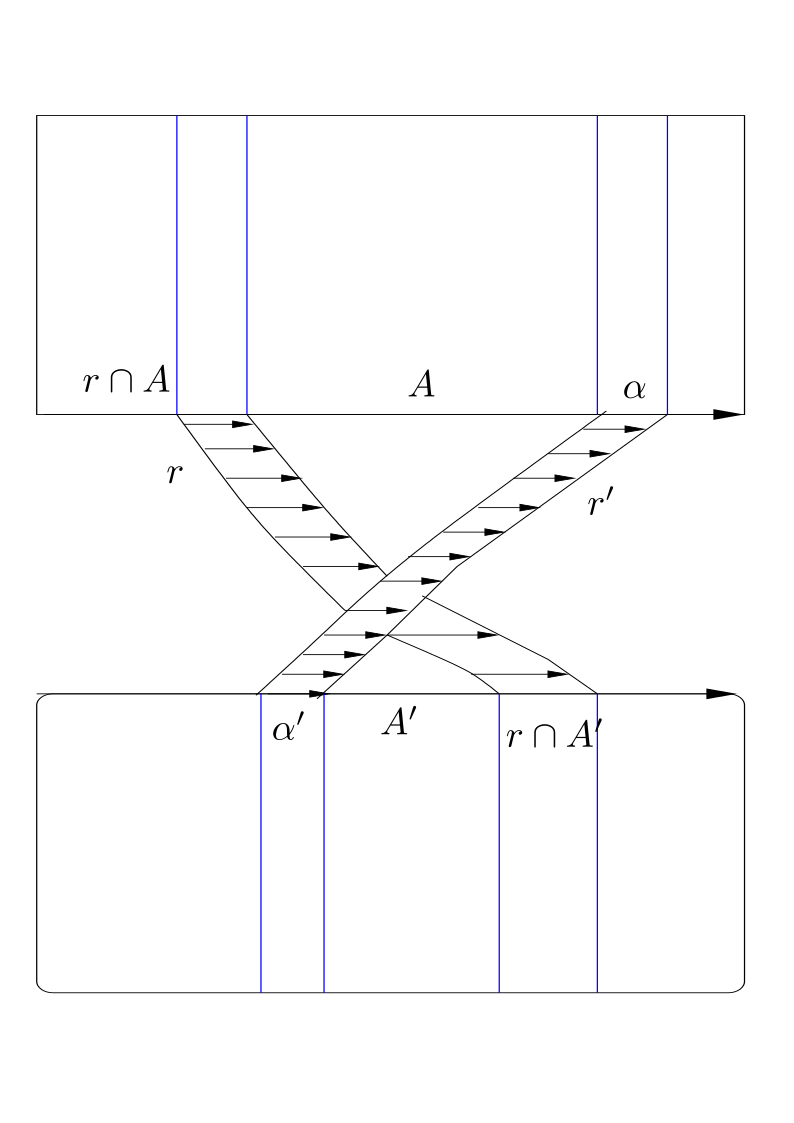}
	\caption{ $(2)$-type obstruction.}
	\label{Fig: obstruction 2}
\end{figure}
\end{defi}

\begin{defi}\label{Defi: p-peridic embrinary}
A horizontal boundary $A$ of a rectangle $R$ in the Markov partition is said to be $p$-\emph{periodic} if $\phi^p(A)\subset A$. If $A$ is $p$-periodic, we refer to any connected component of $A\setminus \pi^{kp}(A)$, with $k\in \NN_{>0}$, as an \emph{embryonic separatrix} of $A$.
\end{defi}

\begin{defi}\label{Defi:  Type 3 obtruction geo}
Let $S_1$, $S_2$, and $S_3$ be three different embryonic separatrices (i.e., they are on periodic sides but not in the image of the iterations of that side), and let $r$ be a ribbon with one horizontal boundary on $S_1$ and the other on $S_2$. We say that the Markov partition has a type-$(3)$ obstruction in $\cR_m$ if there exists another ribbon $r'$ with one side in $S_1$ and the other in $S_3$ (see Figure \ref{Fig: obstruction 3}).

\begin{figure}[hh]
	\centering
	\includegraphics[width=0.5\textwidth]{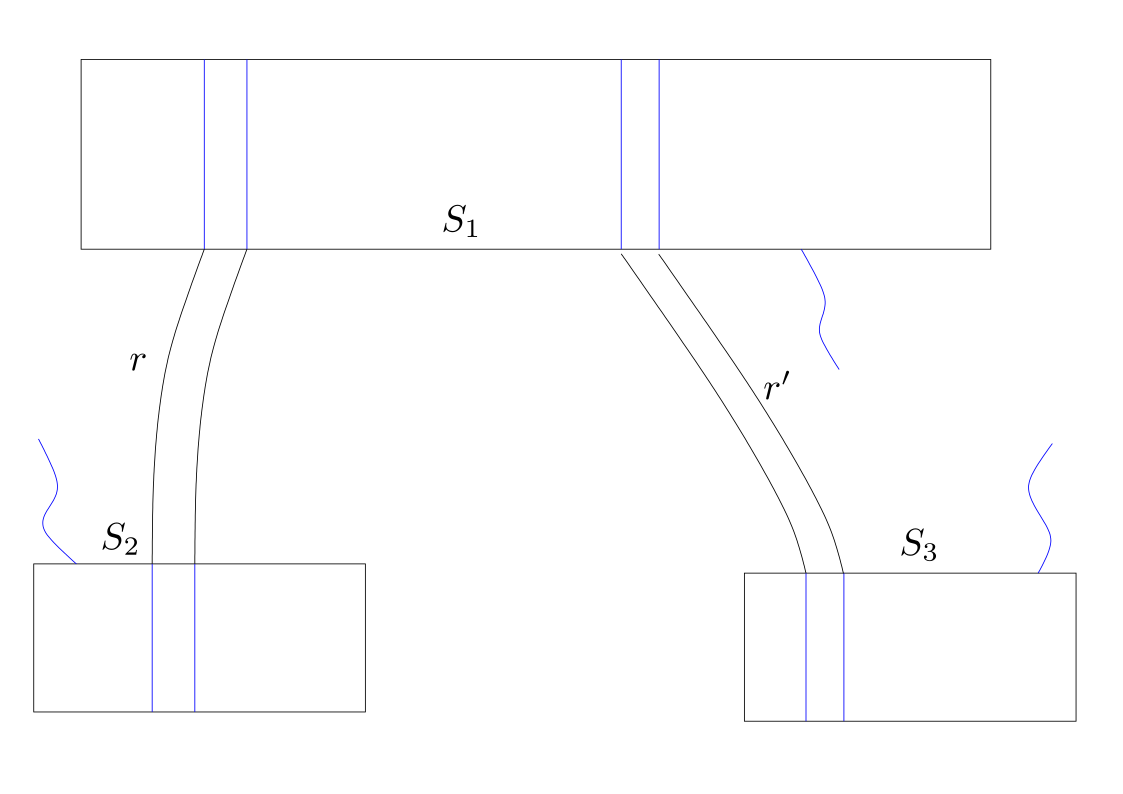}
	\caption{ $(3)$-type obstruction.}
	\label{Fig: obstruction 3}
\end{figure}

\end{defi}

 The following theorem is key to our future analysis, as it allows us to determine whether the genus of $T$ is finite by verifying the obstructions on a bounded number of realizers. This result corresponds to \cite[Theorem 7.4.8]{bonatti1998diffeomorphismes}. Here, we reformulate it in the context of basic pieces, as we will be using it.

\begin{theo}\label{Theo: finite type iff non-obtruction}
Let $T$ be a geometric type without double boundaries, let $(\{R_i\}_{i=1}^n,\phi)$  be a realization of $T$, and let $\cR_{6n}$ be its $6n$-realizer, where $n$ is the number of rectangles in the realization of $T$. The following statements are equivalent:

\begin{itemize}
	\item[i)] The Markov partition does not have obstructions of types $(1)$, $(2)$, and $(3)$ in $\cR_{6n}$.
	\item[ii)] The genus of $T$ is finite, i.e., $\text{gen}(T) < \infty$.
\end{itemize}
\end{theo}

Let's formulate another topological condition on the $6n$-realizer of geometric type $T$. This condition will be important in the problem of realization in the context of pseudo-Anosov homeomorphisms. We begin by assuming that $K$ is a basic piece of a Smale diffeomorphism of surfaces and that this piece has a Markov partition $\cR$ of geometric type $T$.

\begin{defi}\label{Defi: arc}
An $s,u$-arc is an interval contained in the unstable or stable foliation (respectively) of $K$ whose interior does not intersect $K$ but whose ends do.
\end{defi}

\begin{defi}\label{Defi: Impasse geo}
Let $f:S\rightarrow S$ be a surface Smale diffeomorphism, and let $K$ be a saddle-type basic piece of $f$. We call  \emph{topological impasse}, any open disk $\overset{o}{D}\subset S$ that is disjoint from $K$ and whose boundary consists of the union of a $u$-arc with an $s$-arc.
\end{defi}

\begin{figure}[h]
	\centering
	\includegraphics[width=0.6\textwidth]{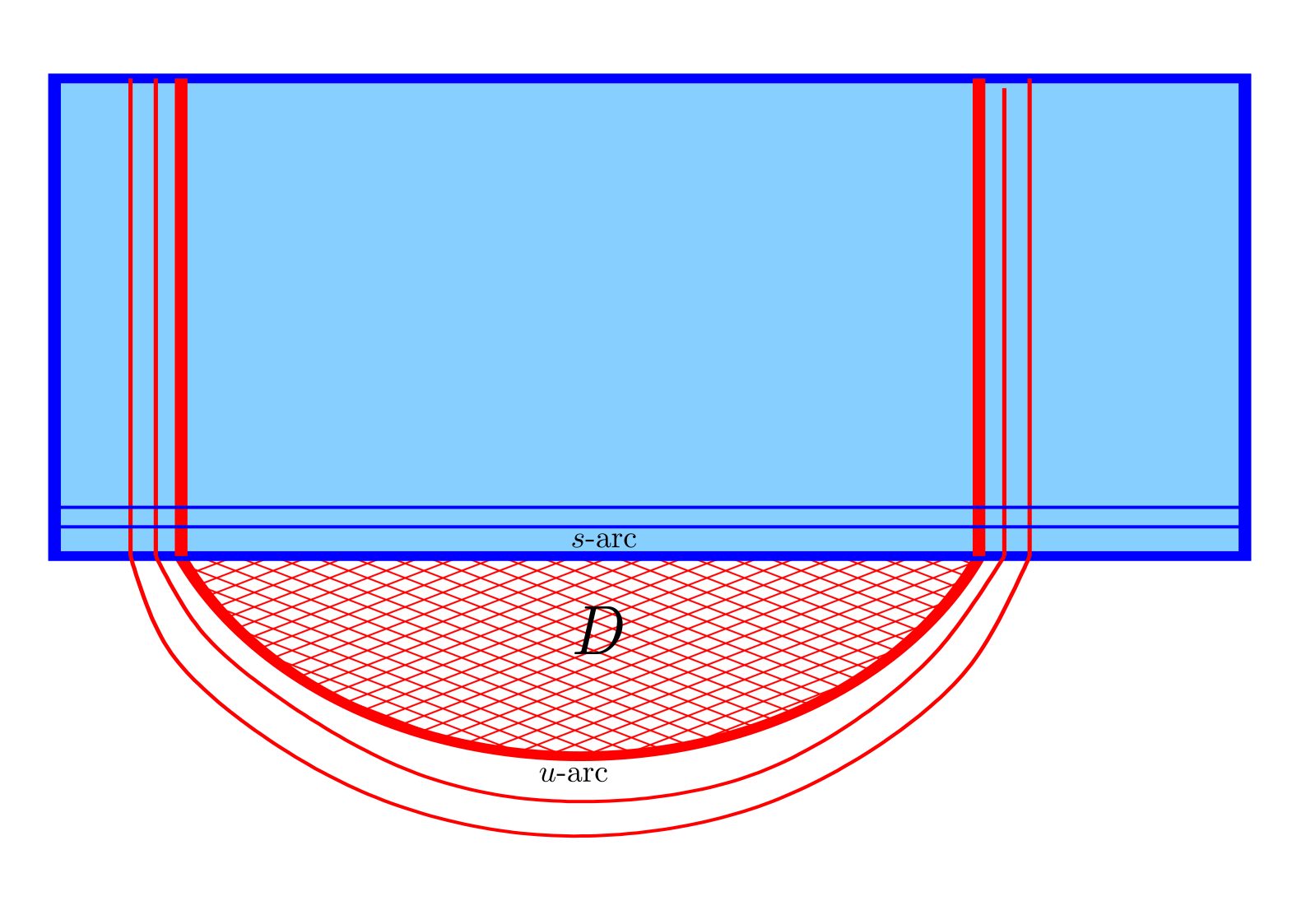}
	\caption{ Topological Impasse}
	\label{Fig: Impasse}
\end{figure}

Two $u$-arcs $\beta$ and $\beta'$ are equivalent if there is a rectangle that has $\beta$ and $\beta'$ as unstable boundaries. This relation defines an equivalence relation between $u$-arcs, and the equivalence classes are rectangles. An $u$-arc on the boundary of an equivalence class is called an extremal $u$-arc. Similarly, an equivalence relation is defined between $s$-arcs, and the concept of $s$-extremal arc is introduced (see \cite[Definition 2.2.6]{bonatti1998diffeomorphismes} for a precise definition). In \cite[Corollary 2.4.8]{bonatti1998diffeomorphismes}, it is proven that the boundary arcs of an impasse are extremal arcs.

\begin{rema}\label{Rema: extremal arc}

The important property of the arcs bounding an impasse is that the $u$-arc is saturated by leaves of the unstable lamination on the side opposite to the impasse, just as the $s$-arc is saturated by leaves of the stable lamination on the opposite side of the impasse.

\end{rema}

%% file: Preliminares/Comimpassegenus.tex
\subsection{ Combinatorial formulation of finite genus and impasse.}\label{Sub-sec: Comb Impasse genus}

The obstructions described above provide a valuable tool for analyzing our Markov partition from a visual perspective. However, in order to obtain an algorithmic characterization of the pseudo-Anosov class of geometric types, we need to translate these three topological obstructions and the impasse condition into a combinatorial formulation that allows us to determine in finite time whether or not these conditions are satisfied.

The key insight lies in Proposition \ref{Prop: properties ribbons}, where we observed that every ribbon of $k$-generation corresponds to the image of a horizontal strip between two consecutive horizontal sub-rectangles of the Markov partition for the map $\phi^k$ under the projection $\pi$. In this section, we will consider the geometric type.

$$
T=(n,\{(h_i,v_i)\}_{i=1}^n,\Phi_T:=(\rho_T,\epsilon_T )).
$$

The trick in our upcoming definitions is that whenever we want to talk about a "ribbon" with a certain property, we need to refer to two consecutive horizontal rectangles and analyze their images under $\phi^k$. This way, a "ribbon" is formed, and we can express the obstructions and the impasse property in terms of the images of these consecutive horizontal rectangles.

\subsubsection{Combinatorial obstructions for finite genus.}
\begin{defi}\label{Defi: Type 1 combinatoric}
	A geometric type $T$ satisfies the \emph{combinatorial condition} of type-$(1)$ if and only if there are:
	
	\begin{itemize}
		\item Three (not necessarily distinct) numbers, $k,i_0,i_1\in \{1,\cdots,n\}$.
		\item Three different  numbers, $l_0,l_1,l_2\in \{1,\cdots,v_{k}\}$ with $l_0<l_1<l_2$.
		\item A pair  $(k',l')\in\cV(T)$, such that $(k',l')\notin \{(k,l):l_0\leq l \leq l_2\}$.
		\item And four pairs, $(i_0,j_0),(i_0,j_0+1),(i_1,j_1)(i_1,j_1+1)\in \cH(T)$
		\end{itemize}
	Such that they satisfy one of the following relations:
\begin{enumerate}
\item If  $\rho_T(i_0,j_0)=(k,l_0)$ and $\rho_T(i_0,j_0+1)=(k,l_2)$, we have two available possibilities:
\begin{itemize}
	\item[i)] $\epsilon_T(i_0,j_0)=1$ and $\epsilon_T(i_0,j_0+1)=-1$.  In such a case:
	 $$
	 \rho_T(i_1,j_1)=(k,l_1), \, \epsilon_T(i_1,j_1)=1 \text{ and } \rho_T(i_1,j_1+1)=(k',l'),
	 $$
	  or well,
	  $$
	  \rho_T(i_1,j_1+1)=(k,l_1), \,  \epsilon_T(i_1,j_1+1)=-1,  \text{ and } \rho_T(i_1,j_1)=(k',l').
	  $$
	
	\item[ii)] $\epsilon_T(i_0,j_0)=-1$ and $\epsilon_T(i_0,j_0+1)=1$. In such a case  
	$$
	\rho_T(i_1,j_1)=(k,l_1), \,  \epsilon_T(i_1,j_1)=-1 \text{ and  }\rho_T(i_1,j_1+1)=(k',l')
	$$
	or well
	$$
	\rho_T(i_1,j_1+1)=(k,l_1), \,  \epsilon_T(i_1,j_1)=1 \text{ and  } \rho_T(i_1,j_1)=(k',l').
	$$.

\end{itemize}
\item  In the symmetric case when, $\rho_T(i_0,j_0)=(k,l_2)$ and $\rho_T(i_0,j_0+1)=(k,l_0)$ there are two options:
\begin{itemize}
	\item[i)] $\epsilon_T(i_0,j_0)=1$ and $\epsilon_T(i_0,j_0+1)=-1$. In such a case: 
	$$
	\rho_T(i_1,j_1)=(k,l_1), \, \epsilon_T(i_1,j_1)=1  \text{ and }  \rho_T(i_1,j_1+1)=(k',l')
	$$
	or well
	$$
	\rho_T(i_1,j_1+1)=(k,l_1), \, \epsilon_T(i_1,j_1+1)=-1  \text{ and }  \rho_T(i_1,j_1)=(k',l')
	$$

	\item[ii)]  $\epsilon_T(i_0,j_0)=-1$ and $\epsilon_T(i_0,j_0+1)=1$. In such a case:  
	$$
	\rho_T(i_1,j_1)=(k,l_1),\, \epsilon_T(i_1,j_1)=-1 \text{ and  } \rho_T(i_1,j_1+1)=(k',l')
	$$
	or well
	$$
\rho_T(i_1,j_1+1)=(k,l_1),\, \epsilon_T(i_1,j_1)=1 \text{ and  } 	\rho_T(i_1,j_1)=(k',l').
	$$
\end{itemize}
\end{enumerate}
	A geometric type has the \emph{combinatorial obstruction} of type-$(1)$ if there exists $m \in \mathbb{N}_{>0}$ such that $T^m$ satisfies the combinatorial condition of type-$(1)$.
\end{defi}

The way we have defined the combinatorial condition of type $(1)$ is by identifying the $1$-generation ribbons of the Markov partition $\cR_0$ in $\cR_1$ that give rise to the topological obstruction of type-$(1)$ in $\cR_1$ as the image of a horizontal stripe contained between two consecutive sub-rectangles. This understanding of ribbons corresponds to two consecutive elements $(i,j)$ and $(i,j+1)$ in $\cH$ and their images under $\rho_T$ in $\cV$, which represent the ribbons. In the following lemma we are going to follow the Figure \ref{Fig: Type one proof} to explain or proof.

\begin{figure}[h]
	\centering
	\includegraphics[width=0.83\textwidth]{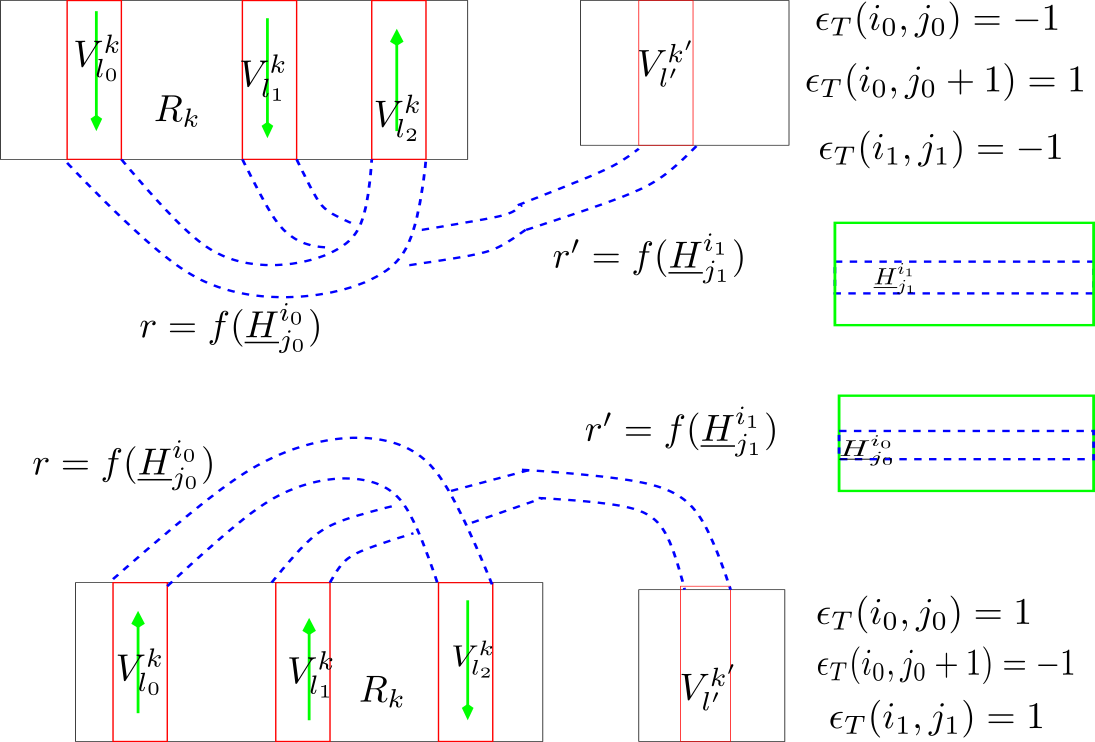}
	\caption{Type $(1)$ combinatorial condition}
	\label{Fig: Type one proof}
\end{figure}

\begin{lemm}\label{Lemm: Equiv com and top type 1}
A geometric type $T$ has the combinatorial condition of type-$(1)$ if and only if the Markov partition in $\cR_1$ exhibits the topological obstruction of type-$(1)$.
\end{lemm}

\begin{proof}
The terms $(k,l_0)$, $(k,l_1)$, and $(k,l_2)$ represent three different vertical sub-rectangles of $R_k$ ordered by the horizontal orientation of $R_k$. The term $(k',l')$ represents a horizontal sub-rectangle of $\cR$ that is not located between the rectangles $V^k_{l_0}$ and $V^k_{l_2}$. Let's illustrate the meaning of the condition in the first case, as the rest are symmetric.

The conditions $\rho(i_0,j_0)=(k,l_0)$ and $\rho(i_0,j_0+1)=(k,l_2)$ are equivalent to the fact that the image under $f$ of the horizontal strip $\underline{H}_{j_0}$ bounded by $H^{i_0}_{j_0}$ and $H^{i_0}_{j_0+1}$ is a ribbon $r=f(\underline{H}_{j_0})$ of the first generation that connects the horizontal boundaries of $V^k_{l_0}$ and $V^{k}_{l_2}$. In this case, in Item $(i)$, the conditions $\epsilon_T(i_0,j_0)=1$ and $\epsilon_T(i_0,j_0+1)=-1$ means that the ribbon $r$ connects the upper boundary of $V^k_{l_0}$ with the upper boundary of $V^{k}_{l_2}$.

The conditions $\rho(i_1,j_1)=(k,l_1)$ and $\rho(i_1,j_1+1)=(k',l')$ indicate that the image of the strip $\underline{H}_{j_1}$ bounded by $H^{i_1}_{j_1}$ and  $H^{i_1}_{j_1+1}$ is a ribbon $r'$ connecting a horizontal boundary of $V^k_{l_1}$ with a horizontal boundary of $V^{k'}_{l'}$. Moreover, since $\epsilon_T(i_1,j_1)=1$, it means that $r'$ connects the upper boundary of $V^k_{l_1}$ with the stable boundary $\alpha$ of $V^{k'}_{l'}$. But the condition $(k',l')\notin \{(k,l): l_0\leq l\leq l_2\}$ ensures that the stable boundary $\alpha$ cannot be situated between the upper boundaries of $V^k_{l_0}$ and $V^k_{l_2}$, as we have specified from the beginning. Therefore the Markov partition $\cR_1$ have the topological obstruction of type $(1)$.

Is not difficult to see that if $\cR_1$ have the topological obstruction of type $(1)$, then the ribbons $r$ and $r'$ given by the definition \ref{Defi: Type 1 obstruction top}, are determined by be the image of a stripe comprised between consecutive sub-rectangles $H^{i_0}_{j_0}$ and $H^{i_0}_{j_0}$ for $r$ and $H^{i_1}_{j_1}$ and $H^{i_1}_{j_1+1}$ for $r'$. The rest of the conditions are totally determined.
\end{proof}

This coding method of representing a ribbon as two consecutive indices $(i,j)$ and $(i,j+1)$ in $\cH$ will be utilized frequently throughout this subsection.

\begin{defi}\label{Defi: Type 2 combinatoric}
	The geometric type $T$ has the type-$(2)$ \emph{combinatorial condition} if and only if there are:
	\begin{itemize}
\item Four different pairs $(k_1,l^1),(k_1,l^2),(k_2,l_1),(k_2,l_2)\in \cV(T)$ with $l_1<l_2$ and $l^1<l^2$.
\item Pairs $(i_1,j_1),(i_1,j_1+1),(i_2,j_2),(i_2,j_2+1)\in \cH(T)$ 
	\end{itemize}
Such that:
	$$
	\rho_T(i_1,j_1)=(k_1,l^1) \text{ and } \rho_T(i_1,j_1+1)=(k_2,l_2).
	$$
	
Additionally, depending on the signs of $\epsilon_T(i_1,j_1)$ and $\epsilon_T(i_1,j_1+1)$ we have the next obstructions:
	
	\begin{enumerate}
\item If $\epsilon_T(i_1,j_1)=\epsilon_T(i_1,j_1+1)$, then we have two options:
\begin{itemize}
\item[i)] If $\rho_T(i_2,j_2)=(k_1,l^2)$ and $\rho_T(i_2,j_2+1)=(k_2,l_1)$ then: $$\epsilon_T(i_1,j_1)=\epsilon_T(i_2,j_1+1)=\epsilon_T(i_2,j_2)=\epsilon_T(i_2,j_2+1)$$.

\item[ii)] If  $\rho_T(i_2,j_2)= (k_2,l_1)$ and $\rho_T(i_2,j_2+1)=(k_1,l^2)$. Then
$$-\epsilon_T(i_1,j_1)=-\epsilon_T(i_2,j_1+1)=\epsilon_T(i_2,j_2)=\epsilon_T(i_2,j_2+1)$$.
\end{itemize}

\item If $\epsilon_T(i_1,j_1)=-\epsilon_T(i_1,j_1+1)$, then we have two options:

\begin{itemize}
\item[i)] If $\rho_T(i_2,j_2)=(k_1,l^2)$ and $\rho_T(i_2,j_2+1)=(k_2,l_1)$ then:
$$
\epsilon_T(i_2,j_2)=\epsilon_T(i_1,j_1) \text{ and  } \epsilon_T(i_2,j_2+1)=\epsilon_T(i_1,j_1+1).
$$
\item[ii)] If $\rho_T(i_2,j_2)=(k_2,l_1)$ and $\rho_T(i_2,j_2+1)=(k_1,l^2)$ then:
$$
\epsilon_T(i_2,j_2)=-\epsilon_T(i_1,j_1+1) \text{ and  } \epsilon_T(i_2,j_2+1)=-\epsilon_T(i_1,j_1).
$$

\end{itemize}

\end{enumerate}

\begin{comment}

$l'_1\in \{1,\cdots,l_1-1,l_1+1, v_{k+1}\}$ and $l'_1\{1,\cdots,l_1-1,l_1+1v_{k+1}\}$ such that,  $l'_1<l_1$ if and only if  $l'_2>l_2$ and there are two options
$$
\Phi(i_2,j_2)=(k_1,l'_1,\epsilon(i_2,j_2)) \text{ and } \Phi(i_2,j_2+1)=(k_2,l'_2,\epsilon(i_2,j_2+2)) 
$$
or well
$$
\Phi(i_2,j_2)=(k_2,l'_2\epsilon(i_2,j_2)) \text{ and } \Phi(i_2,j_2+1)=(k_1,l'_1,\epsilon(i_2,j_2+2)).
$$

\item Suppose $\epsilon(i_1,j_1)=-\epsilon(i_2,j_2+1)$. Then there are $l'_1\{1,\cdots,l_1-1,l_1+1v_{k+1}\}$ and $l'_1\{1,\cdots,l_1-1,l_1+1v_{k+1}\}$ such that, $l'_1<l_1$ if and only if $l'_2<l_2$ and in this situation two possibilities are admissible:
$$
\Phi(i_2,j_2)=(k_1,l'_1,\epsilon(i_2,j_2)) \text{ and } \Phi(i_2,j_2+1)=(k_2,l'_2,\epsilon(i_2,j_2+2)) 
$$
or well
$$
\Phi(i_2,j_2)=(k_2,l'_2\epsilon(i_2,j_2)) \text{ and } \Phi(i_2,j_2+1)=(k_1,l'_1,\epsilon(i_2,j_2+2)).
$$
\end{itemize}
\end{comment}

A geometric type $T$ has the \emph{combinatorial obstruction} of type-$(2)$ if there exists $m\in \NN$ such that $T^m$ satisfies the combinatorial condition of type-$(2)$.
\end{defi}

\begin{figure}[h]
	\centering
	\includegraphics[width=0.9\textwidth]{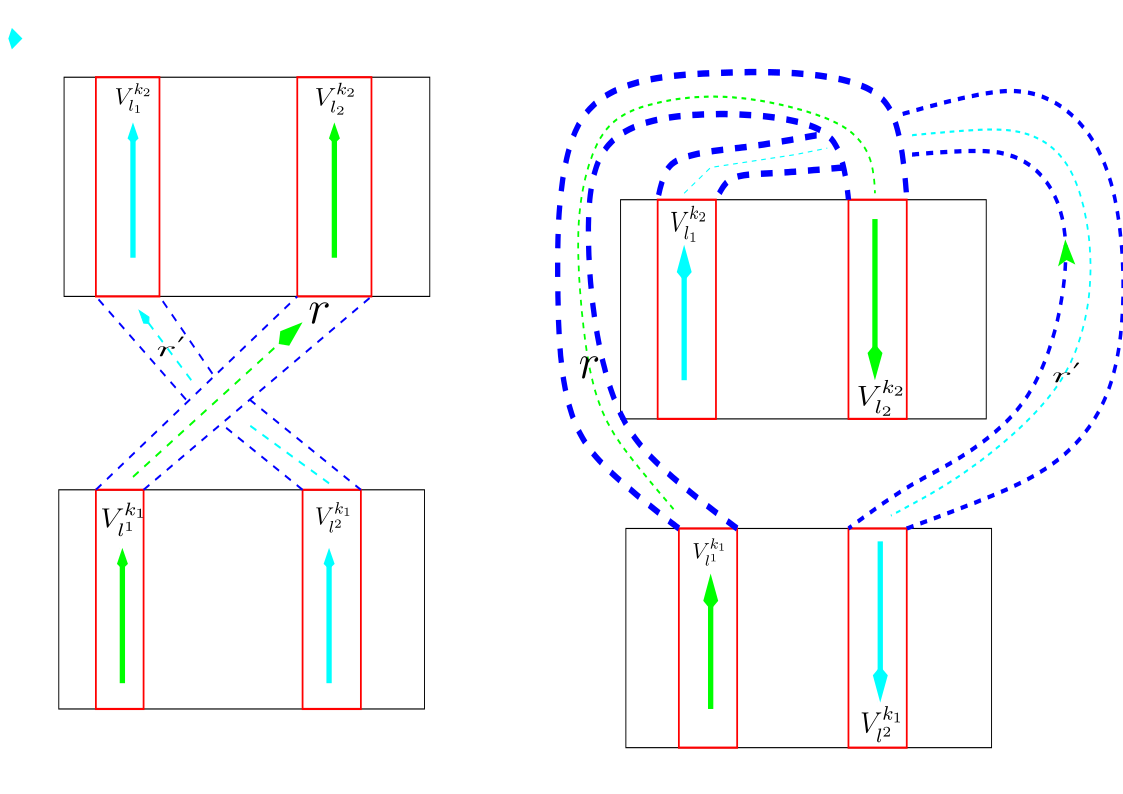}
	\caption{Type $(2)$ combinatorial condition}
	\label{Fig: Type two proof}
\end{figure}

 The Figure \ref{Fig: Type two proof} give some help to understand the proof.
 
\begin{lemm}\label{Lemm: Equiv com and top type 2} 
A geometric type $T$ satisfies the combinatorial condition of type-$(2)$ if and only if the Markov partition $\cR_0$ has the topological obstruction of type-$(2)$ on $\cR_1$.
\end{lemm}

\begin{proof}
The pairs $(k_1,l^1),(k_1,l^2),(k_2,l_1),(k_2,l_2)$ represent the respective vertical sub-rectangles of $\cR$ and the condition $\rho_T(i_1,j_1)=(k_1,l^1) \text{ and } \rho_T(i_1,j^1+1)=(k_2,l_2)$ implies the existence of a rubber $r=f(\underline{H}^{i_1}_{j_1})$ joining a pair of vertical boundaries of the rectangles $V^{k_2}_{l_2}$ and $V^{k_1}_{l^1}$ that correspond to the image  of the upper boundary of $H^{i_1}_{j_1}$ inside $V^{k_1}_{l^1}$ and the image of the inferior boundary of $H^{i_1}_{j_1+1}$ inside $V^{k_2}_{l^2}$.

In the situation when $\epsilon_T(i_1,j_1)=\epsilon_T(i_2,j_1+1)$  we can assume that both are positive (or negative, but the reasoning is the same). 

The condition: $\rho_T(i_2,j_2)=(k_1,l^2)$ and $\rho_T(i_2,j_2+1)=(k_2,l_1)$ implies that the rubber $r'=f(\underline{H}^{i_2}_{j_2})$ joins the stable boundary of $V^{k_2}_{l_2}$ that is the image by $f$ of the inferior boundary of $H^{i_2}_{j_2+1}$ with the stable boundary $V^{k_1}_{l^1}$ that is the image  by $f$ of the upper boundary of $H^{i_2}_{j_2}$. 

Like $\epsilon_T(i_1,j_1)=\epsilon_T(i_2,j_1+1)$ this implies that $r$ joins the upper boundary of $V^{k_1}_{l^1}$ with the inferior boundary of $V^{k_2}_{l_2}$, even more like $$\epsilon_T(i_1,j_1)=\epsilon_T(i_2,j_1+1)=\epsilon_T(i_2,j_2)=\epsilon_T(i_2,j_2+1)=1$$
the $r'$ joins the upper boundary of $V^{k_1}_{l^2}$ (that is the image of the upper boundary of $H^{i_2}_{j_2}$) with the inferior boundary of $V^{k_2}_{l_1}$.
Like $l_1<l_2$ and $l^1<l^2$ the fist combinatorial condition implies that $\cR$ have a topological obstruction of type $(2)$.

Now if  $\rho_T(i_2,j_2)= (k_2,l_1)$ and $\rho_T(i_2,j_2+1)=(k_1,l^2)$, and 
$$\epsilon_T(i_2,j_2)=\epsilon_T(i_2,j_2+1)=-1$$,
The rubber rubber $r'$ have stable boundary in the upper boundary of $V^{k_1}_{l_2}$ as this correspond to the image by $f$ of the inferior boundary of $H^{i_2}_{j_2+1}$ with the  inferior boundary of $V^{k_2}_{l_1}$ that correspond to the image by $f$ of the upper boundary of $H^{i_2}_{j_2}$. Therefore we have a type $(2)$ obstruction.

Consider now that $\epsilon_T(i_1,j_1)=-\epsilon_T(i_1,j_1+1)$. Too fix ideas  $\epsilon_T(i_1,j_1)=1$ and $\epsilon_T(i_1,j_1+1)=-1$. This means that $r$ joins the upper boundary of $V^{k_1}_{l^1}$ (that is the image by $f$ of the upper boundary of $H^{i_0}_{j_0}$)
with the upper boundary of $V^{k_2}_{l_2}$ (that correspond to the image by $f$ of the inferior boundary of $H^{i_1}_{j_2+1}$ ).

 The condition  $\rho_T(i_2,j_2)=(k_1,l^2)$ and $\rho_T(i_2,j_2+1)=(k_2,l_1)$ of the point $(2)$ in the definition together with:
 $$
\epsilon_T(i_2,j_2)=1 \text{ and  } \epsilon_T(i_2,j_2+1)=-1.
$$
implies that $r'$ joints the superior boundary of $V^{k_1}_{l^2}$ that is the image by $f$ of the upper boundary of $H^{i_1}_{j_1}$, with the upper boundary of $V^{k_2}_{l_2}$ that is the image by $f$ of the inferior boundary of $H^{i_1}_{j_1+1}$. Similarly we have the obstruction of type $(2)$ in the Markov partition.

Finally if  $\epsilon_T(i_1,j_1)=1$ and $\epsilon_T(i_1,j_1+1)=-1$ but now $\rho_T(i_2,j_2)=(k_2,l_1), \rho_T(i_2,j_2+1)=(k_1,l^2)$ and 
$$\epsilon_T(i_2,j_2)=1 \text{ and  } \epsilon_T(i_2,j_2+1)=-1.$$

we can deduce that $r'$ joints the upper boundary of $V^{k_1}_{l^2}$ that is the image by $f$ of the inferior boundary of $H^{i_2}_{j_2+1}$ and the upper boundary of $V^{k_2}_{l_1}$ that is the image by $f$ of the upper boundary of $H^{i_2}_{j_2}$. Once again we get a type two obstruction.

If the Markov partition $\cR$ have the type $(2)$ topological obstruction in $\cR_1$the ribbons $r$ and $r'$ determine the horizontal sub-rectangles determined by $(k_1,l^1),(k_1,l^2),(k_2,l_1),(k_2,l_2)$  and we can fix the condition 	$\rho_T(i_1,j_1)=(k_1,l^1)$ and $\rho_T(i_1,j^1+1)=(k_2,l_2)$ to indicate that $r$ joints such rectangles. The condition $(1)$ reflex the case when the two rectangles $R_{k_1}$ and $R_{k_2}$ have cohered orientation and the other case when they have inverse orientation. In any case, we need to remember that $r$ joint the image of the upper boundary of $H^{i_0}_{j_0}$ with the image of the inferior boundary of $H^{i_0}_{j_0}$. The four conditions enunciated are all the possible case when the ribbon $r'$ have stable boundaries in the same stables boundaries components of $\cR$ than $r$.

\end{proof}

\begin{figure}[h]
	\centering
	\includegraphics[width=0.83\textwidth]{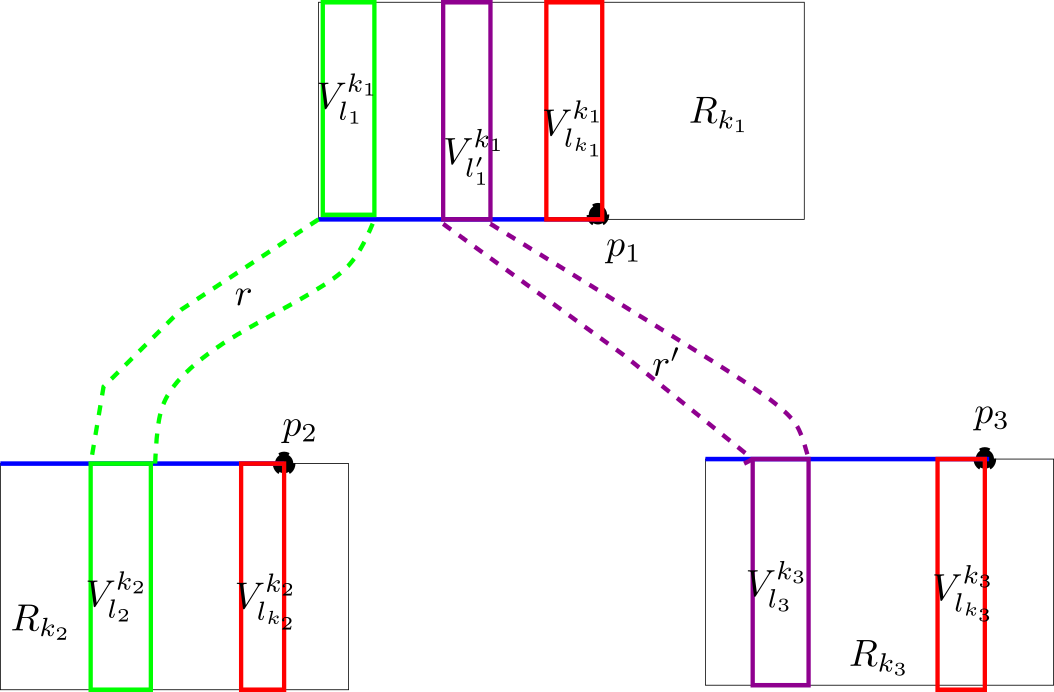}
	\caption{Type $(3)$ combinatorial condition}
	\label{Fig: Type three proof}
\end{figure}

\begin{defi}\label{Defi: Type 3 combinatoric}
A geometric type $T$ satisfies the \emph{combinatorial condition} of type-$(3)$ if there exist indices  described as follows:
\begin{itemize}
\item[i)]  There are indexes $(k_1,j_1)\neq (k_2,j_2)$, and $(k_3,j_3) \in \cH(T)$  were $j_{\sigma}=1$ or $j_{\sigma}=h_{\sigma}$ for $\sigma=k_1,k_2,k_3$, that satisfy:
\begin{eqnarray}\label{Equ: periodic boundary}
\Phi(\sigma,j_{\sigma})=(\sigma,l_{\sigma},\epsilon(\sigma,j_{\sigma}) =1),
\end{eqnarray}

where $l_{\sigma}\in \{1,\cdots,v_{\sigma}\}$. We don't exclude the possibility that $(k_3,j_3)$ is equal to $(k_1,j_1)$ or $(k_2,j_2)$.

\item[ii)] A couple of numbers $l_{1},l_{1}'\in \{1,\cdots, v_{k_1}\}$  such that, either 
$$
l_1<l_1'<l_{k_1} \text{ or } l_{k_1}<l_1<l_1'.
$$

\item[iii)] For  the pair $(k_2,j_2)$, there is a number $l_{2}\in \{1,\cdots, v_{k_2}\}$  with $l_{2}\neq l_{k_2}$.

\item[iv)] There are three possible situations for the pair $(k_3,j_3)$ depending if it is  equal or not to some other pair in item $i)$:

\begin{itemize}
\item[1)] If  $(k_1,j_1)\neq(k_3,j_3) \neq (k_2,j_2) $, there is a number $l_{2} \in \{1,\cdots, v_{k_2}\}$ different form $l_{k_2}$ and a number $l_3\in \{1,\cdots, v_{k_3}\}$ different from $l_{k_3}$.

\item[2)] In case that $(k_2,j_2) = (k_3,j_3)$. There is a numbers $l_3\in \{1,\cdots, v_{k_2}=v_{k_3}\}$, such that: $l_2<l_{k_2}=l_{k_3}<l_3$ or $l_3<l_{k_3}=l_{k_2}<l_2$. 

\item[3)] In case that $(k_1,j_1) = (k_3,j_3)$. There is a number $l_3\in \{1,\cdots, v_{k_1}=v_{k_3}\}$, such that: $l_1<l_1'<l_{k_1}=l_{k_3}<l_3$ or $l_3<l_{k_3}=l_{k_1}<l_1<l_1'$ depending on the situation of item $ii)$. 
\end{itemize}
 
\item[v)] There is a pair of pairs $(i_r,j_r),(i_{r'},j_{r'})\in \cH(T)$ such that $(i_r,j_r+1),(i_{r'},j_{r'}+1)\in \cH(T)$
\end{itemize} 
All of them need to satisfy the next equations:

\begin{eqnarray}
\rho(i_r,j_r)=(k_1,l_1) \text{ and } \rho(i_r,j_r+1)=(k_2,l_2), \text{ or }\\
\rho(i_2,j_2+1)=(k_1,l_1) \text{ and } \rho(i_r,j_r)=(k_2,l_2).
\end{eqnarray}
and at the same time:
\begin{eqnarray}
\rho(i_{r'},j_{r'})=(k_1,l_1') \text{ and } \rho(i_{r'},j_{r'}+1)=(k_3,l_3), \text{ or well }\\
\rho(i_{r'},j_{r'}+1)=(k_1,l_1') \text{ and } \rho(i_{r'},j_{r'})=(k_3,l_3).
\end{eqnarray}
A geometric type $T$ has the \emph{combinatorial obstruction} of type-$(3)$ if there exists an $m\in \NN$ such that $T^m$ satisfies the combinatorial condition of type-$(3)$.
\end{defi}

The distinction between the last combinatorial and topological conditions and obstructions lies in the fact that, for the type $3$ combinatorial condition, we need to consider some power of the geometric type to determine the periodic stable boundaries and their embrionary separatrices. This is why we formulated Lemma \ref{Lemm: Equiv com and top type 3} in terms of combinatorial and topological obstructions, rather than the combinatorial condition of type $3$ itself. In the future (see Lemma \ref{Lemm: T pA class then no condition 3}), we will prove that no geometric type in the pseudo-Anosov class has the combinatorial condition of type $3$. Therefore, its iterations won't have such a condition either. This will be sufficient to establish that any iteration of such a geometric type will don't have the combinatorial condition of type $3$, and the corresponding combinatorial and topological obstructions will not hold for $T$.

\begin{lemm}\label{Lemm: Equiv com and top type 3}
A geometric type have $T$ has the combinatorial obstruction of type $(3)$  if and only if the Markov partition $\cR$ has the topological obstruction of type-$(3)$.
\end{lemm}

\begin{proof}

Assume that $T'$ has the combinatorial obstruction of type $(3)$, meaning there exists an $m \geq 1$ such that for $T := T'^{m}$, and the combinatorial condition of type $(3)$ is satisfied for $T$. We will proceed with our analysis using $T$.

The condition in Item $(i)$ regarding the pairs of indices $(k_1, j_1) \neq (k_2, j_2)$ and  $(k_3, j_3) \in \cH(T)$ implies that the rectangle $R_{\sigma}$ has a fixed stable boundary and a periodic point on it. Let's denote the stable boundaries as $A_{k_1}$, $A_{k_2}$, and $A_{k_3}$. Item $(i)$ indicates that $A_{k_1} \neq A_{k_2}$, but it is possible that $A_{k_3}$ could be equal to one of the other two.

To clarify, let's consider a specific case to fix the ideas. Let's assume that $A_{k_1}$ is the lower boundary of $R_{k_1}$ and $A_{k_2}$ is the upper boundary of $R_{k_2}$. If $A_{k_3}$ is different from the other two stable boundaries, we'll assume that $A_{k_3}$ is the upper boundary of $R_{k_3}$. Note that the other cases are either symmetric or can be entirely determined based on these conventions.

The periodic point $p_1$ that lies on the stable boundary $A_{k_1}$ is contained in the lower boundary of the vertical sub-rectangle $V^{k_1}_{l_{k_1}}$. Then, the conditions $l_1<l_1'<l_{k_1}$ or $l_{k_1}<l_1<l_1'$ imply that the lower boundary of the rectangles $V^{k_1}_{l_1}$ and $V^{k_2}_{l_1'}$ lies on the same embryonic separatrice of $A_{k_1}$. Let's denote this embrionary separatrice as $S_1$. Similarly, the upper boundary of $V^{k_2}_{l_2}$ lies on an embrionary separatrice $S_2$ within $A_{k_2}$. Depending on the situation described in item $iv)$, we can have the following cases:

\begin{itemize}
\item[i)] The periodic point $p_3$ is different from $p_1$ and $p_2$, and the horizontal sub-rectangle $V^{k_3}_{l_3}$ has its upper boundary in an embrionary separatrice $S_{3}$.

\item[ii)] In this case, $p_3=p_2$, but the condition $l_2<l_{k_2}=l_{k_3}<l_3$ or $l_3<l_{k_3}=l_{k_2}<l_2$ implies that the upper boundary of $V^{k_3}_{l_3}$ is in an embrionary  separatrice $S_3$, distinct from $S_2$.

\item[iii)]  In this case, $p_1=p_2$, and the condition $l_1<l_1'<l_{k_1}=l_{k_3}<l_3$ or $l_3<l_{k_3}=l_{k_1}<l_1<l_1'$ implies that the \emph{inferior boundary} of $V^{k_3}_{l_3}$ is in an embryonic separatrice $S_3$ different from $S_1$.
\end{itemize}

The conclusion is that $S_1\neq S_2\neq S_3$. Finally, item $(4)$ along with the conditions about the indexes $(i_r,j_r)$ and $(i_{r'},j_{r'})$ implies that there is a ribbon $r$ from $S_1$ to $S_2$ and another ribbon $r'$ from $S_1$ to $S_3$. Therefore, we have the type $(3)$ topological obstruction in $\cR_0$.

In the converse direction. Imagine that $\cR$ has the topological obstruction of type $(3)$. This means there are three or two different periodic stable boundaries of the Markov partition $A_{k_1}\neq A_{k_2}$ and $A_{k_3}$ with $A_{\sigma}\subset R_{\sigma}$ that contain different embrionary separatrices $S_1,S_2$ and $S_3$. Additionally, there is a ribbon $\underline{r}$ of generation $k$ that joins $S_1$ with $S_2$, and another ribbon $\underline{r}'$ of generation $k'$ that joins $S_1$ with $S_2$.

By taking a certain power of $T' = T^m$ with $m$ a multiple of the period of every stable boundary and greater than $k'$ and $k$, we consider the realization $(\cR,\phi^m)$ of $T$. We can assume that $A_{\sigma}\subset R_{\sigma}$ is a fixed stable boundary for $\phi^m$.
 
We are now in the setting of the combinatorial condition of type $(3)$ for $T$. The fact that the stable boundaries are fixed implies the existence of the indexes $(k_1,j_1)\neq (k_2,j_2)$, and $(k_3,j_3) \in \cH(T')$  in item $(1)$, and clearly, the fixed points are contained in $V^{\sigma}_{l_{\sigma}}$.

The ribbons $\underline{r}$ and $\underline{r}'$, contain ribbons of generation $m$, denoted as $r'$ and $r_2$, respectively. These ribbons joins certain vertical sub-rectangles of the realization of $T$,  $(\cR,\phi^m)$, $V^{k_1}_{l_1}\subset S_1$ with $V^{k_2}_{l_2}\subset S_2$, and $V^{k_1}_{l_1'}\subset S_1$ with $V^{k_3}_{l_3}\subset S_3$. These vertical sub-rectangles produce the indices in items $ii)$, $iii)$, and $iv)$.  The items $2)$ and $3)$ represent the situations where $S_1$ and $S_3$ (or $S_3$ and $S_2$) are in the same stable boundary of a rectangle.

Finally, the ribbons $r$ and $r'$ are determined by two consecutive horizontal sub-rectangles of $(\cR,\phi^m)$. This corresponds to the situation described in item $v)$, and they satisfy the rest of the properties as outlined in the argument.

\end{proof}

\subsubsection{The impasse property.}

Finally we proceed to formulate the impasse condition in terms of the geometric type.

\begin{defi}\label{Defi: Impasse combinatoric}
 Let $T$ be an abstract geometric type of finite genus. Then $T$ has the \emph{impasse property} if there exist $(i,j),(i,j+1)\in \cH(T)$ such that one of the following conditions holds:
 
\begin{eqnarray}
(\rho,\epsilon)(i,j)=(k,l,\epsilon(i,j)) \text{ and } (\rho,\epsilon)(i,j+1)=(k,l+1,-\epsilon(i,j)), \text{ or }\\
(\rho,\epsilon)(i,j)=(k,l+1,\epsilon(i,j)) \text{ and } (\rho,\epsilon)(i,j+1)=(k,l,-\epsilon(i,j)).
\end{eqnarray}
A geometric type $T$ has a \emph{combinatorial impasse} if there exists $m \in \mathbb{N}$ such that $T^m$ has the impasse property.
\end{defi}

Unlike the first three obstructions, where a ribbon is determined by two horizontal sub-rectangles and their respective indices $(i,j),(i,j+1)\in \mathcal{H}$, a topological impasse is defined in terms of a disjoint disk of $K$ and two arcs. Proving the equivalence between the topological and combinatorial formulations just given is a more subtle task, and we address it in the final part of this subsection.

\begin{theo}\label{Theo: Geometric and combinatoric are equivalent}
Let $T=(n,\{(h_i,v_i)\}_{i=1}^n,\Phi)$  be a geometric type of finite genus. Let $f: S \rightarrow S$ be a surface Smale diffeomorphism, and let $K$ be a saddle-type basic piece of $f$ that has a Markov partition of geometric type $T$. The following conditions are equivalent:

\begin{itemize}
\item[i)] The basic piece $K$ has a topological impasse.
\item[ii)] The geometric type $T^{2n+1}$ has the impasse property.
\item[iii)] The geometric type $T$ has a combinatorial impasse.
\end{itemize}
\end{theo}

\begin{proof}
 \textbf{ i) implies ii):} Let  $\cR=\{R_i\}_{i=1}^n$ be a Markov partition of $K$ of geometric type $T$, and let $\overset{o}{D}$ be a topological impasse with $\alpha$ as the $s$-arc and $\beta$ as the $u$-arc, whose union forms the boundary of $D$.

 \begin{lemm}\label{Lemm: impasse disjoin int Markov partition}
The impasse is disjoint from the interior of the Markov partition, i.e.
$$
\overset{o}{D}\cap \cup_{i=1}^n \overset{o}{R_i}=\emptyset.
$$
Moreover, the intersection $\{k_1,k_2\}=\alpha\cap \beta$ is not contained in the interior of the Markov partition.
 \end{lemm}
 
\begin{proof}
Suppose $\overset{o}{D}\cap \overset{o}{R_i}\neq \emptyset$. Let $\{k_1,k_2\}:=\alpha\cap \beta$ be the endpoints of the $s,u$-arcs  on the boundary of $D$. Let us take a point $x$ in $\overset{o}{D}\cap \overset{o}{R_i}$, since $x$ is not a hyperbolic point and have a open neighborhood $U\subset \overset{o}{R_i}$ disjoint of $K$, there must exist a sub-rectangle  $Q\subset \overset{o}{R_i}$  containing $x$ in its interior, whose interior is disjoint from $K$, whose horizontal boundary consists of two disjoint intervals $I,I'$ of $W^s(K)$ and whose vertical boundary consists of two disjoint intervals $J, J'$´ of $W^u(K)$ (See Figure \ref{Fig: Rec Q} ). 
\begin{figure}[h]
	\centering
	\includegraphics[width=0.6\textwidth]{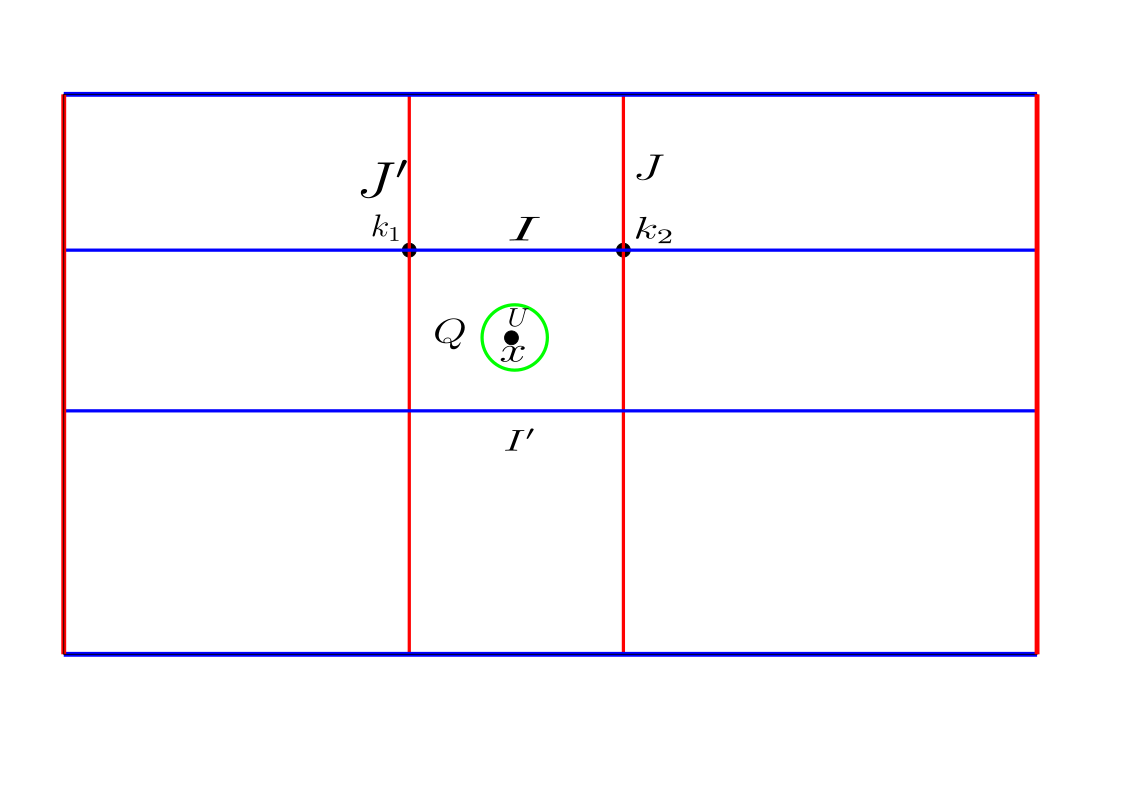}
	\caption{The rectangle $Q$}
	\label{Fig: Rec Q}
\end{figure}

Even more, as $\overset{o}{D}$ is disjoint from $K$ and  $k_1\in K$  is in the closure of $\overset{o}{D}$, the point $k_1$ is a corner point of the rectangle $Q$ and hence $k_1$ is in $\overset{o}{R_i}$.

By similar arguments, we can conclude that $k_2$ is in $\overset{o}{R_i}$. However, this leads to a contradiction, since there would be a horizontal segment $I$ of $R_i$ containing $\alpha$ and a vertical segment $J$ of $R_i$ containing $\beta$ which intersect at two points in the interior of $R_i$ (namely $k_1$ and $k_2$). This is not possible because $R_i$ is an embedded rectangle. Therefore, $\overset{o}{D}$ is disjoint from the interior of $R_i$.

It is evident that if $k_1$ or $k_2$ is in the interior of $R_i$, then $\overset{o}{D}\cap \overset{o}{R_i}\neq \emptyset$, and we appeal to the previous argument to reach a contradiction.

\end{proof}

If $\overset{o}{D}$ is an impasse, then $f^z(\overset{o}{D})$ is also an impasse for all integers $z \in \ZZ$. In the following lemma, we will construct a "well-suited" impasse that will make it easier for us to prove the first implication.

\begin{lemm}\label{Lemm: well suit impasse}
If $K$ has an impasse, then there exists another impasse $D'$ of $K$ such that the $s$-arc $\alpha'$ and the $u$-arc $\beta'$ on the boundary of $D'$ satisfy the following conditions:

\begin{itemize}
\item[i)]  $\beta'$ is in the interior of the unstable boundary of the Markov partition,  $\beta' \subset \overset{o}{\partial^u  \cR}$,
\item[ii)] $f(\beta')$ is not a subset of the unstable boundary of the Markov partition,  $\partial^u\overset{o}{\cR}= \overset{o}{\partial^u  \cR}$.
\item[iii)] $f^{2n+1}(\alpha)\subset \overset{o}{\partial^s \cR}$, where $n$ is the number of rectangles in the Markov partition.
\end{itemize}
\end{lemm}
 
 \begin{proof}
Let $D$ be an impasse of $K$ with $\alpha$ and $\beta$ as its boundary arcs. Let $\{k_1,k_2\}:=\alpha \cap \beta \subset K$ be the extreme points of these arcs. If $k_1$ is a periodic point, then there exists a positive iteration of $k_2$ that intersects $\alpha$ in its interior. However, this is not possible since the interior of $\alpha$ is disjoint from $K$. Therefore, we deduce that $k_1$ and $k_2$ are on the same stable separatrice of a periodic $s$-boundary point, and they are not periodic points themselves. Similarly, $k_1$ and $k_2$ are on the same unstable separatrice of a $u$-boundary point, but they are not periodic.

Since the negative orbit of $k_1$ converges to a periodic point $p_1$ on the stable boundary of the Markov partition, the negative orbit of $k_1$ approaches $p_1$ through points contained in the boundary of the Markov partition. Therefore, there exists $m \in \NN$ such that $f^{-m}(k_1)\in \overset{o}{\partial^u \cR}$.  We claim that if $f^{-m}(k_1)\in \overset{o}{\partial^u_{\epsilon}R_i}$, then $f^{-m}(k_2)$ is in the interior of the same unstable boundary component as $f^{-m}(k_1)$, i.e., $f^{-m}(k_2)\in \overset{o}{\partial^u_{\epsilon}R_i}$.

Indeed, since $f^{-m}(k_1)$ and $f^{-m}(k_2)$ are the boundary points of the $u$-arc $f^{-m}(\beta)$, there are no points of $K$ between them. The point $f^{-m}(k_1)$ lies in the interior of $\partial^u_{\epsilon} R_i$, which means that $f^{-m}(\beta)$ intersects the interior of this stable boundary component. This implies that $f^{-m}(k_2)$ is either an extreme point of $\partial^u_{\epsilon} R_i$ or lies in its interior, i.e., $f^{-m}(k_2) \in \overset{o}{\partial^u_{\epsilon}R_i}$.

If $f^{-m}(k_2)$ were an extreme point of $\partial^u_{\epsilon} R_i$, such an extreme point would not be surrounded by elements of $K$ on either side. This is because on one side we have the $u$-arc $f^{-m}(\beta)$, and on the other side it is on the unstable boundary of a rectangle in the Markov partition. However, this is not possible since $K$ has no double $s$-boundary points. Therefore, we conclude that $f^{-m}(k_2) \in \overset{o}{\partial^u_{\epsilon}R_i}$, as we claimed.

Now let us consider the impasse $D' = f^{-m}(D)$ instead of $D$. We will still use $\alpha$ and $\beta$ to denote the $s$-arc and $u$-arc on the boundary of $D$, but we assume that $\beta \subset \overset{o}{\partial^u \cR}$. Due to the uniform expansion along the unstable leaves, we know that there exists $m \in \NN$ such that $f^m(\beta) \subset \overset{o}{R_i}$, but $f^{m+1}(\beta)$  is no longer contain in  the unstable boundary of the Markov partition. Let us define $D' = f^{m}(D)$ with boundary arcs $\alpha'$ and $\beta'$. This impasse satisfies items $i)$ and $ii)$ of our lemma, since $\beta'$ is contained in the interior of $\overset{o}{\partial^u \cR}$ and $f(\beta')$ is not a subset of $\partial^u \overset{o}{\cR} = \overset{o}{\partial^u \cR}$.

In view of Lemma \ref{Lemm: impasse disjoin int Markov partition}, the points $f(k_1)$ and $f(k_2)$ are not in the interior of the Markov partition. We conclude that $f(k_1)$ and $f(k_2)$ belong to $\partial^s \cR$ because both are on the boundary of $\cR$ but not inside the unstable boundary. By the invariance of the stable boundary of $\cR$ under positive iterations of $f$ and the Pigeonhole principle applied to the $2n$ components of stable boundaries of the Markov partition, we know that $f^{2n}(f(k_1))$ and $f^{2n}(f(k_2))$ lie on periodic stable boundaries of the Markov partition. Moreover, $f^{2n}(f(k_1))$ and $f^{2n}(f(k_2))$ lie on the same stable separatrice, and on each separatrice, there is a single stable boundary component of $\cR$ that is periodic (even though they are all pre-periodic). This implies that $f^{2n}(f(k_1))$ and $f^{2n}(f(k_2))$ are in $\overset{o}{\partial^s_{\epsilon}R_j}$ for a single stable boundary component of the rectangle $R_j$. This proves item $iii)$ of our lemma.

\end{proof}

With the simplification of  Lemma \ref{Lemm: well suit impasse} we deduce that $\beta\subset \partial^u R_i$ and $f^{2n+1}(\alpha)\subset \partial^s R_k$. 

Let us consider the partition $\cR$ viewed as a Markov partition of $f^{2n+1}$. We claim that $\beta$ is an arc joining two consecutive sub-rectangles of $(\cR,f^{2n+1})$, denoted as $H^i_j$ and $H^{i}_{j+1}$. In effect, $\beta$ is a $u$-arc joining two consecutive rectangles or is properly contained in a single sub-rectangle $H$ of $(\cR,f^{2n+1})$. This is because the stable boundaries of $H$ are not isolated from $K$ within $H$, as would be the case if the ends of $\beta$ were on the stable boundary of $H$. Therefore, the only possibility is that $\beta$ is properly contained in a rectangle $H$. In this case, $f^{2n+1}(H)$ is a vertical sub-rectangle of $R_i$ containing $f^{2n+1}(\beta)$ as a proper interval. However, the endpoints of $f^{2n+1}(\alpha)$, which coincide with the endpoints of $f^{2n+1}(\beta)$, are not in the interior of the stable boundary of $R_k$, which contradicts the hypothesis. Thus, we conclude that $\beta$ joins two consecutive sub-rectangles.

Let $H^i_{j}$ and $H^i_{j+1}$ be the consecutive rectangles joined by $\beta$. Suppose that $f^{2n+1}(H^i_j)=V^k_{l}$ with the change of vertical orientation encoded by $\epsilon_{T^{2n+1}}(i,j)$ and $f^{2n+1}(H^i_{j+1})=V^k_{l'}$ with the change in the vertical orientation encoded by $\epsilon_{T^{2n+1}}(i,j+1)$. Between $V^k_{l}$ and $V^k_{l'}$ there are no points of $K$ because the $s$-arc $f^{2n+1}(\alpha)$ joins them in the horizontal direction, this implies that $l'\in \{l+1,l-1\}$. The vertical orientation of $\beta$ is the same as $H^i_j$ and $H^i_{j+1}$, furthermore $f^{2n+1}(\beta)$ joins the horizontal sides of the rectangles $V^k_{l}$ and $V^k_{l'}$ which are on the same stable boundary component of $R_k$,  this implies that the vertical orientation of $f(H^i_j)$ is the inverse of the orientation on $f(H^i_{j+1})$ within $R_k$ (to visualize this it suffices to follow the segment $f^{2n+1}(\beta)$ with a fixed orientation), thus $\epsilon_{T^{2n+1}}(i,j)=-\epsilon_{2n+1}(i,j')$.

The geometric type of the Markov partition $\cR$ for the map $f^{2n+1}$ is $T^{2n+1}$. Our construction implies that $T^{2n+1}$ has the impasse property. With this, we conclude the first implication of the theorem..

 \textbf{ ii) implies iii):} The geometric type $T^{2n+1}$ has the impasse property, which implies that $T$ has a combinatorial impasse.

\textbf{ iii) implies i):} Suppose $T^m$ has the impasse property for $m=2n+1$. Let $H^i_j$ and $H^i_{j+1}$ be consecutive horizontal sub-rectangles in the Markov partition $\cR$ for the map $f^{m}$, as given by the impasse property. This condition implies that if $f^{m}(H^i_j)=V^k_{l}$ and $f^{m}(H^i_{j+1})=V^k_{l'}$, then $l$ and $l'$ are consecutive indexes and have inverse vertical orientations.

The stable segments of  the stable boundary of $R_k$ between $V^k_{l}$ and $V^k_{l'}$ do not contain elements of the maximal invariant set of $f^m\vert_{\cR }$, since they are consecutive sub-rectangles of the partition $\cR$ seen as a Markov partition of $f^m$. In fact they do not contain elements of $K$ since $K=f^m(K)$ and therefore those segments are $s$-arcs of $f$. Furthermore, the image by $f^{m}$ of the two arcs $u$ on the boundary  of $R_i$ that join $H^i_j$ and $H^i_{j+1}$ are $u$-arcs of $f$. Between the two pairs of arcs there is an $s$-arc $\alpha$ and a $u$-arc $\beta$ such that they intersect each other only at their ends. We claim that $\alpha$ and $\beta$ bound a disk $D$ whose  interior is disjoint from $K$.

Suppose that $\gamma:=\alpha \cup \beta$ does not bound a disk. This means that $\gamma$ is not homotopically trivial, and neither are its iterations $\{f^m(\gamma)\}_{m=1}^{\infty}$. In fact, since the stable segments of such curves lie in the stable manifold of a periodic point, there exists a $k\in \NN$ such that $\{f^{mk}(\gamma)\}_{m=1}^{\infty}$ is a set of disjoint curves. We will argue that two of these curves are not homotopic, leading to a contradiction with the fact that the surface $S$ where the basic piece is contained has finite genus.

\begin{lemm}\label{Lemma: gamma not homotopic imagen}
The curve $\gamma$ is not homotopic to any other curve $f^{2mk}(\gamma)$ for $k\in \mathbb{N}$.
\end{lemm}

\begin{figure}[h]
	\centering
	\includegraphics[width=0.4\textwidth]{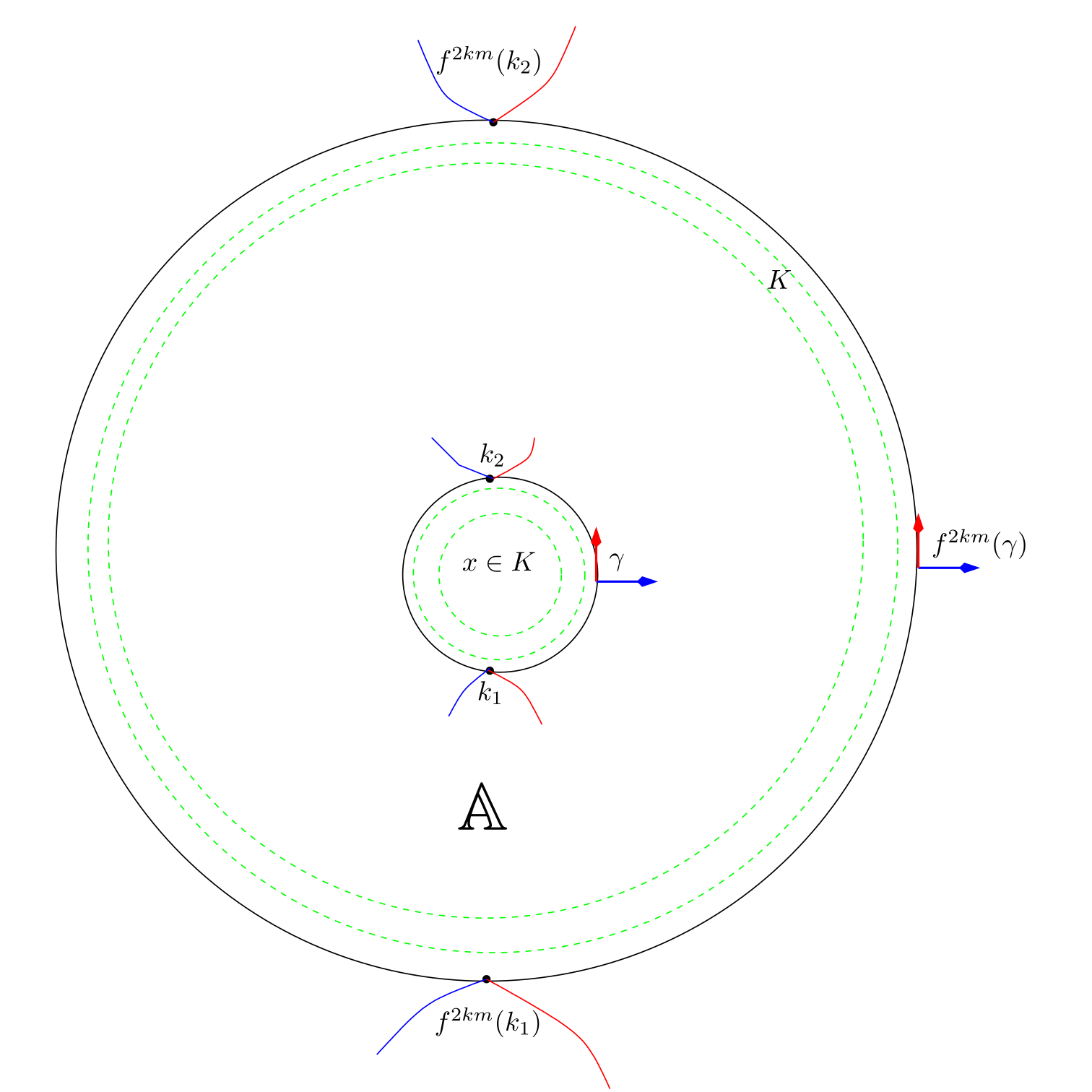}
	\caption{ The Annulus}
	\label{Fig: The annulus}
\end{figure}

\begin{proof}
Suppose $\gamma$ and $f^{m}(\gamma)$ are homotopic, which means they bound a topological annulus $\AA$ in $S$. 

Then we can take small enough sub-annulus $\AA'$ of $\AA$, which has $\gamma$ as one of its boundary components and  is send to another annulus $f^m(\AA')$ that have $f^m(\gamma)$  as a boundary component. 

 We have the following situations:
\begin{itemize}
\item[i)]  $f^m(\overset{o}{\AA'}) \cap \overset{o}{R}\neq \emptyset$ or
\item[ii)] $f^m(\overset{o}{\AA'}) \cap \overset{o}{R} = \emptyset$
\end{itemize}

Lets to understand the fist situation. Take $x\in \gamma$ and a small interval $J$ that have $x$ as and end point and pointing toward the interior of $\AA$, take a interval $I\subset \gamma$ that have $x$ as and end point and give to $I$ an orientation such that $J$ and $I$ generate a frame that is coherent with the orientation of $S$. While $f^m$ preserve the orientation of $S$ it could change the orientation of $I$ and $J$ at the same time, in such a manner that $f^m(J)$ is and interval pointing towards the interior of $\AA$, this is the mechanism behind the situation in item $i)$. Anyways by taking the annulus $\AA$ bounded by $\gamma$ and $f^{2m}(\gamma)$, we can take $\AA'\subset \AA$ small enough such that:
$$f^{2m}(\overset{o}{\AA'}) \cap \overset{o}{R} = \emptyset$$
Therefore, we can assume that $\gamma$ and $f^{2km}(\gamma)$ satisfy the property enunciated in item $ii)$.

Suppose that $\gamma$ is saturated by the hyperbolic set $K$ in the interior direction of $\AA$, and therefore, it is isolated from $K$ in the exterior direction of $\AA$. Then $f^{2km}(\gamma)$ is isolated from $K$ in the interior direction of $\AA$ and saturated by $K$ in the exterior direction of $\AA$. In particular, there exists $x\in K \cap (S \setminus \AA)$.
 
Since ${k_1, k_2}$ and ${f^m(k_1), f^{2km}(k_2)}$ are not periodic points, they do not have free separatrices. Moreover, $k_1$ and $k_2$ lie on the same separatrice. Without loss of generality, let's assume that $k_1$ has a stable separatrice $I$ that does not connect it to a periodic point. The separatrice $I$ cannot remain entirely inside $\AA$ because there exists a point $x \in K$ outside of $\AA$. Non-free separatrices of points $y \in K$ that do not connect $y$ to a periodic point are dense in $K$, so $I$ needs to approach $x$ arbitrarily closely. The boundary of $\AA$ is the union of four arcs, which means that the separatrice $I$ cannot intersect them in their interiors. Therefore, the possible scenarios are as follows:

 \begin{itemize}
\item The intersection point is $k_2$, which is impossible since it would imply that $I$ is inside a closed leaf of $W^s(K)$,
\item The other point is $f^{2km}(k_1)$ or $f^m(k_2)$. This is not possible because in that case, $I$ would need to be contained in the separatrice of $f^{2km}(k_1)$ (or $f^{2km}(k_2)$), which is not the same as the $s$-arc in $f^{2km}(\gamma)$. But this separatrice points towards the exterior of $\AA$
 \end{itemize}
 
So, $I$ couldn't intersect either $f^{2k¨m}(\gamma)$ or $\gamma$. We conclude that $I$ is a subset of $\AA$, which is a contradiction. This ends our proof.
 
\end{proof}

Then, if $\gamma$ is not null-homotopic, we could find an infinite amount of disjoint and non-isotopic curves (See \cite[Lemma 3.2]{juvan1996systems}), which is a contradiction with the finite genus of the surface $S$. It follows that $\gamma$ bounds a disk.

\begin{coro}\label{key}
	
	The disk bounded by $\gamma$ doesn't intersect $K$.
\end{coro}

\begin{proof}
If there exists $x \in K \cap \overset{o}{D}$, then $x$ has at least one non-free stable separatrix $I$. Such a separatrix is not confined within $D$ because there exist points in $K$ that are not in $D$ (they are those on the other side of the $s$-arc $\alpha$ on the boundary of $D)$. Additionally, $I$ is dense in $K$. Therefore, $I$ must intersect the $u$-arc $\beta$ on the boundary of $D$. However, the separatrice $I$ does not intersect the extreme points $\{k_1,k_2\}=\alpha \cap \beta$ because no local separatrices of them intersect the interior of $D$. This leads to a contradiction. Hence, $K \cap \overset{o}{D} = \emptyset$.
\end{proof}

This implies that $D$ is an impasse for $f$.

\end{proof}

\subsection{The genus and the impasse are detected in finite time.}\label{Sub-sec: algoritm finite genus}

\begin{prop}\label{Prop: mixing+genus+impase is algorithm}
	Given any abstract geometric type $T=(n,\{(h_i,v_i)\}_{i=1}^n,\Phi)$, there exists a finite algorithm that can determine whether $T$ satisfies the following properties:
	\begin{enumerate}
		\item The incidence matrix of $T$, $A(T)$ is mixing. 
		\item The genus of $T$ is finite.
		\item  $T$ exhibits an impasse.
	\end{enumerate}
	Furthermore, the number of iterations of $T$ required by the algorithm to determine these properties is upper bounded by $6n$.
\end{prop}

\begin{proof}
	
	We have formulated specific conditions on the combinatorics of $T$ to determine whether $T^n$ has an impasse or one of the three obstructions to have finite genus. Furthermore, according to Theorem \ref{Theo: finite type iff non-obtruction}, determining whether or not $T$ has finite genus requires computing at most $T^{6n}$. To determine whether $T$ has an impasse, Theorem \ref{Theo: Geometric and combinatoric are equivalent} states that it is necessary to compute at most $T^{2n+1}$. Since $n\geq1$, it is clear that $2n+1<6n$, which provides the upper bound we claim in the proposition. Proposition \ref{Prop: algoritm iterations type} asserts that the computation of $T^{6n}$ is algorithmic, thus confirming the existence of an algorithm to determine the finite genus and the presence of an impasse for $T$.

	We have left the mixing criterion for last. The following lemma is the basis for the upper bound on the number of iterations to verify the mixing of the incidence matrix.
	
	\begin{prop}\label{Prop: bound positive incidence matriz }
	A matrix $A$  of size $n \times n$ with non-negative coefficients  is  mixing if and only if  every coefficient in $A^n$ is positive .
	\end{prop}
	
	\begin{proof}	
		If $A^n$ has only positive coefficients, since $A$ don't have negative coefficients, the matrix  $A^{n+1}$ has only positive coefficients and $A$ is mixing.
			
		Consider the directed graph $\mathcal{G}$ with $n$ vertices corresponding to the matrix $A$. In this graph, there is an edge pointing from vertex $v_i$ to vertex $v_j$ if and only if $a_{ij} > 0$ in the matrix $A$. Since $A$ is mixing, it means that there is always a path from any vertex $v_i$ to any other vertex $v_j$ in the graph $\mathcal{G}$.
		
		To prove our Lemma, it is sufficient to demonstrate that there exists a path of length less than or equal to $n$ between any two vertices in the graph.

		Consider a path $\gamma = [v_{i(1)}, \cdots, v_{i(m)}]$ from $v_i = v_{i(1)}$ to $v_j = v_{i(m)}$ in the graph $\mathcal{G}$. Let $\gamma$ have $m$ vertices and be of minimum length among all paths connecting these vertices.
		
		If any two vertices of the path $\gamma$ are equal, it implies the existence of a cycle within $\gamma$. We can eliminate this cycle and obtain a new path $\gamma'$ from $v_i$ to $v_j$ with a shorter length than $\gamma$. This contradicts the assumption that $\gamma$ has the minimum length. Hence, all vertices in $\gamma$ must be distinct.
		
		Therefore, the length of $\gamma$ is necessarily less than or equal to $n$, since there are only $n$ distinct vertices in the graph.
	\end{proof}

	If $T$ is in the pseudo-Anosov class, we have $A^n(T) = A(T^n)$. According to the previous Lemma, if $A(T)$ is a non-negative matrix, then $A(T^n)$ is positive definite. This implies that $A(T^n)$ is mixing and does not have double boundaries. Therefore, for a given geometric type $T$ with $A(T)$ being a non-negative matrix, the positivity of $A(T^n)$ is equivalent to its mixing property, which in turn implies the absence of double boundaries.
	
	\begin{coro}\label{Coro: algoritmic mixing}
		Given a geometric type $T$, if $A(T^n)$ is not positive definite, then $T$ is not in the pseudo-Anosov class.
	\end{coro}
	
	Clearly, we only need to iterate $T$ at most $n$ times, and the upper bound given in the proposition remains valid. 
	This concludes our proof.
\end{proof}

%% file: TotalInvariant/Typeconjugacy.tex
\section{ A total invariant of conjugacy.}\label{Sec: Total Invariant}

\subsection{The geometric type of a geometric Markov partition.}\label{Subsec: Geo type de particion geometrica}
 Let $f: S \rightarrow S$ be a \textbf{p-A} homeomorphism, and let $\cR=\{R_i\}_{i=1}^n$ be a geometric Markov partition of $f$. We are going to associate an abstract geometric type to $(f,\cR)$. For such purpose, let's recall that the horizontal and vertical sub-rectangles of $(f,\cR)$ were introduced at the end of Definition \ref{Defi: Markov partition}. Now we shall label them using the vertical and horizontal directions in each rectangle of $(f,\cR)$ and assign them a geometrization. 

\begin{conv}\label{Defi: Label and geometrization of subrectangles}
Let $(f,\mathcal{R})$ be a geometric Markov partition. If $h_i\geq 1$ is the number of horizontal sub-rectangles of $(f,\mathcal{R})$ contained in $R_i$, the horizontal sub-rectangles of $(f,\mathcal{R})$ contained in $R_i$ are labeled from the \emph{bottom to the top} as $\{H^i_j\}_{j=1}^{h_i}$ with respect to the vertical direction of $R_i$. Similarly, if $v_k\geq 1$ is the number of vertical sub-rectangles of $(f,\mathcal{R})$ contained in $R_k$, they are labeled from \emph{left to right} as $\{V_l^k\}_{l=1}^{v_k}$ with respect to the horizontal direction of $R_k$.

In each horizontal sub-rectangle $H^i_j\subset R_i$ or $V^k_l\subset R_k$, their  the horizontal and vertical foliation have the same orientation than  the rectangles $R_i$ and $R_k$ where they are contained.
\end{conv}

\begin{defi}\label{Defi: geometric type of a Markov partition}
Let $\mathcal{R}=\{R_i\}_{i=1}^n$ be a geometric Markov partition for $f$. The geometric type of $(f,\mathcal{R})$ is given by:
\begin{equation}
T(f,\mathcal{R}):=(n,\{(h_i,v_i)\}_{i=1}^n, \rho,\epsilon),
\end{equation}
where $n$ is the number of rectangles in the family $\mathcal{R}$, and $h_i$ and $v_i$ are the numbers of horizontal and vertical sub-rectangles of $(f,\mathcal{R})$ contained in $R_i$.  In this manner the sets of horizontal and vertical labels of $(f,\mathcal{R})$  are respectively denoted as $\mathcal{H}(f,\mathcal{R})$ and $\mathcal{V}(f,\mathcal{R})$. Finally, if $f(H^i_j)=V^k_l$, the map $\rho:\mathcal{H}(f,\mathcal{R}) \rightarrow \mathcal{V}(f,\mathcal{R})$ is given by $\rho(i,j)=(k,l)$, and $\epsilon(i,j)=1$ if $f$ sends the vertical direction of $H^i_j$ to the vertical direction of $V^k_l$, and $\epsilon(i,j)=-1$ otherwise.
\end{defi}

\begin{defi}\label{Defi: pseudo-Anosov class}
An abstract geometric type $T$ is realized by a \textbf{p-A} $f$ if $f$ has a geometric Markov partition $\mathcal{R}$ such that $T = T(f,\mathcal{R})$. We can also say that $(f,\mathcal{R})$ realizes $T$. The \emph{pseudo-Anosov class} of abstract geometric types $\mathcal{G}\mathcal{T}(\text{pA}) \subset \mathcal{G}\mathcal{T}$ is the set of all those abstract geometric types that are realized by a pseudo-Anosov homeomorphism.
\end{defi}

 Theorem \ref{Theo: conjugated iff  markov partition of same type} is the main result of the present article as the rest of the classification problem \ref{Prob: Clasification} have its foundation on it. 

\begin{theo}\label{Theo: conjugated iff  markov partition of same type}
Let $f:S_f\rightarrow S_f$ and $g:S_g \rightarrow S_g$ be two  pseudo-Anosov homeomorphisms. Then, $f$ and $g$ have a geometric Markov partition of the same geometric type if and only if there exists an orientation preserving homeomorphism between the surfaces $h:S_f\rightarrow S_g$ that conjugates them, i.e., $g=h\circ f\circ h^{-1}$.
\end{theo}

This theorem will be a consequence of the following technical proposition. It's not strange that the major part of this section is devoted to its proof.

\begin{prop}\label{Prop: The relation determines projections}
Let $T\in \mathcal{G}\mathcal{T}(\text{pA})$ whose incidence matrix $A(T)$ is binary.  Let $(\Sigma_{A(T)}, \sigma_{A(T)})$ be the sub-shift associated with $T$ .  Then $T$ determines an equivalence relation $\sim_T$ on $\Sigma_{A(T)}$ with the following property: If $(f, \mathcal{R})$ realizes $T$ and $\pi_{(f,\mathcal{R})}:\Sigma_{A(T)} \to S$ is the induced projection (\ref{Defi: projection pi}), then any pair of codes $\underline{w}, \underline{v} \in \Sigma_{A(T)}$ are $\sim_T$-related if and only if $\pi_{(f,\mathcal{R})}(\underline{w}) = \pi_{(f,\mathcal{R})}(\underline{v})$.
\end{prop}

%%Secciones
\input{TotalInvariant/SubSubshift} %dinamica simbolica generales
\input{TotalInvariant/Subconjugation} %descomposicion de shift y relacion de equivalencia
\input{TotalInvariant/Subproof}

%% file: TotalInvariant/SubSubshift.tex
  \subsection{The sub-shift of finite type of a geometric type.}\label{Subsec: finite type shift Markov partition}

\subsubsection{The incidence matrix the sub-shift of a Markov partition.}
According to \cite[Exposition 10]{fathi2021thurston}, if $\cR$ is a Markov partition of a \textbf{p-A}  the incidence matrix of $(f,\cR)$, $A(f,\cR)$ $a_{ij}$ are $1$ if $f(\overset{o}{R_i}) \cap \overset{o}{R_k} \neq \emptyset$ and $0$ otherwise, without taking into account the number of intersections, information that is crucial for us, therefore we prefer adopt the following definition.

\begin{defi}\label{Defi: Incidence matrix markov partition}
Let $f$ be a \textbf{p-A} homeomorphism and let $\mathcal{R}$ be a Markov partition of $f$. The incidence matrix of $(f, \mathcal{R})$, denoted $A(f,\mathcal{R}) = (a_{i,k})$, has coefficient $a_{ik}$ equal to the number of horizontal sub-rectangles of $R_i$ that $f$ sends to vertical sub-rectangles of $R_k$.
\end{defi}

\begin{defi}\label{Defi: Incidence matrix of a type}
Let $T=\{n,\{(h_i,v_i)\}_{i=1}^n, \rho_T,\epsilon_T\}$ be an abstract geometric type. The  \emph{incidence matrix}  of $T$ is denoted by $A(T)$ and defined as the matrix $A(T) = (a_{ik})$ with coefficients given by:
$
a_{i,k}=\#\{j \in \{1,\cdots,h_i\}: \rho_T=(k,l)\}.
$
\end{defi}

The following lemma is almost and observation.
\begin{lemm}
Let $(f,\cR)$ be  geometric Markov partition then $A(f,\cR):=A(T(f,\cR))$.
\end{lemm}

A square matrix $A=(a_{ij})$ of size $n\times n$ is \emph{non-negative} if for all $1\leq i,j\leq n$, $a_{i,j}\in \mathbb{N}$, and it is \emph{positive defined} if for all $1\leq i,j\leq n$, $a_{i,j}\in \mathbb{N}_+$. If $n \in \mathbb{N}$, the notation for the coefficients of its $n$-th power is $A^n = (a_{i,j}^{(n)})$.

\begin{defi}\label{Defi: Mixing and binary}
Let $A$ be a non-negative matrix, which will be called \emph{mixing} if there exists an $N\in \mathbb{N}_+$ such that $A^N$ is positive defined. A matrix $A$ whose coefficients are either $0$ or $1$ is called \emph{binary}.
\end{defi}

\begin{rema}\label{Rema: Convencion sobre matriz binaria}
	
	Many arguments and definitions that we need to introduce to prove Theorem \ref{Theo: conjugated iff markov partition of same type} would be cumbersome or complicated if $A(T(f,\mathcal{R}))$ has coefficients greater than one. In general, $f(\overset{o}{R_i}) \cap \overset{o}{R_k}$ has more than one connected component, so $a_{i,k} > 1$, and the incidence matrix is not binary.   Anyway,  given a geometric Markov partition $(f,\mathcal{R})$, the set of horizontal sub-rectangles of $\mathcal{R}$, denoted $\mathcal{H}(\mathcal{R})$, is by itself a Markov partition of $f$. In \cite[Definition 3.3]{IntiThesis} (to appear in \cite{IntiII}), the family of horizontal sub-rectangles is endowed with the lexicographic order in their indexes $(i,j)$, and all the sub-rectangles have the same vertical orientation as the rectangle of $\mathcal{R}$ in which they are contained. The geometric Markov partition obtained in this manner is denoted $(f,\mathcal{H}(\mathcal{R}))$ and called \emph{the horizontal refinement} of $(f,\mathcal{R})$. Its geometric type is computed in \cite[Proposition 3.2.]{IntiThesis} (to appear in \cite{IntiII}), which establishes that $A(T,H(\mathcal{R}))$ is binary and justifies such a hypothesis in the rest of this section.
\end{rema}

The set of \emph{bi-infinite} sequences in $n$-digits is denoted $\Sigma = \prod_{\mathbb{Z}} \{1, \cdots, n\}$, and its elements are called \emph{codes}, denoted $\underline{w} = (w_z)_{z \in \mathbb{Z}}$. We endow $\Sigma$ with the topology induced by the metric: $d_{\Sigma}(\underline{x},\underline{y})= \sum_{z\in\mathbb{Z}} \frac{\delta(x_z,y_z)}{2^{|z|}}$, where $\delta(x_z, y_z) = 0$ if $x_z = y_z$ and $\delta(x_z, y_z) = 1$ otherwise. Such topology coincides with product topology and Tychonoff's theorem implies that $\Sigma$ is a compact space. The \emph{shift map} over $\Sigma$ is the homeomorphism $\sigma: \Sigma \rightarrow \Sigma$ given by $(\sigma(\underline{w}))_{z}= w_{z+1}$, for all $z\in \mathbb{Z}$ and any $\underline{w}\in \Sigma$. The \emph{full-shift} (in $n$ symbols) is the dynamical system $(\Sigma,\sigma)$. If $A \in \mathcal{M}{n \times n}({0,1})$ is a binary and mixing matrix, the set $\Sigma_A := {\underline{w} = (w_z){z \in \mathbb{Z}} \in \Sigma : \forall z \in \mathbb{Z}, (a_{w_z,w_{z+1}}) = 1 }$ is a compact and $\sigma$-invariant set (See \cite[Chapter 1]{KitchensSymDym}), and the following definition is meaningful.

\begin{defi}\label{Defi: Sub-shift of finite type}
Let $A \in \mathcal{M}_{n \times n}(\{0,1\})$ be  a binary and mixing matrix. The \emph{sub-shift of finite type} associated with $A$ is the dynamical system $(\Sigma_A, \sigma_A)$, where $\sigma_A := \sigma|_{\Sigma_{A}}$. 
\end{defi}

According to \cite[Lemma 10.21]{fathi2021thurston}, the incidence matrix of any pseudo-Anosov homeomorphism is mixing and we can pose the following definition, observe that $\mathcal{G}\mathcal{T}(\textbf{pA})^{SP}$ is not empty by \cite[Proposition 3.2.]{IntiThesis}.

\begin{defi}\label{Defi: Simbolically prensentable and induced shift}
The set of geometric types in the pseudo-Anosov class whose incidence matrices are binary will be denoted by $\mathcal{G}\mathcal{T}(\textbf{pA})^{sp} \subset \mathcal{G}\mathcal{T}(\textbf{pA})$, and they are called \emph{symbolically presentable} geometric types. Let $T \in \mathcal{G}\mathcal{T}(\textbf{pA})^{sp}$, the \emph{sub-shift induced} by $T$ is given by $(\Sigma_{A(T)}, \sigma_{A(T)})$.
\end{defi}

The sub-shift $(\Sigma_A, \sigma)$ and $f:S\to S$ are related through the \emph{projection} $\pi_{(f,\mathcal{R})}:\Sigma_{A(T)}\to S$ introduced below, such map takes a code $\underline{w} = (x_z)_{z \in \mathbb{Z}} \in \Sigma_{A}$ and sends it to a point $x \in S$ whose orbit follows the itinerary imposed by $\underline{w}$; for all $z \in \mathbb{Z}$, $f^z(x) \in R_{w_z}$. Clearly such projection strongly depends on Markov partition $(f, \mathcal{R})$.
 
\begin{defi}\label{Defi: projection pi}
	Let $T$ be a geometric type in the pseudo-Anosov class, with an incidence matrix $A := A(T)$ that is binary, meaning it has coefficients in ${0,1}$. Let $(f,\cR)$ be a realization of $T$. The \emph{projection} with respect to $(f,\cR)$, denoted as $\pi_f: \Sigma_A \rightarrow S$, is defined by the relation:
		\begin{equation}
		\pi_f(\underline{w})=\bigcap_{n\in \NN} \overline{\cap_{z=-n}^{n} f^{-z}(\overset{o}{R_{w_z}})}.
		\end{equation}
	for every $\underline{w}=(w_z)_{z\in \ZZ} \in \Sigma_A$.	
\end{defi}
 
 \subsubsection{The projection $\pi_{(f,\mathcal{R})}$.}
 
The proof of Proposition \ref{Prop:proyecion semiconjugacion} is essentially the same as given in \cite[Lemma 10.16]{fathi2021thurston}, where $\pi_{(f,\mathcal{R})}$ has a slightly different definition, for them, $\phi_f(\underline{w}) = \bigcap_{z\in \mathbb{Z}} f^{-z}(R_{w_z})$, but in order to prove that $\pi_{(f,\mathcal{R})}$ is well-defined, the authors require the rectangles in $\mathcal{R}$ to be embedded, which is not necessary for  our statement. The intersections of the closures $\bigcap_{z=-n}^{n} f^{-z}(\overset{o}{R_{w_z}})$ are rectangles whose diameters tend to $0$, so for large enough $n$, such intersection is an embedded rectangle, then, the reader can follow the proof in \cite[Lemma 10.16]{fathi2021thurston} to be sure that $\pi_{(f,\mathcal{R})}$ is well-defined, continuous, and a semi-conjugation; therefore, we don't elaborate more on this part. However, the \emph{finite-to-one property} is fundamental for our further study and must introduce a characterization of the codes in the fiber $\pi_{(f,\mathcal{R})}^{-1}(x)$, which we call \emph{sector-codes} of $x$ while establish that $\pi_{(f,\mathcal{R})}^{-1}(x)$ is finite.

\begin{prop}\label{Prop:proyecion semiconjugacion}
If $\cR$ is the Markov partition of a pseudo-Anosov homeomorphism $f:S\rightarrow S$, with a binary incidence matrix $A:=A(\cR)$, then the projection $\pi_f:\Sigma_A \rightarrow S$ is a continuous, surjective, and finite-to-one map. Furthermore, $\pi_f$ semi-conjugates $f$ with $\sigma$, i.e., $f\circ\pi_f=\pi_f \circ \sigma_A$.
\end{prop}

\begin{lemm}\label{Lemm: projection is surjective}
The projection $\pi_f:\rightarrow S$ is surjective.
\end{lemm}

For all $x\in S$, we will construct an element of $\Sigma_A$ that projects to $x$. The \emph{sector codes} of $x$ that we define below will do the job. It was shown in Lemma \ref{Lemm: sector contined unique rectangle} that each sector is contained in a unique rectangle of the Markov partition, and Proposition  \ref{Prop: image secto is a sector} shows that the image of a sector is a sector. This allows  the following definition.

\begin{defi}\label{Defi: Sector codes}
Let $\cR=\{R_i\}_{i=1}^n$ be a Markov partition whose incidence matrix is binary. Let $x \in S$ be a point with sectors $\{e_1(x),\cdots, e_{2k}(x)\}$ (where $k$ is the number of stable or unstable separatrices in $x$).
The \emph{sector code} of $e_j(x)$ is the sequence $\underline{e_j(x)}=(e(x,j)_z)_{z\in \ZZ} \in \Sigma$, given by the rule: $e(x,j)_z:=i$, where $i\in \{1,\dots,n\}$ is the index of the unique rectangle in $\cR$ such that the sector $f^z(e_j(x))$ is contained in the rectangle $R_i$.
\end{defi}

Any $\underline{w}\in\Sigma$ that is equal to the sector code of $e_j(x)$ (for some $j$ ) is called a sector code of $x$. The space $\Sigma$ of bi-infinite sequences is larger than $\Sigma_A$. We need to show that every sector code is, in fact, an \emph{admissible code}, i.e., that $\underline{e_j(x)}\in \Sigma_A$.

\begin{lemm}\label{Lemm: sector code is admisible}
For every $x \in S$, every sector code $\underline{e}:=\underline{e_j(x)}$ is an element of $\Sigma_A$.
\end{lemm}

\begin{proof}
Let $A=(a_{ij})$ the incidence matrix.	The code $\underline{e}=(e_z)$ is in $\Sigma_A$ if and only if for all $z\in \ZZ$, $a_{e_z e_{z+1}}=1$. By definition, this happens if and only if $f(\overset{o}{R_{e_z}})\cap \overset{o}{R_{e_{z+1}}}\neq \emptyset$.

Let $\{x_n\}_{n\in \NN}$ be a sequence converging to $f^z(x)$ and contained in the sector $f^z(e)$. By Lemma \ref{Lemm: sector contined unique rectangle} the sector $f^z(e)$ is contained in a unique rectangle $R_{e_z}$, and we can assume $\{x_n\}\subset R_{e_z}$. Moreover, there exists $N\in \NN$ such that $x_n \in \overset{o}{R_{e_z}}$ for all $n\geq N$. Remember that the sector $f^z(e)$ is bounded by two consecutive local stable and unstable separatrices of $f^z(x)$: $F^s(f^z(x))$ and $F^u(f^z(x))$. If for every $n\in \NN$, $x_n$ is contained in the boundary of $R_{e_z}$, this boundary component is a local separatrice of $x$ between $F^s(f^z(x))$ and $F^u(f^z(x))$, which is not possible.

Since the image of a sector is a sector, the sequence $\{f(x_n)\}$ converges to $f^{z+1}(x)$ and is contained in the sector $f^{z+1}(e)$. The argument in the last paragraphs also applies to this sequence, and $f(x_n)\in \overset{o}{R_{e_{z+1}}}$ for $n$ big enough. This proves that $f(\overset{o}{R_{e_z}})\cap \overset{o}{R_{e_{z+1}}}$ is not empty.
	
\end{proof}

The sector codes of a point $x$ are not only admissible; as the following Lemma shows, they are, in fact, the only codes in $\Sigma_A$ that project to $x$.  

\begin{lemm}\label{Lemm: every code is sector code }
If $\underline{w}=(w_z)\in \Sigma_A$ projects under $\pi_f$ to $x$, then $\underline{w}$ is equal to a sector code of $x$.
\end{lemm}

\begin{proof}
		
		For each $n \in \NN$, we take the rectangle $F_n=\cap_{j=-n}^n f^{-j}(\overset{o}{R_{w_j}})$, which is non-empty because $\underline{w}$ is in $\Sigma_A$. The following properties hold:

	\begin{itemize}
		\item[i)] $\pi_f(\underline{w})=x \in \overline{F_n}$ for every $n\in \ZZ$.
		\item[ii)] For all $n\in \NN$, $F_{n+1}\subset F_n$.
		\item[iii)]	For every $n\in \NN$, there exists at least one sector $e$ of $x$ contained in $F_n$. If this does not occur, there is $\epsilon>0$ such that the regular neighborhood of size $\epsilon$ around $x$, given by Theorem \ref{Theo: Regular neighborhood}, is disjoint from $F_n$. However, $x\in F_n$ by item i).
	
		\item[iv)] If the sector $e \subset H_n$, then for every $m\in  \ZZ$ such that $\vert m\vert \leq n$:
		$$
		f^m(e) \subset f^m(H_n)= \cap_{j=m-n}^{m+n}f^{-j}(\overset{o}{R_{w_{m-j}}}) \subset \overset{o}{R_m}.
		$$ 
		which implies that $e_m=w_m$ for all $m\in \{-n,\cdots,n\}$.
	\end{itemize}
	
	By item $ii)$, if a sector $e$ is not in $F_n$, then $e$ is not in $F_{n+1}$. Together with the fact that for all $n$ there is always a sector in $F_n$ (item $iii)$), we deduce that there is at least one sector $e$ of $x$ that is in $F_n$ for all $n$. Then, we apply point $iv)$ to deduce $e_z=w_z$ for all $z$.
\end{proof}
Let $x$ be a point with $k$ stable and $k$ unstable separatrices. Then $x$ has at most $2k$ sector codes projecting to $x$.
\begin{coro}\label{Coro: Caracterisation fibers}
For all $x\in S$, if $x$ has $k$ separatrices, then $\pi_f^{-1}(x)=\{\underline{e_j(x)}\}_{j=1}^{2k}$. In particular, $\pi_f$ is finite-to-one.
\end{coro}

This ends the proof of Proposition \ref{Prop:proyecion semiconjugacion}.

\subsection{The quotient space  is a surface.}\label{subsec: quotien surface}

There is a very natural equivalence relation in $\Sigma_A$ in terms of the projection $\pi_f$. Two codes $\underline{w}$ and $\underline{v}$ in $\Sigma_A$ are $f$-related, written as $\underline{w} \sim_f \underline{v}$, if and only if $\pi_f(\underline{w}) = \pi_f(\underline{v})$. The quotient space is denoted as $\Sigma_f = \Sigma_A/\sim_f$, and $[\underline{w}]_f$ represents the equivalence class of $\underline{w}$. If $\underline{w} \sim_f \underline{v}$, by definition:
$$
[\pi_f]([\underline{w}]_f)=\pi_f(\underline{w})=\pi_f(\underline{v})=[\pi_f]([\underline{v}]_f).
$$

Even more, since $\pi_f: \Sigma_A \rightarrow S$ is a continuous function, $\Sigma_A$ is compact, and $S$ is a Hausdorff topological space, the \emph{closed map lemma} implies that $\pi_f$ is a closed function. Furthermore, since $\pi_f$ is surjective, it follows that $\pi_f$ is a quotient map. Therefore, the projection $\pi_f$ induces a homeomorphism $[\pi_f]: \Sigma_f \rightarrow S$ in the quotient space. The shift behaves well under this quotient since: $$[\sigma_f([\underline{w}]_f)]_f:=[\sigma([\underline{w}])]_f
$$
and $\underline{w} \sim_f \underline{v}$. The semi-conjugation of $f$ and $\sigma$ through $\pi_f$ implies that:
$$
[\sigma(\underline{w})]\sim_f [\sigma(\underline{v})].
$$
then the map $[\sigma]:\Sigma_f \rightarrow \Sigma_f$ is well defined and is, in fact, a homeomorphism as well, and

\begin{eqnarray*}
[\pi_f]\circ[\sigma]([\underline{w}]_f)=
[\pi_f][\sigma_A(\underline{w})]_f =\pi_f(\sigma(\underline{w})))= \\
f \circ (\pi_f(\underline{w}))= f\circ[\pi_f]([\underline{w}]_f).
\end{eqnarray*}

Therefore, $[\pi_f]$ determines a topological conjugation between $f$ and $[\sigma_A]$. This implies that $\Sigma_f$ is a surface homeomorphic to $S$, and $[\sigma]$ is a pseudo-Anosov homeomorphism. It is almost the ideal situation. Let us summarize this discussion in the following proposition.

\begin{prop}\label{Prop: quotien by f}
	The quotient space $\Sigma_f:=\Sigma_A/\sim_f$ is homeomorphic to $S_f$, and the quotient shift $\sigma_f:\Sigma_f\rightarrow \Sigma_f$ is a generalized pseudo-Anosov homeomorphism, topologically conjugated to $f:S_f \rightarrow S_f$ through the quotient projection:
	$$
	[\pi_f]:\Sigma_f\rightarrow S_f.
	$$
\end{prop}

If we have two pseudo-Anosov maps $f:S_f\rightarrow S_f$ and $g:S_g\rightarrow S_g$ with Markov partitions of the same geometric type, after a horizontal refinement, they have the same incidence matrix $A$ with coefficients in $\{0,1\}$ and are associated with the same sub-shift of finite type $(\Sigma_A,\sigma)$. However, the projections $\pi_f$ and $\pi_g$ are not necessarily the same. In particular, while $\Sigma_A/\sim_f$ is homeomorphic to $S_f$ and $\Sigma_A/\sim_g$ is homeomorphic to $S_g$, we cannot affirm that $S_f$ is homeomorphic to $S_g$.

The problem arises from the fact that given $x\in S_f$ and $y\in S_g$, we don't know if $\pi_f^{-1}(x)\cap\pi_g^{-1}(x)\neq \emptyset$ implies $\pi_f^{-1}(x)=\pi_g^{-1}(x)$. In the proof of the finite-to-one property Lemma, it was shown that every code in $\pi_f^{-1}(x)$ is a sector code of $x$. Therefore, if $\pi_f^{-1}(x)\cap \pi_g^{-1}(y)\neq \emptyset$, there is a common sector code of $x$ and $y$. However, this does not imply a unique (or continuous) correspondence between the sets of sectors of $x$ and $y$. For example, it is possible for $x$ to have a different number of prongs than $y$.

This ambiguity cannot be resolved by the incidence matrix alone because two sector codes are considered different if, after iterations of $\sigma$, they are in different rectangles. If a point $x$ is in the corner of $4$ rectangles and $y$ is in the corner of $3$ rectangles, then $x$ has $4$ different sector codes, while $y$ has only $3$. This discrepancy in the number of incident rectangles at a point is overcome by the implementation of the geometric type.

To address this, we will construct another quotient space of $\Sigma_A$ called $\Sigma_T$ and define an equivalence relation $\sim_T$ in terms of the geometric type $T$.

%% file: TotalInvariant/Subconjugation.tex
\subsection{The equivalent relation $\sim_T$.}\label{Sub-sec:Sub-conjugation}

 In this section, we will define a the equivalence relation that satisfy the properties enumerated  in Proposition \ref{Prop: The relation determines projections}. Given a Markov partition $\cR=\{R_i\}_{i=1}^n$, we have introduced the following notation: the interior of the partition is denoted by $\overset{o}{\mathcal{R}}:=\cup_{i=1}^n \overset{o}{R_i}$, and the stable or unstable boundary of $\mathcal{R}$ is denoted by $\partial^{s,u}\mathcal{R}:=\cup_{i=1}^n \partial^{s,u}R_i$.
 
 The general idea is to start with a geometric type $T$ within the pseudo-Anosov class and define a decomposition of $\Sigma_{A(T)}$ in terms of $T$ into three subsets: $\Sigma_{I(T)}$, $\Sigma_{S(T)}$, and $\Sigma_{U(T)}$. We will then introduce three relations, $\sim_{I,s,u}$, over these sets based on the properties of $T$. These relations will be extended to an equivalence relation $\sim_T$ in $\Sigma_A$. Finally, we will prove that for any pair $(f,\mathcal{R})$ that realizes $T$, $\pi_f(\underline{w})=\pi_f(\underline{v})$ if and only if $\underline{w}\sim_T \underline{v}$. The subsets and the equivalence relation $\sim_T$ will be determined through several steps, which we will outline below:

\begin{itemize} 
	\item 	We define the set of periodic points in $\Sigma_A$ denoted by $\Sigma_P$. It is proven that $\pi_f(\Sigma_P)$ corresponds to the periodic points of $f$, and for every periodic point $p$ of $f$, $\pi_f^{-1}(p)\subset \Sigma_P$.

	\item 	We construct a finite family of positive codes $\mathcal{S}(T)^+$, called $s$-\emph{boundary label codes}. It is proven that every code $\underline{w}\in \Sigma_A$ whose positive part is in $\mathcal{S}(T)^+$ is projected by $\pi_f$ to the stable boundary of a Markov partition $(f,\mathcal{R})$ whose geometric type is $T$. This determines a subset $\underline{\mathcal{S}(T)}\subset \Sigma_{A(T)}$ of $s$-\emph{boundary codes}.

	\item 	A relation $\sim_s$ is defined in $\underline{\mathcal{S}(T)}$, which satisfies the property that two codes in $\underline{\mathcal{S}(T)}$ are $\sim_s$-related, i.e., $\underline{w}\sim_s \underline{v}$ if and only if they project to the same point $\pi_f(\underline{w})=\pi_f(\underline{v})$ in the stable boundary of two adjacent rectangles of the Markov partition.

	\item  	Similarly, we construct a finite family of negative codes $\mathcal{U}(T)^-$, called $u$-\emph{boundary label codes}, and the subset $\underline{\mathcal{U}(T)}\subset \Sigma_{A(T)}$ of codes whose negative part is in $\mathcal{U}(T)^-$. It is proven that $\underline{\mathcal{U}(T)}$ projects by $\pi_f$ to the unstable boundary of rectangles in the Markov partition. A relation $\sim_u$ is defined in $\underline{\mathcal{U}(T)}$ such that if two codes are $\sim_u$-related, they are projected to the same point in the unstable boundary of two adjacent rectangles of the Markov partition.

	\item 	Using the relations $\sim_s$ and $\sim_u$, we define a set of periodic codes $\Sigma_{\mathcal{P}(T)}\subset \Sigma_A$, which are projected by $\pi_f$ to the boundary periodic points of $\mathcal{R}$.
	
	\item 	We define the $s$-\emph{boundary leaf codes} $\Sigma_{\mathcal{S}(T)}\subset\Sigma_A$ as the set of codes $\underline{w}$ for which there exists $k\in \mathbb{Z}$ such that $\sigma^k(\underline{w})\in \underline{\mathcal{S}(T)}$. Similarly, the $u$-\emph{boundary leaf codes} $\Sigma_{\mathcal{U}(T)}\subset\Sigma_A$ is the set of codes $\underline{w}$ for which there exists $k\in \mathbb{Z}$ such that $\sigma^k(\underline{w})\in \underline{\mathcal{U}(T)}$. The relations $\sim_s$ and $\sim_u$ are extended to the sets of $s$-\emph{boundary leaf codes} and $u$-\emph{boundary leaf codes}, respectively, by declaring that $\underline{w},\underline{v}\in \Sigma_{\mathcal{S}(T)}$ are $S$-related, i.e., $\underline{w}\sim_S \underline{v}$, if and only if there is $k\in\mathbb{Z}$ such that $\sigma^k(\underline{w}),\sigma^k(\underline{v})\in \Sigma_{\mathcal{S}(T)}$ and $\sigma^k(\underline{w})\sim_s\sigma^k(\underline{v})$. Similarly, we define the relation $\sim_U$ over $\Sigma_{\mathcal{U}(T)}\subset \Sigma_A$.

	\item We prove that $\pi_f(\Sigma_{\mathcal{S}(T)})$ and $\pi_f(\Sigma_{\mathcal{U}(T)})$ corresponds to the stable and unstable leaves of the periodic boundary points of $(f,\mathcal{R})$, and they are the only codes in $\Sigma_A$ with this property.
	
	\item 	It is proved that if $\underline{w}, \underline{v} \in \Sigma_{\mathcal{S}(T)}, \Sigma_{\mathcal{U}(T)}$ are $S, U$-related, then $\pi_f(\underline{w}) = \pi_f(\underline{v})$.
	
	\item 	We extend the relations $\sim_S$ and $\sim_U$ to a relation $\sim_{S,U}$ in the union $\Sigma_{\mathcal{S}(T)} \cup \Sigma_{\mathcal{U}(T)}$. It is proved that $\underline{w}, \underline{v} \in \Sigma_{\mathcal{S}(T)} \cup \Sigma_{\mathcal{U}(T)}$ are $\sim_{S,U}$-related if and only if they are projected by $\pi_f$ to the same point.

	\item 	The subset $\Sigma_{\mathcal{I}(T)}$ of $\Sigma_A$ is the complement of $\Sigma_{\mathcal{S}(T)} \cup \Sigma_{\mathcal{U}(T)}$ in $\Sigma_A$. $\pi_f(\Sigma_{I})$ is contained in the complement of the stable and unstable foliation of the periodic boundary points of $\mathcal{R}$ and are called \emph{totally interior codes}. We define $\sim_I$ as the identity relation in $\Sigma_{\mathcal{I}(T)}$. It is proved that $\underline{w} \in \Sigma_{I}$ if and only if all the sector codes of $\pi_f(\underline{w})$ are the same.

	\item 	We extend the relations $\sim_{S,U,I}$ to a relation $\sim_T$ in $\Sigma_A$. We prove that it is an equivalence relation in $\Sigma_A$ and satisfies Proposition \ref{Prop: The relation determines projections}.
	
\end{itemize}

\input{TotalInvariant/Subperiodicodes.tex}

\input{TotalInvariant/Subrelation.tex}

%% file: TotalInvariant/Subperiodicodes.tex
\subsubsection{Periodic codes and periodic points.}
Let $T$ be a geometric type whose incidence matrix is binary. Let $f: S \rightarrow S$ be a generalized pseudo-Anosov homeomorphism with a Markov partition of geometric type $T$, denoted as $(f, \mathcal{R})$. Consider $\pi_f: \Sigma_A \rightarrow S$ as the projection induced by $(f, \mathcal{R})$ into $S$. In this subsection, $A := A(T)$ represents the incidence matrix of $T$. The periodic points in the dynamical system $(\Sigma_A, \sigma)$ are denoted as $\text{Per}(\sigma) \subset \Sigma_A$, and the periodic points of $f$ are denoted as $\text{Per}(f) \subset S$. Note that $(f, \mathcal{R})$ implies that $\mathcal{R}$ is seen as a Markov partition of $f$. Hence, to avoid overloading the notation, we may omit the explicit mention of the homeomorphism if it is clear from the context.

The following lemma is classical in the literature when the authors work with the more classical definition of the projection, anyway, this lemma  represents the type of result we aim to prove. On one hand, we have periodic shift codes determined by $T$, and on the other hand, a valid property holds for every pseudo-Anosov homeomorphism with a Markov partition of the prescribed geometric type.

\begin{lemm}\label{Lemm: Periodic to peridic}
Let $T$ be a geometric type in the pseudo-Anosov class. Assume that its incidence matrix is binary, and $(f, \mathcal{R})$ is a pair realizing $T$. If $\underline{w} \in \text{Per}(\sigma)$ is a periodic code, then $\pi_f(\underline{w})$ is a periodic point of $f$. Moreover, if $p$ is a periodic point for $f$, then $\pi^{-1}_f(p) \subset \text{Per}(\sigma)$, meaning that all codes projecting onto $p$ are periodic.
\end{lemm}

\begin{proof}
If $\underline{w}\in \Sigma_{A}$ is periodic of period $k$, the semi-conjugation given by $\pi_f$ implies that:

$$
f^k(\pi_f(\underline{w}))=\pi_f(\sigma^k(\underline{w}))=\pi_f(\underline{w}).
$$

If $p$ is a periodic point of period $P$, and let $\underline{w} \in \Sigma_{A}$ be a code such that $\pi_f(\underline{w}) = p$, then from our proof that $\pi_f$ is finite-to-one, we established that every code projecting to $p$ is a sector code. Thus, we have $\underline{w} = \underline{e}$, corresponding to some sector $e$ of $p$.

If $p$ is a saddle point of index $k$ (with $k \geq 1$), there are at most $2k$ rectangles in $\cR$ that can contain $p$. Let's consider the rectangle $R_{w_0}$. We are aware that under $f$, the sectors incident to $p$ can be permuted, with each sector contained within a single rectangle. The action of $f^{kP}$ on the set of sectors of $p$ is a permutation. Hence, there exists a multiple of the period of $p$, denoted as $Q = km$, such that $f^Q(e) = e$. However, once we observe that a sector $e$ returns to itself after the action of $f^Q$ for all $m \in \{1, 2, \ldots, Q\}$, it  becomes necessary that the sector $f^{Q+m}(e)$ and $f^{m}(e)$ must coincide. Consequently, the sector code $\underline{e} = \underline{w}$ is periodic, with a period $Q$.

\end{proof}

 The following Lemma characterizes the periodic boundary points of $(f,\cR)$ in terms of their iterations under $f$. It is a technical result that we will use in future arguments.

\begin{lemm} \label{Lemm: no periodic boundary points}
	Let $\underline{w} \in \Sigma_A$ be a code such that for every $k \in \mathbb{Z}$, $f^k(\pi_f(\underline{w})) \in \partial^s\mathcal{R}$, then $\pi_f(\underline{w})$ is a periodic point of $f$. Similarly, if $f^k(\pi_f(\underline{w})) \in \partial^u\mathcal{R}$ for all integers $k$, then $\underline{w}$ is a periodic code.
\end{lemm}

\begin{proof}

Suppose $\underline{w}$ is non-periodic. Lemma \ref{Lemm: Periodic to peridic}  implies that $x = \pi_f(\underline{w})$ is non-periodic and lies on the stable boundary of $\mathcal{R}$, which is a compact set with each component having finite $\mu^u$ length. Take $[x,p]^s$ to be the stable segment joining $x$ with the periodic point on its leaf (which exists by Lemma \ref{Lemm: Boundary of Markov partition is periodic}). For all $m \in \mathbb{N}$,
$$
\mu^u([f^{-m}(x), f^{-m}(p)]^s)=\mu^u(f^{-m}[x,p]^s)=\lambda^m\mu^u([x,p]^s),
$$
 which is unbounded. Therefore, there exists $m \in \mathbb{N}$ such that $f^{-m}(x)$ is no longer in $\partial^s\mathcal{R}$, contradicting the hypothesis of our lemma. A similar reasoning applies to the second proposition.
 \end{proof}

This will allow us to provide a characterization of the inner periodic points of $\mathcal{R}$. The proof of Corollary \ref{Coro: interior periodic points unique code} follows from the fact that all sector codes of a point satisfying the hypotheses are equal.

\begin{coro}\label{Coro: interior periodic points unique code}

Let $\underline{w}\in \Sigma_A$ be a periodic code. If $\pi_f(\underline{w})\in \overset{o}{\cR}$, then for all $z\in \ZZ$, $f^z(\pi_f(\underline{w}))\in \overset{o}{\cR}$, which means it stays in the interior of the Markov partition. Moreover, $\pi^{-1}_f(\pi_f(\underline{w}))=\{\underline{w}\}$ have a unique point.

\end{coro}

\begin{proof}
 If, for some $z\in \ZZ$, $f^z(x)$ is in the stable (or unstable) boundary of the Markov partition, as $x$ is periodic and the stable boundary of a Markov partition is $f$-invariant, the entire orbit of $x$ lies in the stable boundary of $\cR$. This leads to a contradiction because at least one point in the orbit of $x$ is in the interior of $\cR$.
 
 From this perspective, all the sectors of $f^z(x)$ are contained within the interior of the same rectangle, and therefore all its sector codes are the same.
\end{proof}

\subsubsection{Decomposition of the surface: totally interior points and boundary laminations.}

The property that the entire orbit of $x$ is always in the interior of the partition is very important because it distinguishes the points that are in the complements of the stable and unstable laminations of periodic boundary points. We are going to give them a name.

\begin{defi}\label{Defi: totally interior points}
Let $(f,\cR)$ be a realization of $T$. The \emph{totally interior points} of $(f,\cR)$ are points $x \in \cup\cR=S$ such that for all $z \in \ZZ$, $f^z(x) \in \overset{o}{\cR}$. They are denoted Int$(f,\cR)\subset S$.
\end{defi}

 Proposition \ref{Prop: Carterization injectivity of pif} characterizes the points of $S$ where the projection $\pi_f$ is invertible as the \emph{totally interior points} of $\cR$.

\begin{prop}\label{Prop: Carterization injectivity of pif}
Let $x$ be any point in  $S$. Then $\vert \pi_f^{-1}(x)\vert =1$ if and only if $x$ is a totally interior point of $\cR$.
\end{prop}

\begin{proof}
If $\vert \pi_f^{-1}(x)\vert = 1$, for all $z\in \ZZ$ the sector codes of $f^z(x)$ are all equal. Therefore, for all $z\in \ZZ$, the sectors of $f^z(x)$ are all contained in the interior of the same rectangle. Furthermore, the union of all sectors determines an open neighborhood of $f^z(x)$, which must be contained in the interior of a rectangle. Therefore, $f^z(x)\in \overset{o}{\cR}$.

On the other hand, if $f^z(x)\in \overset{o}{\cR}$ for all $z\in \ZZ$, then all sectors of $f^z(x)$ are contained in the same rectangle. This implies that the sector codes of $x$ are all equal. In view of Lemma  \ref{Lemm: every code is sector code }, this implies $\vert \pi_f^{-1}(x)\vert = 1$.

\end{proof}

Lemma \ref{Lemm: Caraterization unique codes}  characterizes the stable and unstable manifolds of periodic boundary points of $\cR$ as the complement of the totally interior points. In this manner, we have a decomposition of $S$ as the union of three sets: the totally interior points, the stable lamination of periodic $s$-boundary points $\cF^s(\text{Per}^{s}(f,\cR))$, and the unstable lamination of periodic $u$-boundary points $\cF^u(\text{Per}^{u}(f,\cR))$.

\begin{lemm}\label{Lemm: Caraterization unique codes}
A point $x\in S$ is a totally interior point of $\cR$ if and only if it is not in the stable or unstable leaf of a periodic boundary point of $\cR$.
\end{lemm}

\begin{proof}

Suppose that for all $z\in \ZZ$, $f^z(x)\in \overset{o}{\cR}$ and also $x$ is in the stable (unstable) leaf of some periodic boundary point $p\in \partial^s R_i$ ($p\in \partial^u R_i$). The contraction (expansion) on the stable (unstable) leaves of $p$ implies the existence of a $z\in \ZZ$ such that $f^z(x)\in \partial^s R_i$ ($f^z(x)\in \partial^u R_i$), which leads to a contradiction.

On the contrary, Lemma  \ref{Lemm: Boundary of Markov partition is periodic} implies that the only stable or unstable leaves intersecting the boundary $\partial^{s,u} R_i$ are the stable and unstable leaves of the $s$-boundary and $u$-boundary periodic points. If $x$ is not in these laminations, neither are its iterations $f^z(x)$, and we have $f^z(x)\in \overset{o}{\cR}$ for all $z\in \ZZ$.
\end{proof}

It will be easier to provide a combinatorial definition of codes that project to stable and unstable laminations of periodic $s,u$-boundary points than to define codes that project to a totally interior point based on $T$. However, we can define the latter as the complement of the union of the former.

In the next section, we will demonstrate how to construct codes that project to stable and unstable $\cR$-boundaries and then expand this set to include codes that project to theirs stable and unstable leaves.

\subsubsection{Codes that project to the stable or unstable leaf.}

 Given $\underline{w} \in \Sigma_A$, the task is to determine, in terms of $(\Sigma_A, \sigma)$ and $T$, the subset $\underline{F}^s(\underline{w})$ of codes that are projected by $\pi_f$ to the stable leaf of $\pi_f(\underline{w})$. The following definition is completely determined by $T$ and its sub-shift, this what we use to call a \emph{combinatorial definition}.

\begin{defi}\label{Defi: s,u-leafs}
Let $\underline{w} \in \Sigma_A$. The set of \emph{stable leaf codes} of $\underline{w}$ is defined as follows:
	$$
	\underline{F}^s(\underline{w}):=\{\underline{v}\in \Sigma_A: \exists Z\in \ZZ \text{ such that } \forall z\geq Z, \, v_z=w_z\}.
	$$
Similarly, the \emph{unstable leaf codes} of $\underline{w}$ is the set:
	$$
	\underline{F}^u(\underline{w}):=\{\underline{v}\in \Sigma_A: \exists Z\in \ZZ \text{ such that } \forall z\leq Z, \, v_{z}=w_{z}\}.
	$$
\end{defi}

We will now prove that these sets project to the stable leaf at the point $\pi_f(\underline{w})$, and that they are the only codes that do so. This provides a combinatorial definition of stable and unstable leaves.

Let's introduce some notation. For $\underline{w}\in \Sigma_A$, we define its positive part as $\underline{w}_+ := (w_n)_{n\in \NN}$ (with $0\in \NN$), and its negative part as $\underline{w}_- := (w_{-n})_{n\in \NN}$. We denote the set of positive codes in $\Sigma_A$ as $\Sigma_A^+$, which corresponds to the positive parts of all codes in $\Sigma_A$. Similarly, we define the set of negative codes in $\Sigma_A$ as $\Sigma^-_A$. Moreover, we use $F^s(x)$ to denote the stable leaf $\cF^s(f)$ passing through $x$, and $F^u(x)$ for the corresponding unstable leaf $\cF^u(f)$.

\begin{prop}\label{Prop: Projection foliations}
Let $\underline{w}\in \Sigma_A$. Then we have $\pi_f(\underline{F}^{s}(\underline{w}))\subset F^{s}(\pi_f(\underline{w}))$. Furthermore, assume that $\pi_f(\underline{w})=x$.
\begin{itemize}
\item If $x$ is not a $u$-boundary point, then for all $y\in F^u(x)$, there exists a code $\underline{v}\in \underline{F}^s(\underline{w})$ such that $\pi_f(\underline{v})=y$.

\item  If $x$ is a $u$-boundary point and $\underline{w}_0=w_0$, then for all $y$ in the stable separatrix of $x$ that enters the rectangle $R_{w_0}$, denoted $F^s_0(x)$, there is a code $\underline{v}\in \underline{F}^s(x)$ that projects to $y$, i.e., $\pi_f(\underline{v})=x$.

\end{itemize}

A similar statement applies to the unstable manifold of $\underline{w}$ and its respective projection onto the unstable manifold of $\pi_f(\underline{w})$.
	\end{prop}

\begin{proof}

Let $\underline{v}\in \underline{F}^s(\underline{w})$ be a stable leaf code of $\underline{w}$. Based on the definition of stable leaf codes, we can deduce that $w_z=v_z$ for all $z\geq k$ for certain $k\in \NN$. Consequently, $\pi_f(\sigma^k(\underline{w})),\pi_f(\sigma^k(\underline{v}))\in R_{w_k}$. Since the positive codes of $\underline{w}$ and $\underline{v}$ are identical starting from $k$, they define the same horizontal sub-rectangles of $R_{w_k}$ in which the codes $\sigma^k(\underline{w})$ and $\sigma^k(\underline{v})$ are projected. For $n\in \mathbb{N}$, let $H_n$ be the rectangle determined by:

$$
H_n=\cap_{z=0}^n f^{-z}(R_{w_{z+k}})=\cap_{z=0}^n f^{-z}(R_{v_{z+k}}).
$$
The intersection of all the $H_n$ forms a stable segment of $R_{w_k}$. Furthermore, each $H_n$ contains the rectangles:
$$
\overset{o}{Q_n}=\cap_{z=-n}^{n} f^{-z}(\overset{o}{R_{w_{z+k}}})= \cap_{z=-n}^{n} f^{-z}(\overset{o}{R_{v_{z+k}}})
$$
Therefore, the projections $\pi_f(\underline{v})$ and $\pi_f(\underline{w})$ are in the same stable leaf. 
	
	\begin{figure}[h]
		\centering
		\includegraphics[width=1\textwidth]{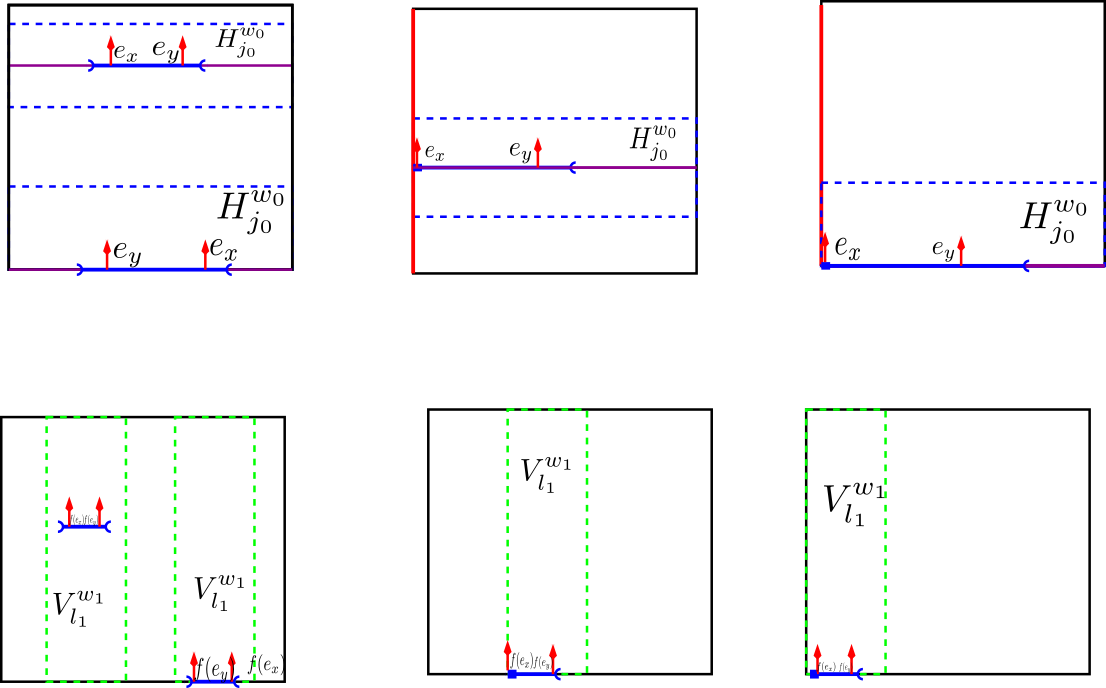}
		\caption{Projections of the stable leaf codes}
		\label{Fig: Proyection fol}
	\end{figure}

For the other part of the argument, let's refer to the image in Figure \ref{Fig: Proyection fol} to gain some visual intuition. Let's address the situation when $x=\pi_f(\underline{w})$ is not a $u$-boundary point. We recall that is implicit that the incidence matrix $A$ is binary.

Let $y\in F^s(\pi_f(\underline{w}))$ such that $y\neq x$. Suppose $y$ is not a periodic point (if it is, we can replace $y$ with $x$). Since  $x$ is not a $u$-boundary point, there is a small interval $I$ properly contained  in the stable segment of  $\cR_{w_0}$ that pass trough $x$. This allows us to apply  the next argument on both stable separatrices of $x$.  By definition of stable leaf, there exist $k\in \mathbb{N}$ such that $f^k(y)\in I$. If we prove there is  $\underline{v}\in \underline{F}^s(\underline{w})$ such that $\pi_f(\underline{v})=f^k(y)$, we obtain that $\sigma^{-k}(\underline{w})\in \underline{F}^s(\underline{w})$  is a code that projects to $y$, i.e $\pi_f(\sigma^{-k}(\underline{v}))=y$. Therefore we can assume  that $y\in I$. This situation is given by the left side pictures in Figure \ref{Fig: Proyection fol}.

There are two situations for $f^z(y)$: either it lies in the stable boundary of $R_{w_0}$ or  not. In any case $x$, the code $\underline{w}$ is equal to  a sector code $\underline{e}_x$ of $x$. The sector $e_x$ is contained in  a unique horizontal sub-rectangle of $R_{w_0}$, $H^{w_0}_{j_0}$, therefore the sector $f(e_x)$ is contained in  $f(H^{w_0}_{j_0})=V^{w_1}_{l_1}$, implying $w_1=\underline{e_x}_1$.

Take the sector   $e_y$ of $y$ in such that in the stable direction points toward $x$ and in the unstable direction point in the same direction than $e_y$. In this manner $e_y$ is contained in the rectangle $H^{w_0}_{j_0}$, therefore $f(e_y)$ is contained in $V^{w_1}_{l_1}$ and $\underline{e_y}_1=w_1$.

 In fact, it turns out that $f^{n}(e_y)$ and $f^n(e_x)$ are contained in the same $R_{w_n}$ for all $n\in \NN$. We can deduce that the positive part of $\underline{e_y}$ coincides with the positive part of $\underline{w}$. Therefore $\underline{e_y}\in \underline{F}^s(\underline{w})$.

%%%%%%%%

In the case where $x\in \partial^u\cR$, there is a slight variation. Let's assume that $\underline{w}=\underline{e_x}$, where $e_x$ is a sector of $x$. This sector code is contained within a unique horizontal sub-rectangle of $R_{w_0}$, $H^{w_0}_{j_0}$ such that $f(H^{w_0}_{j_0})=V^{w_1}_{l_1}$. This horizontal sub-rectangle contain a unique stable interval $I$ that contains $x$ in one of its end points. Consider $y\in I$.

Similar to the previous cases, we can define a sector  $e_y$ of $x$ that is contained in horizontal sub-rectangle $H^{w_0}_{j_0}$, i.e. $\underline{e_y}=w_0$. This implies that $f(e_y)$ is contained in $V^{w_1}_{l_1}$. Then the first term of $\underline{e_y}$ is $w_1$. By applied this process inductively  for all $n\in \NN$, we get that $\underline{e_y}_n=w_n$ and therefore $\underline{e_y}\in \underline{F}^s(\underline{w})$ and projects to $y$.
\end{proof}

The next task is to define, in terms of $T$, those codes that project to the boundary of the Markov partition. We will then use Proposition  \ref{Prop: Projection foliations} to determine the codes that project to the stable and unstable laminations of $s,u$-boundary periodic points.

\subsubsection{Boundary codes and $s,u$-generating functions.}

Let's proceed with the construction of the codes that are projected onto the boundary of the Markov partition, $\partial^{s,u}\cR$. We assume that $\cR=\{R_i\}_{i=1}^n$ is a geometric Markov partition of $f$ with geometric type $T$. For each $i\in \{1,\cdots,n\}$, we label the boundary components of $R_i$ as follows:

\begin{itemize}
\item $\partial^s_{+1}R_i$ denotes the upper stable boundary of the rectangle $R_i$.
\item $\partial^s_{-1}R_i$ denotes the lower stable boundary of $R_i$.
\item $\partial^u_{-1}R_i$ denotes the left unstable boundary of $R_i$.
\item $\partial^u_{+1}R_i$ denotes the right unstable boundary of $R_i$.
\end{itemize}

By using these labeling conventions, we can uniquely identify the boundary components of $\cR$ based on the geometric type $T$.

\begin{defi}\label{Defi; s,u boundary labels of T}
	Let $T=\{n,\{(h_i,v_i)\}_{i=1}^n,\Phi_T\}$  be an abstract geometric type. The $s$-\emph{boundary labels} of $T$ are defined as the formal set:
		$$
	\cS(T):=\{(i,\epsilon): i\in \{1,\cdots,n\} \text{ and } \epsilon\in \{1,-1\} \}, 
	$$
Similarly, the $u$-\emph{boundary labels} of $T$ are defined as the formal set:
	$$
	\cU(T):=\{(k,\epsilon): k\in \{1,\cdots,n\} \text{ and } \epsilon\in \{1,-1\} \}
	$$
\end{defi}

In a while, we will justify such names. It's important to note that this definition was made using only the value of $n$ given by the geometric type $T$, so it does not depend on the specific realization. With these labels, we can formulate in terms of the geometric type the codes that $\pi_f$ projects to $\partial^{u,s}_{\pm 1 }R_i$. The first step is to introduce a \emph{generating function}

We begin by relabeling the stable boundary component of $\mathcal{R}$ using the next  function:
$$
\theta_T:\{1,\cdots,n\}\times \{1,-1\}\rightarrow \{1,\cdots,n\} \times\cup_{i=1}^n \{1,h_i\}_{i=1}^n \subset \cH(T)
$$ 
that is defined as:

\begin{equation}\label{Equa: theta T relabel}
\theta_T(i,-1)=1 \text{ and } \theta(i,1)=h_i.
\end{equation}

The effect of $\theta_T$ is to choose the sub-rectangle of $R_i$ that contains the stable boundary component in question. Specifically, the boundary $\partial^s_{-1}R_1$ is contained in $\partial^s_{-1}H^i_1$, and the boundary $\partial^s_{+1}R_1$ is contained in $\partial^s_{+1}H^i_{h_i}$. This allows us to track the image of a stable boundary component, as the image of $\partial^s_{+1}R_i$ under $f$ should be contained in the stable boundary of $f(H^i_{h_i})$. But remember, the image of $\partial^s_{-1}R_i$ is a boundary component of $f(H^i_1) = V^k_l$, and the pair $(k,l)$ is uniquely determined by $\Phi_T$. By considering the value of $\epsilon_T(i,h_i)$, we can trace the image of the upper boundary component of $H^i_{h_i}$ in a more precise manner. If $\epsilon_T(i,h_i) = 1$, it indicates that the map $f$ does not alter the vertical orientations. As a result, the image of the upper boundary component of $H^i_{h_i}$ will remain on the upper boundary component of $f(H^i_{h_i})$. i.e. in this example:
$$
f(\partial^s_{+1}R_i)\subset f(\partial^s_{+1}H^i_{h_i}) \subset \partial^s_{+1}R_k.
$$ 
On the contrary, if $\epsilon_T(i,h_i)=-1$, it means that the map $f$ changes the vertical orientation. This has the following implication:
 $$
 f(\partial^s_{+1}R_i)\subset f(\partial^s_{+1}H^i_{h_i})\subset \partial_{-1}R_k.
 $$

 We don't really care about the index $l$ in $(k,l) \in \mathcal{V}(T)$, so it's convenient to decompose $\rho_T$ into two parts: $\rho_T:=(\xi_T,\nu_T)$. In this decomposition
$$
\xi_T:\cH(T)\rightarrow \{1,\cdots, n\},
$$
is defined as $\xi_T(i,j)=k$ if and only if $\rho_T(i,j)=(k,l)$.

 Let's continue with our example. In order to determine where $f$ sends the upper boundary component $\partial^s_{+1}R_i$, we need to identify the rectangle in $\mathcal{R}$ to which $f$ maps $H^i_j$. This can be determined using the following relation:
 $$
 \xi_T(i,\theta_T(1))=\xi_T(i,h_i)=k.
 $$ 
We now incorporate the change of orientation to determine the boundary component of $R_k$ that contains $f(\partial^s_{+1}H^i_{h_i})$. If we assume that $f(\partial^s_{+1}R_i) \subset \partial^s_{\epsilon}R_k$, then the value of $\epsilon$ can be determined using the following formula:
  $$
  \epsilon=+1 \cdot\epsilon_T(i,h_i) = +1 \cdot \epsilon_T(i,\theta_T(i,1)).
  $$   
This procedure is illustrated in Figure  \ref{Fig: theta T}.
 	\begin{figure}[h]
 	\centering
 	\includegraphics[width=0.6\textwidth]{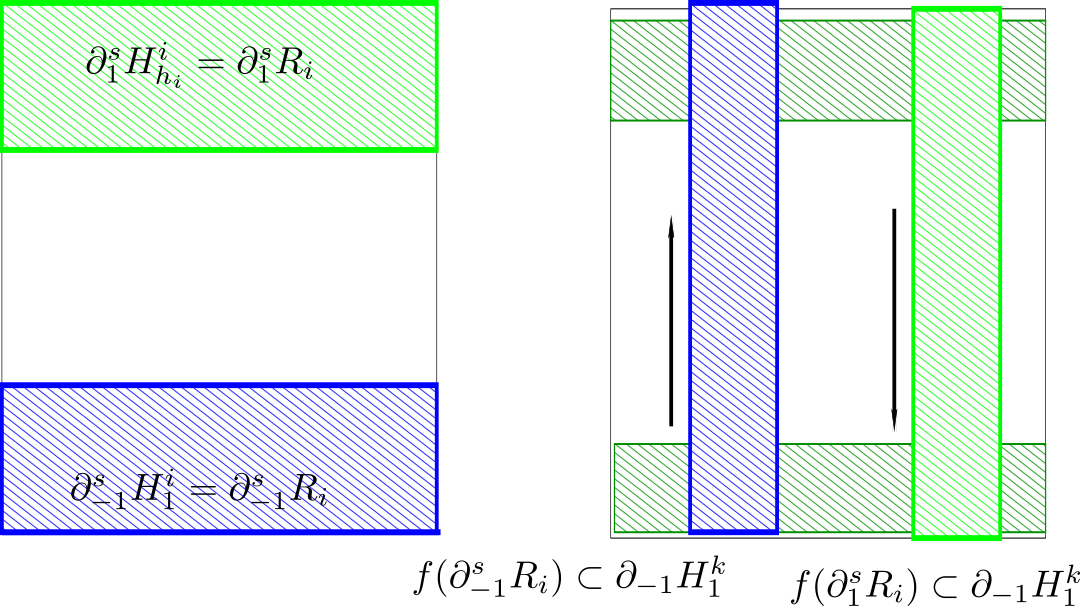}
 	\caption{The effect of $\theta_T$}
 	\label{Fig: theta T}
 \end{figure}
 The $S$-generating function utilizes $\theta_T$ and $\xi_T$ to generalize this idea.

\begin{defi}\label{Defi: s-boundary generating funtion}
 The $s$-generating function of $T$ is the function 
 	$$
 \Gamma(T):\cS(T) \rightarrow \cS(T),
 $$
defined for every $s$-boundary label $(i_0,\epsilon_0) \in \mathcal{S}(T)$ by the following formula:

\begin{equation}\label{Equ: Gamma generation funtion}
	\Gamma(T)(i_0,\epsilon_0)=(\xi_T(i_0,\theta_T(i_0,\epsilon_0)),
\epsilon_0 \cdot \epsilon_T(i_0,\theta_T(i_0,\epsilon_0)) ).
\end{equation}
	
\end{defi}

The $u$-boundary generating function can be defined as either the $s$-boundary generating function of $T^{-1}$ (associated with $f^{-1}$), or it can be defined directly as follows.

\begin{defi}\label{Defi: u-boundary generating funtion}
The $u$-\emph{generating function} of $T$ is
$$
\Upsilon(T):\cU(T)\rightarrow  \cU(T).
$$
defined as the $s$-\emph{generating function} of $T^{-1}$, $\Gamma(T^{-1})$.
\end{defi}

 The orbit of a label $(i_0,\epsilon_0)\in \cS(T)$ under $\Gamma(T)$ is defined as:
$$
\{(i_m,\epsilon_m)\}_{m\in \NN}:=\{\Gamma(T)^m(i_0,\epsilon_0) : m\in\NN\}. 
$$
We define $\varsigma_T:\cS(T)\rightarrow {1,\dots,n}$ as the projection onto the first component of the funtion $\Gamma_T$, i.e $\varsigma_T(i,\epsilon)=k$ if and only if $\Gamma_T(i,j)=(k,\epsilon)$. We will focus on the $s$-generating function. The ideas for the unstable case naturally extend, but sometimes we are going to recall them too.

The generating functions allow us to create a positive (or negative) code in $\Sigma^+$ for every $s$-boundary label of $T$ by composing $\Gamma(T)$ ($\Upsilon(T)$) with $\varsigma_T$ to obtain:
 
\begin{eqnarray}\label{Equa: $s$-boundary code +}
\underline{I}^{+}(i_0,\epsilon_0):=\{ \varsigma_T\circ\Gamma(T)^m(i_0,\epsilon_0)\}_{m\in\NN}=\{i_m\}_{m\in \NN}
\end{eqnarray}

Similarly, for $(k_0,\epsilon_0)\in \cU(T)$, we can construct the negative code in $\Sigma^{-}$ given by:

\begin{eqnarray}\label{Equa: $u$-boundary code -}
\underline{J}^{-}(k_0,\epsilon_0):=\{ \varsigma_T\circ\Upsilon(T)^m(k_0,\epsilon_0)\}_{m\in\NN}=\{k_{-m}\}_{m\in \NN}
\end{eqnarray}

These positive and negative codes are of great importance and deserve a name.

\begin{defi}\label{Defi: positive negative boundary codes }
	
	 Let $(i_0,\epsilon_0)\in \cS(T)$. The $s$-\emph{boundary positive code} of $(i_0,\epsilon_0)$ is denoted as $\underline{I}^+(i_0,\epsilon_0)\in \Sigma^+$, and it is determined by Equation \ref{Equa: $s$-boundary code +}. The set of $s$-boundary positive codes of $T$ is denoted as:
	 $$
	\underline{\cS}^{+}(T)=\{\underline{I}^{+}(i,\epsilon): (i,\epsilon)\in \cS(T)\} \subset \Sigma^+.
	$$
	Let $(k_0,\epsilon_0)\in \cU(T)$. The $u$-\emph{boundary negative code }of $(k_0,\epsilon_0)$ is denoted as $\underline{J}^{-}(k_0,\epsilon_0)=\{k_{-m}\}_{m=0}^{\infty}$ and is determined by Equation \ref{Equa: $u$-boundary code -}. The set of $u$-boundary negative codes of $T$ is denoted as
		
	$$
	\underline{\cJ}^{-}(T)=\{\underline{J}^{-}(k,\epsilon): (k,\epsilon)\in \cU(T)\}\subset \Sigma^{-}
	$$
\end{defi}
 
 Lemma \ref{Lemm: positive code admisible} below asserts that every $s$-boundary positive code of $T$ corresponds to the positive part of at least one element of $\Sigma_A$. The analogous result for $u$-boundary negative codes is true using a fully symmetric approach with the $u$-generating function. Therefore, we will only provide the proof for the stable case.

\begin{lemm}\label{Lemm: positive code admisible}
	Every $s$-boundary positive code $\underline{I}^{+}(i_0,\epsilon_0)$ belongs to $\Sigma_A^+$. Similarly, every $u$-boundary negative code $\underline{J}^{-}(k_0,\epsilon_0)$ belongs to $\Sigma_A^-$.
\end{lemm}

\begin{proof}
The incidence matrix $A$ of $T$ has a relation to any realization $(f,\cR)$ of $T$, which we can exploit to visualize the situation, anyway observe that our arguments only use the properties of $A$ or $T$. 
	In the construction of the positive code $\underline{I}^{+}(i_0,\epsilon_0)$, the term $i_1$ corresponds to the index of the unique rectangle in $\cR$ such that $f(H^{i_0}_{\theta_T(1,\epsilon_0)})=V^{i_1}_{l_1}$ (where $\theta_T(i_0,\epsilon)=1$ or $h_i$). In either case, we have:
	
	$$f^{-1}(\overset{o}{R{i_1}})\cap \overset{o}{R_{i_0}}=\overset{o}{H^{i_0}_{\theta_T(1,\epsilon_0)}´}\neq \emptyset,
	$$
	
	 which implies $a_{i_0,i_1}=1$. By induction, we can indeed extend this result to the entire positive code $\underline{I}^{+}(i_0,\epsilon_0)$.
	 
	 Similarly, for the case of $u$-boundary negative codes, we can use the symmetric approach by considering $T^{-1}$ and the inverse of the incidence matrix $A^{-1}$. Additionally, we can consider the realization $(f^{-1},\cR)$ to obtain a visualization. 
\end{proof}

Proposition \ref{Prop: s code Inyective} is the first step in proving that each $s$-boundary label uniquely determines a set of codes that projects to a single boundary component of the partition.

\begin{prop}\label{Prop: s code Inyective}	
	The map $I:\cS(T)\rightarrow \Sigma_A^+$ defined by $I(i,\epsilon):=\underline{I}^+(i,\epsilon)$ is  injective map. Similarly, the map $J:\cU(T)\rightarrow \Sigma_A^-$ defined by $J(i,\epsilon):=\underline{J}^-(i,\epsilon)$ is also injective.
\end{prop}

\begin{proof}
	
	Clearly $I$ is a  well defined map as $\Gamma^n(T)$ is itself a well defined map with the same domain. 	If $i_0 \neq i'_0$, then the sequences $\underline{I}^+(i_0,\epsilon_0)$ and $\underline{I}^+(i'_0,\epsilon'_0)$ will differ in their first term, making them distinct. The remaining case involves considering $(i_0,1)$ and $(i_0,-1)$.
	
	Let's denote $\underline{I}^+(i_0,1) = {i_m}$ and $\underline{I}^+(i_0,-1) = {i'_m}$ as their positive codes. We will analyze the sequences ${\Gamma(T)^m(i_0,1)}$ and ${\Gamma(T)^m(i_0,-1)}$ and show that there exists an $m \in \NN$ such that $i_m \neq i'_m$. Our approach starts with the following technical lemma:
	
	\begin{lemm}\label{Lemm: positive h-i}
If $T$ is in the pseudo-Anosov class and $A$ is a binary matrix, then there exists $M \in \mathbb{N}$ such that if  $\underline{I}^+(i_0,1) = \{i_m\}_{m\in \NN}$, then  $h_{i_M} > 1$.
	\end{lemm}
	
	\begin{proof}
		Since $T$ is in the pseudo-Anosov class, there exists a realization $(f,\mathcal{R})$. The infinite sequence ${i_m}$ takes a finite number of values (at most $n$), so there exist natural numbers $m_1 < m_2$ such that $R_{i_{m_1}} = R_{i_{m_2}}$. If such an $M$ does not exist, then for all $m_1 \leq m \leq m_2$, $h_{i_m} = 1$. This implies that $f^{m_2-m_1}(R_{i_{m_1}}) \subset R_{i_{m_2}} = R_{i_{m_1}}$, which is a vertical sub-rectangle of $R_{i_{m_1}}$.
	
	By uniform expansion in the vertical direction, the rectangle $R_{i_{m_1}}$ reduces to a stable interval, and this is not permitted in a Markov partition, leading us to a contradiction. 	
	\end{proof}
	
In view of Lemma $M:=\min\{m\in \NN: h_{i_m}>1\}$ exists. If there is $0\leq m\leq M$ such that $i_m\neq i'_m$, our proof is complete. If not, the following Lemma addresses the remaining situation.
	
	\begin{lemm}\label{Lemm: diferen positive code}
	Suppose that for all $0\leq m\leq M$, $i_m=i_m'$, then $i_{M+1}\neq i'_{M+1}$.
	\end{lemm}
	
	\begin{proof}
		
	Observe that for all $0\leq m\leq M$, if $\Gamma(T)^m(i_0,+1)=(i_m,\epsilon_m)$ then $$\Gamma(T)^m(i_0,-1)=(i_m,-\epsilon_m)$$.
	
	 They have the same index $i_m=i'_m$ but still have inverse $\epsilon$ part, i.e., $\epsilon_m=-\epsilon'_m$. In fact, the first term $\epsilon_0=-\epsilon'_0$, without lost of generality $\epsilon_0=1$  and $\epsilon'_0=-1$,   by hypothesis and considering that  $1=h_{i_0}=h_{i'_0}$ we infer that:
		$$
		\theta_T(i_0,\epsilon_0)=(i_0,h_{i_m})=(i'_0,1)=\theta_T(i'_0,\epsilon_0'),
		$$
Therefore:	
		$$
		\epsilon_1=\epsilon_0\cdot \epsilon_T(i_0,h_{i_0})=-\epsilon'_0\epsilon_T(i'_0,1)=-\epsilon'_1.
		$$
		then continue the argument by induction. In particular: $\Gamma(T)^M(i_0,1)=(i_M,\epsilon_M)$ and  $\Gamma(T)^M(i'_0,-1)=(i_M,-\epsilon_M)$.
		
		 The incidence matrix of $T$ have $\{0,1\}$ has coefficients ${0,1}$, and since that $1\neq h_{i_M}$,  if $\rho_T(i_M,1)=(k,l)$ then $\rho_T(i_M,h_{i_M})=(k',l')$ where $k\neq k'$.
		 
	Consider the  case when, $\theta_T(i_M,\epsilon_M)=(i_M,1)$ and  $\theta_T(i_M,\epsilon'_M)=(i_M, h_{i_M})$. Lets apply the formula of $\Gamma(T)$:
	$$
	\Gamma(T)^{M+1}(i_0,1)=\Gamma(T)(i_M,\epsilon_M)=(\xi_T(i_M,1),\epsilon_M\cdot \epsilon_T(i_M,1))=(k,\epsilon_{M+1})
	$$
	and 				
		$$
	\Gamma(T)^{M+1}(i_0,-1)=\Gamma(T)(i_M,-\epsilon_M)=(\xi_T(i_M,h_{i_M}),-\epsilon_M\cdot \epsilon_T(i_M,h_{i_M}))=(k',\epsilon'_{M+1})
	$$	
	therefore $i_{M+1}=k\neq k'=i'_{M+1}$.
	
		 The situation $\theta_T(i_M,\epsilon_M)=(i_M,h_{i_M})$ and  $\theta_T(i_M,\epsilon'_M)=(i_M, 1)$ is treated similarly. The has been lemma  proved.

	\end{proof}

The proposition follows from the previous lemma. The result for negative codes associated with $u$-boundary labels is proven using a fully symmetric approach with the $u$-generating function.
\end{proof}

If $\Gamma(T)(i,\epsilon)=(i_1,\epsilon_1)$, applying the shift to this code gives another $s$-boundary positive code. That is clearly given by 
$$
\sigma(\underline{I}^+(i,\epsilon))=\underline{I}^+(i_1,\epsilon_1),
$$
where $\Gamma(T)(i_0,\epsilon_0)=(i_1,\epsilon_1)$. Since there are $2n$ different $s$-boundary positive codes, there exist natural numbers $k_1\neq k_2$ with $k_1,k_2\leq 2n$ such that $\sigma^{k_1}(\underline{I}^+(i,\epsilon))=\sigma^{k_2}(\underline{I}^+(i,\epsilon))$, and the code $\underline{I}^+(i,\epsilon)$ is pre-periodic. This implies the following corollary.

\begin{coro}\label{Coro: preperiodic finite s,u boundary codes}
There are exactly $2n$ different $s$-boundary positive codes and $2n$ different $u$-boundary negative codes. Furthermore, every $s$-boundary positive code and every $u$-boundary negative code is pre-periodic under the action of the shift $\sigma$. 

Moreover, for every $s$-boundary positive code $\underline{I}^+(i,\epsilon)$, there exists $k\leq 2n$ such that $\sigma^k(\underline{I}^+(i,\epsilon))$ is periodic. Similarly, for every $u$-boundary negative code $\underline{J}^-(i,\epsilon)$, there exists $k\leq 2n$ such that $\sigma^{-k}(\underline{J}^-(i,\epsilon))$ is periodic.
\end{coro}

Now we are ready to define a family of admissible codes that project onto the stable and unstable leaves of periodic boundary points of $\cR$.

\begin{defi}\label{Defi: s,u-boundary codes}
	The set of $s$-\emph{boundary codes} of $T$ is:
	\begin{equation}
	\underline{\cS}(T):=\{\underline{w}\in \Sigma_A: \underline{w}_+\in \underline{\cS}^{+}(T)\}.
	\end{equation}
The set of  $u$-\emph{boundary codes} of $T$ is 
\begin{equation}
\underline{\cU}(T):=\{\underline{w}\in \Sigma_A: \underline{w}_-\in \underline{\cU(T)}^{-}(T)\}.
\end{equation}
\end{defi}

Next, Proposition \ref{Prop: positive codes are boundary}  states that the projection of $s$-boundary codes of $T$ through $\pi^f$ is always contained in the stable boundary of the Markov partition $(f,\cR)$. Similarly, Proposition \ref{Prop: boundary points have boundary codes} ensures that these are the only codes in $\Sigma_A$ that project to the stable boundary. This provides us with a symbolic characterization of the boundary of the Markov partition. As we have done so far, we will provide a detailed proof for the stable case, noting that the unstable version follows by symmetry.

\begin{prop}\label{Prop: positive codes are boundary}
Let $(f,\cR)$ be a pair consisting of a homeomorphism and a partition that realizes $T$. Suppose $(i,\epsilon)\in \cS(T)$ is an $s$-boundary label of $T$, and let $\underline{w}\in \underline{\cS}(T)$ be a code such that $\underline{I}^+(i,\epsilon)=\underline{w}_+$. Then, $\pi_f(\underline{w})\in \partial^s_{\epsilon} R_{i}$.

Similarly, if $(i,\epsilon)\in \cU(T)$ is a $u$-boundary label of $T$, and $\underline{w}\in \underline{\cU}(T)$ is a code such that $\underline{J}^-(i,\epsilon)=\underline{w}_-$, then $\pi_f(\underline{w})\in \partial^u_{\epsilon}R_i$.

\end{prop}

\begin{proof}
	Make $(i,\epsilon)=(i_0,\epsilon_0)$ to make coherent the following notation and $\underline{I}^+(i,\epsilon)=\{i_m\}_{m\in \NN}$. 	For every $s\in \NN$ define the rectangles $\overset{o}{H_s}=\cap_{m=0}^s f^{-1}(\overset{o}{R_{i_m}})$. The limit of the closures of such rectangles when $s$ converge to infinity is a unique stable segment of $R_{i_0}$. The proof is achieved  if $\partial^s_{\epsilon_0}R_{i_0}\subset H_s:=\overline{ \overset{o}{H_s} }$ for all $s\in \NN$. In this order of ideas is enough to argument that for all $s\in \NN$:
	$$
	f^s(\partial^s_{\epsilon_0}R_{i_0}) \subset R_{i_{s}}.
	$$
	We are going to probe  by induction over $s$ something more specific:
	$$
	f^s(\partial_{\epsilon_0} R_{i_0})\subset \partial_{\epsilon_s}R_{i_s}.
	$$  
	
\emph{Base of induction}: For $s=0$. This is the case as $f^0(\partial^s_{\epsilon_0}R_{i_0})\subset \partial^s_{\epsilon_0}R_{i_0}$.
	
\emph{Hypothesis of induction }: Assume  that $f^s(\partial_{\epsilon_0} R_{i_0})\subset \partial_{\epsilon_s}R_{i_s}$ and  $f^s(\partial^s_{\epsilon_0}R_{i_0})\subset R_{i_s}$.

\emph{Induction step}: We are going to prove that
$$
f^{s+1}(\partial_{\epsilon_0} R_{i_0})\subset \partial_{\epsilon_{s+1}}R_{i_{s+1}}
$$ 

 and then  that $f^{s+1}(\partial^s_{\epsilon_0}R_{i_0})\subset R_{i_{s+1}}$. For that reason consider the two following cases:
	
	\begin{itemize}
		\item $\epsilon_{s}=1$. In this situation $f^s(\partial^s_{\epsilon_0}R_{i_0})\subset H^{i_s}_{h_{i_s}}$.  Hence  $f^{s+1}(\partial^s_{i_0}R_{i_0})\subset f(H^{i_s}_{h_{i_s}}) \subset R_{i'_{s+1}}$. Where $R_{i'_{s+1}}$ is the only rectangle such that $\xi_T(i_s,h_{i_s})=(i'_{s+1})$.		
		Even more $f^{s+1}(\partial^s_{\epsilon_0} R_{i_0})\subset \partial^s_{\epsilon'_{s+1}} R_{i'_{s+1}}$, where $\epsilon'_{s+1}$ obey to the formula $\epsilon'_{s+1}= \epsilon_T(i_s,h_{i_s})=\epsilon_s \cdot \epsilon_T(i_s,h_{i_s})$.
		
		\item  $\epsilon_{s}=-1$. In this situation  $f^s(\partial^s_{\epsilon_0}R_{i_0})\subset H^{i_s}_1$. Hence $f^{s+1}(\partial^s_{i_0}R_{i_0})\subset f(H^{i_s}_{1}) \subset R_{i'_{s+1}}$, where $R_{i_{s+1}}$ is the only rectangle such that $\xi_T(i_s,1)=(i'_{s+1}$. 
		Even more, $f^{s+1}(\partial^s_{\epsilon_0} R_{i_0})\subset \partial^s_{\epsilon'_{s+1}} R_{i'_{s+1}}$, where  $\epsilon'_{s+1}$ obey to the formula $\epsilon'_{s+1}= -\epsilon_T(i_s,1)=\epsilon_s \cdot \epsilon_T(i_s,1)$.
	\end{itemize}
	In bot situations:
	$$
	f^{s+1}(\partial^s_{\epsilon_0})R_{i_0})\subset \partial^s_{\epsilon'_{s+1}}R_{i'_{s+1}}
	$$
and they follows the rule:
	$$
	(i'_{s+1},\epsilon'_{s+1})=(\xi_T(i_s,\theta_T(i_s,\epsilon_s)),\epsilon_s\cdot \epsilon_T(i_s,\theta_T(\epsilon_s)))=\Gamma(T)^{s+1}(i_0,\epsilon_0)=(i_{s+1},\epsilon_{s+1}).
	$$
Therefore $f^{s+1}(\partial_{\epsilon_0} R_{i_0})\subset \partial_{\epsilon_{s+1}}R_{i_{s+1}}$, as we claimed and the result is proved.
The unstable case is totally symmetric.

\end{proof}

Proposition \ref{Prop: positive codes are boundary} provides justification for naming the $s$-boundary and $u$-boundary labels of $T$ as such, as they generate codes that are projected to the boundary of the Markov partition. The next proposition asserts that these codes are the only ones that have such a property.

\begin{prop}\label{Prop: boundary points have boundary codes}
	If $\underline{w}\in \Sigma_A$ projects to the stable boundary of the Markov partition $(f,\cR)$ under $\pi_f$, i.e., $\pi_f(\underline{w})\in \partial^s \cR$, then it follows that $\underline{w}\in \underline{\cS}(T)$. Similarly, if $\pi_f(\underline{w})\in \partial^u\cR$, then $\underline{w}\in \underline{\cU}(T)$.
	
\end{prop}

\begin{proof}

Like $A(T)$ has coefficients ${0,1}$, the sequence $\underline{w}_+$ determines, for all $m\in \mathbb{N}$, a pair $(w_m,j_m)\in \mathcal{H}(T)$ such that $\xi_T(w_m, j_m)=w_{m+1}$ and a number $\epsilon_T(w_m,j_m)=\epsilon_{m+1} \in {1,-1}$. It is important to note that, like  $f^m(x)\in \partial^s\cR$, in fact,  there are only two  cases (unless $h_{w_m}=1$ where they are the same):
$$
w_{m+1}=\xi_T(w_m,1) \text{ or well } w_{m+1}=\xi_T(w_m,h_{w_m}).
$$

This permit to define $\epsilon_m\in \{-1,+1\}$ as the only number such that:
\begin{equation}\label{Equa: determine epsilon m}
w_{m+1}=\xi_T(w_m,\theta_T(w_m,\epsilon_m)).
\end{equation}
Even more $\epsilon_m$ determine $\epsilon_{m+1}$ by the formula:
$$
\epsilon_{m+1}=\epsilon_m\cdot \epsilon_T(w_m,\theta_T(w_m,\epsilon_m)).
$$
In resume:
\begin{equation}\label{Equa: w determine by gamma}
\Gamma(T)(w_m,\epsilon_m)=(w_{m+1},\epsilon_{m+1})
\end{equation}
follow the rule dictated by the $s$-generating function. So if we know $\epsilon_M$ for certain $M\in \NN$ we can determine  $\sigma^{M}(\underline{w})_+=\underline{I}^+(w_M,\epsilon_M)$, so it rest to determine at least one $\epsilon_M$. 

Lemma \ref{Lemm: positive h-i} can be adapted to this context to prove the existence of a minimal $M\in \mathbb{N}$ such that $h_{w_M}>1$. Then, equation \ref{Equa: determine epsilon m} determines $\epsilon_M$. Now, we need to recover $\epsilon_0$, but we proceed backwards. Since $h_{w_m}=1$ for all $m<1$, we have:

$$
\epsilon_M=\epsilon_{M-1}\cdot \epsilon_T(w_{M-1},1)
$$
and then $\epsilon_{M-1}=\epsilon_M \cdot \epsilon_T(w_{M-1},1)$. By applied this procedure we can determine $\epsilon_0$ and then using \ref{Equa: w determine by gamma} to get that
$$
\Gamma(T)^m(w_0,\epsilon_0)=(w_m,\epsilon_m).
$$
The conclusion is that $\underline{w}_+=\underline{I}^{+}(w_0,\epsilon_0)$. 

The unstable boundary situation is analogous.
\end{proof}

Therefore, $\underline{\cS}(T)$ and $\underline{\cU}(T)$ are the only admissible codes that project to the boundary of a Markov partition. With this in mind, we can distinguish the periodic boundary codes from the non-periodic ones. It is important to note that the cardinality of each of these sets is less than or equal to $2n$.

\begin{defi}\label{Defi: s,u-boundary periodic codes}
	The set of $s$-\emph{boundary periodic codes} of $T$ is:
	\begin{equation}
\text{ Per }(\underline{\cS(T)}):=\{\underline{w}\in \underline{\cS}(T): \underline{w} \text{ is periodic }\}.
	\end{equation}
	The set of  $u$-\emph{boundary periodic codes} of $T$ is 
	\begin{equation}
\text{ Per }(\underline{\cU(T)}):=\{\underline{w}\in \underline{\cU}(T): \underline{w} \text{ is periodic }\}.
	\end{equation}
\end{defi}

\subsection{Decomposition of $\Sigma_{A}$.}
 Now we can describe the stable and unstable leaves of $f$ corresponding to the periodic points of the boundary.

\begin{defi}\label{Defi: stratification Sigma A}
	We define the $s$-\emph{boundary leaves codes} of $T$:
	\begin{equation}
	\Sigma_{\cS(T)}=\{\underline{w}\in \Sigma_A:   \exists k\in\NN \text{ such that } \sigma^k(\underline{w})\in \underline{\cS(T)}\}.
	\end{equation}
	We define the $u$-\emph{boundary leaves codes} of $T$
	\begin{equation}
	\Sigma_{\cU(T)}=\{\underline{w}\in \Sigma_A:  \exists k\in\NN \text{ such that } \sigma^{-k}(\underline{w})\in \underline{\cU(T)}\}.
	\end{equation}
	Finally, the \emph{totally interior codes} of $T$ are
	$$
	\Sigma_{\cI nt(T)}=\Sigma_A\setminus(\Sigma_{\cS(T)} \cup \Sigma_{\cU(T)}).
	$$
\end{defi}

The importance of such a division of $\Sigma_A$ is that its projections are well determined, as indicated by the following lemma.

\begin{lemm}\label{Lemm: Projection Sigma S,U,I}

	A code $\underline{w}\in \Sigma_A$ belongs to $\Sigma_{\cS(T)}$ if and only if its projection $\pi_f(\underline{w})$ is within the stable leaf of a boundary periodic point of $\cR$.
	
A code $\underline{w}\in \Sigma_A$ belongs to $\Sigma_{U(T)}$ if and only if its projection $\pi_f(\underline{w})$ lies within the unstable leaf of a boundary periodic point of $\cR$.

	A code $\underline{w}\in \Sigma_A$ belongs to $\Sigma_{\cI nt(T)}$ if and only if its projection $\pi_f(\underline{w}$ is contained within Int$(f,\cR)$>
\end{lemm}

\begin{proof}
The $s$-boundary leaf codes satisfy Definition  \ref{Defi: s,u-leafs}, and according to Proposition \ref{Prop: Projection foliations}, if $\underline{w}\in \Sigma_{\cS(T)}$, their projection lies on the same stable manifold as an $s$-boundary component of $\cR$. Stable boundary components of a Markov partition represent the stable manifold of an $s$-boundary periodic point, and for each $\underline{w}\in \Sigma_{\cS(T)}$, there exists $k=k(\underline{w})\in \mathbb{N}$ such that $\sigma^k(\underline{w})_+$ is a periodic positive code (Corollary \ref{Coro: preperiodic finite s,u boundary codes}). This positive code corresponds to a periodic point on the boundary of $\cR$ within whose stable manifold $\pi_f(\underline{w})$ is contained. This proves  one direction in the first assertion  of the lemma.

If $\underline{v}$ is a periodic boundary code ,$\pi_f(\underline{v})\in \partial^s\cR$, by definition $\underline{v}\in \underline{\cS(T)}$.  Suppose that $\pi_f(\underline{w})$ is on the stable leaf of $\underline{v}$, then there exist $k \in \NN$ such that $\pi_f(\sigma^k(\underline{w}))$ is on the same stable boundary of $\cR$ as $\pi_f(\underline{v})$.  The proposition\ref{Prop: boundary points have boundary codes} implies that $\sigma^k(\underline{w})\in \underline{\cS(T)}$, hence  $\underline{w}\in \Sigma_{\cS(T)}$ by definition. This completes the proof of the first assertion of the lemma. A similar argument proves the unstable case.	
	
 The conclusion of these items is that  $\underline{w}\in \Sigma_{\cS(T)}\cap \Sigma_{\cU(T)}$ if and only if $\pi_f(\underline{w})\in \cF^{s,u}(\text{ Per }^{s,u}(\cR))$. 
 
   As proved in the Lemma \ref{Lemm: Caraterization unique codes} totally interior points are disjoint from the stable and unstable lamination generated by boundary periodic points. If $\underline{w}\in  \Sigma_{\cI nt(T)}$ and $\pi_f(\underline{w})$ is on the stable or unstable leaf of a $s,u$-boundary periodic point, we have seen that $\underline{w}\in \Sigma_{\cS(T),\cU(T)}$ which is not possible, therefore $\pi_f(\underline{w})\in $Int$(f,\cR)$. In the conversely direction, if  $\pi_f(\underline{w})$ is not in the stable or unstable lamination of $s,u$-boundary points, $\underline{w}\notin (\Sigma_{\cS(T)})\cup \Sigma_{\cU(T)}$, so $\underline{w}\in\Sigma_{\cI nt(T)}$. This ends the proof.
 \end{proof}
We have obtained the decomposition 
$$
\Sigma_{A(T)}=\Sigma_{\cI nt(T)} \cup \Sigma_{\cS(T)} \cup \Sigma_{\cU(T)}
$$ 
 and have characterized the image under $\pi_f$ of each of these sets. In the next subsection we use this decomposition to define relations on each of these parts and then extend them to an equivalence relation in $\Sigma_A$.

%% file: TotalInvariant/Subrelation.tex
\subsubsection{Relations in $\Sigma_{\cS(T)}$ and $\Sigma_{\cU(T)}$}

Let $\underline{w}\in \Sigma_{\cS(T)}\setminus$Per$(\sigma_A)$ be non-periodic $s$-boundary leaf code of $T$. Since it is not a periodic boundary points of $T$ there exists a number $k:=k(\underline{w})\in \ZZ$ with the following properties: 
\begin{itemize}
\item $\sigma^k(\underline{w})\notin \underline{\cS(T)}$, but
\item $\sigma^{k+1}(\underline{w})\in \underline{\cS(T)}$
\end{itemize}
 i.e. $\pi_f(\sigma^k(\underline{w}))\notin\partial^s\cR$ but $\pi_f(\sigma^{k+1}(\underline{w}))\in\partial^s\cR$. This easy observation lead to the next Lemma.

\begin{lemm}\label{Lemm: minumun k}
The number $k:=k(\underline{w})$ is the unique integer such that $f^{k}(\pi_f(\underline{w})) \in \cR\setminus\partial^s\cR$ and for all $k'>k(\underline{w})$,  $f^{k'}(\pi_f(\underline{w}))\in \partial^s\cR \setminus$Per$(f)$.
\end{lemm}

Let $x=\pi_f(\underline{w})$, as consequence of Lemma \ref{Lemm: minumun k} there are indices $i\in \{1,\cdots,n\}$ and $j\in \{1,\cdots,h_{i}-1\}$ such that:
\begin{itemize}
\item $f^k(x)$ is the rectangle $R_i$ but not in its stable boundary, i.e. $f^k(x)\in R_{i}\setminus \partial^s R_i$ 
\item $f^k(x)$ is in two adjacent horizontal sub-rectangles of $R_i$, i.e. $f^k(x)\in \partial^s_{+1} H^i_j$ and $x\in \partial_{-1}^s H^i_{j+1}$
\end{itemize}

\begin{comment}
In fact as $x$ is not a periodic point it is not a singularity, so $x$ and $f^k(x)$ have at most $4$ sector codes (because it have at most $4$ sectors). Then $x$ could be in at most $4$ rectangles of $\cR$, but the fact $f^k(x)$ is not in a stable boundary of the  Markov partition implies $f^k(x)$ is in at most $2$ rectangles of $\cR$.
%Take $\phi_T(i,j)=(i_0,l_0,\epsilon_T(i,j))$ and $\phi_T(i,j+1)=(i'_0,l'_0,\epsilon_T(i,j+1))$.
\end{comment}

As $f^{k+1}(x)\in \partial^s \cR$, the rectangles  $H^i_j$ and $H^i_{j+1}$ determine the following conditions:

\begin{itemize}
\item As $f^k(x)\in \partial^s_{+1}H^i_j$,  we have that $f^{k+1}(x)\in \partial^s_{\epsilon_0} R_{i_0}$ where:
\begin{equation}
i_0=\xi_T(i,j) \text{ and  } \epsilon_0=\epsilon_T(i,j)
\end{equation}
\item  As $f^k(x)\in \partial^s_{-1}H^i_{j+1}$ we have that $f^{k+1}(x)\in \partial^s_{\epsilon'_0} R_{i'_0}$ where
\begin{equation}
i_0=\xi_T(i,j+1)  \text{ and  } \epsilon'_0=-\epsilon_T(i,j+1)
\end{equation}
\end{itemize}

This mechanism of identification is illustrated in Figure  \ref{Fig: stable identification}
\begin{figure}[h]
	\centering
	\includegraphics[width=0.6\textwidth]{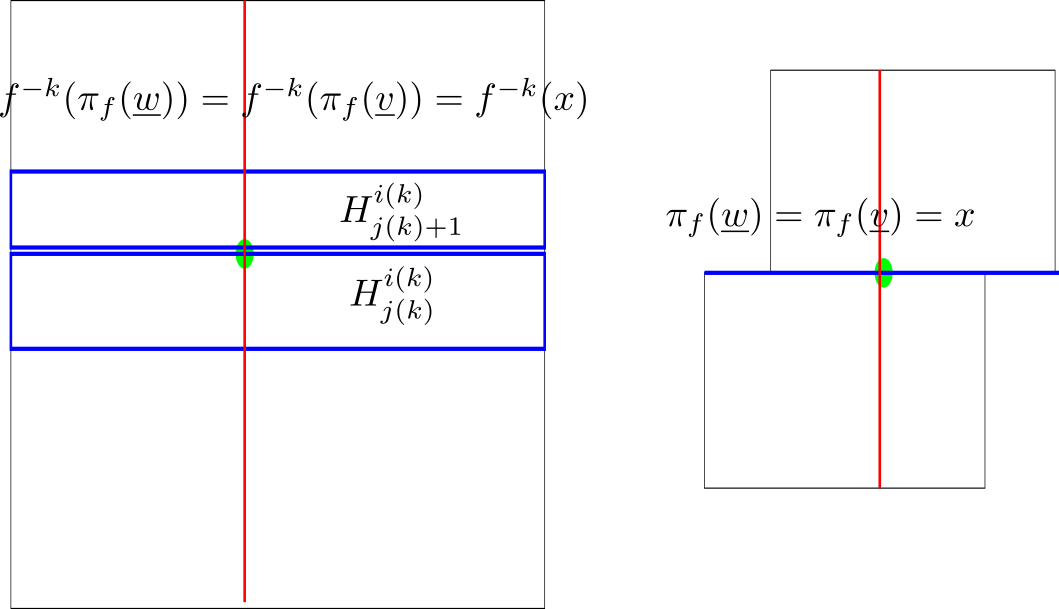}
	\caption{The stable identification mechanism}
	\label{Fig: stable identification}
\end{figure}

This analysis suggest the following definition

\begin{defi}\label{Defi: s realtion in Sigma S no per}
Let $\underline{w},\underline{v}\in \Sigma_{\cS(T)}\setminus$Per$(\sigma_A)$, they are $s$-related and its write $\underline{w}\sim_{s} \underline{v}$ if they are equals or: There exist  $k\in \ZZ$ such that
\begin{itemize}
\item[i)]  $\sigma^k(\underline{w}),\sigma^k(\underline{w})\notin \underline{\cS(T)}$ but $\sigma^{k+1}{\underline{w}},\sigma^{k+1}(\underline{v})\in \underline{\cS(T)}$ 

\item[ii)] The zero terms of $\sigma^k(\underline{w})$ and  $\sigma^k(\underline{v})$ are equals.i.e  $v_k=w_k$ and $h_{w_k}=h_{v_k}>1$. 

\item[iii)] Lets take $i=v_k=w_k$. There exist $j\in \{1, \cdots,h_i -1\}$ such that only one of the next options occurs:
\begin{equation}\label{Equa: Sim s obtion 1}
\xi_T(i,j)=w_{k+1} \text{ and then }  \xi_T(i,j+1)=v_{k+1},
\end{equation}
or
\begin{equation}\label{Equa: Sim s obtion 2}
\xi_T(i,j)=v_{k+1} \text{ and then }  \xi_T(i,j+1)=w_{k+1}.
\end{equation}

\item[iv)] Suppose the positive code of $\sigma^{k+1}(\underline{w})$  is equal to the $s$-boundary code $\underline{I}^{+}(w_{k+1},\epsilon_w)$ and the positive code of $\sigma^{k+1}(\underline{v})$  is equal to the $s$-boundary code $\underline{I}^{+}(v_{k+1},\epsilon_v)$, then:

If equation \ref{Equa: Sim s obtion 1} happen:
\begin{equation}\label{Equa: sim s epsilon 1}
\epsilon_w=\epsilon_T(w_{k},j) \text{ and } \epsilon_v=-\epsilon_T(v_{k},j+1)
\end{equation}
If equation \ref{Equa: Sim s obtion 2} happen:
\begin{equation}\label{Equa: sim s epsilon 2}
\epsilon_v=\epsilon_T(v_{k},j) \text{ and } \epsilon_w=-\epsilon_T(w_{k},j+1)
\end{equation}

\item[v)] The negative codes  $\sigma^{k}(\underline{w})_{-}$ and $\sigma^{k}(\underline{v})_{-}$ are equal.
\end{itemize}
\end{defi}

Our challenge is to show that $\sim_s$ is an equivalence relation in $\Sigma_{\cS(T)}$ and that two $\sim_s$-related  codes project by $\pi_f$ to the same point. The basis of the argument will be that when $\underline{w},\underline{v}\in \underline{\cS(T)}$ and are $\sim_s$-related, they project to adjacent stable boundaries of the Markov partition.  We subsequently can generalize the result to the whole set $\Sigma_{\cS(T)}\setminus$Per$(\sigma_A)$. The situation with the periodic points will be discussed later.

\begin{prop}\label{Prop: sim s equiv in Sigma non per}
The relation $\sim_{s}$ is an equivalence relation in $\Sigma_{\cS(T)}\setminus$Per$(\sigma_A)$.
\end{prop}

\begin{proof}
Reflexivity and symmetry are fairly obvious from the definition and we will concentrate on transitivity. Lets to assume $\underline{w}\sim_{s}\underline{v}$ and $\underline{v}\sim_s\underline{u}$. 

The number $k\in \ZZ$ of item $i)$  is unique because it is characterized as:
$$
k(\underline{w})=\min\{z\in \ZZ: \sigma^k(\underline{w})\in \underline{w}\}-1,
$$
 therefore $k:=k(\underline{w})=k(\underline{v})=\underline{u}$ is the same for the three codes.

The $k$-terms of such codes are equal too, as $w_k=v_k$ and $v_k=u_k$, lets take $i=v_k=w_k=u_k$. Without lost of generality we could assume that $j\in \{1,\cdots,h_{i}-1\}$ is such that:
$$
\xi_T(i,j+1)=w_{k+1} \text{ and } \xi_T(i,j)=v_{k+1},
$$ 

Suppose too that $\sigma^{k+1}(\underline{w})_+=\underline{I}^+(w_{k+1},\epsilon_w)$  and $\sigma^{k+1}(\underline{v})_+=\underline{I}^+(v_{k+1},\epsilon_v)$. They satisfy the relations:
$$
\epsilon_w=-\epsilon_T(i,j+1) \text{ and } \epsilon_v=\epsilon_T(i,j).
$$

Like  $\underline{v}\sim_s \underline{u}$ there is a unique $j'\in \in \{1,\cdots,h_{i}-1\}$ such that:
$$
\xi_T(i,j')=u_{k+1}
$$
but the relation, $\underline{v}\sim_s \underline{u}$ implies that, $j'=j+1$ or $j-1$. If we prove that $j'=j+1$ necessarily $\underline{u}=\underline{w}$ and we have finished.   Lets to analyses the situation of $j'=j-1$.

 Suppose $\sigma^{k+1}(\underline{u})_+=\underline{I}^+(u_{k+1},\epsilon_u)$, as $j=(j-1)+1$ we have in the situation of Equation \ref{Equa: Sim s obtion 1} and then we applied Equation \ref{Equa: sim s epsilon 1} to obtain that:
$$
\epsilon_v= -\epsilon_T(i,j) \text{ and } \epsilon_u=\epsilon_T(i,j-1).
$$
Therefore, $\epsilon_T(i,j)=-\epsilon_T(i,j)$ and this is a direct contradiction.  Then $j'=j+1$ and the positive part of $\sigma^k(\underline{w})$ coincides with the positive part of $\sigma^k(\underline{u})$.
 
  Item $v)$ implies, that the negative part of $\sigma^{k}(\underline{w})$ is equal to the negative part of $\sigma^{k}(\underline{v})$, and  the negative part of   $\sigma^{k}(\underline{v})$ is equal to the negative part of $\sigma^{k}(\underline{u})$. Hence, the negative part of   $\sigma^{k}(\underline{w})$ is equal to the negative part  of $\sigma^{k}(\underline{u})$, which implies that $\underline{w}\sim_s \underline{u}$. 
\end{proof}

\begin{rema}\label{Rema: two codes in sim s}
We deduce from the proof that the equivalence class of every code $\underline{w}\in \Sigma_{\cS(T)}\setminus$Per$(\sigma_A)$ have at most $2$ elements.  In fact it has exactly $2$-elements. The number $k(\underline{w})$ always exists  and as $\sigma^k(\underline{w})\notin \underline{\cS(T)}$, its projection is in the complement of $\partial^s\cR$, this implies $\pi_f(\sigma^k{\underline{w}})\notin \partial^s\cR$ but  $\pi_f(\sigma^{k+1}{\underline{w}})\in \partial^s\cR$ therefore is in the boundary of two consecutive horizontal sub-rectangles of $R_{w_k}$, this gives rise to two  different codes $\sim_s$-related.
\end{rema}

It remains to extend the relation $\sim_s$ to the $s$-boundary periodic codes that were defined in \ref{Defi: s,u-boundary periodic codes} by Per$(\underline{\cS(T)})$. Observe first that, Per$(\sigma)\cap\Sigma_{\underline{\cS(T)}}=$Per$(\underline{\cS(T)})$ because if $\underline{w}\in$Per$(\sigma)\cap\Sigma_{\cS(T)}$ there is $K\in \ZZ$ such that for all $k\geq K$, $\sigma^k(\underline{w})\in \underline{\cS(T)}$ but there exist $r>k$ such that $\sigma^r(\underline{w})=\underline{w}$, so $\underline{w}\in \underline{\cS(T)}$ since the beginning. The relation $\sigma_s$ could be extended to $\Sigma_S$ as follows.

\begin{defi}\label{Defi: sim s in per}
Let $\underline{\alpha},\underline{\beta}\in$Per$\underline{\cS(t)}$ be  $s$-boundary periodic codes. They are $s$-related, $\underline{\alpha}\sim_s \underline{\beta}$,  if and only if they are equal or there are $\underline{w},\underline{v} \in \Sigma_{\cS(T)}\setminus$Per$(\sigma)$ such that:
\begin{itemize}
	\item $\underline{w}\sim_{s}\underline{v}$,
	\item There exists $k\in \ZZ$ such that,  $\sigma^k(\underline{w})_+=\underline{\alpha}_+$ and   $\sigma^k(\underline{v})_+=\underline{\beta}_+$.
\end{itemize}
\end{defi}

\begin{prop}\label{Prop: sim s equiv in Sigma S }
The relation $\sim_s$ in $\Sigma_{\cS(T)}$ is an equivalence relation.
\end{prop}

\begin{proof}
Two codes in $\Sigma_{\cS(T)}$ are $s$-related if and only if both are periodic or non-periodic.We have already determined the non-periodic situation, it remains to analyze what happens in the periodic setting. But reflexivity is for free, while symmetry and transitivity are inherited from the relation in $\Sigma_{\cS(T)}\setminus$Per$(\sigma)$ to the periodic codes in view of the definition we have given.
\end{proof}

Now we are ready to see that $\sim_s$ related codes  are projected to the same point.

\begin{prop}\label{Prop: s-relaten implies same projection}
Let $\underline{w},\underline{v}\in \Sigma_{\cS(T)}$ be to $s$-boundary leaves codes. If  $\underline{w}\sim_{s}\underline{v}$ then $\pi_f(\underline{w})=\pi_f(\underline{v})$.
\end{prop}

\begin{proof}

Assume that $\underline{w},\underline{v}$ are $s$-related. For simplicity assume that the $k$ in the definition is equal to zero. Like the negative part of such codes is equal $\pi_f(\underline{w})$ and $\pi_f(\underline{v})$ are in the same unstable segment of $R_{w_0}$. Take $w_0=v_0=i$.
		
	We shall prove that $x_w:=\pi_f(\underline{w})$ and $x_v:=\pi_f(\underline{v})$ are in the same stable segment of $R_i$. We can assume that $x_w\in H^{i}_{j+1}$ and $x_v \in H^{i}_{j}$, condition imposed by Item $iii)$ of Definition \ref{Defi: s realtion in Sigma S no per}.
	
	The point $x_w$ is on the stable boundary $\delta_w=+1,-1$ of $H^{i}_{j+1}$ because its image is on the stable boundary of $\cR$. If $\delta_w=+1$ then, following the dynamics of $f$, the point $f(x_w)$ is the boundary $\epsilon_w=\epsilon_T(i,j+1)$, which is not possible since $\epsilon_w=-\epsilon_T(i,j+1)$ by definition of the relation $\sim_s$. Then $\delta_w=-1$ and $x_w$ is on the lower boundary of $H^i_{j+1}$.
	
	Analogously $x_v$ is the boundary component $\delta_v=+1,-1$ of $H^{i}_j$. 	If $\delta_v=-1$,  following the dynamics of $f$, $f(x_v)$ is the boundary component of $R_{v_1}$ corresponding to $\epsilon_v=-\epsilon_T(i,j)$ which is a contradiction with the condition $\epsilon_v=\epsilon_T(i,j)$ imposed by $\sim_s$. Therefore $\delta_v=+1$ and is on the upper boundary of $H^i_j$.
	
	Then $\pi_f(\underline{w})$ and $=\pi_f(\underline{v})$ are on the lower boundary of $H^{i}_{j+1}$ and the upper boundary of $H^{i}_j$ respectively, which are the same stable segment of $R_i$, so $\pi_f(\underline{w})=\pi_f(\underline{v})$.

\end{proof}

In the same spirit of $\sim_s$ there is a equivalence relation $\sim_u$ for the elements in $\Sigma_{\cU(T)}$, there is a easy ways to determine such a relation that appeals to the idea of the inverse of the geometric type $T$ as introduced in \ref{Defi: inverse of Type}. As the Observation \ref{Defi: inverse of Type} says if $(f,\cR)$ represent $T$, then $(f^{-1},\cR)$ represent $T^{-1}$. In this setting $f^{-1}(V^k_l)=H^i_j$ if and only if $f(H^i_j)=V^k_l$ and $f^{-1}$ preserve the horizontal direction restricted to $V^k_l$ if and only if $f$ preserve the horizontal direction restricted to $H^i_j$. We summarize this discussion as follows:

\begin{lemm}\label{Lemm: geometric type of inverse}
If $\cR$ is a geometric Markov partition for the pseudo-Anosov homeomorphism $f$ with geometric type $T$, then:
\begin{itemize}
\item $\cR$ is a geometric Markov partition for $f^{-1}$ in which the horizontal and vertical directions of $(f,\cR)$ become the vertical and horizontal directions of $(f^{-1},\cR)$ respectively.
\item The geometric type of $\cR$ view as Markov partition for $f^{-1}$ is $T^{-1}$.
\end{itemize}
\end{lemm}

With this notion we could define the relation $\sim_u$ in $\Sigma_{\cU(T)}$ as the relation $\sim_s$ in $\Sigma_{\cS(T^{-1})}$, this formulation is sometimes useful. However, we prefer to give the following definition which is fully symmetric to  \ref{Defi: s realtion in Sigma S no per} and illuminates the mechanism by which codes are identified in $\Sigma_{\cU(T)}$, this insight will be use later when extending the $\sim_s$ and $\sim_u$ relations to $\Sigma_A$.

\begin{defi}\label{Defi: u realtion in Sigma U no per}
	Let $\underline{w},\underline{v}\in \Sigma_{\cU(T)}\setminus$Per$(\sigma)$, they are $u$-related and write $\underline{w}\sim_{u} \underline{v}$ if and only if they are equal or:
	\begin{itemize}
		\item[i)] There exist $z\in \ZZ$ such that $\sigma^z(\underline{w}),\sigma^z(\underline{v})\notin \underline{\cU(T)}$ but $\sigma^{z-1}(\underline{w}),\sigma^{z-1}(\underline{v})\in \underline{\cU(T)}$.
		
		\item[ii)] The $0$-terms, $\sigma^z(\underline{w})_0=w_z$ and  $\sigma^z(\underline{v})_0=v_z$ are equals, i.e. $w_z=v_z:=k\in\{1,\cdots,n\}$  and the number $v_k>1$ (given by $T$).
		
		\item[iii)] There is $l\in \{1, \cdots,v_{k}-1\}$ such that only one of the two possibilities happen:
	\begin{equation}\label{Equa: Sim u 1 option}
	\xi_{T^{-1}}(k,l)=(w_{z-1}) \text{ and  } \xi_{T^{-1}}(k,l+1)=(v_{z-1})
	\end{equation}
		or
	\begin{equation}\label{Equa: Sim u 2 option}
	\xi_{T^{-1}}(k,l)=(v_{z-1}) \text{ and  } \xi_{T^{-1}}(k,l+1)=(w_{z-1})
	\end{equation}
		
		\item[iv)] Suppose the negative code of $\sigma^{z-1}{\underline{w}}$  is equal to the $u$-boundary code $\underline{J}^-(w_{z-1},\epsilon_w)$ and  the negative code of  $\sigma^{z-1}(\underline{v})$  is equal to the $u$-boundary code $\underline{J}^-(v_{z-1},\epsilon_v)$, then:
		
		If Equation \ref{Equa: Sim u 1 option} holds:
	\begin{equation}\label{Equa: Sim u 1 epsilon}
\epsilon_w=\epsilon_{T^{-1}}(k,l) \text{ and } \epsilon_v=-\epsilon_{T^{-1}}(k,l+1)
	\end{equation}
		In the case Equation \ref{Equa: Sim u 2 option} happen:
	\begin{equation}\label{Equa: Sim u 2 epsilon}
		\epsilon_v=-\epsilon_{T^{-1}}(k,l) \text{ and } \epsilon_w=-\epsilon_{T^{-1}}(k,j+1)
	\end{equation}
		\item[v)] The positive codes of $\sigma^z(\underline{w})$ is equal to the positive code of $\sigma^z(\underline{v})$.
	\end{itemize}
\end{defi}

Similarly to Proposition \ref{Prop: sim s equiv in Sigma non per} we could prove $\sim_{u}$ is a equivalent relation in $\Sigma_{U}\setminus$Per$(\sigma_A)$ and then extend the relation to the periodic boundary points as in  \ref{Defi: sim s in per}.

\begin{defi}\label{Defi: sim u in per}
	Let $\underline{\alpha},\underline{\beta}\in$Per$(\Sigma_{\cU(T)})$ be periodic $u$-boundary codes. They are $u$-related if and only if there exist $\underline{w},\underline{v}\in \Sigma_{\cU(T)}\setminus$Per$(\sigma$ such that:
	\begin{itemize}
		\item $\underline{w}\sim_{u}\underline{v}$,
		\item There exist $p\in \ZZ$ such that the negative  codes  $\sigma^p(\underline{w})_-$ is equal to the negative code of $\underline{\alpha}$ and the negative code  $\sigma^p(\underline{v})_-$ coincide with the negative code of $\underline{\beta}$.
	\end{itemize}
\end{defi}

With the same techniques used in \ref{Prop: sim s equiv in Sigma S } we could prove 

\begin{prop}\label{Prop: sim u equiv in Sigma U }
	The relation $\sim_u$ in $\Sigma_{\cU(T)}$ is of equivalence.
\end{prop}

And similarly 

\begin{prop}\label{Prop: u-relaten implies same projection}
	Let $\underline{w},\underline{v}\in \Sigma_{\cU(T)}$ be two $s$-boundary leaves codes, if $\underline{w}\sim_{u}\underline{v}$ then $\pi_f(\underline{w})=\pi_f(\underline{v})$.
\end{prop}

We have defined $\sim_{s,u}$ in terms of the combinatorial information of the geometric type $T$ so that they do not depend on the homeomorphism-partition representation $(f,\cR)$ of $T$. Note that the Propositions \ref{Prop: s-relaten implies same projection} and \ref{Prop: u-relaten implies same projection} give a sufficient condition for two codes to project to the same point but this is not a necessary condition, since there could be more codes projecting to the same point, for this reason we need to complete our relations to $\sim_T$ in $\Sigma_A$.  We end this subsection by defining an equivalence relation $\sim_I$ on the totally interior codes $\Sigma_{\cI nt(T)}$. As proved in Lemma \ref{Lemm: Projection Sigma S,U,I} totally interior codes are the only codes that project to totally interior points of any representation $(f,\cR)$ of $T$, moreover any totally interior point of $(f,\cR)$ has associated to it a unique totally interior code that projects to it, this inspires the following definition.

\begin{defi}\label{Defi: I sim relation}
Let $\underline{w},\underline{v}\in \Sigma_{\cI nt(T)}$be two totally interior codes, they are $I$-related and we write $\underline{w}\sim_I\underline{v}$ if and only if $\underline{w}=\underline{v}$.
\end{defi} 

The following result follows from  Proposition \ref{Prop: Carterization injectivity of pif} where we have characterized totally interior points as having a single code projecting to them. The equivalent relation part is totally trivial because the relation $\sim_I$ is the equality between codes.

\begin{prop}\label{Prop: totally interior points projection sim I}
The relation $\sim_I$ is an equivalence relation in $\Sigma_{\cI nt(T)}$ and two codes $\underline{w},\underline{v}\in \Sigma_{\cI nt(T)}$ are $\sim_I$ related if and only if $\pi_f(\underline{w})=\underline{v}$, i.e they project to the same point.
\end{prop}

\subsubsection{ The equivalent relation $\sim_T$ on $\Sigma_{A(T)}$}

Finally we are ready to define the relation $\sim_T$ in $\Sigma_A$ that Proposition \ref{Prop: The relation determines projections} claim to exist, it consist essentially in the relation generated by $\sim_s,\sim_u$ and $\sim_I$.

\begin{defi}\label{Defi: Sim-T equivalent relation}
Let $\underline{w},\underline{v}\in \Sigma_{A}$ they are $T$-related and write $\underline{w}\sim_T\underline{v}$ if and only if any of the following disjoint situations occurs:
\begin{itemize}
	\item[i)] $\underline{w},\underline{v}\in \Sigma_{\cI nt(T)}$ and $\underline{w}\sim_I\underline{v}$, i.e. they are equals.
	\item[ii)] $\underline{w},\underline{v}\in \Sigma_{\cS(T)} \cup \Sigma_{\cU(T)}$ and there exist a finite number of codes $\{\underline{x_i}\}_{i=1}^m \subset  \Sigma_{\cS(T)} \cup \Sigma_{\cU(T)}$ such that:
\begin{equation}
	\underline{w}\sim_{s}\underline{x_i}\sim_{u} \underline{x_2}\sim_{s} \cdots \sim_{s} \underline{x_m}\sim_{u} \underline{v},
\end{equation}
or
\begin{equation}
\underline{w}\sim_{u}\underline{x_i}\sim_{s} \underline{x_2}\sim_{u} \cdots \sim_{u} \underline{x_m}\sim_{u} \underline{v}.
\end{equation}
\end{itemize}
\end{defi}

\begin{prop}
The relation $\sim_T$ is an equivalence relation in $\Sigma_A$.
\end{prop}

\begin{proof}
For  $\underline{w}\in \Sigma_{\cI nt(T)}$ the relation is reflexive, symmetric and transitive by virtue of equality. 
In the situation $\underline{w},\underline{v} \Sigma_{\cS(T)} \cup \Sigma_{\cU(T)}$, the reflexivity and symmetry are free, the transitivity comes from the concatenation of the codes and the fact that $\sim_{s}$ and $\sim_{u}$ are transitive as well.

\end{proof}

 If $\underline{w}\sim_{s,u}\underline{v}$  then $\sigma(\underline{w})\sim_{s,u}\sigma(\underline{v})$ because the numbers $k,z\in \ZZ$ at Item $i)$ of the definitions \ref{Defi: s realtion in Sigma S no per} and  \ref{Defi: u realtion in Sigma U no per} are changed to $k-1$ and $z-1$ respectively and the rest of the conditions hold for $\sigma(\underline{w})$ and $\sigma(\underline{v})$. The same property holds for the relation $\sim_I$. This remark implies the following lemma which will simplify some arguments.

\begin{lemm}\label{lemma: simT related iterations related}
If two codes $\underline{w},\underline{v}\in \Sigma_A$ are $\sim_T$ related then for all $k\in \ZZ$, $\sigma^k(\underline{w})\sim_T\sigma^k(\underline{v})$.
\end{lemm}

The only property that remains to be corroborated in the relation $\sim_T$ to obtain Proposition\ref{Prop: The relation determines projections}   is the one that relates it to the projection. That is the content of the following result.

\begin{prop}\label{Prop: T related iff same projection}
Let $\underline{w},\underline{v}\in \Sigma_A$ be any admissible codes. Then $\underline{w} \sim_T\underline{v}$ if and only if $\pi_f(\underline{w})=\pi_f(\underline{v})$.
\end{prop}

\begin{proof}

If  $\underline{w}\sim_T\underline{v}$ then we have two options:

\begin{itemize}
\item $\underline{w},\underline{v}\in \Sigma_{\cI nt(T)}$ and they are $\sim_I$-related. By Proposition \ref{Prop: totally interior points projection sim I}  this happen if and only if $\pi_f(\underline{w})=\pi_f(\underline{w})$. This situation is over.
\item  $\underline{w},\underline{v}\in \Sigma_{\cS(T)}\cup\Sigma_{\cU(T)}$. Using alternately Propositions \ref{Prop: s-relaten implies same projection} and \ref{Prop: u-relaten implies same projection} we deduce
$$
\pi_f(\underline{w})=\pi_f(\underline{x_1})=\cdots=\pi_f(\underline{x_m}) =\pi_f(\underline{v}).
$$
and complete a direction of the proposition.
\end{itemize}

Now suppose that  $x=\pi_f(\underline{w})=\pi_f(\underline{v})$, we need to prove that they are $\sim_T$ related. Since the only codes that project to the same point are sector codes of the point (Lemma \ref{Lemm: every code is sector code }), $\underline{w},\underline{v}$ are sector codes of $x$ and the following Lemma implies our Proposition.

\begin{lemm}\label{Lemm: sector codes}
Let $\{\underline{e_i}\}_{i=1}^{2k}$ be the sector codes of the point $x=\pi_f(\underline{e_i})$. Then $\underline{e_i}\sim_T \underline{e_j}$ for all $i,j\in \{1,\cdots,2k\}$.
\end{lemm}

\begin{proof}
If $x$ is a totally interior point the Corollary\ref{Coro: interior periodic points unique code} implies that all sector codes of $x$ are equal to a then $\underline{e_j}\sim_T\underline{e_j}$. In the Figure \ref{Fig: Sim T items i y ii} this correspond to the situation when $f^k(x)$ have all its quadrants, like in item $b)$.

The remaining situation is when  $x\in \cF^s(\text{ Per }^s(\cR))\cup \cF^u(\text{ Per }^u(\cR))$, i.e. $x$ is in stable or unstable lamination generated by $s,u$-boundary periodic points, we will concentrate on this case.  Numbering the sectors of $x$ in cyclic order, we consider three situations depending on where $x$ is located:
\begin{itemize}
\item[i)] $x\in \cF^s(\text{ Per }^s(\cR))$ but $x\notin \cF^u(\text{ Per }^u(\cR))$ (Items $c)$ and $f)$ in Figure \ref{Fig: Sim T items i y ii})
\item[ii)] $x\in \cF^u(\text{ Per }^u(\cR))$ but $x\notin \cF^s(\text{ Per }^s(\cR))$ (Items $a)$ and $d)$ in Figure \ref{Fig: Sim T items i y ii}).
\item[iii)]  $x\in \cF^s(\text{ Per }^s(\cR))\cap\cF^u(\text{ Per }^u(\cR))$ ((Items $b)$ and $e)$ in Figure \ref{Fig: Sim T items i y ii})).
\end{itemize}

In either case we first consider that $x$ is not periodic, which means that no sector code $\underline{e_j}$  is periodic and the point $x$ have $4$ sectors because  is not a periodic point and then it is not a singularity.

\textbf{Item} $i)$ (Look at (Items $c)$ and $f)$ in Figure \ref{Fig: Sim T items i y ii})).
 There exist $k\in \ZZ$  such that $f^k(x)\notin\partial^s\cR$ and like $x\notin  \cF^u(\text{ Per }^u(\cR))$, $f^z(x)\notin \partial^u\cR$ for all $z\in \ZZ$. In particular $f^z(x)\in \overset{o}{\cR}$ for all $z\leq k$, we conclude that : $f^k(x)$ have only four quadrants like $x$ and according with  the definition of the $\sim_s$ relation ( discussion accompanied by the  Figure \ref{Fig: stable identification}):
\begin{itemize}
\item For all $z< k$, all the quadrants of $f^{z}(x)$ are in the same rectangle: $\underline{e_1}_{z}=\underline{e_2}_{z}=\underline{e_3}_{z}=\underline{e_4}_{z}$.
\item $f^{k+n}(x)\in \partial^s\cR$ and therefore its quadrants are like in   (Items $c)$ and $f)$ in Figure \ref{Fig: Sim T items i y ii}) then: $\underline{e_1}_{k+n}= \underline{e_2}_{k+n}$ and $\underline{e_3}_{k+n}= \underline{e_4}_{k+n}$

\item Therefore:  $\underline{e_1}\sim_{s}\underline{e_4}$ and $\underline{e_2}\sim_{s}\underline{e_3}$. 

\item Like $f^z(x)\notin \partial^u\cR$ for all $z\in \ZZ$ the sector codes of $x$ satisfy that: $\underline{e_2}=\underline{e_1}$ and $\underline{e_3}=\underline{e_4}$. 
\end{itemize}
In conclusion $\underline{e_2}\sim_u\underline{e_1}$ and $\underline{e_3}\sim_u\underline{e_4}$ and then: 
 $$
 \underline{e_1}\sim_s\underline{e_4}\sim_u\underline{e_3}\sim_s \underline{e_2}\sim_u\underline{e_1}.
 $$ 
 so $\underline{e_i}\sim_T\underline{e_j}$ for $i,j=1,2,3,4$.
 
\textbf{Item} $ii)$ is argued in the same way. Look at Figure \ref{Fig: Sim T items i y ii} to get the intuition of \textbf{Item} $i)$ and realize that the same idea applies for the second item. In this case up some iteration $f^k(x)$ is in are a rectangle like in items $a)$ or $d)$ in the image and the negative iterations of $f^k(x)$ keep this configurations. Therefore the  terms of the sector codes $\underline{e_1}_{k-n}=\underline{e_4}_{k-n}$ and $\underline{e_2}_{k-n}=\underline{e_3}_{k-n}$ and like $f^{k+n}(x)\in \overset{o}{\cR}$ the therms $\underline{e_{\sigma}}_{k+n}$ is the same for $\sigma=1,2,3,4$: this implies that
\begin{itemize}
\item $\underline{e_1} \sim_u \underline{e_4}$ and $\underline{e_2}\sim_u \underline{e_3}$.
\item  $\underline{e_1}\sim_s \underline{e_2}$ and $\underline{e_3}\sim_s \underline{e_4}$.
\end{itemize}

Finally: $\underline{e_i}\sim_T\underline{e_j}$ for $i,j=1,2,3,4$.
\begin{figure}[h]
	\centering
	\includegraphics[width=0.9\textwidth]{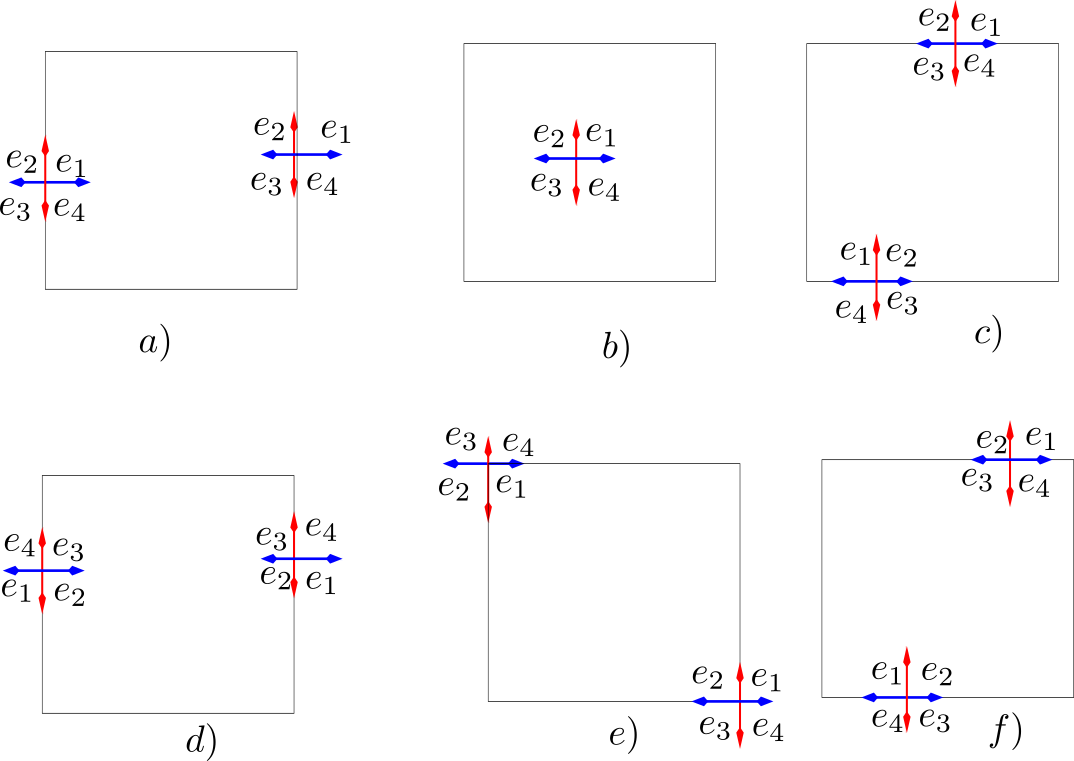}
	\caption{The sector codes that are identify.}
	\label{Fig: Sim T items i y ii}
\end{figure}

\textbf{Item} $iii)$ is the most technical situation. There are unique numbers $k(s),k(u)\in \ZZ$ defined as:
$$
k(s):=\max\{k\in \ZZ:f^k(x)\notin \partial^s \cR \text{ but } f^{k+1}(x)\in \partial^s \cR\},
$$
and 
$$
k(u)=\min\{k\in \ZZ:f^k(x)\notin \partial^u \cR \text{ but } f^{k-1}(x)\in \partial^u \cR\},
$$

The $f$-invariance of $\cR^s$ implies that for all $k>k(s)$, $f^k(x)\in \partial^s \cR$, but for all $k\leq k(s)$,  $f^k(x)\notin \partial^s \cR$. Similarly the $f^{-1}$ invariance of $\partial^u \cR$ implies that for all $k<k(u)$, $f^k(x)\in \partial^u\cR$ but for all $k\geq k(u)$ $f^k(x)\notin \partial^u \cR$.

They are three possibilities to detail: $k(u)<k(s)$, $k(u)=k(s)$ and $k(u)>k(s)$. Lets to divide the proof in this cases.

\textbf{First case} $k(u)<k(s)$. This inequality implies that $f^{k(s)}(x)\in \overset{o}{\cR}$ (and $f^{k(u)}(x)\in \overset{o}{\cR}$), this conditions  implies:

\begin{itemize}
\item The sector codes  take the same value for all $k(u)\leq  k\leq k(s)$.
	
\item  For all $k\geq k(s)$ the configuration of the sectors of $f^k(x)$ is like in Items $c)$ or $f)$ in in figure \ref{Fig: Sim T items i y ii}.

\item For all $k< k(u)$ the configurations of the sector of  $f^{k}(x)$ is like in Items  Item $a)$ or $d)$ in in figure \ref{Fig: Sim T items i y ii}.

\end{itemize}

 In view of Lemma \ref{lemma: simT related iterations related} we can deduce that: 
$$
\underline{e_1}\sim_s\underline{e_4}\sim_u\underline{e_3}\sim_s \underline{e_2}\sim_u\underline{e_1}.
$$
Hence $\underline{e_i}\sim_T\underline{e_j}$ for all $i=1,\cdots,4$.

\textbf{Second case} $k(u)=k(s)$. Analogously, $f^{k(u)}(x)=f^{k(s)}(x)\in \overset{o}{\cR}$ and we repeat the analysis of the previous case to deduce that $x$ have $4$ sector codes and all of them are $\sim_T$ related.  Figure \ref{Fig: k(u) less than k(s)} illustrates the ideas behind our arguments.

\begin{figure}[h]
	\centering
	\includegraphics[width=0.7\textwidth]{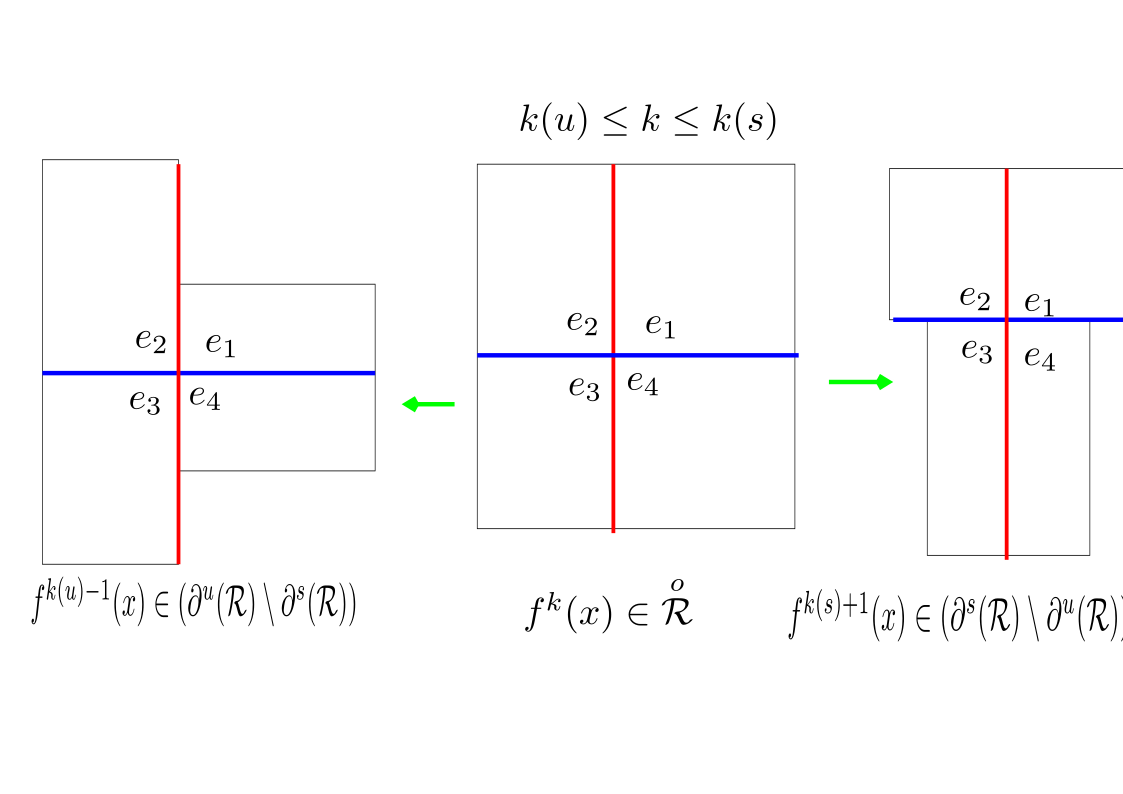}
	\caption{Situation $k(u)\leq k(s)$}
	\label{Fig: k(u) less than k(s)}
\end{figure}

\textbf{Third case} $k(s)<k(u)$.   In this situation $f^{k(s)}(x)\notin \partial^s \cR$ but $f^{k(s)}\in \partial^u\cR$ and for all $k$,  between $k(s)$ and $k(u)$, $f^k(x)$ is a corner point, and this have the following consequences:

\begin{itemize}
\item For all $n\in \NN$ $f^{k(s)-n}(e_1)$ and $ f^{k(s)-n}(e_4)$  are in the same rectangle of $\cR$ and similarly  $f^{k(s)}(e_2)$ and $f^{k(s)}(e_3)$ are in the same rectangle of $\cR$.
\item For all $n\in \NN$, $f^{k(u)+n}(e_1)$ and $f^{k(u)+n}(e_4)$ are in the same rectangle, and  similarly $f^{k(u)+n}(e_2)$ and $f^{k(u)+n}(e_3)$, are in the same rectangle.
\end{itemize}

Now for all $n\in \NN$, the negative part of the codes $\sigma^{k(s)-n}(e_1)$ and  $\sigma^{k(s)-n}(e_1)$ are equal. But for $k\in\{k(s),\cdots,k(u)\}$, $f^{k}(x)$ is in the stable boundary of $\cR$ and the configuration of the sector is like in Items $e)$, for $k\geq k(s)+1$ the configuration of the sector is like in $c)$ or $f)$. This lets us to conclude that $\underline{e_1}\sim_s \underline{e_4}$. The same argument applied for $\underline{e_2}$ and $\underline{e_3}$.

Using Lemma \ref{lemma: simT related iterations related} we deduce that:
$$
\underline{e_1}\sim_s \underline{e_4} \text{ and  } \underline{e_2}\sim_s \underline{e_3}.
$$

 Finally  for all $n\in \NN$, the sector codes of $f^{k(u)+s}(x)$ are like in configurations $c)$ and $f)$ so thee positive therms of $\sigma^{k(u)+1}(e_1)$ and  $\sigma^{k(u)+1}(e_2)$ are equal and $f^{k(u)}(x)$ is a corner point like in items $e)$, in particular the sector codes $\sigma^{k(u)}(\underline{e_1})$ and $\sigma^{k(u)}(\underline{e_1})$ share a unstable boundary of the Markov partition and then $\underline{e_1}\sim_u \underline{e_2}$. This process applied for the rest of sectors. Once again we can use Lemma \ref{lemma: simT related iterations related} to deduce that:
$$
\underline{e_1}\sim_u \underline{e_2} \text{ and  } \underline{e_3}\sim_u \underline{e_4}.
$$  
Then we put all together to obtain that:
$$
\underline{e_1}\sim_s\underline{e_4}\sim_u\underline{e_3}\sim_s \underline{e_2}\sim_u\underline{e_1}.
$$
proving that $\underline{e_i}\sim_T \underline{e_j}$ for $i,j=1,2,3,4$. This mechanism is illustrated in Figure \ref{Fig: k(s) less than k(s)}.

\begin{figure}[h]
	\centering
	\includegraphics[width=0.6\textwidth]{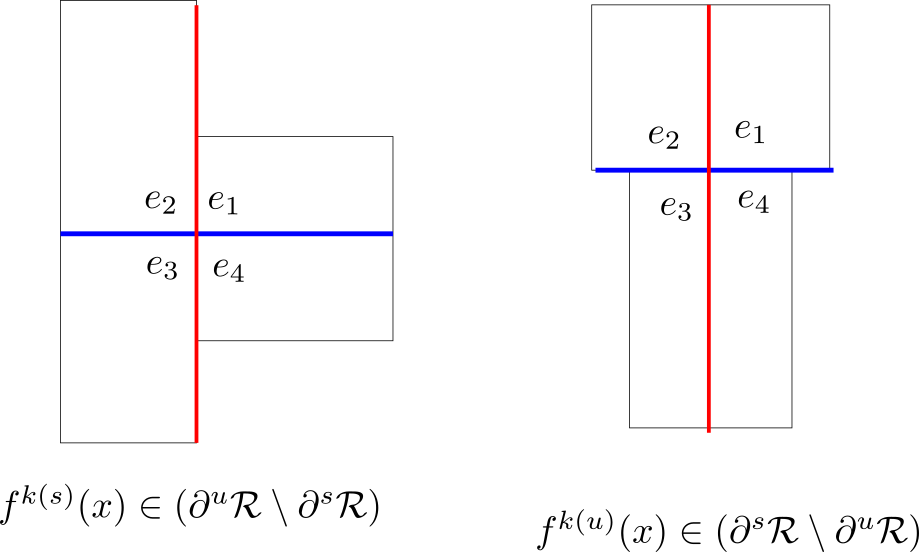}
	\caption{Situation $k(s)\leq k(u)$}
	\label{Fig: k(s) less than k(s)}
\end{figure}

It remains to address the periodic coding case. Let $x$ be a periodic point with $2k$-sectors, we label them with the cyclic order with respect to the surface orientation.  Two adjacent $e_i,e_{i+1}$-sectors are $s$-related if they share a stable local separatrix, as indeed, the stable separatrix lies in unique pair of rectangles and the boundary code of such sides are $s$-related. The reason is the mechanism of stable boundary identification that appears in Figure  \ref{Fig: stable identification} and that motivated the definition of the $\sim_s$-relation. If $ e_i,e_{i+1}$ share the same unstable separatrice they are $u$-related for the same reason. In this way one can get from one sector of $x$ to another by means of a finite number of intermediate sectors, which are alternatively $\sim_s$ and $\sim_u$-related. Then $\underline{e_i}\sim_T\underline{e_j}$ for all $i,j=1,\cdots, 2k$. This argument ends the proof.
\end{proof}

\end{proof}

%% file: TotalInvariant/Subproof.tex
\subsection{The geometric type is a total conjugacy invariant.}

All the work of this section is condensed in the following Proposition.

\begin{prop}\label{Prop:  cociente T}
	Let $T$ be a symbolically modelable geometric type and $(f,\cR)$ a realization of $T$, where $f:S\rightarrow S$ is a pseudo-Anosov homeomorphism (with or without Spines). Let $A:=A(T)$ be the incidence matrix of $T$, assume that $A$ is a binary matrix,  $(\Sigma_A,\sigma)$ the  sub-shift of finite-type associated to $T$ and $\pi_f:\Sigma_A\rightarrow S$ is the projection.
	
	The quotient space $\Sigma_T=\Sigma_A/\sim_T$ is equal to $\Sigma_f:=\Sigma_A/\sim_{f}$. Therefore $\Sigma_T$ is a closed surface. The sub-shift of finite type $\sigma$ pass to the quotient by $\sim_T$ to a homeomorphism $\sigma_T:\Sigma_T \rightarrow \Sigma_T$ which is a  generalized pseudo-Anosov homeomorphism topologically conjugate to $f:S\rightarrow S$ via the quotient homeomorphism:
	$$
	[\pi_f]: \Sigma_T=\Sigma_f \rightarrow S.
	$$ 
\end{prop}

\begin{proof}
As the proposition \ref{Prop: The relation determines projections} indicates, the relation $\sim_f$ given by $\underline{w}\sim_f\underline{v}$ if and only if $\pi_f(\underline{w})=\pi_f(\underline{w})$ is equal to $\sim_T$, regardless of the pair $(f,\cR)$ that represent $T$. Therefore the quotient spaces $\Sigma_T$ and $\Sigma_f$ are the same.

The Proposition \ref{Prop: quotien by f} implies that $S$ and $\Sigma_f$ are homeomorphic, hence $\Sigma_T$ is homeomorphic to $S$ which is a closed surface. Moreover the shift map $\sigma$ passes to the quotient under $\sim_f$ to a homeomorphism $[\sigma]$ which is topologically conjugate to $f$ via the quotient map $[\pi_f]$. On the other hand the shift $\sigma$ passes to the quotient under $\sim_T$ too, so we have the homeomorphism $\sigma_T:\Sigma_T\rightarrow \Sigma_T$, which is identical to $[\sigma]$ and  we conclude  as in Proposition \ref{Prop: quotien by f} that:
	$$
[\pi_f]^{-1}\circ f\circ [\pi_f]= \sigma_T.
	$$
Finally, the Proposition \ref{Prop: pA closed up conjugation} implies that $\sigma_T:\Sigma_T \rightarrow \Sigma_T$ is  pseudo-Anosov because it is topologically conjugate to a pseudo-Anosov homeomorphism. 
\end{proof}

Observe that if $(g,\cR_g)$ and $(f,\cR_f)$ are two pair that represent $T$. It is possible to consider the product homeomorphism $[\pi_f]^{-1}\circ [\pi_g]$, because if $\underline{w}\in \pi_f^{-1}(x)\cap \pi^{-1}_g(y)\neq \emptyset$, any other codes $\underline{X}\in \pi^{-1}(x)$ and $\underline{Y}\in \pi^{-1}_g(y)$ are $\sim_T$ related, i.e. $\underline{X}\sim_T\underline{w}\sim_T\underline{Y}$, therefore  $ \pi_f^{-1}(x)= \pi^{-1}_g(y)$. It is time to prove our main Theorem \ref{Theo: conjugated iff  markov partition of same type}

\begin{theo*}
	Let $f:S_f\rightarrow S_f$ and $g:S_g \rightarrow S_g$ two pseudo-Anosov homeomorphism maybe with Spines. If $f$ and $g$ had  Markov partitions  of same geometric type then there exist an orientation preserving homeomorphism between the surfaces $h:S_f\rightarrow S_g$ that conjugate them $g=h\circ f\circ h^{-1}$.
\end{theo*}

\begin{proof}
We already know that if $f$ and $g$ are conjugate they have Markov partitions of the same geometric type, so let's focus on the other direction of the proof.

Let $T := T(f, \cR_f) = T(g, \cR_g)$ be the geometric type of the Markov partitions of $f$ and $g$. If $T$ does not have a binary incidence matrix, we can consider a horizontal refinement of $(f, \cR)$ and $(g, \cR)$ in such a manner that the incidence matrix of such refinement is binary and the number of horizontal and vertical sub-rectangles in the respective Markov partitions is bigger than $2$.  The geometric types of these refinements remain the same. Therefore, we can assume from the beginning that $T$ has an incidence matrix $A$ with coefficients in $\{0, 1\}$ and $v_i,h_i>2$

 The quotient spaces of $\Sigma_A$ by $\sim_f$ and $\sim_g$ are equal to $\Sigma_T$ as proved in Proposition \ref{Prop:  cociente T}, i.e.  $\Sigma_g=\Sigma_{T}=\Sigma_f$. Moreover, the quotient shift  $\sigma_{T}$ is topologically conjugate to $f$ via $[\pi_f]:\Sigma_T \rightarrow S_f$ and to $g$ via $[\pi_g]:\Sigma_T \rightarrow S_g$. Therefore $f$ is topologically conjugate to $g$ by $h:=[\pi_g]\circ[\pi_f]^{-1}:S_f \rightarrow S_g$  which is a well-defined homeomorphism. It rest to prove that $h$ preserve the orientation. 
 
 \begin{lemm}
Let $H^i_j$ and $\underline{H}^i_j$ the respective horizontal sub-rectangles of $\cR_f$ and $\cR_g$, then $h(H^i_j)=\underline{H}^i_j$.
 \end{lemm}
 
 \begin{proof}
The set  $R(i,j)=\{\underline{w}: w_0=i \text{ and } \rho(i,j)=(w_1,l_0) \}$ is such that: $\pi_f(R_i,j)=H^i_j$ and  $\pi_g(\underline{H}^i_j)$. Consider the set $R(i,j)_T$ like the equivalent classes of $\Sigma_T$ that contains and element of $R(i,j)$. Clearly $[\pi_g]\circ[\pi_f]^{-1}(H^i_j)=[\pi_g](R(i,j)_T)=\underline{H}^i_j$
 \end{proof}

This have the following consequence.

\begin{coro}
The homeomorphism $h$ restricted to every rectangle $R_i$ preserve the traverse orientation of its vertical and horizontal foliations. In particular $h$ preserve the orientation restricted to $R_i$.
\end{coro}

\begin{proof}
Let $R_i$ be a rectangle in $\cR_f$ such that $x$ is its bottom-left corner point. This implies that the stable separatrix $I$ that determines the lower boundary of $R_i$ points in the opposite direction from $x$. Similarly, for the left stable boundary $J$ of $R_i$, it also points away from $x$. It's worth noting that within $R_{i}$, there are two sub-rectangles $H^i_1$ and $V^i_1$, whose the intersection defines a sector $e_x$ associated with $x$.

Now, consider $h(x)$, which is a corner point of $\underline{R}_i$. The rectangles $h(H^i_1)=\underline{H}^i_1$ and $h(V^i_1)=\underline{V}^i_1$ that are incident in $h(x)$ define a sector $h(e_x)$. It's necessary that the rectangle $\underline{H}^i_2$, the image of $H^i_2$, be adjacent to $\underline{H}^i_1$. This adjacency ensures the coherent ordering of horizontal sub-rectangles within $\underline{R}_i$, in accordance with the horizontal orientation of $\underline{R}_i$. Simultaneously, $h(V^i_2)$ becomes the vertical sub-rectangle of $\underline{R}_i$ next to $\underline{V}^i_1$. The sequence of rectangles $\underline{V}^i_l$ retains its coherence with the horizontal orientation of $\underline{R}_i$.

As both $h_i$ and $v_i$ are greater than $2$, $h$ simultaneously maps positive-oriented vertical and horizontal sub-rectangles of $R_i$ to positive-oriented vertical and horizontal sub-rectangles of $\underline{R}_i$. This also implies that $h$ preserves the transversal orientation of the foliation, and consequently, it maintains the orientation restricted to $R_i$.

\end{proof}

Finally, if the rectangles $R_i$ and $R_j$ intersect at a stable boundary point $x$, we can assign vertical orientations to $R_i$ and $R_j$ such that $x$ becomes the lower boundary point of $R_i$ and the upper boundary point of $R_j$. Necessary adjustments can be made to the stable directions of these rectangles to ensure that their orientations remain unchanged. By doing so, we can define $\underline{R_i}$ and $\underline{R_j}$ with vertical orientations induced by $h$. Since $h$ preserves the transversal orientation of the foliations in both $R_i$ and $R_j$, it also preserves the horizontal orientations of these rectangles. Consequently, $h$ preserve the orientation in the union of $R_i$ and $R_j$.

\end{proof}

The space $\Sigma_T$ and the homeomorphism $\sigma_T$ are constructed in terms of $T$ and as we have seen represent the conjugacy class of pseudo-Anosov homeomorphisms realizing $T$. In this sense we think that the symbolic dynamical system $(\Sigma_T,\sigma_T)$ deserves a name.

\begin{defi}\label{Defi: symbolic model}
Let $T$ be a geometric type in the pseudo-Anosov class, with an incidence matrix $A$ that is binary. Let $(f, \cR)$ be a pair representing $T$. The pair $(\Sigma_T, \sigma_T)$ represents the symbolic model of the geometric type $T$.
\end{defi}

Finally, we have obtained a combinatorial representation of a geometric type $T$ in the pseudo-Anosov class by solving Topic 1.III in \ref{Prob: Clasification}.

%% file: Realization/Realization.tex
\section{The pseudo-Anosov class of geometric types.}\label{Chapter: Realization}

The objective of this section is to prove the following Theorem.

\begin{theo}[Algorithmic characterization]\label{Theo: caracterization is algoritmic}
	There exists a finite and explicit algorithm that determine whether or not given geometric type $T$ belongs to the pseudo-Anosov class. Such algorithm requires calculating at most $6n$ iterations of the geometric type $T$, where $n$ is the first parameter of $T$.	
\end{theo}

In Proposition \ref{Prop: mixing+genus+impase is algorithm}, we presented a finite algorithm to determine whether the last properties in Item $iii)$ of Proposition \ref{Prop: pseudo-Anosov iff basic piece non-impace} are satisfied or not. This is the main theoretical result of this section, and it is evident that Theorem \ref{Theo: caracterization is algoritmic} is a direct consequence of it.

 \begin{prop}\label{Prop: pseudo-Anosov iff basic piece non-impace}
	Let $T$ be an abstract geometric type. The following conditions are equivalent.
\begin{itemize}
	\item[i)] The geometric type $T$ is realized as a mixing basic piece of a surface Smale diffeomorphism without impasse.
\item[ii)] The geometric type $T$ is in the pseudo-Anosov class.
\item[iii)] The geometric type $T$ satisfies the following properties:
	\begin{enumerate}
		\item  The incidence matrix $A(T)$ is mixing
		\item The genus of $T$ is finite
		\item $T$ does not have an impasse.
	\end{enumerate}
\end{itemize}
\end{prop}
 
 The proof of \ref{Prop: pseudo-Anosov iff basic piece non-impace} is divided in three parts: In Subsection \ref{Sec: type basic piece then pA}  we get Item $i)$ implies Item $ii)$, in Section \ref{Sec: Tipes PA finite genus but not impasse}   Item $ii)$ implies Item $iii)$ and Section \ref{Sec: finite genus no impas implies basic piece} close the cycle by proof that Item $iii)$ implies Item $i)$.

\input{Realization/SubPiecetoPseudo.tex} 
\input{Realization/SubPseudotoPiecce.tex}

\input{Realization/SubCombinatorialform.tex}

%% file: Realization/SubPiecetoPseudo.tex
\subsection{The geometric type of a mixing basic piece without impasse.}\label{Sec: type basic piece then pA}

In this section, our goal is to prove that if a geometric type $T$ is realized as a mixing, saddle-type basic piece of a surface Smale diffeomorphism without impasse, then $T$ is in the pseudo-Anosov class, in this manner we obtain that Item $(1)$ implies Item $(2)$ in the Proposition \ref{Prop: pseudo-Anosov iff basic piece non-impace}. Our starting point is the following theorem by Christian Bonatti and Emmanuelle Jeandenans.(see \cite[Theorem 8.3.1]{bonatti1998diffeomorphismes}):

\begin{theo}\label{Theo: Basic piece projects to pseudo-Anosov}

	Let $f$ be a Smale diffeomorphism of a compact and oriented surface $S$, $K$ be a non-trivial saddle-type basic piece of $f$, $\Delta(K)$ be its domain, and suppose that $K$ does not have an impasse. Then, there exists a compact surface $S'$, a generalized pseudo-Anosov homeomorphism $\phi$ on $S'$, and a continuous and surjective function $\pi: \Delta(K) \rightarrow S'$ such that:
	
	$$
	\pi\circ f\vert_{\Delta(K)}=\phi\circ \pi.
	$$

Furthermore, the semi-conjugation $\pi$ is injective over the periodic orbits of $K$, except for the finite number of periodic $s,u$-boundary points..
\end{theo}

The surface $S'$ has an induced orientation through $\pi$, we are going to clarify this point in Lemma \ref{Lemm: S is oriented} . Using the function $\pi$ and the orientation of $S'$, we will establish a more detailed version of Theorem \ref{Theo: Basic piece projects to pseudo-Anosov} because we aim to compare the geometric types of the basic piece and the pseudo-Anosov homeomorphism. We have defined the induced partition Markov partition by an orientation-preserving homeomorphism in the context of pseudo-Anosov homeomorphisms. 
There is a similar construction using $\pi$ that produces the induced Markov partition, which we will refer to in the following Proposition. We will provide more information about this concepts before starting the proof, but let's present the statement first.

\begin{prop}\label{Prop: type of basic piece is type of pseudo-Anosov}
	 Let $f:S\rightarrow S$ be a Smale surface diffeomorphism. Let $K$ be a mixing saddle-type basic piece of $f$ without impasse, and $\mathcal{R}=\{R_i\}_{i=1}^n$ be a geometric Markov partition of $K$ of geometric type $T$. Let $\pi:\Delta(K) \rightarrow S'$ be the projection, and let $\phi$ be the generalized pseudo-Anosov homeomorphism described in Theorem \ref{Theo: Basic piece projects to pseudo-Anosov}.
	 Then, the induced partition $\pi(\mathcal{R})=\{\pi(R_i)\}_{i=1}^n$ is a geometric Markov partition of $\phi$ of geometric type $T$.
\end{prop}

The proof requires two main results. First, Proposition \ref{Prop: pi cR is Markov partition} demonstrates that $\pi(\cR)$ is a Markov partition for $\phi$. Under certain natural assumptions regarding the orientation of the rectangles in $\pi(\cR)$,  $\pi(\cR)$ possesses a well-defined geometric type. Subsequently, Lemma \ref{Lemm: pi send sub-rec in sub rec} establishes the relationships between the vertical and horizontal sub-rectangles of $\cR$ and $\pi(\cR)$. Finally in the sub subsection \ref{Subsec: Proof Proposition} we provide a concise proof of Proposition \ref{Prop: type of basic piece is type of pseudo-Anosov}.

%%%%%%%%%%%%%
The Bonatti-Jeandenans Theorem \ref{Theo: Basic piece projects to pseudo-Anosov} contains all the ideas that we are going to use achieve Proposition \ref{Prop: type of basic piece is type of pseudo-Anosov}. This imposes the task of explaining how the function $\pi$ collapses the invariant laminations of $f$. In a simplified manner, the strategy in the proof of Bonatti-Jeandenans can be divided into the following parts:

\begin{itemize}
\item An equivalence relation $\sim_{R}$ is defined on the domain of $K$, $\Delta(K)$.

\item It is proven that the quotient space of $\Delta(K)$ by $\sim_{R}$ is a compact surface $S'$, and the diffeomorphism $f$ induces a generalized pseudo-Anosov homeomorphism $\phi$ on $S'$.

\item The function $\pi:S\rightarrow S'$  associates every point in $S$ with its equivalence class, i.e., it is the canonical projection.

\item  It is proved that $\pi$ is a semi-conjugation with the desired properties.
\end{itemize}

Understand the equivalence relation $\sim_{R}$ is to understand the function $\pi$. For this reason, we will explain how they work. The facts about $K$ and its invariant manifolds that were proven in \cite[Proposition 2.1.1]{bonatti1998diffeomorphismes} and recalled in Proposition \ref{Prop: sub-boundary points 2.1.1} will be especially useful at this point.
 
 \subsubsection{The equivalent relation $\sim_{R}$ described by parts.}
 
Assume that $f:S\rightarrow S$ is a Smale surface diffeomorphism and $K$ is a non-trivial, mixing, saddle-type basic piece of $f$ without double boundary points nor impasse. These are the only general assumptions in this section. 
 
 \subsubsection{Neighbor points and adjacent separatrices}
 
 In Figure \ref{Fig: Neigboor}, there are four periodic points of a saddle type basic piece $K$ with non-double boundaries, let us explain the notations:
 
 \begin{itemize}
\item[i)] The point $p_1$ is a $s$-boundary point but not a $u$-boundary point, while the point $p_3$ is a $u$-boundary point but not a $s$-boundary point.

\item The points $p_2$ and $p_3$ are $s$ and $u$ boundary points, respectively; we refer to them as \emph{corner points}.

\item The non-free separatrices of these periodic points are labeled as $W^s_{1,2}(p_1)$, $W^{s,u}_1(p_2)$, $W^u_{1,2}(p_3)$, and $W^{s,u}_1(p_4)$.

\item The unstable intervals $J_{p_1,p_2}(1,1)$ and $J'_{p_1,p_2}(1,1)$ are $u$-arches that have one endpoint in $W^s_1(p_1)$ and the other in $W^s_1(p_2)$. Similarly, $I_{p_3,p_4}(2,1)$ and $I'_{p_3,p_4}(2,1)$ are $s$-arcs that connect $W^u_2(p_3)$ with $W^u_1(p_4)$.

\item The rest of the arcs follows the same login in their notations.

 \end{itemize}

We are going  to discuss the notion of \emph{neighborhood point}. Take a $s$-boundary periodic point of $K$, in our example we take $p_1$ and let $x\in W^s_1(p_1) \cap K$. Like $x$ is not a $u$-boundary point, there exist a $u$-arc, like $J_{p_1,p_2}(1,1)$, with one end point in $x$ and the other in a point $y\in W^s_1(p_2)$, where $p_2$ is a $s$-boundary point and $W^s_1(p_2)$ is non-free separatrice. In this setting,  $p_2$ is what we call a \emph{neighbor point} of $p_1$ and $W^s_1(p_2)$ is an \emph{adjacent separatrice}  of $W^s_1(p_1)$.

This procedure can be applied to any $u$ or $s$ periodic boundary point. For example $p_2$ is a $u$-boundary point. Therefore if $x\in W^u_1(p_2)\cap K$ there is a $s$-arc, $I_{p_2,p_3}(1,1)$, that have and end point in $x$ and the other in a unstable separatrice of a $u$-boundary point, in the case of the picture such point is $p_3$, we said that $p_3$ is a neighbor point of $p_2$ and $W^u_1(p_2)$ is an adjacent separatrice of $W^u_1(p_2)$.

In general, let $p$ a $s$ (or $u$)-boundary point and $W^s_1(p)$ ( resp. $W^u_1(p)$) a non-free stable (unstable) separatrice. There are two properties that we what to establish:
\begin{enumerate}
\item The adjacent separatrice of $W^s_1(p)$ (resp. $W^u_1(p)$) is unique, and 
\item The adjacent separatrice of $W^s_1(p)$ (resp. $W^u_1(p)$) is different than $W^s_1(p)$ (resp. $W^u_1(p)$).
\end{enumerate}

Since the basic piece $K$ does not have an impasse, there cannot be a $u$-arc joining two points in the same stable separatrice. As a result, the adjacent separatrice of $W^s_1(p)$ cannot be $W^s_1(p)$ itself. However, it is possible that the adjacent separatrice of $W^s_1(p)$ is the other stable separatrice of $p$. We will soon see that this configuration leads to the formation of a spine.  The adjacent separatrices will be well determined  if, for every pair of points $x_1$ and $x_2$ on the same stable separatrice $W^u_1(p)$, the $u$-arcs (see $J_1$ and $J_2$ in Figure \ref{Fig: Neigboor}) that pass through $x_1$ and $x_2$ have their other endpoints in the same stable separatrice.

\begin{figure}[h]
	\centering
	\includegraphics[width=0.8\textwidth]{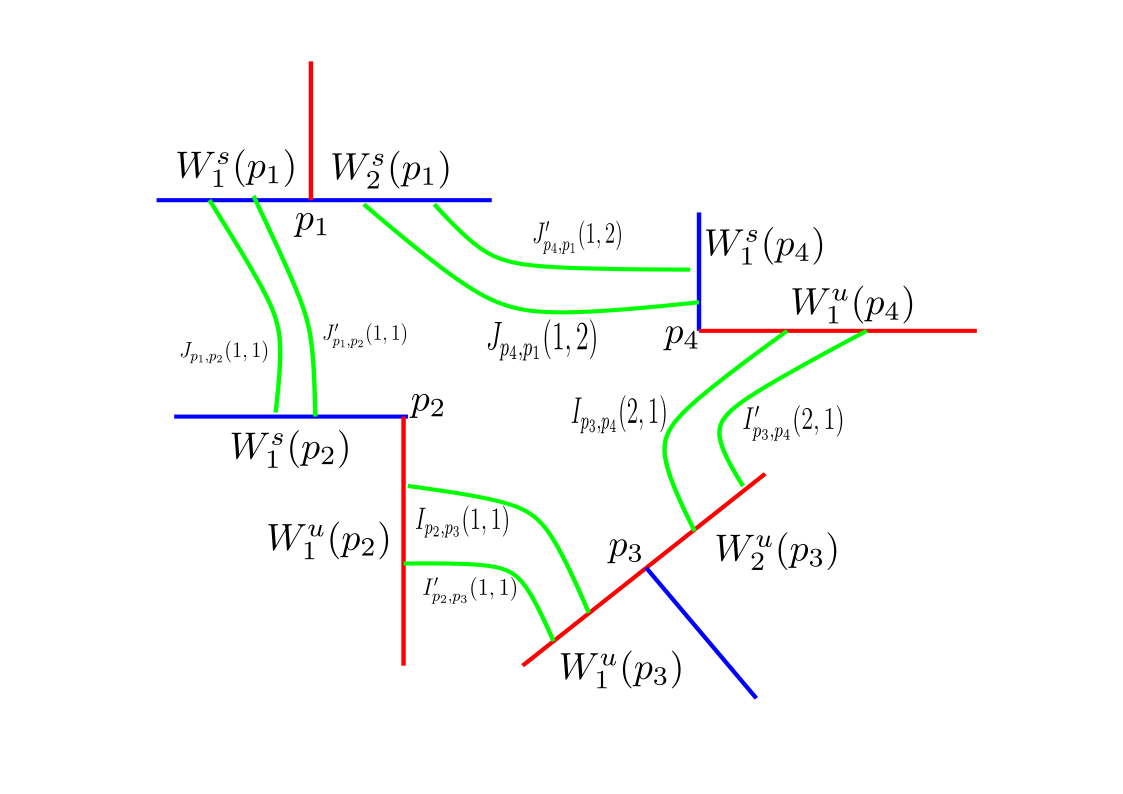}
	\caption{ Neighbors points.}
	\label{Fig: Neigboor}
\end{figure}

These assertions were proven in \cite[Proposition 2.4.9]{bonatti1998diffeomorphismes}. The hypothesis that $K$ does not have an impasse is crucial (see \cite[Definition 2.3.2]{bonatti1998diffeomorphismes} for understanding the statement). We summarize this discussion in the following lemma.

\begin{lemm}\label{Lemm: well define neighboor }
Let $K$ be a saddle-type basic piece without an impasse. If a $u$-arc $J_1$ joins two $s$-boundary stable separatrices $W^s_1(p_1)$ and $W^s_2(p_2)$, then:
\begin{itemize}
\item[i)] The separatrices are different, i.e., $W^s_1(p_1) \neq W^s_2(p_2)$.
\item[ii)]  If $J_2$ is any other $u$-arc with one endpoint in $W^s_1(p_1)$, then the other endpoint of $J_2$ is in $W^s_2(p_2)$.
\end{itemize} 
\end{lemm}

\begin{defi}
Let $p$ and $q$ be periodic boundary points of the same nature, i.e. $p$ is $s$-boundary (resp. $u$-boundary) if and only if $q$ is $s$-boundary (resp. $u$-boundary) . They are \emph{neighbor points} if one of the following conditions is satisfied:
\begin{enumerate}
\item They are $s$-boundary points and have non-free stable separatrices $W^s_1(p)$ and $W^s_2(q)$, with $W^s_1(p) \neq W^s_2(q)$, for which there exists a $u$-arc with one endpoint in $W^s_1(p)$ and the other in $W^s_2(q)$. In this case, we say that $W^s_1(p)$ is the \emph{adjacent separatrice} of $W^s_2(q)$.

\item They are $u$-boundary points and have unstable separatrices $W^u_1(p)$ and $W^u_2(q)$, with $W^u_1(p) \neq W^u_2(q)$, for which there exists an $s$-arc with one endpoint in $W^u_1(p)$ and the other in $W^u_2(q)$. In this case, we say that $W^u_1(p)$ is the \emph{adjacent separatrice} of $W^u_2(q)$.
\end{enumerate}

\end{defi}

\subsubsection{The cycles}

Next, we define an equivalence relation on the set of periodic $s$-boundary points and $u$-boundary points.

\begin{defi}\label{Defi: relation cycle}
Let $p$ and $q$ two boundary points of $K$ we have the next relations:
 \begin{itemize}
\item  If $p$ and $q$ are $s$-boundary points, then $p$ and $q$ are $\sim_s $-related if they are neighbors points.
\item  If $p$ and $q$ are $u$-boundary points, then $p$ and $q$ are $\sim_u $-related if they are neighbors points.
\item The points $p$ and $q$ are $\sim_c$-related if there is a sequence of boundary periodic points $\{p = p_0, \ldots, p_{k-1} = q\}$ such that $p_i \sim_s p_{i+1}$ or $p_i \sim_u p_{i+1}$ for $i \in \mathbb{Z}/k$ (the integers modulo $k$).
 \end{itemize}
\end{defi}

Is not difficult to see that the relation $\sim_c$ is reflexive, symmetric and transitive, therefore the next lemma is immediate.

\begin{lemm}\label{Lem: equiv cycle}
	The relation $\sim_c$ is a equivalence relation in the finite set of boundary periodic points of $K$, and every equivalent $\mathcal{C}$ class is called a  \emph{cycle}. We denote $\mathcal{C}=\{p_1, \ldots, p_k\}$ to indicate that $p_1\sim_{s,u}p_{i+1}$, in this manner we give an order to the cycle.
\end{lemm}

It might not be immediately apparent that if we follow the cycle, we return to the original point, i.e.,  $\mathcal{C}=\{p_i: i\in \ZZ/k \}$, and in particular $p_1=p_k$. However, this is proven in \cite[Lemma 8.3.4]{bonatti1998diffeomorphismes}. Another important feature is that when we follow the adjacent separatrices that determine the cycle, we always return to the same stable or unstable separatrice. In other words, the cycle is closed, this property is ensured by \cite[Lemma 8.3.5]{bonatti1998diffeomorphismes}, which proves that the number of corner points in a cycle is even. Thus, if we start with a stable separatrice, we will end with a stable separatrice of the same point.

In \cite[Proposition 3.1.2]{bonatti1998diffeomorphismes}, it is proven that every connected component of $\Delta(K)\setminus \delta(K)$ (See definitions \ref{Defi: Domain of K} and \ref{Defi: restricted domain} resp.) is homeomorphic to $\mathbb{R}^2$ and is periodic under the action of the diffeomorphism $f$. The boundary of each connected component has two possibilities: it is either the union of two free separatrices of a corner point or it consists of an infinite chain of arcs and their closures. The possibility of having an infinite chain of arcs is excluded by \cite[Proposition 2.6.4, Lemma 2.6.7]{bonatti1998diffeomorphismes}, in conjunction with the fact that there are no impasses in $K$. Consequently, every curve formed by a corner point and its free separatrices corresponds to the boundary of a connected component of $\Delta(K)\setminus \delta(K)$. Let $p$ be a corner point. Based on the previous observations, we can define a set $C(p) \subset \Delta(K)$ as the connected component of $\Delta(K)\setminus\delta(K)$ that is bounded by the (stable or unstable) free separatrices of $p$.

\begin{defi}\label{Defi: P(C) region of cycle}
Let $\mathcal{C} = \{p_1, \ldots, p_k\}$ be a cycle. We denote $\mathcal{P}(\mathcal{C})$ as the union of points in the cycle together with their free separatrices, and the sets $C(p)$ for all the corner points in the cycle.
\end{defi}

\subsubsection{Four different equivalent classes}

Now we can define certain families in $\Delta(K)$ that will serve as models for the equivalence classes for the relation $\sim_R$. They are:

\begin{defi}\label{Defi: equivalen clases sim-r}
		Let $\Delta(K)$ be the domain of the basic piece $K$, consider the four distinguished types of subsets of the domain of $K$.
\begin{itemize}
	\item[i)] The \emph{singletons} $\{x\in K: \, x \text{ is not a boundary point}\}$.
	
	\item[ii)] The \emph{minimal rectangles}, i.e., rectangles whose boundary consists of two unstable arcs and two stable arcs.
	
	\item[ii)] The arcs included in a non-boundary invariant manifold.
	
	\item[iii)] The sets $\cP(C)$, where $C$ is a cycle.
\end{itemize}
\end{defi}

By considering all sets of the form described above, we obtain a partition of $\Delta(K)$ that is invariant under $f$ as its proved in \cite[Lemma 8.3.7]{bonatti1998diffeomorphismes}. Moreover, this partition enables us to define an equivalence relation ( See \cite[Definition 8.3.8]{bonatti1998diffeomorphismes}).

\begin{defi}\label{Defi: sim R equiv relation}
We define $\sim_{R}$ as the equivalence relation where the classes are the components of $\Delta(K)$ that correspond to one of the four families previously described.
\end{defi}

The surface $S'$ in the Bonatti-Jeandenans Theorem \ref{Theo: Basic piece projects to pseudo-Anosov} is the quotient of $S$ by this equivalence relation,  i.e $S'=S/\sim_R$, and the map $\pi: S \rightarrow S'$ is the projection that maps each point to its equivalence class.

\subsubsection{ The projection of a Markov partition is a Markov partition.}

In this subsection we are going to prove that the projection of the Markov partition $\cR$ by $\pi$ is a Markov partition for $\phi$. For that reason, there are two important aspects to consider: the behavior of rectangles in a Markov partition $(\cR,f)$ under the action of $\pi$, and the behavior of the complement of this Markov partition in $\Delta(K)$. The first situation is summarized in \cite[Lemma 8.4.2]{bonatti1998diffeomorphismes}, which we reproduce below as Lemma \ref{Lemm: projection of rectangle by pi}. Essentially, the lemma states that the interior of a rectangle in the Markov partition projects to the interior of a rectangle under the pseudo-Anosov homeomorphism $\phi$ (see Theorem \ref{Theo: Basic piece projects to pseudo-Anosov}).

\begin{lemm}\label{Lemm: projection of rectangle by pi}
	Let $\cO$ the interior of a rectangle where every boundary component is not isolated in $\cO$ (i.e it boundary is accumulated by $K$ from the inside). The quotient of $\cO$ by the equivalence relation is homeomorphic to  a rectangle $(a_1,a_2)\times (b_1,b_2)\subset \RR^2$. 
	
	Moreover, the stable lamination of $\cO$ given by the segment of stable manifold of $K$, $W^s(K)\cap \cO$ has image the foliation by horizontal lines of the rectangle $(a_1,a_2)\times (b_1,b_2)$ and the unstable lamination of $\cO$ induced by $W^u(K)\cap \cO$ has  image the foliation by vertical lines of such rectangle.
	
	Finally, if the coordinates of the rectangle $(a_1, a_2) \times (b_1, b_2)$ are given by $(s, t)$, then the \emph{Margulis measures} $\upsilon^s$ and $\upsilon^u$ of $f$ pass to induce the measures $dt$ and $ds$, respectively.
	
\end{lemm}

The proof relies on an alternative formulation the projection $\pi$ inside these types of rectangles. It begins by considering a parametrization of the rectangle $\overline{\cO}$, $r:[-1,1]^2 \rightarrow \overline{\cO}$ with coordinates $(s,t)$.  Then, a function is defined $\psi:\cO \to (a_1,a_2)\times(b_1,b_2)\subset \RR^2$ as follows:
$$
\psi=(\psi_1,\phi_2)(m)=(\epsilon\int_{0}^{s}d\upsilon^u,\nu \int_{0}^{t} d\upsilon^s).
$$
Where $m \in \cO$ has coordinates $(s,t)$, $\epsilon = +1$ if $s \geq 0$ and $-1$ otherwise, and $\nu = +1$ if $t \geq 0$ and $-1$ otherwise. 

 The pre-image of any point in the rectangle $(a_1,a_2)\times(b_1,b_2)$ under this map coincides with the equivalence classes determined by $\sim_R$ inside $\cO$. Therefore,  $(\psi_1,\phi_2)$  is a covering map. By taking the quotient we obtain a homeomorphism, that intrinsically depends on $r$:
$$
\Psi^{-1}=(\Psi_1,\Psi_2):  \cO/\sim_{R} \rightarrow (a_1,a_2)\times (b_1,b_2)
$$
The homeomorphism has an inverse $\Phi: (a_1,a_2)\times (b_1,b_2) \to \cO/\sim_R$: ,  which serves as a parametrization of $\pi(\cO)$.
The next lemma is implicit in the Bonatti-Jeandenans proof of Lemma \ref{Lemm: projection of rectangle by pi}.

\begin{lemm}\label{Lemm: psi bonatti }
	The function $\psi$ is continuous and respect the parametrization $r$, $\phi_{1,2}$ is  non-decreasing along stable and unstable laminations in the rectangle $\cO$. Even more
	$$
	\pi(\cO)=\Psi( (a_1,a_2)\times (b_1,b_2) ).
	$$
\end{lemm}

\subsubsection{The induced orientation}

Once we have a the parametrization $\Psi^{-1}$ of  $\pi(\cO)$, whose inverse coincides with $\pi$. We can used $\Psi$ to endow $\pi(\cO)$ with a vertical and horizontal orientations, that is coherent with the vertical and horizontal orientation of $\cO$ that are induced by the parametrization $r$.

\begin{lemm}\label{Lemm: pi preserve transversals}
The homeomorphism $\Phi: (a_1,a_2)\times (b_1,b_2) \times \cO$ preserve the transversal orientations of the vertical and horizontal foliations.
\end{lemm}

\begin{proof}
The parametrization  $r$ of $\cO$ determines the vertical and horizontal positive orientation, furthermore with this orientations, $\psi$ is non-decreasing along the stable and unstable leaves of $\cO$, and when we pass to the quotient $\Psi^{-1}$ is increasing along the stable and unstable leafs of $\pi(\cO)$. This property ensures that $\Phi$ preserves the transversal orientations.
\end{proof}

\begin{defi}\label{defi: orientations induced by pi}
The orientation given to $\pi(\cO)$ by the homeomorphism $\Psi$ is  the \emph{orientation induced} by $\pi$ in $\pi(\cO)$. The vertical and horizontal orientations are $\cO$ are such given by the parametrization $\Phi$. With all this conventions we refer to $\pi(\cO)$ as the \emph{oriented rectangle induced by $\pi$}.
\end{defi}

\subsubsection{The image of rectangle is a rectangle.} The interiors of the rectangles in a Markov partition  satisfy the properties of Lemma \ref{Lemm: projection of rectangle by pi}, but our desire is the projection  by $\pi$ of every rectangle be a rectangle. Therefore, we need to extend the homeomorphism $\Psi$ to a parametrization of a rectangle.

\begin{lemm}\label{Lemm: pi(R) is a rectangle}
	For all $i\in \{1,\cdots,n\}$, $\pi(R_i)$ is a rectangle 
\end{lemm}

\begin{proof}

Take a rectangle $R$ in the Markov partition. This rectangle has non-isolated boundaries from its interior. In view of Lemma \ref{Lemm: projection of rectangle by pi}, $\pi(\overset{o}{R_i})$ is homeomorphic to a certain afine rectangle $\overset{o}{H}:=(a_1,a_2)\times (b_1,b_2)$ through the parametrization:

$$
\Phi:\overset{o}{H}\rightarrow \pi(\overset{o}{R}).
$$
Such parametrization sends (open) oriented vertical and horizontal intervals of $(a_1,a_2)\times (b_1,b_2)$ to oriented unstable and stable intervals of $\pi(\overset{o}{R})$, respectively.

It remains to show that this parametrization can be continuously extended to the closed rectangle $H=[a_1,a_2]\times [b_1,b_2]$ in such a way that its image is $\pi(R)$. Moreover, it is necessary for $H$ to be a homeomorphism restricted to any vertical or horizontal interval of $H$ whose image is contained in a single stable or unstable leaf of $\pi(R_i)$, as expected. In particular, the image of every boundary of the affine rectangle is contained in a single leaf.

Take a vertical segment $J\subset (a_1,a_2)\times [b_1,b_2]$. The closure of $\overset{o}{J}$ relative to the unstable leaf in which it is contained consists of adding the two boundary points to the image to obtain a closed interval: $\overline{\Psi(\overset{o}{J})}=\Psi(\overset{o}{J})\cup \{B_1,B_2\}$. Since $\pi(R)$ has the induced geometrization by $\pi$, $\Psi(\overset{o}{J})$ is positively oriented. Let $B_1$ be the lower point and $B_2$ be the upper point of $\overline{\Psi(\overset{o}{J})}$. The natural way to extend $\Psi$ to $J$ is given by the formula:

$$
\Psi(s,b_1)=B_1 \text{ and } \Psi(s,b_2)=B_2.
$$

We can do this for every unstable leaf and every stable leaf. The corners are the only points that remain to be determined. They will be naturally determined if we can prove that the image of the interior of the horizontal boundary of $[a_1,a_2]\times [b_1,b_2]$ lies in a single stable leaf, and the same holds for the image of the vertical boundary. Let us explain the idea.

Assume for a moment that the parametrization $\Psi$ sends the left side $\overset{o}{J}$ of $[a_1,a_2]\times (b_1,b_2)$ into an open interval contained in a unique unstable leaf $L^u$. Furthermore, assume that the lower boundary $I$ of $H$ is such that $\Psi(\overset{o}{I})$ lies in a unique stable leaf $F^s$. Then, $\Psi(\overset{o}{J})$ has a unique lower boundary point $A$, and $\Psi(\overset{o}{I})$ has a well-defined left endpoint $B$. These points are forced to be the same.

To visualize that they coincide, consider a sequence $J_n$ of vertical intervals of $[a_1,a_2]\times [b_1,b_2]$ that converges to the left side of $J$, and similarly, a decreasing sequence of intervals $I_n$ that converges to the lower side of the rectangle. The sequence of points $\{x_n\}=\Psi(\overset{o}{I_n})\cap \Psi(\overset{o}{J_n})$ (which lies in the interior of the rectangle) converges simultaneously to the lower point $A$ of $\Psi(I)$ and the left point $B$. Therefore, $\Psi$ is well-defined at the lower left corner point. The same argument can be applied to the other corner configurations.

So lets prove that the boundary is included in a  single leaf. Let be $s,s'\in(a_1,a_2)$ different numbers,  we need to observe that $\Psi(s,b_1)$  and $\Psi(s',b_1)$ are in the same stable leaf.

\begin{figure}[h]
	\centering
	\includegraphics[width=0.6\textwidth]{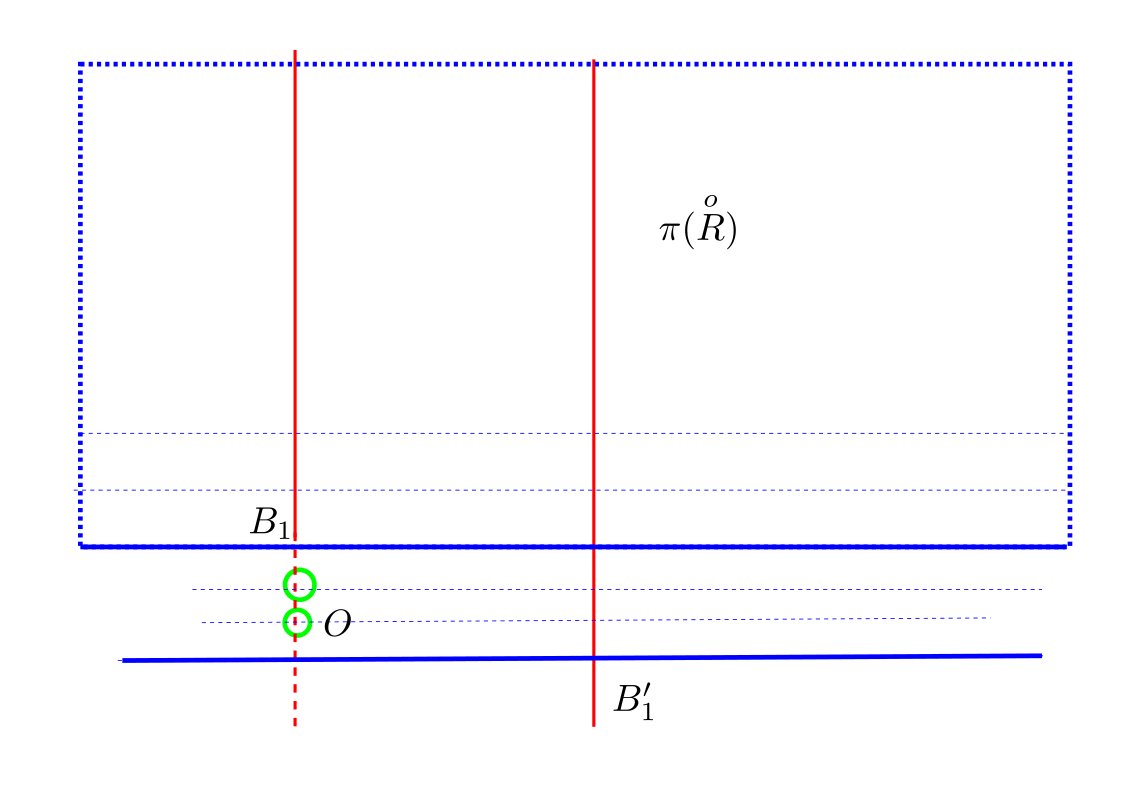}
	\caption{ Same unstable boundary.}
	\label{Fig: Same boundary}
\end{figure}

If $\Psi(s_1,b_1)=B_1$ and $\Psi(s'b_1)=B'_1$ are not in the same leaf, we would have a configuration like Figure \ref{Fig: Same boundary}. Using the horizontal orientation, we can assume that $B_1<B'_1$. However, this would create a problem because,  the intersection of a stable leaf near but above $B'_1$ with the vertical leaf passing trough $B_1$ would induce a point $O$ in the interior of the rectangle $\pi(\overset{o}{R})$. According to the trivial product structure in the interior, this configuration implies that $\Psi(s',b_1)=B'_1$ is in the interior of the corresponding unstable interval $\pi(\overset{o}{R})$, not in its boundary as intended. Therefore, $B_1$ and $B'_1$ must be in the same leaf.

Let $x$ be a stable boundary point of $h(R)$ and let $U$ be an open set containing $x$. The transverse measures of the pseudo-Anosov homeomorphism $\phi$ are Borel measures, and we can observe the existence of an open rectangle with a small enough diameter that can be sent inside $U$ under $\Psi$. This allows us to prove the continuity of $\Psi$ at the boundary points.

Take $J$ to be the interval such that $\Psi(J)$ has its lower point equal to $x$. Let $\epsilon > 0$. There exist vertical segments $\overset{o}{J_1}, \overset{o}{J_2} \subset H$ such that, in the horizontal order of rectangle $H$, $J_1 \leq J < J_2$, with equality only if $J$ is the left side of $H$, in which case it is not difficult to make the necessary adjustments to our arguments.

We have $\Psi(\overset{o}{J_1}) < \Psi(\overset{o}{J}) < \Psi(\overset{o}{J_2})$. We know that the length of every stable interval between $\Psi(\overset{o}{J_1})$ and $\Psi(\overset{o}{J_2})$ is a constant $C(1, 2) > 0$ (due to the invariance of the transverse measures under leaf isotopy). We can choose $J_1$ and $J_2$ to be sufficiently close to each other so that $C(1, 2) < \epsilon/4$.

Now, let $\overset{o}{I}$ be a horizontal leaf in $H$ such that the unstable segments contained into $\Psi(\overset{o}{J_1})$ and $\Psi(\overset{o}{J_2})$, with one endpoint in $\overset{o}{I}$ and the other in the stable leaf trough $x$, have $\mu^s$ length less than $\epsilon/4$.

Continuing from the previous point, we can choose the horizontal leaf $\overset{o}{I}$ to be sufficiently close to $x$ so that the unstable segments contained between $\Psi(\overset{o}{J_1})$ and $\Psi(\overset{o}{J_2})$, with one endpoint in $\overset{o}{I}$ and the other in the stable leaf of $x$, have length less than $\epsilon/4$.

In this way, by considering the rectangle whose boundary is given by the intersections of $I', J_1, J_2$, and the lower boundary of $H$, we can ensure that its image is contained in a rectangle with diameter less than $\epsilon$. For $\epsilon$ small enough, we can provide an open set $H$ whose image is completely contained in $U$. This completes the argument for continuity.
Finally we are going to see that:
$$
\Psi(H)=\pi(R).
$$

Let $x\in \partial^s R$. Without loss of generality, we can assume that $x$ is on the lower boundary of $R$. In order to show that $\pi(R)\subset \Psi(H)$, we need to prove that $\pi(x)$ is in $\Psi(H)$. The equivalence class of $x$ under the relation $\sim_R$ can either be a minimal rectangle or an unstable arc. In either case, there is a point in $K$ that belongs to the same equivalence class as $x$. Since all points in the equivalence class collapse to the same point, we can assume that $x\in K$. Let $x_n\in K\cap \overset{o}{R_n}$ be a sequence converging to $x$. By the continuity of $\pi$, we have $\pi(x_n)\rightarrow \pi(x)$.

After taking a refinement, we can assume that $\pi(x_{n+1})< \pi(x_{n})$ with respect to the vertical and horizontal orientation of the foliations induced by $\pi(\overset{o}{R})$. For every $\pi(x_n)$, there exists a unique point $y_n\in \overset{o}{H}$ such that $\Psi(y_n)=\pi(x_n)$. Moreover, we can assume that the sequence $(y_n)$ is decreasing with respect to both the vertical and horizontal orientation of $H$.

For each point $y_n$, it passes through a unique pair of stable and unstable intervals of $H$, denoted as $I_n$ and $J_n$ respectively. Since the intervals are decreasing and bounded from below with respect to the vertical order, there is a horizontal interval $I$ to which $I_n$ converges. Similarly, $J_n$ converges to an interval $J$, which may or may not be in the boundary of $H$. In any case, there exists a point $y\in I\cap J$ that is accumulated by the sequence $(y_n)$. By continuity, we have $\lim \pi(x)=\pi(x_n)=\lim \Psi(y_n)=\Psi(y)$.

To show the other inclusion $\Psi(H)\subset \pi(R)$, we consider a point $x\in H$ and aim to find a point $y\in R$ such that $\pi(y)=\Psi(x)$. We will focus on the case when $x$ lies on the inferior boundary of $H$, and the argument can be adapted for other boundary cases.

Let $J\subset H$ be the vertical interval passing through $x$. The set $\Psi(J)$ is a vertical interval in $\pi(R)$. Consider $\pi^{-1}(\Psi(\overset{o}{J}))$, which is a horizontal sub-rectangle of $R$ that may be reduced to an open interval. In any case, the stable boundary of this sub-rectangle belongs to a unique equivalence class under $\sim_R$. This equivalence class contains a point $k\in K\cap \partial R$. Let $J'$ be the unstable interval of $R$ that contains $k$ as one of its boundary points.
By continuity and the orientation-preserving properties of $\Psi$ and $\pi$, we have $\Psi(J)=\pi(J')$. Moreover, $\Psi(x)$ corresponds to the inferior boundary of $\pi(J')$ under this identification. Thus, we have found a point $y\in R$ such that $\pi(y)=\Psi(x)$, completing the proof.

\end{proof}

 \subsubsection{The projection of a Markov partition is a Markov partition}

\begin{prop}\label{Prop: pi cR is Markov partition}
	Let $\cR=\{R_i\}_{i=1}^n$ be a Markov partition of $f$. The family of rectangles $\pi(\cR)=\{\overline{\pi(\overset{o}{R_i})}\}_{i=1}^n$ is a  Markov partition of $\phi$.
\end{prop}

\begin{lemm}\label{Lemm: disjoint interiors}
	If $i\neq j$, then  $\pi(\overset{o}{R_i})\cap \pi(\overset{o}{R_j})= \emptyset$
\end{lemm}

\begin{proof}
	
	Please observe that $\pi\left(\overset{o}{R_i}\right)=\overset{o}{\left(\pi(R_i)\right)}$.  If $x\in \pi(\overset{o}{R_i})$, then $\pi^{-1}(x)$ is an equivalence class determined by the relation $\sim_R$, and we need to argue that $\pi^{-1}(x)\subset \overset{o}{R_i}$. This implies that if $\pi(\overset{o}{R_i})\cap \pi(\overset{o}{R_j}) \neq \emptyset$, then $\overset{o}{R_i}=\overset{o}{R_j}$.

	The points in $\pi^{-1}(x)$ do not belong to a cycle since cycles do not project to the interior of the rectangles. If $\pi^{-1}(x)$ is a singleton in the basic set, then $\pi^{-1}(x)\subset \overset{o}{R_i}$. If $\pi^{-1}(x)$ is a minimal rectangle, its stable (or unstable) boundary is different from the stable boundary of $R_i$ because $\pi^{-1}(x)$ is a minimal rectangle whose interior is disjoint from the basic piece $K$. Since the stable boundary of $R_i$ is isolated from one side, these two leaves are different, and the minimal rectangle lies in the interior of $R_i$. In the case where $\pi^{-1}(x)$ is an arc of a non-isolated leaf, we have $\pi^{-1}(x)\subset \overset{o}{R_i}$ because its endpoints do not belong to the stable boundary of $R_i$.
\end{proof}

\begin{lemm}\label{Lemm: the rectangles pi R cover S}
	The function $\pi$ restricted to $K$ is surjective, and $\cup_{i}^n \pi(R_i)=S'$.
\end{lemm}

\begin{proof}
	It is enough to observe that in any equivalence class of $\sim_{R}$ there is an element of $K$. Since $K$ is contained in $\cup_{i=1}^n R_i$, we have our result.
\end{proof}

\begin{lemm}\label{Lemm: Boundaries are f-invariant}
	The stable boundary of $\pi(\cR)$, $\partial^s \pi(\cR) = \cup_{i=1}^n \partial^s \pi(R_i)$, is $\phi$-invariant. Similarly, the unstable boundary of $\pi(\cR)$, $\partial^u \pi(\cR) = \cup_{i=1}^n \partial^u \pi(R_i)$, is $\phi^{-1}$-invariant.
	
\end{lemm}

 Proposition \ref{Prop: pi cR is Markov partition} follows from  Lemmas \ref{Lemm: pi(R) is a rectangle}, \ref{Lemm: disjoint interiors}, \ref{Lemm: the rectangles pi R cover S}, and \ref{Lemm: Boundaries are f-invariant}.

\begin{proof}
We need to show that $\partial^s \pi(\cR)=\cup_{i=1}^n\partial^s\pi(R_i)$ is $f$-invariant and $\partial^u \pi(\cR)=\cup_{i=1}^n\partial^u\pi(R_i)$ is $f^{-1}$-invariant. 
Given $i\in\{1,\cdots,n\}$, it is clear that $\partial^s\pi(R_i)=\pi(\partial^sR_i)$,This implies the following contentions and equalities between sets:
\begin{eqnarray*}
	\phi(\partial^s\pi(R_i))=\phi(\pi(\partial^sR_i))=\\
	\pi(f(\partial^sR_i))\subset \pi(\partial^s \cR)=\partial^s\pi(\cR).
\end{eqnarray*}
This proves $\partial^s \pi(\cR)$ is $\phi$-invariant. A similar argument proves the unstable boundary of $\pi(\cR)$ is  $\phi^{-1}$-invariant.
\end{proof}

\subsubsection{The surface $S$ is oriented}

Before to prove that a geometric Markov partition of $f$ projects to a geometric Markov partition we shall to determine  that the surface $S'$ is orientable. In the next lemma, observe that the orientation induced by $\pi$ in each rectangles of the Markov partition $\pi(\cR)$ is consistent along the boundary of the partition, which means that the orientations of adjacent squares in $\pi(\cR)$ match along their shared boundaries.

\begin{lemm}\label{Lemm: S is oriented}
	The surface $S'$ is orientable. 
\end{lemm}

\begin{proof}
Like the union of all $\pi(R_i)$ covers the whole surface $S'$, in order to obtain a orientation on it, is necessary to check that the orientations of the rectangles $\{\pi(R_i)\}_{i=1}^n$ are coherent in their boundaries. Let us study the situation of two adjacent stable boundaries
Let $x_i\in \partial^s R_i$ and $x_j\in \partial^s R_j$ be two points such that $\pi(x_i)=\pi(x_j)$. We will consider the case where they are not periodic $s$ or $u$-boundary points (i.e., they do not belong to a cycle). Without loss of generality, let's assume that $x_i$ and $x_j$ belong to the same unstable leaf $F^u$ and are both contained in $K$. 

Possibly after changing the vertical orientation of $R_i$ and/or $R_j$, it can be assumed that the horizontal direction of the stable boundaries where the points $x_i$ and $x_j$ of these rectangles meet, point to the respective periodic $s$-boundary points. This change the geometrization of the Markov partition, but  for this new geometrization, there is a unique orientation of the unstable leaf $F^u$ containing $J$ such that the vertical and horizontal orientations of $R_i$ and $R_j$, with respect to this geometrization, match the orientation induced by $\pi$. When we collapse the $u$-arc $J$ with end point $x_i$ and $x_j$, the stable and unstable directions of the leaves passing trough $\pi(x_i)=\pi(x_j)$ preserve their geometrization in every rectangle, and the resulting orientation at $\pi(x_j)$ is coherent for the two adjacent rectangles $\pi(R_i)$ and $\pi(R_j)$. (see \ref{Fig: Gluing orientation}).

\begin{figure}[h]
	\centering
	\includegraphics[width=0.4\textwidth]{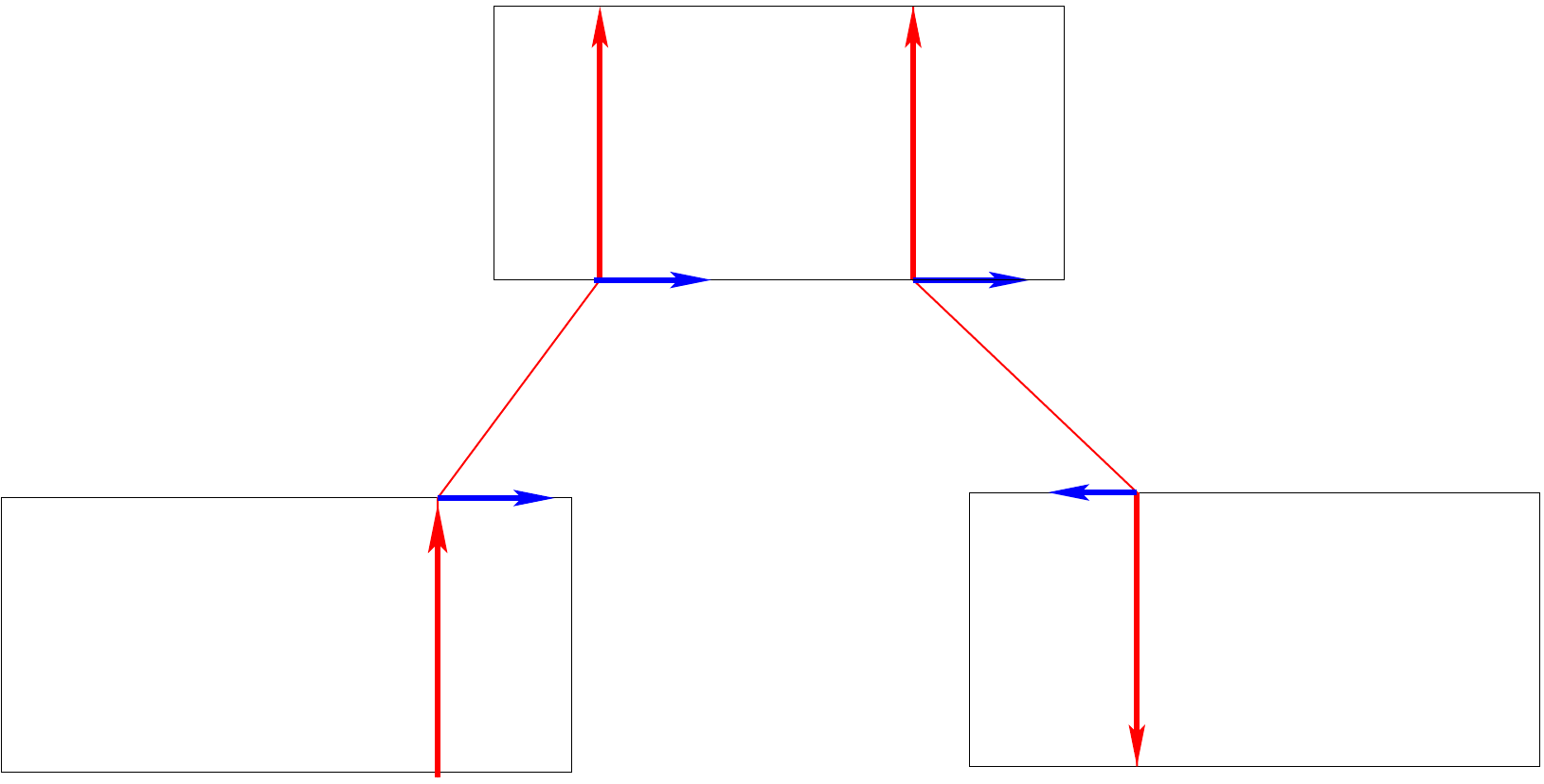}
	\caption{ The projection by $\pi$ match in the boundary}
	\label{Fig: Gluing orientation}
\end{figure}

The situation for a cycle is similar. Consider a cycle $[p_1,\cdots, p_k]$ and assign orientations to the non-free stable separatrices of each point in the cycle, pointing towards the periodic boundary point. Let's take a pair of rectangles $R_i$ and $R_j$ that contain two neighboring points $p_i$ and $p_j$ in the cycle. The stable boundaries of these rectangles intersect the two adjacent stable separatrices determined by $p_i$ and $p_j$.

We can orient $R_i$ and $R_j$ by completing their vertical orientations with the previously chosen horizontal orientation. After collapsing these two rectangles, $\pi(R_i)$ and $\pi(R_j)$ become two rectangles that share a stable interval in their boundary. The horizontal orientation of $\pi(R_i)$ and $\pi(R_j)$ is induced by the orientation of the two adjacent stable separatrices determined by $p_i$ and $p_j$. This forces a coherent orientation when we glue $\pi(R_i)$ with $\pi(R_j)$. This orientation consistency applies to all rectangles in the cycle, and the cyclic order around $\pi(p_i)$ must also be preserved.

\end{proof}

\begin{defi}\label{Defi: Orientation S'}
We endow the surface $S'$ with the orientation induced by the quotient homeomorphism $[\pi]: S/\phi \rightarrow S'$. 
\end{defi}

\subsubsection{The geometric type of the projection}
  The orientation of $S'$ given in Definition \ref{Defi: Orientation S'}, assigns to each rectangle in $\pi(\cR)$ the orientation induced by $\pi$ and this orientation is consistent along the boundary of the partition.

\begin{defi}\label{Defi: geometric partition induced by}
Let $\cR$ a geometric Markov partition of $f$ whose geometric type is $T$. Let $\pi(\cR)$ the geometric Markov partition of $\phi$ that consist of the  rectangles  whose orientation is induced by $\pi$,  its geometric type is denoted by $\pi(T)$ and we call $\pi(\cR)$ the geometric Markov partition induced by $\pi$.
\end{defi}

 Proposition \ref{Prop: type of basic piece is type of pseudo-Anosov} is consequence of this other proposition
 
 \begin{prop}\label{Prop: same type projection }
The geometric type $T$ is equal to $\pi(T)$
 \end{prop}
The proof will arrive from another Lemma.

\subsubsection{Same sub-rectangles:} For each $i \in \{1, \ldots, n\}$, the rectangle $R_i$ of the Markov partition $\cR = \{R_i\}_{i=1}^n$ for the Smale diffeomorphism $(f, S)$ has horizontal sub-rectangles $\{H^i_j\}_{j=1}^n$, which are enumerated in increasing order with respect to the vertical order in $R_i$. Similarly, the horizontal sub-rectangles of the rectangle $\pi(R_i)$ of the Markov partition $\pi(\cR)$ for the pseudo-Anosov homeomorphism $(\phi, S')$ are enumerated as $\{\underline{H^i_j}\}_{j=1}^{\underline{h_i}}$ in increasing order with respect to the vertical orientation of $\pi(R_i)$ induced by $\pi$.

\begin{lemm}\label{Lemm: pi send sub-rec in sub rec}
With the previous notation, we observe that $h_i = \underline{h_i}$, meaning that the number of horizontal sub-rectangles in $R_i$ is equal to the number of horizontal sub-rectangles in $\pi(R_i)$. Furthermore, for each $j \in \{1, \ldots, n\}$, we have $\pi(H^i_j) = \underline{H^i_j}$.

Similarly, if $\{V^k_l\}_{l=1}^{v_k}$ are the vertical sub-rectangles of $R_k$ and $\{\underline{V^k_l}\}_{l=1}^{\underline{v_k}}$ are the vertical sub-rectangles of $\pi(R_k)$, we have $v_k = \underline{v_k}$, indicating that the number of vertical sub-rectangles in $R_k$ is equal to the number of vertical sub-rectangles in $\pi(R_k)$. Moreover, we have $\pi(V^k_l) = \underline{V^k_l}$ for each $l \in \{1, \ldots, v_k\}$.

\end{lemm}

\begin{proof}
		 
Let $H:=H^i_j$ be a connected component of $R_i \cap f^{-1}(R_k)$, which is an horizontal sub-rectangle of the Markov partition $\cR$. Therefore, $\overset{o}{H}$ is a connected component of $\overset{o}{R_i} \cap f^{-1}(\overset{o}{R_k})$ for certain indexes of rectangles.
	 
 Now, let $\pi(\overset{o}{H})$ is a connected component of $\pi(\overset{o}{R_i}) \cap \pi(f^{-1}(\overset{o}{R_k}))= \overset{o}{\pi(R_i)} \cap \phi^{-1}(\pi(\overset{o}{R_k}))$. Therefore $\pi(\overset{o}{H})$ is an horizontal sub-rectangle of $\pi(\cR)$. If $H_1,H_2$ are two different  horizontal sub-rectangles of $R_i$. Their interior don't intersect and a adaptation of the argument in \ref{Lemm: disjoint interiors} implied $\pi(\overset{o}{H_1})\cap \pi(\overset{o}{H_2})=\emptyset$. Therefore there is a bijection between the horizontal sub-rectangles of $R_i$ and those of $\pi(R_i)$.
 
The projection $\pi$ preserve the vertical order of the rectangles, then the rectangles $H^i_{j}$ and $H^i_{j+1}$ are projected in adjacent rectangles preserving the order, hence $\pi(H^i_{j})=\underline{H^i_{j}}$.

The proof of the assertions concerning the vertical sub-rectangles is completely similar.
\end{proof}

\begin{coro}\label{Coro: Same sub-rectangle}
For all $i\in \{1,\cdots, n\}$ and $j\in \{1,\cdots,h_i\}$, $\phi(\underline{H^i_j})=\underline{V^k_l}$ if and only if $f(H^i_j)=V^k_l$.
\end{coro}

\begin{proof}

If $f(H^i_j)=V^k_l$, then $\pi(f(H^i_j))=\pi(V^k_l)$. Using the semi-conjugation, we obtain that $\phi(\pi(H^i_j))=\phi(\underline{H^i_j})=\underline{V^k_l}$.

Conversely, if $\phi(\underline{H^i_j})=\underline{V^k_l}$, then $\underline{H^i_j}=\pi(H^i_j)$ and $\underline{V^k_l}=\pi(V^k_l)$. In the case that $f(H^i_j)=V^{k'}{l'}$, it is clear that $\underline{V^k_l}=\pi(f(H^i_j))=\pi(V^{k'}{l'})=\underline{V^{k'}_{l'}}$, which implies $k'=k$ and $l'=l$.

\end{proof}

\begin{rema}\label{Rema: pi preserve the orientation}
As the orientation in $\pi(R_i)$ is induced by $\pi$, it is clear that $\pi$ preserves the vertical orientation relative to $R_i$ and $\pi(R_i)$. Similarly, $\pi$ preserves the horizontal orientation relative to $R_k$ and $\pi(R_k)$.

This implies that $f$ restricted to $H^i_j$ preserves the vertical orientation if and only if $\phi$ restricted to $\underline{H^i_j}$ preserves the vertical orientation.
\end{rema}

\subsubsection{Proof of proposition \ref{Prop: type of basic piece is type of pseudo-Anosov}. } \label{Subsec: Proof Proposition}
 After \ref{Coro: Same sub-rectangle} and Remark \ref{Rema: pi preserve the orientation} we can conclude the proof of Proposition \ref{Prop: type of basic piece is type of pseudo-Anosov}.
 
\begin{proof}
	Let  $T(\cR)=\{n,\{(h_i,v_i)\}_{i=1}^n, (\rho,\epsilon)\}$ be the geometric type of $\cR$, and let $\pi(T)=\{n',\{(h'_i,v'_i)\}_{i=1}^{n'}, (\rho',\epsilon')\}$  be the geometric type of $\pi(\cR)$.
	
	The number of rectangles in $\cR$ and $\pi(\cR)$ is equal to $n$, and for every $i\in \{1,\cdots, n\}$, Lemma \ref{Lemm: pi send sub-rec in sub rec} implies that $h_i=h_i'$ and $v_i=v_i'$.
	
	The function $\rho(i,j)=(k,l)$ if and only if $f(H^i_j)=V^k_l$. Corollary \ref{Coro: Same sub-rectangle} states that this occurs if and only if $\phi(\underline{H^i_j})=\underline{V^k_l}$. In terms of the geometric type, this is equivalent to the fact that $\rho'(i,j)=(k,l)$.
	
	Finally, Remark \ref{Rema: pi preserve the orientation} tells us that $\epsilon(i,j)$ is $1$ if and only if $\epsilon'(i,j)$ is $1$. This completes our proof.

\end{proof}

%% file: Realization/SubPseudotoPiecce.tex
\subsection{The pseudo-Anosov class have finite genus and no impasse.}\label{Sec: Tipes PA finite genus but not impasse}

In this section, we assume that $T$ is a geometric type in the pseudo-Anosov class. We are going to prove that its incidence matrix is mixing, it genus is finite and does not display any impasse. The arguments of this section use the equivalence between the combinatorial and topological conditions of finite genus and impasse.

Nowadays the following Proposition is a  classic result, its proof could be found for example in \cite{farb2011primer} and \cite{fathi2021thurston} and directly implies  the first Item of Proposition \ref{Prop: pseudo-Anosov iff basic piece non-impace}.

\begin{prop}\label{Prop: Incidence matrix mixing}
	Let $\cR$ be the Markov partition of a generalized pseudo-Anosov homeomorphism  then its incidence matrix is \emph{mixing}, i.e. there exists a number $n\in \NN$ such that for all $i,j\in {1,\cdots,n}$, $a_{i,j}^{(m)}>0$.
\end{prop}

According to Proposition \ref{Prop: Types transitive have a realization} any geometric type without double boundaries admits a realization. The following Lemma can be applied to a geometric type $T$ in the pseudo-Anosov class and implies that $T$ does not have double boundaries.

\begin{lemm}\label{Lemm: T mixing implies non double boundaries}
If the incidence matrix of $T$ is mixing, then $T$ does not have double boundaries.
\end{lemm}

\begin{proof}

	Let $\{e_i\}_{i=1}^n$  be the canonical basis of $\mathbb{R}^n$, and let $A$ be the incidence matrix of $T$. We recall that the entry $a_{ij}$ represents the number of horizontal sub-rectangles of $R_i$ that are mapped to $R_j$. In particular, if $h_i=1$, there exists a unique $j_0$ such that $f(R_i)\cap R_{j_0}$ is a horizontal sub-rectangle of $R_{j_0}$. This implies that $a_{ij_0}=1$, and for all $j\neq j_0$, $a_{ij}=0$.
	
	The condition of double $s$-boundaries implies the existence of $\{i_s\}_{s=1}^S\subset\{1,\cdots,n\}$  such that$A(e_{i_s})=e_{i_{s+1}}$ for $1\leq s<S$ and $A(e_{i_S})=e_{i_1}$. This forces $A^m$ to not be positive definite for any $m\in \mathbb{N}$, which contradicts the mixing property of the incidence matrix.
	
	Therefore, if $T$ has a mixing incidence matrix, it does not have double $s$-boundaries.
	
\end{proof}

This lemma have the following corollaries.

\begin{coro}\label{coro: pseudo Anosov non-doble boundaries}
	If $T$ is in the pseudo-Anosov class, then $T$ does not have double boundaries.
\end{coro}

\begin{coro}\label{Coro: T pA then realization}
If $T$  is a geometric type in the pseudo-Anosov class, $T$ it admits a realization.
\end{coro}

\subsubsection{Some notation related to orientations}

 Let $T$ be a geometric type in the pseudo-Anosov class that we denoted by:
$$
T:=\{n,\{h_i,v_i\}_{i=1}^n, \Phi:=(\rho,\epsilon)\}.
$$
In future discussions, we will compare two different objects associated to a to the geometric type $T$:

\begin{itemize}
\item[i] A \emph{realization} by \emph{disjoint rectangles} of the geometric type $T$. We shall denote such realization by a pair \emph{partition/diffeomorphism}: $\{\cR=\{R_i\}_{i=1}^n, \phi\}$. 

\item[ii)] A \emph{geometric Markov partition} by non-disjoint rectangles, $\tilde{\cR}=\{\tilde{ R_{i}}\}_{i=1}^n,$ of a pseudo-Anosov homeomorphism $f:S\rightarrow S$ with geometric type $T$.
\end{itemize}

We shall to fix some notations:

\begin{itemize}
\item[i)] For all $i\in \{1,\cdots,n\}$, $\tilde{R_i}$ is a rectangles in the Markov partition and $R_i$ is a rectangle in the realization.

\item[ii)]  The vertical and horizontal sub-rectangles of the realization are denote by:
$$
\{V^k_l: (k,l)\in \cV(T)\} \text{ and } \{H^i_j: (i,j)\in \cH(T)\}
$$.
\item[iii)] The vertical and horizontal sub-rectangles of the Markov partition are denoted by:
$$
\{\tilde{V^k_l}:(k,l)\in \cV(T)\} \text{ and  } \{ \tilde{H^i_j}: (i,j)\in \cH(T) \}.
$$

\item[iv)] If $\alpha$ is the lower boundary of a vertical sub-rectangle $V^k_l$ in the realization, then $\tilde{\alpha}$ is the lower boundary of the corresponding vertical sub-rectangle $\tilde{V^k_l}$ in the Markov partition. Similarly, if $\alpha$ is the left boundary of $H^i_j$, then $\tilde{\alpha}$ is the left boundary of $\tilde{H^i_j}$.

\item[v)] Take two vertical sub-rectangles $V^k_{l_1}$ and $V^k_{l_2}$, with $l_1<l_2$, and let $\alpha$ and $\beta$ be the horizontal boundaries of these sub-rectangles. Suppose both boundaries are contained in either the upper or lower boundary of $R_k$. Without loss of generality, let's assume they are in the lower boundary of $R_k$. Using the fact that $l_1<l_2$, we can write $\alpha=[a_1,a_2]^s<\beta=[b_1,b_2]^s$, and we denote 
$$
\gamma=[a_2,b_1]^s
$$
 as the curve contained in the lower boundary of $R_k$ and lying between $\alpha$ and $\beta$.

\item[vi)] The curve $\tilde{\gamma}$ is the curve contained in the lower stable boundary of $\tilde{R_k}$ that corresponds to the curve obtained by taking the union of the lower boundaries of all vertical sub-rectangles in the Markov partition $\tilde{\cR}$ that lie between $\tilde{V^k_{l_1}}$ and $\tilde{V^k_{l_2}}$. This is
$$
\tilde{\gamma}:=\cup_{l_1<l<l_2}\partial^s_{-}\tilde{V^k_l}.
$$

\end{itemize}

With respect to Item $v)$ and $vi)$: if $l_2=l_1 +1$ the curve $\tilde{\gamma}$ consist in a point but $\gamma$ is a non-trivial interval.

\begin{rema}\label{Rema: geometric type rules the images}
Suppose that: $\Phi_T(i,j)=(k,l,\epsilon)$. This implies that: $\phi(\tilde{H^i_j}) = \tilde{ V^k_l}$ and  $f(H^i_j)=V^k_l$.
\end{rema}

\subsubsection{Orientation} Here we should define three different orientations of a curve contained in the stable boundary of the Markov partition and the realization. This will complete the notation that we require to deduce that the realization doesn't have the topological obstructions by looking at what kind of obstruction it induces in the Markov partition.

Let $H^i_j$ and $V^k_l$ be a horizontal and a vertical sub-rectangle of $R_i$ in the realization. They have orientations induced by the rectangle $R_i$ in which they are contained. The function $\epsilon_T$ in the geometric type $T$ measures the change in relative vertical orientation induced by the action of $f$.

\begin{defi}\label{Defi: induced and glin for realization}
	Let $\alpha$ be a horizontal boundary component of $V^k_l. $ Since $R_k$ is a geometrized rectangle, there are two types of orientations for this curve:
\begin{itemize}
	\item The \emph{induced orientation} of $\alpha$ is the one that is inherited from the horizontal orientation of $R_k$. This means that $\alpha$ is oriented in the same direction as the horizontal sides of $R_k$.
	
	\item  	The \emph{gluing orientation} of $\alpha$, defined as follows: If $\epsilon_T(i,j) = 1$, then $\alpha$ has the induced orientation. However, if $\epsilon(i,j) = -1$, then $\alpha$ has the inverse orientation of the induced orientation. In other words, if the change in relative orientation between $H^i_j$ and $V^k_l$ is negative, then the gluing orientation of $\alpha$ is the opposite of the induced orientation.
	
\end{itemize}
\end{defi}

Similarly, for the Markov partition, we can define the induced orientation and gluing orientation of the curve  $\tilde{\alpha}$.

\begin{defi}\label{Defi: Induced/gluin in Markov}
	Let  $\tilde{\alpha}$ be the corresponding curve in the Markov partition:
\begin{itemize}
	\item The \emph{induced orientation} of $\tilde{\alpha}$ is the one that is inherited from the horizontal orientation of $\tilde{R}_k$.
	
	\item  The \emph{gluing orientation} of $\tilde{\alpha}$ is defined in a similar manner as for the concretization. If $\epsilon(i,j) = 1$, then $\tilde{\alpha}$ has the induced orientation. However, if $\epsilon(i,j) = -1$, then $\tilde{\alpha}$ has the inverse orientation of the induced orientation.
\end{itemize}
\end{defi}

 If $\tilde{\alpha}$ is contained in the stable boundary of two rectangle $R_1$ and $R_2$, the induced orientation of $\tilde{\alpha}$ coincides with its glue orientation if and the horizontal orientation of $R_1$ and $R_2$ is coherent along the stable boundaries that contains $\tilde{\alpha}$.

The boundary of a geometric Markov partition is contained in the stable leaves of periodic points and we can define another orientation in such stable boundaries.

\begin{defi}\label{Def: dynmaic orientation}
	Let $\tilde{O}$ be a periodic point of $f$ and $\tilde{L}$ be a stable separatrix of $\tilde{O}$. The  \emph{dynamic orientation} of the separatrix  $\underline{L}$ is defined by declaring that: for all $x \in \tilde{L}$, $f^{2Pm}(x) < x$, where $P$ is the period of $f$. 
\end{defi}

This orientation points towards the periodic point. Suppose that $\tilde{L}$ intersects the stable boundary of $\tilde{R_i}$ in an interval $\tilde{I}$, which may or may not contain a periodic point $\tilde{P}$. The \emph{dynamic orientation} of $\tilde{I}$ is the orientation induced by the dynamic orientation of the separatrix  $\tilde{L}$ within the interval. Let's address the case of the realization.

\begin{defi} \label{Defi: Dynac orinted for realization}
Suppose $O$ is a periodic point located on the horizontal boundary $L$ of $R_i$. Let $I$ be a connected component of $L \setminus \{O\}$. It can be observed that $I$ is mapped into  a smaller interval under the action of $\phi$, and after a certain number of iterations, it becomes contained within the interior of $I$.  This contraction property in  us to define the \emph{dynamical orientation} of $I$ in such a way that it points towards the periodic point $O$.
\end{defi}

Suppose that in the realization, there is a ribbon $r$ joining the boundary $\alpha$ of $V^k_l$ with the boundary $\beta$ of $V^{k'}_{l'}$. This configuration implies that in the Markov partition of $f$, the boundary $\tilde{\alpha}$ of $\tilde{V^k_l}$ is identified with the boundary $\tilde{\beta}$ of $\tilde{V^{k'}_{l'}}$. This identification preserves the \emph{gluing orientation} of $\tilde{\alpha}$ and $\tilde{\beta}$. However, it is possible for the curves to have different induced orientations. The induced orientation may vary along the curves, but the gluing orientation remains the same, ensuring coherence in the identification process.

\subsubsection{Finite genus and not impasse} After establishing the formalism, we proceed to prove that $T$ has no impasse and none of the obstructions for finite genus. We begin by assuming that $T$ satisfies the topological conditions for have fine genus and not impasse in the first realizer $\cR_1$, which implies the existence of certain ribbons connecting the boundaries $\alpha$ and $\beta$ of vertical sub-rectangles in the realization. This implies that $\tilde{\alpha}$ and $\tilde{\beta}$ are identified in the Markov partition, and this identification must be consistent with the induced, gluing, and dynamic orientations previously defined. By carefully considering these orientations, we will be able to derive certain contradiction, thereby establishing that $T$ has finite genus and not impasse.

In the next lemma we are going to prove that if there is a stripe that joints two stable intervals inside a rectangle $R_k$ of the realization implies that such rectangle contain a spine in one of its horizontal boundaries, like is indicated in Figure \ref{Fig: Colapse point}. 

\begin{figure}[h]
	\centering
	\includegraphics[width=0.4\textwidth]{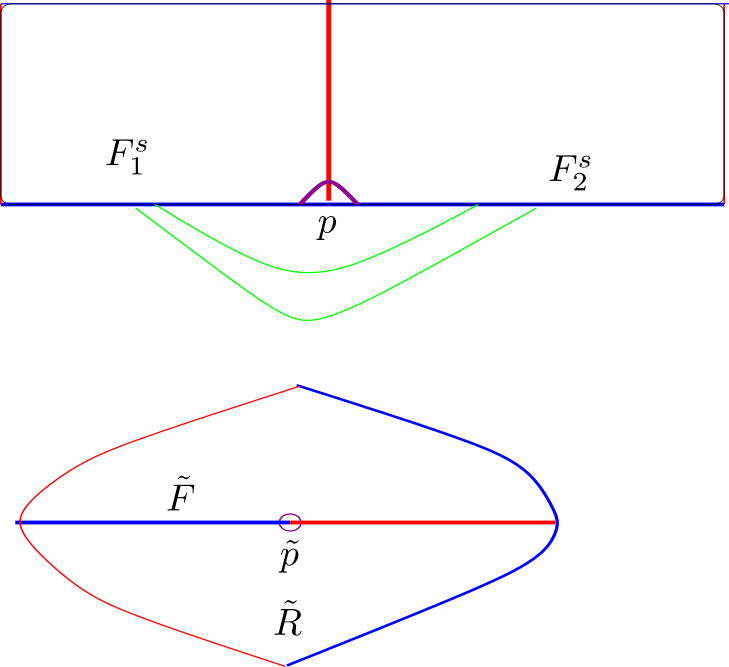}
	\caption{ Two stable separatrices collapse in a Spine }
	\label{Fig: Colapse point}
\end{figure}

\begin{lemm}\label{Lemm: gamma periodic points}
Consider a rectangle $R_k$, with the horizontal orientation of $R_k$ suppose there exists a ribbon joining the boundaries $\alpha=[a_1,a_2]^s\subset V^k_{l_1}$ and $\beta:=[b_1,b_2]^s \subset V^k_{l_2}$, assume that $l_1<l_2$ and  let $\gamma=[a_2,b_1]^s$ the curve between $\alpha$ and $\beta$. Let  $\tilde{\alpha}:=[a_1',a_2']^s$, $\tilde{\beta}:=[b_1',b_2']^s$ and $\tilde{\gamma}:=[a_1',b_2']$ be the previously defined curves with the induced orientation in $\tilde{R_k}$. In this circumstances:
\begin{itemize}
\item[i)] The curve $\tilde{\gamma}$ is a closed interval with one end point equal to a Spine type periodic point of $f$.
\item[ii)] The curve $\gamma$ contains a periodic point.
\end{itemize}
\end{lemm}

\begin{proof}

The curve $\tilde{\gamma}$ is a closed manifold of dimension one, it is either an interval or a circle. Since $\tilde{\gamma}$ is contained in the stable leaf $I$ of a periodic point $\tilde{ O}$ of the pseudo-Anosov homeomorphism $f$, it couldn't be a circle therefore is closed interval $I$. We claim that such interval must contains an end  point of a stable separatrix  of a periodic point  $\tilde{O}$ of $f$, and then such end point is $\tilde{ O}$.

 Look at figure \ref{Fig: tildeRk} to follow our next arguments, the curve $\tilde{\gamma}$ is the union of all the lower  (or upper) boundaries of $\tilde{ V^k_l}$ with $l_1<l<l_2$ an it is contained in inferior boundary of $\tilde{R_k}$. It could be a single point if $l_2=l_1$ or the not trivial interval, $\tilde{\gamma}:=[a_1',b_2']$, lets  to consider the last situation, let $\tilde{o}$ the extreme point of $I$ that is not in $\tilde{\alpha}$ nor $\tilde{\beta}$.
 
 \begin{figure}[h]
 	\centering
 	\includegraphics[width=0.4\textwidth]{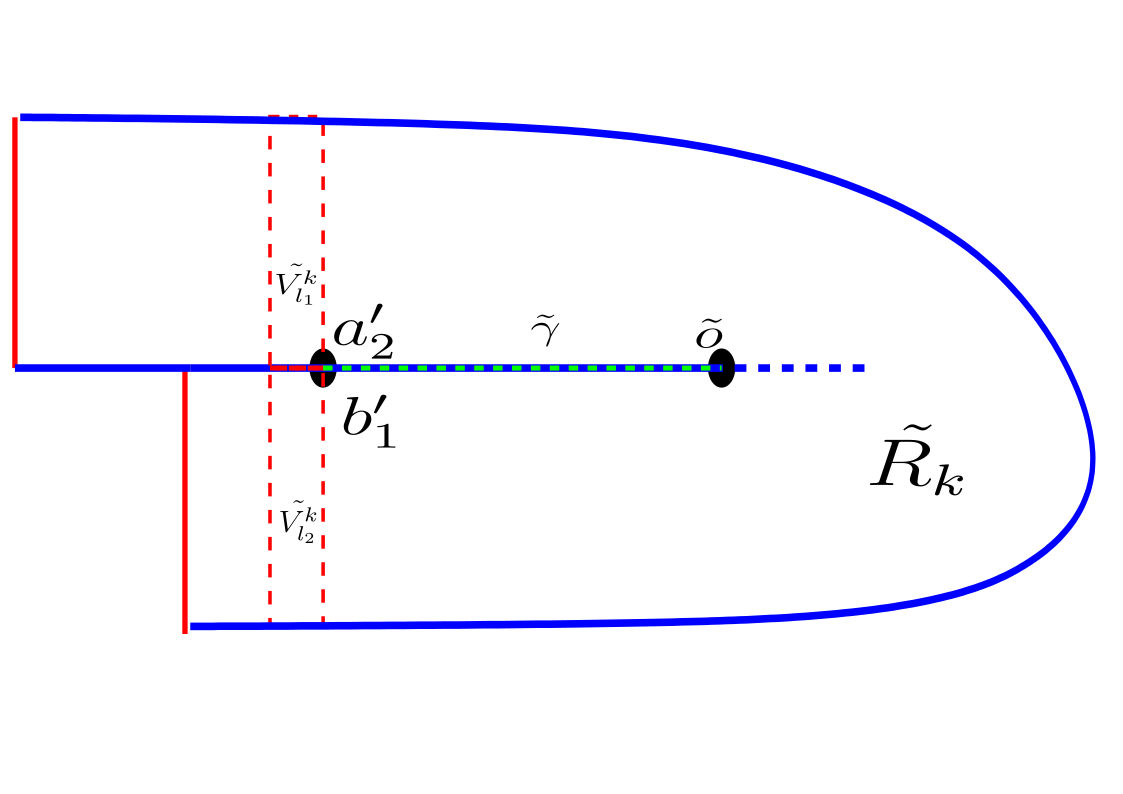}
 	\caption{Rectangle $\tilde{R_k}$ }
 	\label{Fig: tildeRk}
 \end{figure}

Let $\Psi: H := [0,1]\times[0,1] \rightarrow \tilde{R_k}$ the parametrization of the rectangle $\tilde{R_k}$. Following the induced horizontal orientation of $\tilde{R_k}$ the point $\tilde{o}$ lies in the interior of the lower boundary of $\tilde{R_k}$ and we can assume that $\Psi(t_0,0) = \tilde{o}$ for some $t_0 \in (0,1)$. Since $\tilde{R_k}$ is a rectangle and there exist two closed intervals $A_1$ and $A_2$ in $H$ such that by the action of $\Phi$ these intervals are identified, i.e. $\Phi(A_1) =\tilde{ \alpha}= \tilde{\beta}=\Psi(A_2)$, then there exists closed  semicircle $U'\Subset \overset{o}{H}\cup \partial^s H$ with center at $(t_0,0)$ that contains its  that $\Psi(U')=U'\Subset \overset{o}{R_k}$ (see Figure) . Consequently, $\tilde{o}$ has an open neighborhood $\overset{o}{U}$ that is contained in the interior of $\tilde{R_k}$, that have center in $\tilde{o}$. Moreover, it is clear that the diameter of $U'$ was collapsed by the action of $\Phi$ in a closed interval with a extreme in $\tilde{o}$ and therefore $\tilde{o}$ has only one separatrice. This allows us to deduce that $\tilde{o}$ is a Spine, and so it is equal to the periodic point $\tilde{O}$.

 \begin{figure}[h]
	\centering
	\includegraphics[width=0.5\textwidth]{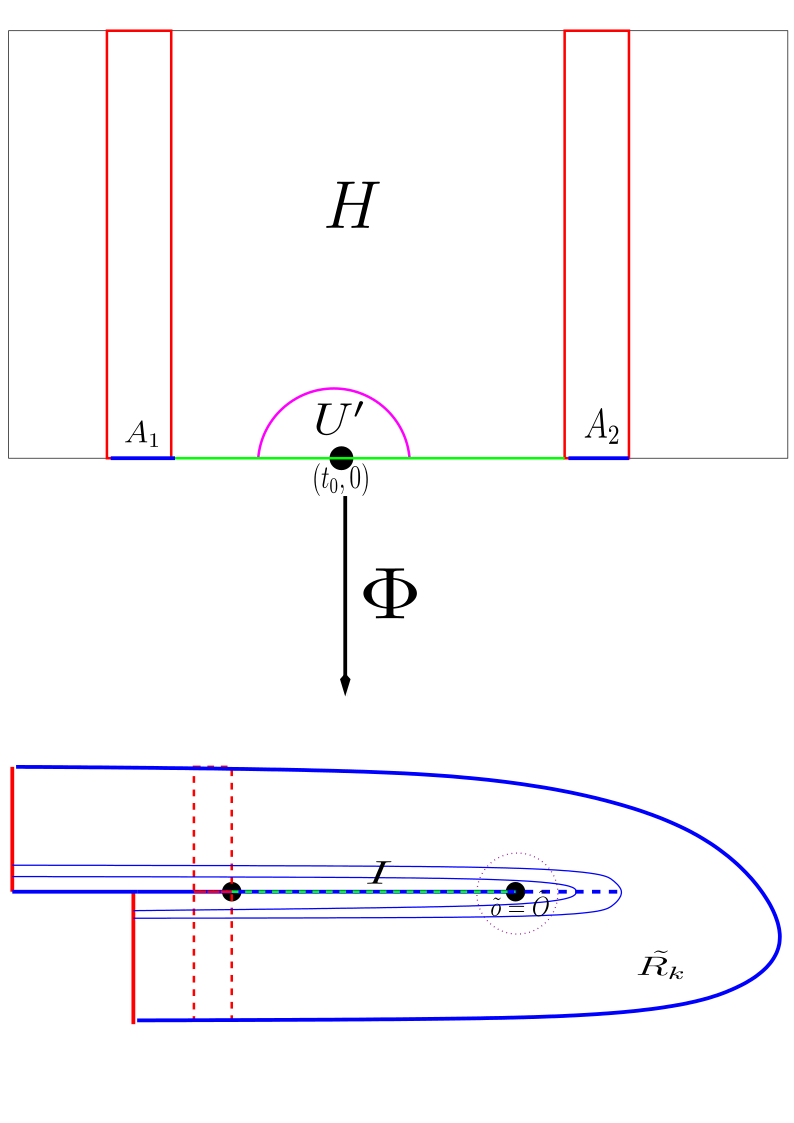}
	\caption{The projection by $\Phi$ of the rectangle $H$}
	\label{Fig: Rectangle H}
\end{figure}

The situation when $I$ is a point is similarly proved.

\end{proof}

\begin{coro}\label{Coro: T pA clas no com impasse}
A geometric type $T$ in the pseudo-Anosov class does not have combinatorial impasses.
\end{coro}

\begin{proof}
First, observe that the geometric type $T$ does not have the impasse condition. If there were an impasse, it would be in a situation as described in the previous lemma, where $\tilde{\alpha}$ and $\tilde{\beta}$ are the horizontal boundaries that are identified. The curve $\tilde{\gamma}$ would be reduced to a periodic point, and $\gamma$ would have a periodic point. However, $\gamma$ is the $s$-arc given by the impasse, so it cannot have periodic points. This contradiction implies that $T$ does not have the impasse condition.

Now, if $T$ represents (via a Markov partition) the pseudo-Anosov homeomorphism $f$, then $f^n$ has a Markov partition of geometric type $T^n$. Since $f^n$ is pseudo-Anosov, $T^n$ is also in the pseudo-Anosov class and does not have the impasse condition. This implies that $T$ does not have an impasse.
\end{proof}

\begin{lemm}\label{Lemm: T pA class then no condition 1}
A geometric type $T$ in the pseudo-Anosov class does not satisfy the combinatorial condition of type $(1)$.
\end{lemm}

\begin{proof}

Suppose $T$ satisfies the combinatorial condition of type $1$. This implies that the $1$-realization of $T$ has the topological obstruction of type $1$, according to the Lemma \ref{Lemm: Equiv com and top type 1}. Therefore, there exist curves $\alpha$ and $\alpha'$ in the boundary of a rectangle $R_k$ in the realization that are joined by a ribbon $r$. Let $\gamma$ the curve that is comprised between $\alpha$ and $\alpha'$ with respect the vertical orientation of $R_k$.  Then  there exists another ribbon $r'$ with a horizontal boundary $\beta$ that is contained within the interior of $\gamma$, while the other boundary of $r'$ is located outside of $\gamma$.

The curve  $\tilde{\gamma}$ contains  $\tilde{\beta}$  in particular $\tilde{\gamma}$ is a non-trivial interval and due to the identification induced by the ribbon $r'$ it contains $\tilde{\beta'}$. However, like it was established in Lemma \ref{Lemm: gamma periodic points}, the stable boundary of $R_k$ is collapsed into a Spine, the curve $\gamma$ have a periodic point $p$ in their interior and moreover bout separatrices of $p$ are identified. 

Is not difficult to see that the  gluing orientation and the dynamical orientation of $\gamma$ must be the same, for that reason look at Figure \ref{Fig: Type 1 obstruction} where:

\begin{itemize}
\item The horizontal h orientation of the stable boundary of $R_k$ is indicated with a blue solid arrow.
\item The induced orientation in $\tilde{\gamma}$ is the indicated with two dotted blue lines.
\item The gluing orientation in $\tilde{\gamma}$ is the indicated with two dotted black lines,and clearly they point towards the spine $\tilde{ O}$.
\end{itemize}

  Hence, any interval $I\subset \gamma$ that is  one separatrice of $p$ must be identify with another interval $I'$ in the other stable separatrice of $p$, and they are the only intervals that are identify. In particular $\beta$ must me identify with another interval inside $\gamma$, like this interval is $\beta'$ we deduce that $\beta'$ is inside $\gamma$ and this is a contradiction.

\begin{figure}[h]
	\centering
	\includegraphics[width=0.7\textwidth]{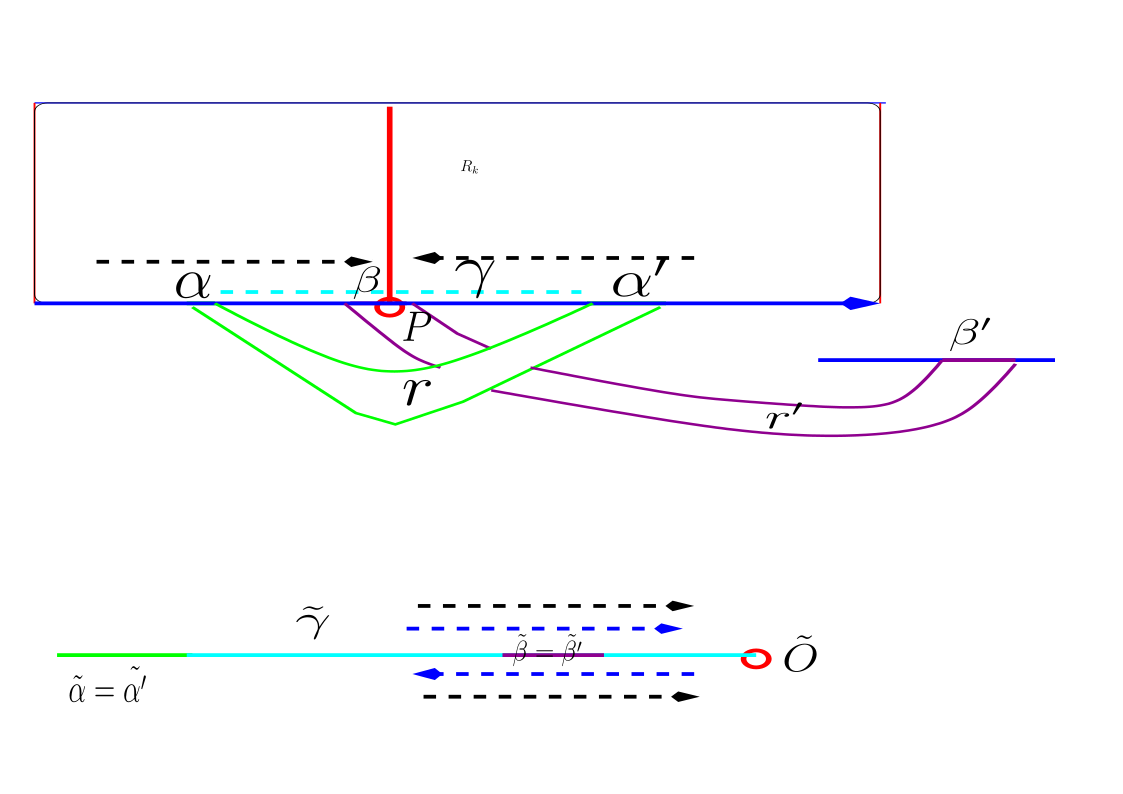}
	\caption{ Type $1$ obstruction}
	\label{Fig: Type 1 obstruction}
\end{figure}
\end{proof}

If $T$ is a geometric type in the pseudo-Anosov class and $f$ is a generalized pseudo-Anosov homeomorphism that realizes $T$, it follows that $T^n$ is the geometric type of a Markov partition of $f^n$. Since $f^n$ is also pseudo-Anosov, we can conclude that $T^n$ does not have the type $(1)$ property or an impasse. In particular, for any positive integer $n$, $T^{6n}$ does not have the type $(1)$ property, and $T^{2n+1}$ does not have an impasse. This discussion leads to the following corollary.

\begin{coro}\label{Lemm: T pA class then no obstruction 1}
A geometric type $T$ in the pseudo-Anosov class don't have the type-$(1)$ obstruction or impasse.
\end{coro}

\begin{figure}[h]
	\centering
	\includegraphics[width=0.7\textwidth]{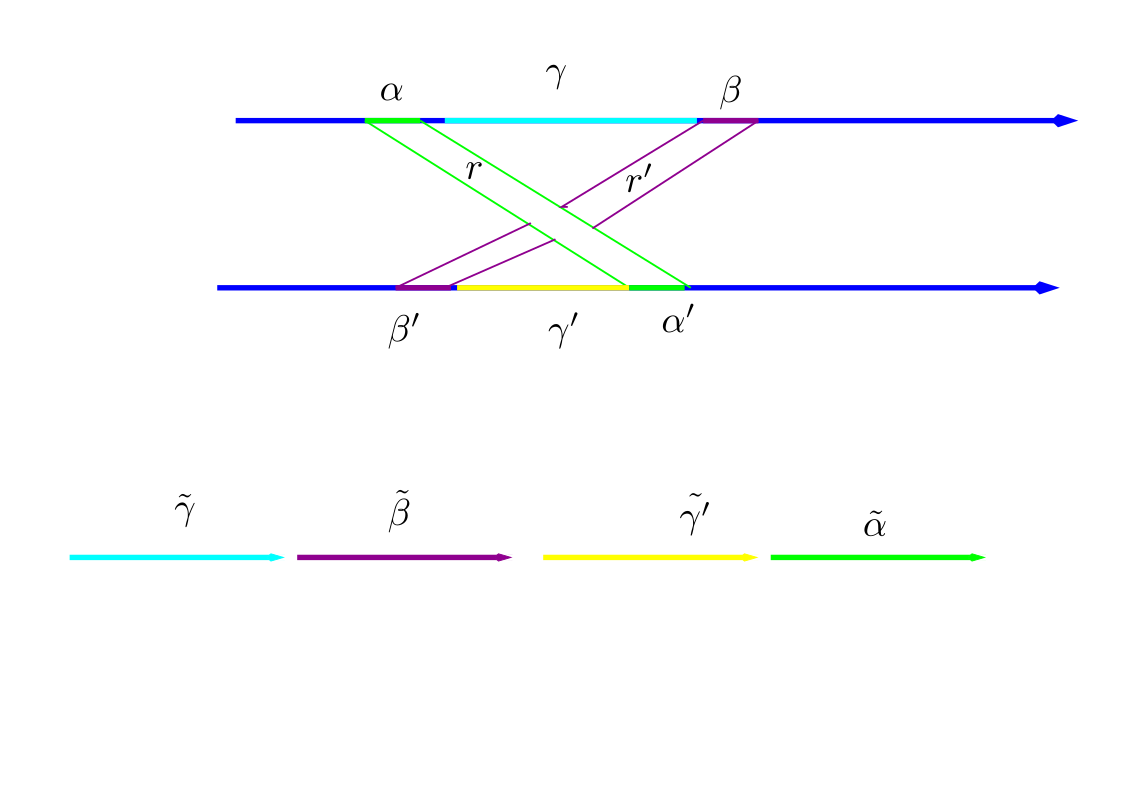}
	\caption{ Type $2$ obstruction}
	\label{Fig: Type 2 obstruction}
\end{figure}

\begin{lemm}\label{Lemm: T pA class then no condition 2}
A geometric type $T$ in the pseudo-Anosov class does not have the combinatorial condition of type-$(2)$. 
\end{lemm}

\begin{proof}

Assume this is not true. In view of Lemma \ref{Lemm: Equiv com and top type 2}, $T$ has the topological condition of type-$(2)$ in its $1$-realization. This implies the existence of a ribbon $r$ joining two curves $\alpha$ and $\alpha'$ that lie on different sides of the stable boundary of the Markov partition. Similarly, there is another ribbon $r'$ joining curves $\beta$ and $\beta'$, where $\beta$ is on the same stable side as $\alpha$ and $\beta'$ is on the same side as $\alpha'$. These curves correspond to the boundaries of vertical sub-rectangles in the realization. The type-$(2)$ condition imposes certain conditions about the orientation on these curves, namely:

\begin{itemize}
\item  The gluing orientation of $\alpha$ is equal to its induced orientation, and the gluing orientation of $\alpha'$ is equal to its respective induced orientation.

\item  The gluing orientation of $\beta$ coincides with its induced orientation if and only if the r gluing and the induced orientations of $\beta'$ coincides.
\end{itemize}

As usual, $\gamma$ is the curve between $\alpha$ and $\beta$ with respect to the induced orientation: $\alpha \leq \gamma \leq \beta$. Similarly, $\gamma'$ is the curve between $\beta'$ and $\alpha'$ with respect to the induced orientation: $\beta' \leq \gamma' \leq \alpha'$. The same orientation properties hold for the curves $\tilde{\alpha}$, $\tilde{\alpha'}$, $\tilde{\beta}$, $\tilde{\beta'}$, $\tilde{\gamma}$, and $\tilde{\gamma'}$ in the Markov partition of $f$.

Furthermore, the dynamic orientation of $\tilde{\alpha}$ matches its induced orientation if and only if the dynamic orientation of $\tilde{\alpha'}$ matches the induced orientation. We adopt the convention that the induced orientation, gluing orientation, and dynamic orientation of $\tilde{\alpha}$ and $\tilde{\alpha'}$ are the same and fixed from now on. Even more, we assume that the dynamical orientation and gluing orientation of $\tilde{\beta}$ and $\tilde{\beta'}$ are the same, as they both point towards the same periodic point. We shall analyze two disjoint situations: either the dynamical orientation of $\tilde{\beta}$ is the same as that of $\tilde{\alpha}$, or it is different.

If the dynamical orientation of $\tilde{\alpha}$ and $\tilde{\beta}$ are the same, we define the curve $\tilde{L}$ as follows:
$$
\tilde{L}:=\tilde{\alpha'}\cdot \tilde{\gamma'}\cdot \tilde{\beta}\cdot \tilde{\gamma}.
$$

This curve $\tilde{L}$ is a simple closed curve because its entire interior is oriented with respect to the dynamical orientation in a stable leaf. However, this is not possible because the stable foliation of $f$ does not have closed leaves.

If the dynamical orientations of $\tilde{\alpha}$ and $\tilde{\beta}$ are different, this implies that $\tilde{\gamma}$ contains a periodic point $\tilde{O}$, and we have the following inequalities with respect to the induced orientation: 
$$
\tilde{\alpha} < \tilde{O} < \tilde{\beta}.
$$ 
Like the induced and dynamical orientations of $\tilde{\alpha}$ and $\tilde{\alpha'}$ coincide: $\tilde{\alpha'}<\tilde{O}$. Similarly, the dynamical orientation of $\tilde{\beta}$ matches that of $\tilde{\beta'}$ and they are the inverse of the induced orientation, so $\tilde{O}<\tilde{\beta'}$. This implies that, with respect to the induced orientation:
$$
\tilde{ O}<\tilde{\beta'}< \tilde{\alpha'}<\tilde{O}
$$
so the only way for this to occur is to have a closed leaf, which contradicts the fact that the stable foliation of $f$ does not have closed leaves.
This ends our proof by contradiction.
\end{proof}

\begin{coro}\label{Lemm: T pA class then no obstruction 2}
A geometric type $T$ in the pseudo-Anosov class does not have the type-$(2)$ obstruction.
\end{coro}

\begin{proof}

It is sufficient to prove that $T^m$ does not have the type-$(2)$ condition. However, $T^m$ is a geometric type in the pseudo-Anosov class associated with $f^m$. By Lemma \ref{Lemm: T pA class then no condition 2}, we know that $T^m$ does not have the type-$(2)$ condition. Therefore, it follows that $T$ does not have the type-$(2)$ obstruction.
\end{proof}

%%%%%%%%% Type 3 condition%%%%%

\begin{lemm}\label{Lemm: T pA class then no condition 3}
A geometric type $T$ in the pseudo-Anosov class does not have the combinatorial condition of type $(3)$.
\end{lemm}

\begin{figure}[h]
	\centering
	\includegraphics[width=0.7\textwidth]{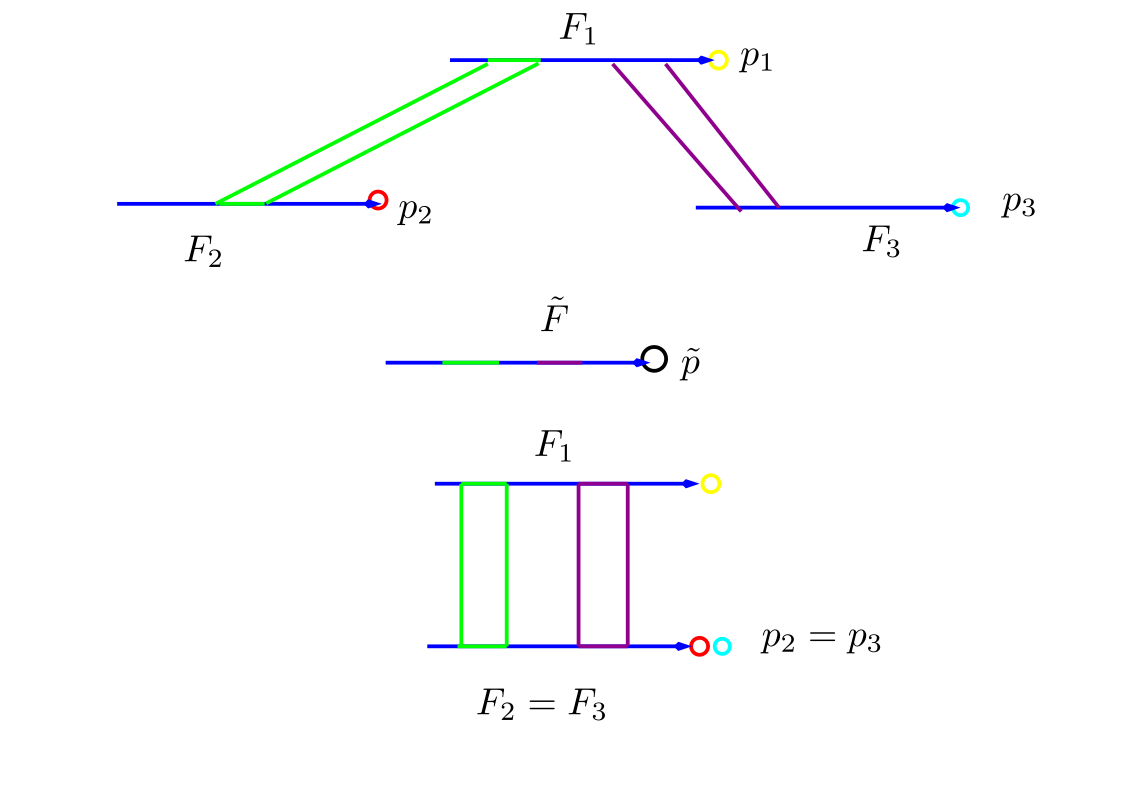}
	\caption{ Type $3$ obstruction}
	\label{Fig: Type 3 obstruction}
\end{figure}

\begin{proof}

Suppose $T$ has the combinatorial condition of type $(3)$. Then, according to Lemma \ref{Lemm: Equiv com and top type 3}, the first realization of $T$ has the topological obstruction of type $(3)$. Let $\alpha$ and $\alpha'$ be the curves identified by the ribbon $r$, and let $r'$ be the other ribbon given by the combinatorial condition that connects $\beta$ and $\beta'$. Suppose that $\alpha$ and $\beta$ are contained in the separatrix  $F^s_1(p_0)$, while $\beta'$ is contained in $F^s_1(p_1)$ and $\alpha'$ is contained in $F^s_1(p_2)$, where $p_0$, $p_1$, and $p_2$ are periodic points. By hypothesis  $F^s_1(p_0)$, $F^s_1(p_1)$, and $F^s_1(p_2)$ are all distinct and the curves are all embrionary . Now, let us analyze what happens in the Markov partition.

By the identifications, the curves $\tilde{\alpha}$, $\tilde{\alpha'}$, $\tilde{\beta}$, and $\tilde{\beta'}$ are all contained in the same stable separatrix  $\tilde{F}$ associated with a periodic point $\tilde{O}$ and they don't contain this point (remember they comes from embrionary separatrices).

Now, the separatrix  $\tilde{F}$ has only two sides. On one side, we have the rectangle $\tilde{R_k}$ which contains $\tilde{\alpha}$ and $\tilde{\beta}$ in their periodic boundary. On the other side, we have the rectangle $\tilde{R_a}$ which contains $\tilde{\alpha'}$, and $\tilde{R_b}$ which contains $\tilde{\beta'}$.

If $\tilde{R_a}$ is not the same as $\tilde{R_b}$, then only one of these rectangles can have a periodic stable boundary. However, our hypothesis states that both $\alpha'$ and $\beta'$ are in embrionary separatrices, which means that both $\tilde{R_a}$ and $\tilde{R_b}$ should have periodic boundaries containing these curves. This leads to a contradiction.

Therefore, the rectangle $R_a$ of the Markov partition contains, in one of its periodic boundaries, the curves $\alpha'$ and $\beta'$. However, by the coherence of the dynamical orientation with the induced orientations, $\alpha'$ and $\beta'$ should be in the same stable separatrix. This contradicts our assumption.

\end{proof}

\begin{coro}\label{Lemm: T pA class then no obstruction 3}
A geometric type $T$ in the pseudo-Anosov class doesn't have the type  $(3)$ combinatorial obstruction.
\end{coro}

\begin{proof}

Similarly to the other obstruction, for all $m\in \mathbb{N}$, $T^m$ is the geometric type of the pseudo-Anosov homeomorphism $f^m$. This implies that $T^m$ doesn't have the type $(3)$ combinatorial condition.
\end{proof}

%% file: Realization/SubCombinatorialform.tex
\subsection{Finite genus and no-impasse implies basic piece without impasse.}\label{Sec: finite genus no impas implies basic piece}

In this part we are going to prove the next implication of Proposition \ref{Prop: pseudo-Anosov iff basic piece non-impace}. It correspond tho the following proposition.

\begin{prop}\label{Prop: Charaterization pseudo Anosov class}
A geometric type $T$  which satisfy the following properties:
	\begin{enumerate}
	\item The incidence matrix in transitive.
	\item The genus of $T$ is finite.
	\item $T$ don't have impasse.
	\end{enumerate}
is realized as a \emph{mixing} basic piece of a surface Smale diffeomorphism without topological impasse
\end{prop}

\begin{proof}
 The Theorem \ref{Theo: finite genus iff realizable}  due to Bonatti-Langevin-Beguin implies that $T$ is realized as a saturated basic set $K$ of a surface Smale diffeomorphism . Like the incidence matrix is mixing, we deduce $K$ is mixing and then it is basic piece (Look at \cite[Proposition 18.7.7]{hasselblatt2002handbook}). The Theorem \ref{Theo: Geometric and combinatoric are equivalent} implies that $K$ does not have a topological impasse.
\end{proof}

%% file: Equivalence/Equivalence.tex
\section{Equivalent representations.}

\subsection{The formal DA of a geometric type.}\label{Sec: Formal DA}

Let $T$ be a geometric type in the pseudo-Anosov class. Our previous results showed the existence of a closed surface $S$ and a Smale diffeomorphism $\phi: S \rightarrow S$, that have a  mixing saddle-type basic piece $K(T)$ which has a Markov partition of geometric type $T$. Therefore, we can define $\Delta(T) := \Delta(K(T))\subset S$ as the domain of $K(T)$. Moreover, the diffeomorphism $\phi$ uniquely determines (up to conjugation) a diffeomorphism in the domain $\phi_T:\Delta(T) \rightarrow \Delta(T)$. This justify the following definition.

\begin{defi}\label{Defi: Formal DA}
		Let $T$ be a geometric type of the pseudo-Anosov class. Let $S$ be a closed surface and $\phi: S\rightarrow S$ be a Smale surface diffeomorphism having a mixing saddle-type basic piece $K(T)$ with a geometric Markov partition of geometric type $T$, let $\Delta(T)\subset S$ be the domain of the basic piece $K(T)$.  We define the \emph{formal derived from Anosov} of $T$ to be the triplet
	$$
	\textbf{DA}(T):=(\Delta(T),K(T),\phi_T).
	$$
where $\phi_T: \Delta(T) \to \Delta(T)$ is the restriction of the diffeomorphism $\phi$ restricted to the domain of $K(T)$.
	
\end{defi}

The following proposition is consequence of the Bonatti-Langevin results that were presented in  \ref{Sub-sec: Domain basic piece}.

\begin{prop}\label{Defi: Unique DA}
	Let $T$ be a geometric type in the pseudo-Anosov class. Up to a topological conjugacy in their respective domains, that is the identity in the non-boundary periodic points of $K(T)$, there exists a unique \textbf{DA}$(T)$.
\end{prop}

Take $T\in \cG\cT(\textbf{p-A})$ be a geometric type in the pseudo-Anosov class and a \textbf{p-A} homeomorphism $f$ with a geometric Markov  partition $\cR$ of geometric type $T$. The formal derived from Anosov of $T$ coincides with the classical idea of open  the stable and unstable manifolds of the periodic point of $f$ in the boundary of $\cR$, after this process, a saddle-type basic piece without singularities is  obtained and each periodic point on the boundary of the Markov partition $\cR$ becomes a corner point of the basic piece.  However, in  \textbf{DA}$(T)$ the process is different as the periodic points on $\partial^s\cR$ become $s$-boundary points of the basic piece $K(T)$ while the periodic points on  $\partial^u\cR$ become $u$-boundary points of $K(T)$.  When comparing two saddle-type basic pieces obtained as the formal \textbf{DA} of two different geometric types in the pseudo-Anosov class, the first obstruction for such basic pieces to be topologically conjugate is to have the same number of periodic $s$ and $u$-boundary points. In the next sub-section, we introduce the joint refinement of two geometric types that permits us to overcome this first difficulty

\subsection{The joint refinement of two geometric types.} 

Let $T_1$ and $T_2$ be geometric types in the pseudo-Anosov class, and suppose $(f,\cR)$ and $(g,\cR)$ are two geometric Markov partitions of pseudo-Anosov homeomorphisms $f$ and $g$ such that $T_1=T(f,\cR_f)$ and $T_2=T(g,\cR_g)$. Let 
$$
N(\cR_f)=\max{Per(p): p \in \partial \cR_f \text{ is periodic for } f } 
$$
 and 
$$
N(\cR_g)=\max\{Per(p): p \in \partial \cR_g \text{ is periodic for } g \} 
$$

Finally, let $N(f,g)=\max\{N(\cR_f),N(\cR_g)\}$ and take $Per(f,N(f,g))$ as the set of periodic points of $f$ with period less than or equal to $N(f,g)$ and similarly $Per(g,N(f,g))$.   Along \cite[Chapter 5.4]{IntiThesis} (to papers in \cite{IntiII}) we have carefully define the  joint refinement of $\cR_f$ with respect to $\cR_g$ and we have proved that the family of rectangles in \ref{Defi: Joint refinament} is a geometric  Markov partition for $f$, it symmetrization detailed and we compute its geometric type $T_{(T_g,T_f)}$, finally we has proved it have the  corner property \ref{Defi: Corner type partition }. 

\begin{defi}\label{Defi: Joint refinament}
Consider $(\cR_f)_{\cR_g}$ the family of rectangles obtained by cutting each rectangle $R_i\in \cR_f$ along the stable and unstable segments of all periodic points in $Per(f,N(f,g))$ contained in $R_i$. These rectangles are given a lexicographic order and the orientation induced by the $R_i$ in which they are contained, and we call $(\cR_f)_{\cR_g}$ the joint refinement of $\cR_f$ with $\cR_g$.Moreover, the joint refinement  $(f, (\cR_f)_{\cR_g}$ is a  geometric partition of $f$ whose geometric type is denote  $T_{(T_g,T_f)}$.
\end{defi}

 The number $N(f,g)$ is equal to $N(g,f)$, and the corner points of $\cR_{(R_f,R_g)}$ coincide with the set of periodic points of $f$ with periods less than or equal to $N(f,g)$. By construction, the geometric Markov partition $\cR_{(R_f,R_g)}$ has the corner property (\cite[Corollary 5.14.]{IntiThesis}) and all its boundary points are corners of any rectangle that contains them.

\begin{defi}\label{Defi: Corner type partition }
	A geometric Markov partition $\cR$ of $f$ have the \emph{corner property} if every rectangle $R\in \cR$ that contains a boundary point $p\in P_b(f,\cR)$ have $p$ as a corner point. A geometric type $T$ have the \emph{corner property} if every periodic boundary code is a  corner peridic code, i.e $\underline{B(T)}=\underline{S(T)}\cap \underline{U(T)}$.
\end{defi}

The following corollary correspond to \cite[Lemma 5.50.]{IntiII}

\begin{lemm}\label{Lemm: Corner prop equivalence}
	The geometric Markov partition $\cR$ has the corner property if and only if its geometric type $T$ has the corner property.
\end{lemm}

 Here we have choose just introduce the characterization of the joint refinement that we require in order to complete our proofs but it explicit construction and computations is carried along \cite[Chapter 5]{IntiThesis} since we must use this construction, we pose it as a result.

\begin{coro}\label{Coro. algortim compute joint}
There exist a constructive algorithm to compute the joint refinement of two geometric types in the pseudo-Anosov class.
\end{coro}

\subsection{The lifting of a geometric  Markov partition of $f$ to the formal \textbf{DA}$(T)$.}

The important property of the joint refinement it that every periodic boundary point is a corner point of any rectangle that contain such point. We must use this property to prove the following Proposition. 

\textbf{Notation:} Once we have fixed $\cR_f$ and $\cR_g$ to avoid an overwhelming notation, we make the following conventions:
\begin{eqnarray}
\cR_{f,g}:=\cR_{(\cR_f,\cR_g)} \text{ and } \cR_{g,f}:=\cR_{(\cR_g,\cR_f)}, \\
T_{f,g}:=T_{(T_f,T_g)} \text{ and } T_{g,f}:=T_{(T_g,T_f)}.
\end{eqnarray}

The formal derived from pseudo-Anosov of such geometric types are:
\begin{eqnarray}
\textbf{DA}(T_{f,g}):=\{\Delta(T_{f,g}), K(T_{f,g}),\Phi_{T_{f,g}}) \} \text{ and } \\
\textbf{DA}( T_{g,f}):=\{\Delta( T_{g,f}), K( T_{g,f}),\Phi_{ T_{g,f}})\} 
\end{eqnarray}

\begin{prop}\label{Prop: preimage is Markov Tfg type}
	Let $f:S_f \rightarrow S_f$ and $g:S_g \rightarrow S_g$ be two generalized pseudo-Anosov homeomorphisms with geometric Markov partitions, $\cR_f$ and $\cR_g$, whose geometric types, $T_{f}$ and $T_g$, have binary incidence matrix. Suppose that $f$ and $g$  are  topologically conjugate through an orientation preserving homeomorphism $h:S_f\rightarrow S_g$ and let:
	$$
	\pi:\Delta(T_{g,f})\rightarrow S_g,
	$$
	be  the projection given by Proposition \ref{Prop: type of basic piece is type of pseudo-Anosov}. In this situation:
	\begin{itemize}
		\item[(1)] The induced geometric partition, $h(\cR_{f,g})$ is a geometric Markov partition of $g$ whose geometric type $ T_{f,g}$ have the corner property.
		
		\item[2)] A  point $p\in S_f$ is a periodic boundary point of the Markov partition $(f,\cR_{f,g})$ if and only if $h(p)$ is a periodic boundary point of $(g,\cR_{g,f})$. In particular, the corner periodic points of $h(\cR_{f,g})$ are the corner periodic points of $\cR_{\cR_{g,f}}$.
		
		\item[(3)] There exit a unique Markov partition $\underline{h}(\cR_{f,g})=\{\underline{h}(R_i)\}_{i=1}^n$ of $K(T_{g,f})\subset \Delta(T_{(g,f)}))$ such that:
		$$
		\pi\left(\overset{o}{\underline{h}}(R_i)\right)=\overset{o}{(h(R_i))}.
		$$
		\item[(4)] Keep the previous label of the rectangles in  $\underline{h}(\cR_{f,g})$ and give to the rectangle $\underline{h}(\cR_{f,g})$  the unique vertical and horizontal orientation  such that $\pi$ is increasing along both the vertical and horizontal leaves within $\underline{h}(R_i)$. With this geometrization, the geometric type of $\underline{h}(\cR)$ is $T_{f,g}$.
	\end{itemize}
\end{prop}

	\begin{defi}\label{Defi: induced geometric Markov partition}
	Let $f:S\rightarrow S$ be a generalized pseudo-Anosov homeomorphism, and let $h:S\rightarrow S$ be an orientation-preserving homeomorphism, and $g:=h\circ f \circ h^{-1}$. If $\cR$ is a geometric Markov partition of $f$, the \emph{induced geometric Markov partition} of $g$ by $h$ is the Markov partition $h(\mathcal{R})=\{h(R_i)\}_{i=1}^n$. For each $h(R_i)$, we choose the unique orientation in the vertical and horizontal foliations such that $h$ preserves both orientations at the same time.
\end{defi}

\begin{proof}

	\textbf{Item} $(1)$. 	In Definition  \ref{Defi: induced geometric Markov partition}, we introduced the induced geometric Markov partition $h(\cR_{f,g})$ and theorem \ref{Theo: conjugated iff  markov partition of same type} establishes that $h(\cR_{f,g})$ is a geometric Markov partition with the same geometric type as $\cR_{f,g}$, specifically $T_{f,g}$.  By lemma \ref{Lemm: Corner prop equivalence} since $\cR_{f,g}$ has the corner property if and only if  $T_{f,g}$ has the corner property.

\textbf{Item} $(2)$.	By definition of the joint refinement, a periodic point $p$ of $f$ is a corner point  of $\cR_{f,g}$ if and only if it is a period is less than or equal to $N(f,g)$. But $h(p)$ has the same period that $p$, therefore $h(p)$ is a periodic point of $g$ with period less than or equal to $N(g,f)=N(f,g)$, then $h(p)$ is corner point of $\cR_{g,f}$, by definition it is a boundary point. This proves the second item of the proposition.

	\textbf{Item} $(3)$. We are going to prove the second statement through a series of lemmas. 	The reader can refer to Figure \ref{Fig: rectangle split} for a visual representation of our notation.
	
	\begin{figure}[h]
		\centering
		\includegraphics[width=0.7\textwidth]{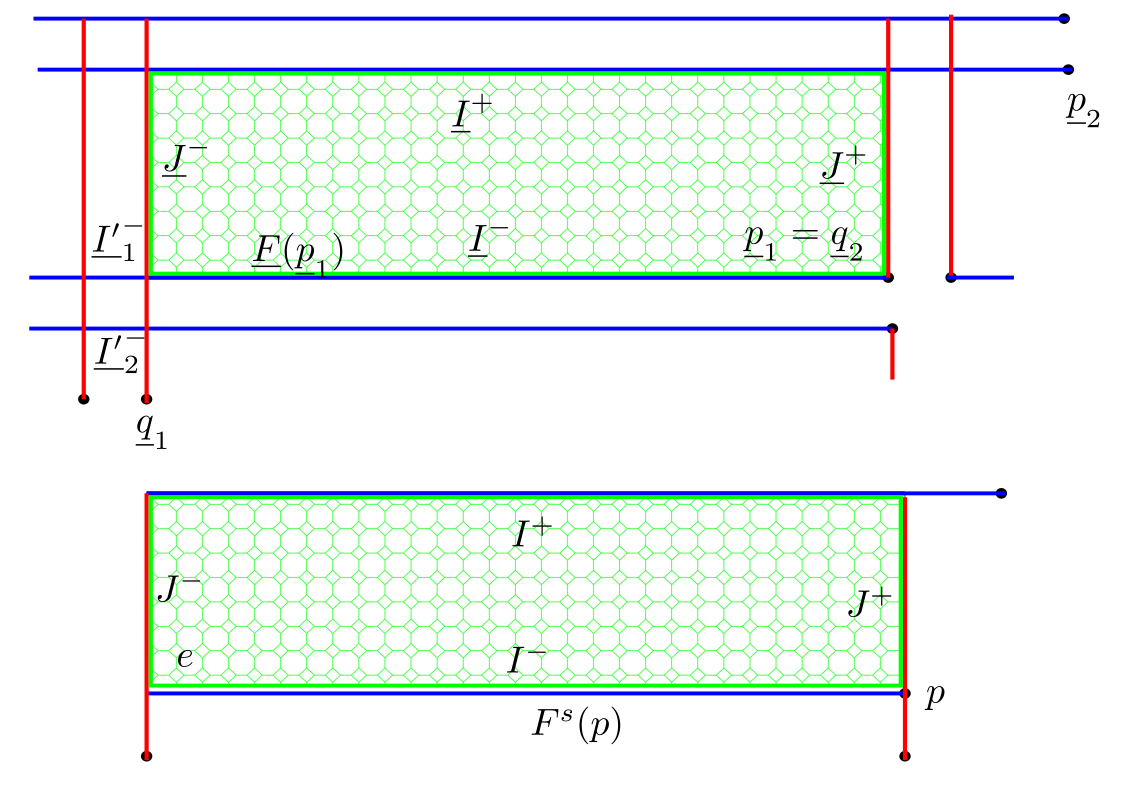}
		\caption{The lift of a rectangle $R$. }
		\label{Fig: rectangle split}
	\end{figure}

	\begin{lemm}\label{Lemm: preimage pi rectangle}
		Let $R$ be a rectangle in the Markov partition $h(\cR_{f,g})$. Then $\pi^{-1}(R)$ contains a unique rectangle $\underline{R}\subset \Delta(T_{g,f})$ with the following properties:
		\begin{itemize}
			\item The interior of $\underline{R}$ is mapped under $\pi$ to the interior of $R$, i.e., $\pi(\overset{o}{\underline{R}})=\overset{o}{R}$, and the image of $\underline{R}$ is the rectangle $R$, i.e., $\pi(\underline{R})=R$.
			
			\item The horizontal  boundary of $\underline{R}$ is contained in a stable separatrice of an $s$-boundary periodic point of $K(T_{g,f})$, and its vertical boundary is in an unstable separatrice of an $u$-boundary point.

			\item The rectangle $\underline{R}$ is isolated from the basic piece $K(T_{g,f})$ from its exterior, i.e., there is an open neighborhood $U$ of $\underline{R}$ such that $K(T_{g,f})\cap U =K(T_{g,f}) \cap \underline{R}$.
		\end{itemize}
	\end{lemm}
	
	\begin{proof}
		
		Denote the left side of $R$ as $J^-$, its right side as $J^+$, its lower boundary as $I^-$, and its upper boundary as $I^+$. All these intervals are contained within the invariant manifold of a corner periodic point of $h(\cR_{f,g})$. Moreover, the boundary points of $h(\cR_{f,g})$ are all corner points and as proven in Item $(2)$, they coincide with the corner points of $\cR_{g,f}$. This implies that neither $J^{\pm}$ nor $I^{\pm}$ contain a periodic point in their interior; therefore, they are contained within a unique separatrice of a corner point. Now, let's analyze the set $\pi^{-1}(I^-)$ to get some conclusions that can be applied to the remaining intervals $J^+, J^-, I^+$ in the boundary of $R$.
		
		Assume that $I^{-}$ is contained in the  stable separatrice $F^s(p)$ of the corner point $p$.  Inside the domain  $\textbf{DA}(T_{g,f})$ there are two $s$-boundary periodic points of $K(T_{g,f})$, $\underline{p}_1$ and $\underline{p}_2$ with different stable separatrices $\underline{F}^s(\underline{p}_1)$ and $\underline{F}^s(\underline{p}_2)$, such that:  $\pi^{-1}(\overset{o}{I^-})$ has two connected components that are are contained in such separatrices:
		$$
		\underline{\overset{o}{I'}}_1\subset F^s(\underline{p}_1) \text{ and } \underline{\overset{o}{I'}}_2\subset F^s(\underline{p}_2).
		$$
		and $\pi(\underline{\overset{o}{I'}}_1)=\pi(\underline{\overset{o}{I'}}_2)=\overset{o}{I^-}$. 
		Such stable intervals are contained in the stable boundary of  $s$-boundary periodic point, therefore only one of these intervals have points that can be approximated by singletons (Item $i)$ in Definition \ref{Defi: equivalen clases sim-r} ) of the form $\pi^{-1}(x_n)\in K(T_{g,f})$, where $x_n\in \overset{o}{R}$ and $x_n$ converge to a point in $I^-$. Assume that $\overset{o}{\underline{I'}_1}$  is such an interval. 
		
		Let's understated the closure of $\overset{o}{\underline{I'}_1}$ inside the stable separatrice $\underline{F}^s(p_1)$.  We need to consider two different of extreme points in $I^-$:
		\begin{itemize}
			\item[i)] The interval $I^{-}$ has a periodic point $p$ as one of its endpoints. In this case, $\pi^{-1}(p)$ is contained in a cycle. The endpoint of $\underline{I'}^-=\overline{\overset{o}{\underline{I'}_1}}$ which projects to $p$ needs to be contained in such a cycle. However, at the same time, such a point can be approximated by points in $K(T_{g,f})\cap \overset{o}{\underline{I'}_1}$ and is, therefore, an element of $K(T_{g,f})$. The unique hyperbolic points in the cycle are periodic points, so the endpoint of $\underline{I'}^-$ that projects to $p$ is the periodic point $\underline{p}_1$, which is approximated by points in  $\underline{F}^s(\underline{p}_1)$.
			
			\item[ii)]   The endpoint $O\in I^-$ is not periodic. In this case, $\pi^{-1}(O)$ is a minimal rectangle. Similar to before, the endpoint of $\underline{I'}^-=\overline{\overset{o}{\underline{I'}_1}}$ which projects to $O$ needs to be contained in such a minimal rectangle and be an element of the basic piece. Therefore, it is the unique corner point $\underline{O}$, in the minimal rectangle  that is saturated by points in $K(T_{g,f})\cap \pi^{-1}(\overset{o}{\underline{I'}_1})$.
		\end{itemize}
		
		In this setting: If $I^-$ have a periodic end point:
		$$
		\underline{I}^{-}:=[\underline{O},\underline{p}_1]^s\subset F^s(\underline{p}_1).
		$$
		If $I^-=[O_1,O_2]^s$ have no periodic end points, we can applied the analysis in Item $i)$ to construct point $\underline{O_1}$ and $\underline{O_2}$ such that:
		
		$$
		\underline{I}^{-}:=[\underline{O_1},\underline{O_1}]^s\subset F^s(\underline{p}_1).
		$$ 
		
		\begin{rema}\label{Rema: boundary points s-u}
			The extreme points of $\underline{I}^-$ are, in any case, $u$-boundary points. Therefore,  $\underline{I}^-$ is contained within an open interval $\underline{I}_1^-\subset \underline{F}^s(\underline{p}_1)$ such that $K(T_{g,f})\cap \underline{I}_1^- = K(T_{g,f})\cap \underline{I}^-$ and $\pi(\underline{I'}_1^-)=I^-$.  Furthermore, $\underline{I}^{-}$  is the smallest interval contained in $\underline{F}^s(\underline{p}_1)$ that projects to $I'$. These two properties determine it uniquely.
		\end{rema}

		The same construction applies to the other boundary components. That is, there are unique periodic corner points $\underline{p}_2$, $\underline{q}_1$, and $\underline{q}_2$ of the basic piece $K(T_{g,f})$ and unique separatrices of such points that contain minimal intervals:
		
		\begin{eqnarray*}
			\underline{I}^{+}\subset \underline{F}^s(\underline{p}_2)\\
			\underline{J}^{-}\subset \underline{F}^s(\underline{q}_1)\\
			\underline{J}^{+}\subset \underline{F}^s(\underline{q}_2)
		\end{eqnarray*}

		These intervals project into the respective boundary components of $R$. The stable separatrice $\underline{F}^s(\underline{p}_2)$ is the only stable separatrice that is saturated by stable leaves of non $s$ boundary periodic points that project to $R$. Clearly, these intervals are unique, as they are minimal with these properties.
		
		We claim that $\vert \underline{I}^-\cap\underline{J}^-\vert =1$. Let $O$ be the corner of $R$ determined by $I^-$ and $J^-$, and let $e$ be the sector of $O$ that is determined by the left-inferior corner of $R$. Then $\pi^{-1}(O)$ could be:
		
		\begin{itemize}
			\item Part of a cycle if $O$ is a periodic point. In this cycle, all the boundary points are corner points and has only one sector saturated by points in $K(T_{g,f})$. Let $\{x_n\}\subset \overset{o}{R}$ be a sequence on non-boundary periodic codes that converges to $0$ in the sector $e$, then $\pi^{-1}(x_n)$ is a singleton and then: $\lim_{n\to \infty} \pi^{-1}(x_n)$ is simultaneously the periodic endpoint of  $\underline{I}^{-}$ and the periodic endpoint of $\underline{J}^{-}$, i.e   $\underline{p}_1=\underline{q}_1$.

			\item  	A minimal rectangle if $O$ is not periodic. In this case, there is only one corner of the minimal rectangle where $\underline{I}_{-}$ and $\underline{J}^{-}$ intersect. Such  is the only one accumulated by singletons in $\pi^{-1}(\overset{o}{R})$ that converge to $\underline{O}$.
		\end{itemize}

		The same property follows for the other vertices of $R$. Furthermore, since the basic piece $K(T_{g,s})$ doesn't contain impasses, we deduce that $\underline{I}^-$ is in a different stable separatrice than $\underline{I}^+$, and then $\underline{I}^-\cap \underline{I}^+=\emptyset$. In conclusion, $\underline{I}^- \cup \underline{I}^{+} \cup \underline{J}^- \cup \underline{J}^+$ forms a closed curve. We shall see what happens with the interior of $\tilde{R}$.

		Take any point $x\in \overset{o}{R}$. The set $\pi^{-1}(x)$ is not part of a cycle because the only $s$ and $u$ boundary points of $K(T_{g,f})$ are projected to corner points of the original Markov partition $h(\cR_{f,g})$. Therefore, $\pi^{-1}(x)$ is a rectangle and is trivially bi-foliated. We can deduce that if $\overset{o}{J}\subset \overset{o}{R}$ is an unstable segment, its preimage $\pi^{-1}(\overset{o}{J})$ is a rectangle without its horizontal boundary. Clearly, the horizontal boundary of $\pi^{-1}(\overset{o}{J})$ has one connected component intersecting $\underline{I}^-$ and the other intersecting $\underline{I}^+$. The same holds for any horizontal segment $\overset{o}{I}\subset \overset{o}{R}$.

		In conclusion, $\underline{I}^- \cup \underline{I}^{+} \cup \underline{J}^- \cup \underline{J}^+$ bounds a region that is trivially bi-foliated, and therefore, it is a rectangle denoted as $\underline{R}$. It's evident that its projection is $R$, and its interior projects to the interior of $R$.
		
		The stable boundary of $\cR$ consists of two disjoint intervals contained in the $s$ and $u$ boundary leaves, as we observed in Remark \ref{Rema: boundary points s-u}. By taking slightly larger intervals that contain $\underline{I}^-,\underline{I}^{+},\underline{J}^-$, and $\underline{J}^+$, we can construct a small neighborhood $U$ of $\underline{R}$ such that $K(T_{g,f})\cap U =K(T_{g,f}) \cap \underline{R}$.
	\end{proof}

	As a consequence, for all rectangles $h(R_i)\in h(\cR_{f,g})$, there is a unique rectangle $\underline{h(R_i)} \subset \textbf{DA}(T_{g,f})$ as described in Lemma \ref{Lemm: preimage pi rectangle} that projects to $h(R_i)$. Furthermore, these rectangles have their boundaries isolated of the basic piece from its exterior. In particular, we have the following corollary:
	
	\begin{coro}\label{Coro: Disjoint rectangles}
		If $i\neq j$, then $\underline{h(R_i)} \cap \underline{h(R_j)}=\emptyset$.
	\end{coro}

	Let
	$$
	\underline{h}(\cR_{f,g}):=\{\underline{h(R_i)}\}_{i=1}^n
	$$
	be the family of rectangles constructed in Lemma \ref{Lemm: preimage pi rectangle} for each rectangle in the Markov partition $h(\cR_{f,g})$.

	\begin{lemm}\label{Lemm: Union rectangle is K}
		$K(T_{g,f})\subset \cup_{i=1}^n \underline{h(R_i)}$
	\end{lemm}
	
	\begin{proof}
		
		The set of periodic points of $f$ that are not $(s,u)$-boundary is dense inside every rectangle $h(R_i)$ and is contained in its interior. Denote the set of these periodic points by $P(i)$. Then $\pi^{-1}(P(i))$ consists of all the periodic points of $\phi_{T_{g,f}}$ that are contained in the interior of $\pi^{-1}(h(R_i))$ and there fore  that are not $(s,u)$-boundary. Such points  are dense in $K(T_{g,f})\cap R_i$, and then:
		
		$$
		K(T_{g,f})\subset \cup_{i=1}^n \overline{\pi^{-1}(P(i))} \subset \cup_{i=1}^n \overline{\overset{o}{\underline{h(R_i)}}} = \cup_{i=1}^n \underline{h(R_i)}
		$$
		
	\end{proof}

	\begin{lemm}\label{Lemm; h(R -gf) is invariant}
		The stable boundary of the partition $\underline{h}(\cR_{f,g})$ is $\Phi_{g,f}$-invariant.
	\end{lemm}

	\begin{figure}[h]
		\centering
		\includegraphics[width=0.7\textwidth]{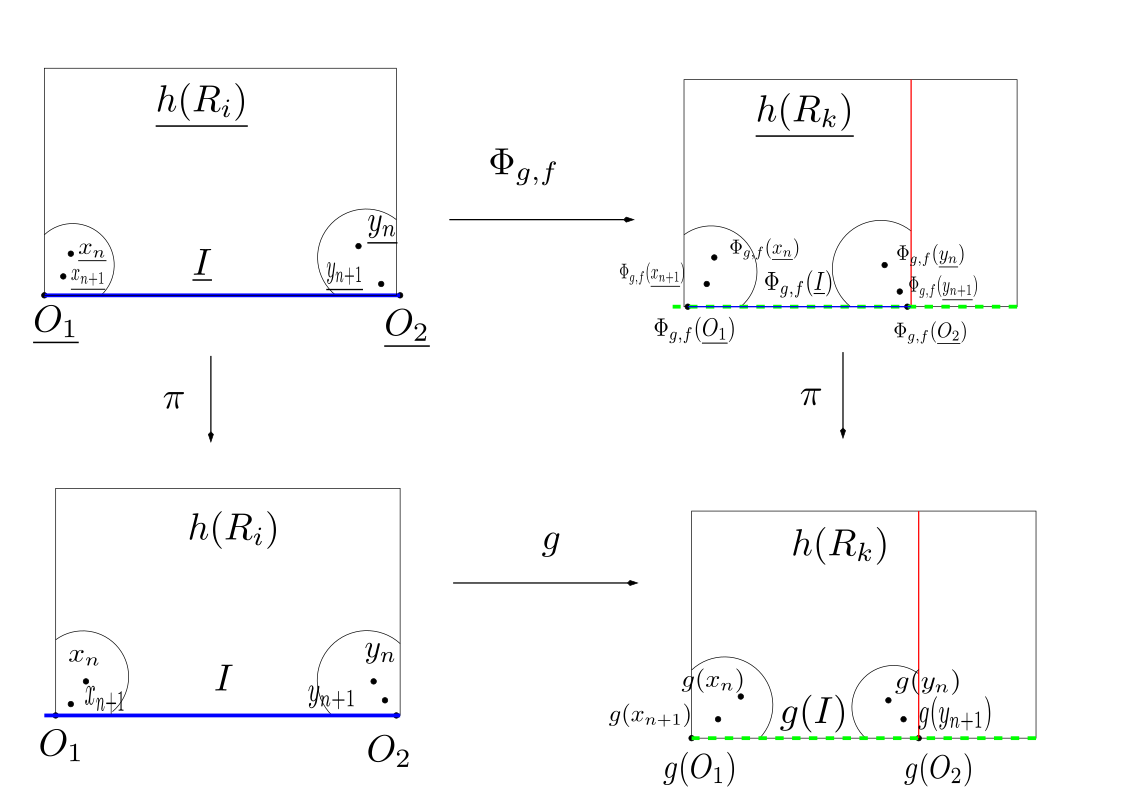}
		\caption{The stable boundary of the partition $\underline{h}(\cR_{f,g})$}
		\label{Fig: dete boundary}
	\end{figure}

	\begin{proof}
		
		Let $\underline{h(R_i)}\in \underline{h}(\cR_{f,g})$ be a rectangle. Take $\underline{I}$ as the lower boundary of such a rectangle, and let $\underline{O_1}$ and $\underline{O_2}$ its end points, the analysis for the upper boundary is the same. 
		
		There is a sequence ${\underline{x}_n}\subset \overset{o}{\underline{h(R_i)}}$ consisting of non $s,u$-boundary periodic points that converges to $\underline{O_1}$ as shown in Figure \ref{Fig: dete boundary}. Similarly, there exists a sequence $\{\underline{y}_n\} \subset \overset{o}{\underline{h(R_i)}}$ of non $s,u$-boundary periodic points that converges to $\underline{O_2}$. In these points, $\pi$ is a one-to-one map and serves as a conjugation restricted to their image. This implies that:
		
		$$
		\Phi_{g,f}(\underline{O_1})=\lim_{n\rightarrow \infty} \Phi_{g,f}(\underline{x}_n)= \lim_{n\rightarrow \infty}  \pi^{-1}\circ g \circ \pi (\underline{x_n}).
		$$
		and 
		
		$$
		\Phi_{g,f}(\underline{O_2})=\lim_{n\rightarrow \infty} \Phi_{g,f}(\underline{y}_n)= \lim_{n\rightarrow \infty}  \pi^{-1}\circ g \circ \pi (\underline{y_n}).
		$$

		By taking sub sequences of $\{\underline{x}_n\}$ and $\{\underline{y}_n\}$, we can assume that $x_n := \pi(\underline{x_n})$ and $y_n := \pi(\underline{y_n})$ are always contained in the interior of the rectangle $h(R_i).$ Furthermore, since $h(\cR_{f,g})$ is a Markov partition of $g$, we can assume that the lower horizontal sub-rectangle $H := H^i_1$ contained in $h(R_i)$ of the Markov partition $(g, \cR_{f,g})$ contains to $\{x_n\}$ and $\{y_n\}$ within its interior. Therefore, if $g(H) = V^k_l$, then $\{g(x_n)\}$ and $\{g(y_n)\}$ are contained in $\overset{o}{V^k_l} \subset \overset{o}{h(R_k)}$ and they converge to $g(O_1)$ and $g(O_2)$, which are in the stable boundary  of $\underline{h(R_k)}$.

		We claim that $\Phi_{g,f}(\underline{I})$ is a stable interval contained in a single stable boundary component of $\underline{h(R_k)}$.	 Suppose this is not the case, then $\Phi_{g,f}(\underline{I})$ contains an $s$-arc, $\underline{\alpha}'=[\underline{a}',\underline{b}']^s$  joining two corners of the rectangle $\underline{h(R_k)}$, where $\underline{a}'$ is on the lower boundary of  $\underline{h(R_i)}$ and $\underline{b}'$ is on the upper boundary of $\underline{h(R_i)}$ , as shown in Figure \ref{Fig: Same s bound}. This implies that $\underline{I}$ contains in its interior the $s$-arc
		 $\Phi_{g,f}^{-1}(\underline{\alpha}'):=\underline{\alpha}=[\underline{a},\underline{b}]^s$, that we recall it is contained in the inferior boundary of $\underline{h(R_i)}$
		 
		 	\begin{figure}[h]
		 	\centering
		 	\includegraphics[width=0.7\textwidth]{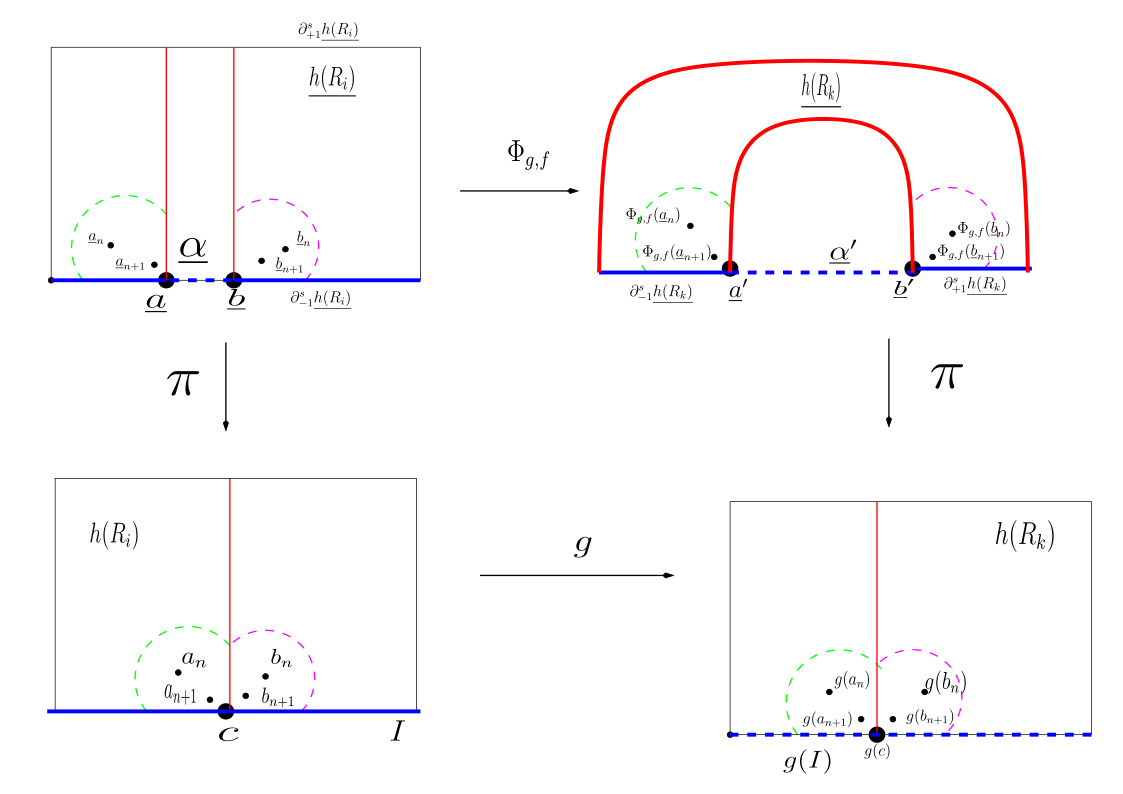}
		 	\caption{$\Phi_{g,f}(\underline{I})$  is in single stable boundary component}
		 	\label{Fig: Same s bound}
		 \end{figure}

		   Just as we did before for the endpoints of  $\underline{I}$ we can take decreasing successions $\{\underline{a}_n\},\{\underline{b}_n\}\subset  \overset{o}{\underline{h(R_i)}}$  of periodic non-corner points converging to $\underline{a}$ and $\underline{b}$ in two distinct sectors (the green and purple semicircles in Figure \ref{Fig: Same s bound}). Let us define,  $a_n:=\pi(\underline{a}_n)$ and $b_n:=\pi(\underline{b}_n)$. They are decreasing sequences (for the vertical order of $h(R_i)$) that converge to the same point  $c=\pi(\underline{a})0\pi(\underline{b})\in I$.  Clearly $\{g(a_n)\}$ is a decreasing (or increasing) sequence for the vertical order of $h(R_k)$ if and only if $\{g(b_n)\}$ is a decreasing (resp. increasing) sequence, and both converge to $g(c)$, without loss of generality suppose they are decreasing. Thus $\pi^{-1}(g(a_n))$ and $\pi^{-1}(g(b_n))$ are decreasing for the vertical order of $\underline{h(R_i)}$. Moreover 
		 $$
		\lim \pi^{-1}(g(a_n))= \lim \Phi_{g,f}(\underline{a_n})=\phi_{g,f}(\underline{a})= \underline{a}'
		$$
and 
			 $$
		\lim \pi^{-1}(g(b_n))= \lim \phi_{g,f}(\underline{b_n})=\phi_{g,f}(\underline{b})=\underline{b}'
		$$

		This implies that $\underline{a}'$ and $\underline{b}'$ are on the lower boundary of $\underline{h(R_i)}$ which is a contradiction.	 Then $\Phi_{g,f}(\underline{I})$ is contained in a single stable boundary component of $\underline{h(R_k)}$ and we conclude that the stable boundary of $\pi^{-1}(h(\cR_{f,g}))$ is $f$-invariant.

	\end{proof}
	
	The $\Phi^{-1}{g,f}$-invariance of $\partial^u \underline{h}(\cR_{f,g})$ is similarly proved. Consequently, $\underline{h}(\cR_{f,g})$ is a Markov partition of the basic piece $K(T_{g,f})$.

	\textbf{Item} $(4)$.  Considering the geometrization of $\underline{h}(\cR_{f,g})$  that was indicated in our last item, let $T'$ be its geometric type. It was proven in Proposition \ref{Prop: type of basic piece is type of pseudo-Anosov} that:
	
	$$		
	\pi(\underline{h}(\cR_{f,g}))=h(\cR_{f,g})
	$$		 
	
	is a geometric Markov partition of $g$ with geometric type $T'$. Since the geometric type of $h(\cR_{f,g})$ is $T_{f,g}$, we can conclude that $T = T_{f,g}$, which confirms the last statement of our proposition.
	
\end{proof}

\subsection{Equivalent geometric types: The Béguin's algorithm.}\label{Chapter: Algorithm}

Let $T_f$ and $T_g$ be two geometric types in the pseudo-Anosov class, and let $(f, \cR)$ and $(g, \cR_g)$ be realizations of such geometric types by generalized pseudo-Anosov homeomorphisms. Their joint refinement as introduced in \ref{Defin: Comun refinament} must be denoted as $T_{f, g}$ and $T_{g,f)}$. Their  formal derived form Anosov, denoted as \textbf{DA}$(T_{f, g})$ and \textbf{DA}$(T_{g,f})$, were introduced in \ref{Defi: Formal DA}. Two geometric types, $T_1$ and $T_2$ are  \emph{strongly equivalents} if \textbf{DA}$(T_1)$ has a Markov partition of geometric type $T_2$ and \textbf{DA}$(T_2)$ has a Markov partition of geometric type $T_1$. Proposition \ref{Prop: preimage is Markov Tfg type} yields to the following corollary.

\begin{coro}\label{Coro: equivalence pA and DA}
	Let $f$ and $g$ be generalized pseudo-Anosov homeomorphisms with geometric Markov partitions $\cR_f$ and $\cR_g$ of geometric types $T_f$ and $T_g$, respectively. Let $\cR_{f,g}$ be the joint refinement of $\cR_f$ with respect to $\cR_g$, and let $\cR_{g,f}$ be the joint refinement of $\cR_g$ with respect to $\cR_f$, whose geometric types are $T_{f,g}$ and $T_{g,f}$, respectively.  Under these hypotheses: $f$ and $g$ are topologically conjugated through an orientation preserving homeomorphism if and only if $T_{f,g}$ and $T_{g,f}$ are strongly equivalent.
\end{coro}

\begin{proof}

	Let's assume that $f$ and $g$ are topologically conjugated through an orientation-preserving homeomorphism. Proposition \ref{Prop: preimage is Markov Tfg type} establishes that, in this situation, the saddle-type basic piece $K_{f,g}$ in \textbf{DA}$(T_{f,g})$ has a geometric Markov partition with geometric type $T_{g,f}$; therefore, $T_{f,g}$ and $T_{g,f}$ are strongly equivalent.
	
	If $T_{f,g}$ and $T_{g,f}$ are strongly equivalent, then the saddle-type basic piece $K_{f,g}$ in \textbf{DA}$(T_{f,g})$ has a geometric Markov partition $\cR$ with geometric type $T_{g,f}$.  Let 
	
	$$\pi:\textbf{DA}((T_{f,g})\to S_f$$
	
	be the projection introduced in Theorem \ref{Theo: Basic piece projects to pseudo-Anosov}.  Proposition \ref{Prop: type of basic piece is type of pseudo-Anosov} states that $\pi(\cR)$ is a geometric Markov partition of $f$ with type $T_{g,f}$. Theorem \ref{Theo: conjugated iff  markov partition of same type} implies that $f$ and $g$ are topologically conjugated through an orientation-preserving homeomorphism.
	
\end{proof}

Part of the classification  of saddle-type basic pieces for surface Smale diffeomorphisms was carried out in 2004 by François Béguin in \cite{beguin2004smale}. In that work, he developed an algorithm that determines in finite time when two geometric types associated with non-trivial saddle-type basic pieces of surface Smale diffeomorphism are strongly equivalent.

\begin{theo}[ Béguin \cite{beguin2004smale}]\label{Theo: Beguin algorithm}
	Let $T_1$ and $T_2$ be geometric types realizable as basic pieces $K_1$ and $K_2$ of Smale surface diffeomorphism. There exists a finite algorithm that determines if  $T_1$ and $T_2$ are strongly equivalent.
\end{theo}

Now we can state a similar result when is applied to geometric types in the pseudo-Anosov class.

\begin{theo}\label{Theo: algorithm conjugacy class}
	Let $T_f$ and $T_g$ be two geometric types within the pseudo-Anosov class. Assume that $f: S \rightarrow S_f$ and $g: S_g \rightarrow S_g$ are two generalized pseudo-Anosov homeomorphisms with geometric Markov partitions $\cR_f$ and $\cR_g$, having geometric types $T_f$ and $T_g$ respectively.  We can compute the geometric types $T_{f,g}$ and $T_{g,f}$ of their joint refinements through the algorithmic process described in \cite[Chapter 5]{IntiII}, and the homeomorphisms $f$ and $g$ are topologically conjugated by an orientation-preserving homeomorphism if and only if the algorithm developed by Béguin determines that $T_{f,g}$ and $T_{g,f}$ are strongly equivalent.
\end{theo}

\begin{proof}
	
	In view of Corollary \ref{Coro: equivalence pA and DA},  the algorithm determines that $T_{f,g}$ and $T_{g,f}$ are strongly equivalent if and only if   $f$ and $g$ are conjugated. 
\end{proof}

%% file: pAClassification.bbl
\begin{thebibliography}{BdCH23}

\bibitem[BdCH23]{boyland2023unimodal}
Philip Boyland, Andr{\'e} de~Carvalho, and Toby Hall.
\newblock Unimodal measurable pseudo-anosov maps.
\newblock {\em arXiv preprint arXiv:2306.16059}, 2023.

\bibitem[B{\'e}g99]{beguin1999champs}
Fran{\c{c}}ois B{\'e}guin.
\newblock {\em Champs de vecteurs hyperboliques en dimension 3}.
\newblock PhD thesis, Dijon, 1999.

\bibitem[B{\'e}g02]{beguin2002classification}
Fran{\c{c}}ois B{\'e}guin.
\newblock Classification des diff{\'e}omorphismes de smale des surfaces: types
  g{\'e}om{\'e}triques r{\'e}alisables.
\newblock In {\em Annales de l'institut Fourier}, volume~52, pages 1135--1185,
  2002.

\bibitem[B{\'e}g04]{beguin2004smale}
Fran{\c{c}}ois B{\'e}guin.
\newblock Smale diffeomorphisms of surfaces: a classification algorithm.
\newblock {\em Discrete \& Continuous Dynamical Systems}, 11(2\&3):261, 2004.

\bibitem[BH95]{bestvina1995train}
Mladen Bestvina and Michael Handel.
\newblock Train-tracks for surface homeomorphisms.
\newblock {\em Topology}, 34(1):109--140, 1995.

\bibitem[BLJ98]{bonatti1998diffeomorphismes}
Christian Bonatti, R{\'e}mi Langevin, and Emmanuelle Jeandenans.
\newblock {\em Diff{\'e}omorphismes de Smale des surfaces}.
\newblock Soci{\'e}t{\'e} math{\'e}matique de France, 1998.

\bibitem[Bow75]{BowenAnosovbook}
Rufus Bowen.
\newblock {\em Equilibrium States and the Ergodic Theory of Anosov
  Diffeomorphisms}, volume 470 of {\em Lecture Notes in Mathematics}.
\newblock Springer, 1975.

\bibitem[CD]{IntiII}
Inti Cruz~Díaz.
\newblock Geometric markov partitions for pseudo-anosov homeomorphisms with
  fine combinatorial properties.
\newblock {\em In Preparation}.

\bibitem[CD23]{IntiThesis}
Inti Cruz~Diaz.
\newblock {\em An Algorithmic Classification of Generalized Pseudo-Anosov
  Homeomorphisms via Geometric Markov Partitions}.
\newblock PhD thesis, UBFC, 2023.
\newblock Available upon request.

\bibitem[dC05]{Andre2005extensions}
Andr{\'e} de~Carvalho.
\newblock Extensions, quotients and generalized pseudo-anosov maps.
\newblock {\em Graphs and patterns in mathematics and theoretical physics},
  73:315--338, 2005.

\bibitem[dCH04]{Hall2004unimodal}
Andr{\'e} de~Carvalho and Toby Hall.
\newblock Unimodal generalized pseudo-anosov maps.
\newblock {\em Geometry \& Topology}, 8(3):1127--1188, 2004.

\bibitem[FLP21]{fathi2021thurston}
Albert Fathi, Fran{\c{c}}ois Laudenbach, and Valentin Po{\'e}naru.
\newblock {\em Thurston's Work on Surfaces (MN-48)}, volume~48.
\newblock Princeton University Press, 2021.

\bibitem[FM11]{farb2011primer}
Benson Farb and Dan Margalit.
\newblock {\em A primer on mapping class groups (pms-49)}.
\newblock Princeton university press, 2011.

\bibitem[Hir87]{hiraide1987expansive}
Koichi Hiraide.
\newblock Expansive homeomorphisms of compact surfaces are pseudo-anosov.
\newblock {\em Proc. Japan Acad}, 63, 1987.

\bibitem[HK02]{hasselblatt2002handbook}
Boris Hasselblatt and Anatole Katok.
\newblock {\em Handbook of dynamical systems}.
\newblock Elsevier, 2002.

\bibitem[JMM96]{juvan1996systems}
Martin Juvan, Aleksander Malni{\v{c}}, and Bojan Mohar.
\newblock Systems of curves on surfaces.
\newblock {\em journal of combinatorial theory, Series B}, 68(1):7--22, 1996.

\bibitem[Kit98]{KitchensSymDym}
Bruce~P. Kitchens.
\newblock {\em Symbolic dynamics. One-sided, two-sided and countable state
  Markov shifts}.
\newblock Universitext. Springer, 1998.

\bibitem[Los93]{los1993pseudo}
J{\'e}r{\^o}me Los.
\newblock Pseudo-anosov maps and invariant train tracks in the disc: A finite
  algorithm.
\newblock {\em Proceedings of the London Mathematical Society}, 3(2):400--430,
  1993.

\bibitem[Mos83]{mosher1983pseudo}
Lee Mosher.
\newblock {\em Pseudo-Anosovs on punctured surfaces}.
\newblock PhD thesis, 1983.

\bibitem[Mos86]{mosher1986classification}
Lee Mosher.
\newblock The classification of pseudo-anosovs.
\newblock {\em Low-dimensional topology and Kleinian groups (Coventry/Durham,
  1984)}, 112:13--75, 1986.

\bibitem[Shu13]{shub2013global}
Michael Shub.
\newblock {\em Global stability of dynamical systems}.
\newblock Springer Science \& Business Media, 2013.

\bibitem[Thu88]{thurston1988geometry}
William~P Thurston.
\newblock On the geometry and dynamics of diffeomorphisms of surfaces.
\newblock {\em Bulletin of the American mathematical society}, 19(2):417--431,
  1988.

\end{thebibliography}
